\documentclass[preprint, 12pt]{imsart}
\usepackage[top=2.8cm, bottom=2.8cm, left=2.8cm, right=2.8cm]{geometry}

\usepackage{xr}
\makeatletter
\newcommand*{\addFileDependency}[1]{
  \typeout{(#1)}
  \@addtofilelist{#1}
  \IfFileExists{#1}{}{\typeout{No file #1.}}
}
\makeatother

\newcommand*{\myexternaldocument}[1]{%
    \externaldocument{#1}%
    \addFileDependency{#1.tex}%
    \addFileDependency{#1.aux}%
}
\myexternaldocument{PaperPiecewiseConstant_Final_supp}

\usepackage{amsthm,amsmath,amssymb}
\usepackage{parskip}
\usepackage[numbers]{natbib}
\usepackage{hypernat}
\usepackage{marvosym}
\usepackage[hang,small,bf]{caption}
\usepackage{mathtools}

\usepackage[colorlinks]{hyperref}
\hypersetup{
    colorlinks,%
    citecolor=brown,%
    filecolor=black,%
    linkcolor=blue,%
    urlcolor=black
}
\usepackage{hypernat}
\usepackage{enumerate}

\usepackage{color}

\usepackage{amsthm}
\usepackage{amsmath}
\usepackage{amssymb}
\usepackage{thmtools} 
\usepackage{xcolor}
\usepackage{subcaption}
\usepackage{textcomp}

\usepackage{tikz}
\usepackage{mathtools}
\mathtoolsset{showonlyrefs}

\newtheoremstyle{exampstyle}
{8pt} 
{8pt} 
{\it} 
{} 
{\bfseries} 
{.} 
{.5em} 
{} 

\theoremstyle{exampstyle}

\newtheorem{theorem}{Theorem}[section]
\newtheorem{proposition}[theorem]{Proposition}
\newtheorem{lemma}[theorem]{Lemma}

\newtheorem{remark}{Remark}[section]
\declaretheorem[name={Definition}, sibling=remark]{definition}

\startlocaldefs

\newcommand{\eat}[1]{}

\DeclarePairedDelimiter{\parens}{(}{)}
\DeclarePairedDelimiter{\braces}{\{}{\}}
\DeclarePairedDelimiter{\brackets}{[}{]}
\DeclarePairedDelimiter{\abs}{|}{|}
\DeclarePairedDelimiter{\norm}{\|}{\|}
\DeclarePairedDelimiter{\tvnorm}{\|}{\|_{\textup{TV}}}
\DeclarePairedDelimiter{\ceil}{\lceil}{\rceil}
\DeclarePairedDelimiter{\floor}{\lfloor}{\rfloor}
\DeclarePairedDelimiter{\inner}{\langle}{\rangle}

\DeclareMathOperator*{\argmin}{\arg\!\min}

\newcommand{\op}[1]{\operatorname{#1}}
\renewcommand{\bar}[1]{\overline{#1}}
\renewcommand{\hat}[1]{\widehat{#1}}
\renewcommand{\tilde}[1]{\widetilde{#1}}
\newcommand{\E}{\mathbb{E}}
\renewcommand{\P}{\mathbb{P}}

\newcommand{\R}{\mathbb{R}}

\newcommand{\defn}{\coloneqq}
\newcommand{\qt}[1]{\qquad\text{#1}}

\newcommand{\sign}{\op{sign}}
\newcommand{\mysign}{\mathfrak{s}}
\newcommand{\mysignalt}{\tilde{\mysign}}
\newcommand{\logplus}{\log_+}
\newcommand{\Hamming}{d_{\textup{H}}}
\newcommand{\KL}{D_{\textup{KL}}}
\newcommand{\phitilde}{\tilde{\phi}}

\renewcommand{\b}[1]{\boldsymbol{\mathbf{#1}}}
\renewcommand{\vec}[1]{\b{#1}}
\newcommand{\ivec}{\vec{i}}
\newcommand{\lvec}{\vec{\ell}}
\newcommand{\jvec}{\vec{j}}
\newcommand{\kvec}{\vec{k}}
\newcommand{\nvec}{\vec{n}}
\newcommand{\mvec}{\vec{m}}

\newcommand{\qvec}{\vec{q}}
\newcommand{\onevec}{\vec{1}}
\newcommand{\zerovec}{\vec{0}}
\newcommand{\avec}{\vec{a}}
\newcommand{\bvec}{\vec{b}}

\newcommand{\alphavec}{\vec{\alpha}}

\newcommand{\etavec}{\vec{\eta}}
\newcommand{\uvec}{\vec{u}}
\newcommand{\vvec}{\vec{v}}
\newcommand{\xvec}{\vec{x}}
\newcommand{\zvec}{\vec{z}}
\newcommand{\yvec}{\vec{y}}

\newcommand{\xivec}{\vec{\xi}}

\newcommand{\evec}{\vec{e}}
\newcommand{\noisevec}{\vec{\noise}}

\newcommand{\noise}{\xi}
\newcommand{\numobs}{n}
\newcommand{\numrect}{R}
\newcommand{\usedim}{d}
\newcommand{\altdim}{p}
\newcommand{\noisestd}{\sigma}
\newcommand{\Normal}{\mathcal{N}}
\newcommand{\coefplain}{\beta}
\newcommand{\coef}{\b{\coefplain}}
\newcommand{\coefstar}{\coef^*}
\newcommand{\coefplainstar}{\coefplain^*}
\newcommand{\coefhat}{\b{\hat{\beta}}}
\newcommand{\coeftilde}{\b{\tilde{\beta}}}
\newcommand{\coefplaintilde}{\tilde{\coefplain}}

\newcommand{\paramplain}{\theta}
\newcommand{\param}{\b{\paramplain}}
\newcommand{\paramhat}{\b{\hat{\param}}}
\newcommand{\paramstar}{\b{\param}^*}
\newcommand{\paramtilde}{\b{\tilde{\param}}}

\newcommand{\paramplaintilde}{\tilde{\paramplain}}
\newcommand{\paramplainbar}{\bar{\paramplain}}
\newcommand{\paramplainstar}{\paramplain^*}

\newcommand{\LASSOfit}{\paramhat_{\textup{LASSO}}}

\newcommand{\EMfitfun}{\fhat_{\textup{EM}}}
\newcommand{\EMfitfuntwo}{\ftilde_{\textup{EM}, \LASSOrad}}

\newcommand{\HKfitfun}{\fhat_{\textup{HK}\zerovec, \LASSOrad}}

\newcommand{\GFun}{G}
\newcommand{\GWidth}{w}

\newcommand{\LatticeDesign}{\mathbb{L}_{\numobs_1, \ldots, \numobs_\usedim}}

\newcommand{\QVolume}{\Delta}

\newcommand{\URevalplain}{\vvec}
\newcommand{\UReval}[1]{\URevalplain(#1)}

\newcommand{\ConvexSet}{\mathcal{K}}

\newcommand{\fhat}{\hat{f}}
\newcommand{\fstar}{f^*}
\newcommand{\ftilde}{\tilde{f}}
\newcommand{\gtilde}{\tilde{g}}
\newcommand{\Loss}{\mathcal{L}}
\newcommand{\MinimaxRisk}{\mathfrak{M}_{\noisestd, \LASSOrad, \usedim}(\numobs)}
\newcommand{\MinimaxRiskEM}{\mathfrak{M}_{\textup{EM}, \noisestd, \LASSOrad, \usedim}(\numobs)}

\newcommand{\MinimaxRiskHK}{\mathfrak{M}_{\textup{HK}, \noisestd, \LASSOrad, \usedim}(\numobs)}
\newcommand{\Risk}{\mathcal{R}}

\newcommand{\coefhatcm}{\coefhat_{\textup{EM}}}
\newcommand{\coefhatcmplain}{\hat{\coefplain}_{\textup{EM}}}
\newcommand{\coefhathk}{\coefhat_{\textup{HK}\zerovec, \LASSOrad}}
\newcommand{\coefhathkplain}{\hat{\coefplain}_{\textup{HK}\zerovec, \LASSOrad}}

\newcommand{\LASSOrad}{V}
\newcommand{\HKVar}{V_{\textup{HK}\zerovec}}
\newcommand{\VVar}[1]{V^{(#1)}}

\newcommand{\Ind}{\mathbb{I}}
\newcommand{\IndUpperRight}[1]{\Ind_{[#1, \onevec]}}
\newcommand{\IndSet}{\mathcal{Q}}

\newcommand{\NNLSSetPlain}{\mathcal{D}}
\newcommand{\NNLSSet}{\NNLSSetPlain_{\numobs_1, \ldots, \numobs_\usedim}}

\newcommand{\IndexSet}{\mathcal{I}}

\newcommand{\KIndexSet}{\mathcal{K}}

\newcommand{\Ball}{\mathcal{B}}
\newcommand{\Probspace}{\Omega}
\newcommand{\Probspacesm}{\Probspace_0}
\newcommand{\dirac}{\delta}
\newcommand{\Proj}{\Pi}

\newcommand{\LASSOBall}{\mathcal{C}}
\newcommand{\TCone}{\mathcal{T}}

\newcommand{\TConeLASSO}{\TCone_{\LASSOBall(\LASSOrad)}}
\newcommand{\TConeSub}{T}

\newcommand{\OneJumpClass}{\tilde{\rpcplain}^2_1}
\newcommand{\OneJumpClassStrong}{\rpc_1}
\newcommand{\FClass}{\mathcal{F}}
\newcommand{\EMClassPlain}{\FClass_{\textup{EM}}}
\newcommand{\EMClass}{\EMClassPlain^\usedim}

\newcommand{\DFClass}{\FClass_{\textup{DF}}^\usedim}
\newcommand{\DFClassTwo}{\FClass_{\textup{DF}}^2}
\newcommand{\MClassPlain}{\FClass_{\textup{M}}}
\newcommand{\MClass}{\MClassPlain^\usedim}

\newcommand{\TVrad}{\bigrad}

\newcommand{\rect}{R}
\newcommand{\Partition}{\mathcal{P}}
\newcommand{\SubRect}{Q}
\newcommand{\LL}{L}
\newcommand{\LLup}{\LL_u}
\newcommand{\LLlw}{\LL_\ell}

\newcommand{\MIndexSet}{\mathcal{M}}
\newcommand{\IIndexSet}{\mathcal{I}}

\newcommand{\Imat}{\b{I}}
\newcommand{\Altdesignmat}{\b{A}}

\newcommand{\rad}{t}
\newcommand{\radtilde}{\tilde{\rad}}

\newcommand{\radstar}{\rad_*}

\newcommand{\radcover}{\epsilon}
\newcommand{\smallrad}{r}
\newcommand{\bigrad}{R}
\newcommand{\covernum}{N}
\newcommand{\packnum}{M}

\newcommand{\sparrect}{k}

\newcommand{\pospart}{\b{\pi}}
\newcommand{\negpart}{\b{\nu}}

\newcommand{\rpc}{\rpcplain^\usedim}
\newcommand{\rpcplain}{\mathfrak{R}}
\newcommand{\Diff}{D}
\newcommand{\edgecoef}{t}

\endlocaldefs

\begin{document}

\begin{frontmatter}
\title{Multivariate extensions of isotonic regression and total variation denoising via entire monotonicity and Hardy-Krause variation}
\runtitle{Entire monotonicity and Hardy-Krause variation}

\begin{aug}
\author{\fnms{Billy} \snm{Fang}\thanksref{t1}, \ead[label=e1]{blfang@berkeley.edu}}
\author{\fnms{Adityanand} \snm{Guntuboyina}\thanksref{t2},  \ead[label=e2]{aditya@stat.berkeley.edu}}
\and
\author{\fnms{Bodhisattva} \snm{Sen}\thanksref{t3}  \ead[label=e3]{bodhi@stat.columbia.edu}}
\affiliation{
University of California, Berkeley\thanksmark{a1} and
Columbia University\thanksmark{a2}
}

\thankstext{t1}{Supported by a NSF Graduate Research Fellowship}
\thankstext{t2}{Supported by NSF CAREER Grant DMS-1654589}
\thankstext{t3}{Supported by NSF Grants DMS-1712822 and AST-1614743}

\runauthor{Fang, Guntuboyina, and Sen}

\address{393 Evans Hall\\
Berkeley, CA 94720 \\
\printead{e1}\\
\phantom{E-mail:\ }}

\address{423 Evans Hall\\
Berkeley, CA 94720 \\
\printead{e2}\\
\phantom{E-mail:\ }}

\address{1255 Amsterdam Avenue \\
New York, NY 10027\\
\printead{e3}
}
\end{aug}

\begin{abstract}
We consider the problem of nonparametric regression when the covariate
is $\usedim$-dimensional, where $\usedim \ge 1$. In this paper we
introduce and study two nonparametric least squares estimators (LSEs)
in this setting --- the entirely monotonic LSE and the constrained
Hardy-Krause variation LSE. We show that these two LSEs are natural
generalizations of univariate isotonic regression and univariate total
variation denoising, respectively, to multiple dimensions. We discuss
the characterization and computation of these two LSEs obtained from
$n$ data points. We provide a detailed study of their risk properties
under the squared error loss and fixed uniform lattice design. We show
that the finite sample risk of these LSEs is always bounded from above
by $\numobs^{-2/3}$ modulo logarithmic factors depending on
$\usedim$; thus these nonparametric LSEs avoid the curse of
dimensionality to some extent. We also prove nearly matching minimax
lower bounds. Further, we illustrate that these LSEs
are  particularly useful in fitting rectangular piecewise constant
functions. Specifically, we show that   the risk of the entirely
monotonic LSE is almost parametric (at most $1/n$ up to logarithmic
factors) when the true function is well-approximable by a rectangular
piecewise constant entirely monotone function with not too many
constant pieces. A similar result is also shown to hold for the
constrained Hardy-Krause variation LSE for a simple subclass of
rectangular piecewise constant functions. We believe that the proposed
LSEs yield a novel approach to estimating multivariate functions using
convex optimization that avoid the  curse of dimensionality to some
extent.
\end{abstract}



\end{frontmatter}

\section{Introduction}\label{section:introduction}
Consider the problem of nonparametric regression where the goal is to
estimate an unknown regression function $\fstar : [0, 1]^\usedim \to
\R$ ($d \ge 1$) from noisy observations at fixed design points
$\xvec_1, \ldots, \xvec_\numobs \in [0, 1]^\usedim$.
Specifically, we observe responses $y_1, \ldots, y_n$ drawn according to the model
\begin{equation}\label{obmod}
\qquad y_i = \fstar(\xvec_i) + \noise_i,
\qt{where $\;\;\noise_i \overset{\text{i.i.d.}}{\sim} \Normal(0, \noisestd^2) \quad$
for $i = 1, \ldots, \numobs$},
\end{equation}
$\sigma^2 > 0$ is unknown, and the purpose is to  nonparametrically
estimate $\fstar$ known to belong to a prespecified function class. In
the univariate ($\usedim = 1$) case, two such important function classes
are: (i) the class of \textit{monotone nondecreasing} functions in
which case $\fstar$ is usually estimated by
the isotonic least squares estimator (LSE) (see e.g.,~\citet{RWD88},
\citet{groeneboom2014nonparametric},
\citet{BBBB72}, \citet{Brunk55}, \citet{AyerEtAl55}); and (ii) the
class of functions whose \textit{total variation} is bounded by a specific
constant in which case it is natural to estimate $f^*$ by total
variation denoising (see
e.g.,~\citet{rudin1992nonlinear},~\citet{mammen1997locally},
\citet{chambolle2010introduction}, \citet{condat2013direct}).
Both these estimators --- isotonic regression and
total variation denoising --- have a long
history and are very well-studied. For
example, it is known that both these estimators produce piecewise
constant fits and have  finite sample risk (under the squared error
loss) bounded from above by a constant multiple of $n^{-2/3}$ (see e.g.,~\citet{MW00},~\citet{Zhang02},~\citet{mammen1997locally}). Moreover, it is well-known that both
these estimators are especially useful in fitting piecewise constant
functions where their risk is almost parametric (at most $1/n$ up to
logarithmic factors); see e.g.,~\citet{guntuboyina2017nonparametric},
\citet{dalalyan2017tvd}, and~\citet{guntuboyina2017spatial} and the references therein.

In this paper, we try to answer the following question: ``What is a
natural generalization of univariate isotonic regression and
univariate total variation denoising to multiple dimensions?'' To
answer this question we introduce and study two (constrained) LSEs for
estimating $f^*: [0, 1]^\usedim \to \R$ where $d \ge 1$. We show that both
these LSEs yield rectangular piecewise constant fits and have finite
sample risk that is bounded from above by $n^{-2/3}$ (modulo
logarithmic factors depending on $d$), thereby avoiding the
curse of dimensionality to some extent. Further, we study the characterization and
computation of these two estimators: the LSEs are obtained as
solutions to convex optimization problems --- in fact, quadratic
programs with linear constraints --- and are thus easily
computable. Moreover, as in the case $\usedim=1$, we illustrate that these
LSEs are particularly useful in fitting rectangular piecewise constant
functions and can have  almost parametric risk (up to logarithmic
factors). These results are directly analogous to the univariate
results mentioned in the previous paragraph and thus justify our
claim that our proposed estimators are natural multivariate
generalizations of univariate isotonic regression and univariate total
variation denoising.

Our first estimator is the LSE over $\EMClass$, the class of
\emph{entirely monotone} functions on $[0, 1]^\usedim$:
\begin{equation}\label{equation:cmfitfun}
\EMfitfun \in \argmin_{f \in \EMClass}
\frac{1}{\numobs} \sum_{i=1}^\numobs (y_i - f(\xvec_i))^2.
\end{equation}
The class $\EMClass$ of entirely monotone functions is formally
defined in \autoref{SECTION:CM_HK}. Entire monotonicity is an
existing generalization in multivariate analysis of the univariate
notion of monotonicity (see e.g., \cite{aistleitner2014functions,
  leonov1996total, young1924discontinuities,
  hobson1950theory}). Indeed,
in the univariate case when $\usedim = 1$, the class
$\EMClassPlain^1$ is precisely the class of nondecreasing functions on
$[0, 1]$ and thus, for $\usedim = 1$, the estimator
\eqref{equation:cmfitfun} reduces to the usual isotonic LSE. For $\usedim = 2$, the class
$\EMClassPlain^2$ consists of all functions $f: [0, 1]^2 \rightarrow
\R$ which satisfy both $f(a_1, a_2) \leq f(b_1, b_2)$ and
\begin{equation}\label{sp2}
f(b_1, b_2) - f(a_1, b_2) - f(b_1, a_2) + f(a_1, a_2) \ge 0,
\end{equation}
for every $0 \le a_1 \le b_1 \le 1$ and $0 \le a_2 \le b_2 \le 1$.
The formal definition of $\EMClass$ for general $\usedim \geq 1$ is given
in~\autoref{SECTION:CM_HK}.
We remark that in general, entire monotonicity
is different from the usual notion of monotonicity in
classical multivariate isotonic regression \cite{RWD88};
see \autoref{lemma:cm_m} for a connection between these two notions.
We also remark that $\EMClass$ is closed under translation and
nonnegative scaling; that is, if $f \in \EMClass$,
then $af+b \in \EMClass$ for any $a \ge 0$ and $b \in \R$.
Additionally, the collection of right-continuous functions in
$\EMClass$ is precisely the collection of cumulative distribution functions
of nonnegative measures on $[0, 1]^\usedim$
(see  \autoref{lemma:em_right_continuous}).

Our terminology of entire monotonicity is taken from
\citet{young1924discontinuities}. As a word of caution, we note that
some authors (e.g., \citet{aistleitner2014functions}) use the term
``completely monotone'' in place of ``entirely monotone.'' We use the
latter terminology because ``completely monotone'' has been used in
the literature for
other notions (see e.g.,~\cite{widder1941princeton, gao2010many,
  feller2015completely}) which are unrelated to our definition of
entire monotonicity. Entire  monotonicity has also been referred by
other names in the literature (for example, it has been referred to as
``quasi-monotone'' in \citet{hobson1950theory}).

The second main estimator that we study in this paper involves
$\HKVar(\cdot)$, the \emph{variation
  in the sense of Hardy and Krause (anchored at $\zerovec$)}, which we
shorten to \emph{Hardy-Krause variation} or \emph{HK$\zerovec$
  variation}. The HK$\zerovec$ variation of
a univariate function $f : [0, 1] \to \R$ is simply the  total
variation of the function, i.e.,
\begin{equation}\label{equation:TV_def}
\HKVar(f)
= \sup_{0 = x_0 < x_1 < \dots < x_k = 1} \sum_{i = 0}^{k-1} \abs{f(x_{i+1})
  - f(x_{i})},
\end{equation}
where the supremum is over all $k \ge 1$ and all partitions $0 = x_0 < x_1 < \dots < x_k
= 1$ of $[0, 1]$. Thus HK$\zerovec$ variation is a generalization
of one-dimensional total variation to multiple dimensions. For $\usedim =
2$, HK$\zerovec$ variation is defined in the following way: for $f:[0,1]^2 \to \R$,
\begin{equation}\label{hk2}
\begin{split}
  \HKVar(f) &\defn \HKVar(x \mapsto f(x, 0)) + \HKVar(x \mapsto f(0, x))
  \\
&\phantom{{}\defn} \quad
+ \sup \sum_{0 \leq l_1 < k_1, 0 \leq l_2 < k_2}
\left|
    f(x_{l_1 + 1}^{(1)}, x_{l_2+1}^{(2)}) - f(x_{l_1}^{(1)}, x_{l_2+1}^{(2)})
\right.
\\
&\phantom{{} \defn \quad + \sup \sum_{0 \leq l_1 < k_1, 0 \leq l_2 < k_2}{}}
\quad \left.
  - f(x_{l_1 + 1}^{(1)}, x_{l_2}^{(2)}) + f(x_{l_1}^{(1)}, x_{l_2}^{(2)})
\right|
\end{split}
\end{equation}
where the first two terms in the right hand side above are defined via
the univariate definition \eqref{equation:TV_def} and the supremum in
the third term above is over all pairs of partitions $0 = x_0^{(1)} <
x_1^{(1)} < \dots < x_{k_1}^{(1)} = 1$ and $0 = x_0^{(2)} <
x_1^{(2)} < \dots < x_{k_2}^{(2)} = 1$ of $[0, 1]$.  Note that a
special role is played in the first two terms of the right hand side
of \eqref{hk2} by the point $(0, 0)$ and this is the reason for the
phrase ``anchored at $\zerovec$''. For smooth functions $f: [0, 1]^2
\rightarrow \R$, it can be shown that
\begin{equation*}
  \HKVar(f) = \int_0^1 \int_0^1 \left|\frac{\partial^2 f}{\partial x_1
    \partial x_2} \right| dx_1 dx_2 + \int_0^1 \left|\frac{\partial
    f(\cdot, 0)}{\partial x_1} \right| dx_1 + \int_0^1 \left|\frac{\partial
    f(0, \cdot)}{\partial x_2} \right| dx_2
\end{equation*}
and, from the first term in the right hand side above, it is clear
that the HK$\zerovec$ variation is related to the $L^1$ norm of the
mixed derivative. The definition of HK$\zerovec$ variation for general
$d \geq 1$ is given in~\autoref{SECTION:CM_HK}. HK$\zerovec$ variation
is quite different from the usual definition of multivariate total
variation (see e.g., \citet[Chapter 5]{ziemer2012weakly}) as explained
briefly in \autoref{SECTION:CM_HK}.

Functions that are piecewise constant on axis-aligned rectangular
pieces (see
\autoref{definition:peecee}) have finite HK$\zerovec$ variation as
explained in \autoref{SECTION:CM_HK}.
More generally, the collection of right-continuous functions of finite HK$\zerovec$
variation is precisely the same as the collection of cumulative distribution
functions of finite signed measures (see \autoref{lemma:hk_right_continuous}).
An example of a function with infinite HK$\zerovec$ variation
is the indicator function of an open $\usedim$-dimensional ball
contained in $[0, 1]^\usedim$ (see \cite[Sec. 12]{owen2005multidimensional}).


Our second estimator is the constrained LSE over functions with
HK$\zerovec$  variation bounded by some tuning parameter $\LASSOrad >
0$:
\begin{equation}\label{equation:hkfitfun}
\HKfitfun \in \argmin_{f : \HKVar(f) \le \LASSOrad}
\frac{1}{\numobs} \sum_{i=1}^\numobs (y_i - f(\xvec_i))^2.
\end{equation}
This estimator is a generalization of total variation denoising to $\usedim \geq
2$ because in the case $\usedim = 1$, HK$\zerovec$ variation coincides with
total variation and, thus, the above estimator performs univariate total
variation denoising, sometimes also called trend filtering of first
order \cite{rudin1992nonlinear, mammen1997locally,
  chambolle2010introduction, condat2013direct, kim2009ell_1,
  tibshirani2014adaptive}. This generalization is different from the
usual multivariate total variation denoising as in
\citet{rudin1992nonlinear} (see \autoref{section:discussion} for more
discussion on how $\HKfitfun$ is different from the multivariate total
variation regularized estimator). It is also possible to define the
HK$\zerovec$ variation estimator in the following penalized form:
\begin{equation}\label{equation:hkfitfun_pen}
\fhat_{\textup{HK}\zerovec, \lambda} \in \argmin_{f}
\frac{1}{\numobs} \left\{ \sum_{i=1}^\numobs (y_i - f(\xvec_i))^2 +
  \lambda \HKVar(f) \right\}
\end{equation}
for a tuning parameter $\lambda >0$. In this paper, we shall focus on the
constrained form in \eqref{equation:hkfitfun} although analogues of
our results for the penalized estimator \eqref{equation:hkfitfun_pen}
can also be proved.

Before proceeding further, let us note that entire monotonicity is
related to HK$\zerovec$ variation in much the same way as univariate
monotonicity is related to univariate total variation. Indeed, for
functions in one variable, the following two properties are
well-known:
\begin{enumerate}
\item Every function $f: [0, 1] \rightarrow \R$ of bounded variation can be written as the difference of two monotone functions $f = f_+ - f_-$ and the total variation of $f$ equals the sum of the variations of $f_+$ and $f_-$.
\item If $f: [0, 1] \rightarrow \R$ is nondecreasing, then its total variation on $[0, 1]$ is simply $f(1) - f(0)$.
\end{enumerate}
These two facts generalize almost verbatim to entire monotonicity and
HK$\zerovec$ variation (see \autoref{lemma:hkvar_properties}). Thus,
in some sense, entire monotonicity is to Hardy-Krause variation as
monotonicity is to total variation.

Although the terminology of ``entire monotonicity'' does not seem to
have been used previously in the statistics literature, entirely
monotone functions are closely related to cumulative distribution functions
of nonnegative measures which appear routinely in
statistics. HK$\zerovec$ variation has appeared previously in
statistics in the
literature on quasi-Monte Carlo (see e.g.,
\cite{owen2005multidimensional,   guo2006quasi}) as well as in the
power analysis of certain sequential detection problems (see e.g.,
\cite{prause2017sequential}). Additionally \citet{benkeser2016highly} (see also
\cite{MR3724476,van2017finite,van2017uniform,van2019efficient})
considered the class $\{f : \HKVar(f) \le \LASSOrad\}$ in their ``highly
adaptive LASSO'' estimator and exploited its connections to the LASSO in a
setting that is different from our classical nonparametric regression
framework. They also used the terminology of  ``sectional variation
norm'' to refer to the Hardy-Krause variation (see also \cite[Section
2]{gill1995inefficient}).  An estimator very similar to
\eqref{equation:hkfitfun}  was proposed by \citet{mammen1997locally}
for $d = 2$ when the design points take values in a uniformly spaced
grid (this estimator of \cite{mammen1997locally} is described in
\autoref{section:esld}). Also, \citet{lin2000tensor} proposed an
estimator in the context of the Gaussian white noise model that bears
some similarities to \eqref{equation:hkfitfun} (this connection is
detailed in \autoref{section:discussion}).

%


The goal of this paper is to analyze the properties of the estimators
\eqref{equation:cmfitfun} and~\eqref{equation:hkfitfun}. Here is a
description of our main results. \autoref{SECTION:COMPUTATION}
concerns the computation of these estimators. Note that, as stated,
the optimization problems defining our estimators~\eqref{equation:cmfitfun}
and~\eqref{equation:hkfitfun} are convex
(albeit infinite-dimensional). We show that, given arbitrary data
$(\xvec_1, y_1), \dots, (\xvec_n, y_n)$, the two
estimators~\eqref{equation:cmfitfun}  and~\eqref{equation:hkfitfun}
can be computed by solving a nonnegative least squares (NNLS) problem and a
LASSO problem respectively, with a suitable design matrix that
only depends on the design-points $\xvec_1, \dots, \xvec_n$. It is
interesting to note that the design matrices in the two
finite-dimensional problems for computing~\eqref{equation:cmfitfun}
and~\eqref{equation:hkfitfun} are exactly the same. Our main results
in this section (\autoref{proposition:cm_nnls} and
\autoref{proposition:hk_lasso}) imply that $\EMfitfun$ and $\HKfitfun$
can be taken to be of the form
\begin{equation}\label{introexp}
  \EMfitfun = \sum_{j=1}^\altdim (\coefhatcmplain)_j \cdot
\Ind_{[\zvec_j, \onevec]} ~~ \text{ and } ~~ \HKfitfun = \sum_{j=1}^\altdim
(\coefhathkplain)_j \cdot \Ind_{[\zvec_j, \onevec]}
\end{equation}
for some $\zvec_1, \dots, \zvec_p$ that only depend on the design
points $\xvec_1, \dots, \xvec_n$ and vectors $\coefhatcm$ and
$\coefhathk$ in $\R^{\altdim}$ which are obtained by solving the NNLS
problem \eqref{equation:cm_nnls} and the LASSO problem
\eqref{equation:hk_lasso} respectively. Here $\Ind_{[\zvec_j,
  \onevec]}$ denotes the indicator of the rectangle $[\zvec_j,
\onevec]$ (defined via \eqref{clorec}). Because NNLS and LASSO
typically lead to sparse solutions, the vectors $\coefhatcm$ and
$\coefhathk$ will be sparse which clearly implies that $\EMfitfun$
and $\HKfitfun$ as given above \eqref{introexp} will be piecewise constant
on axis-aligned rectangles. Therefore our estimators give rectangular
piecewise constant fits to data and this generalizes the fact that
univariate isotonic regression and total variation denoising yield
piecewise constant fits. In the case when the design points $\xvec_1,
\dots, \xvec_n$ form an equally spaced lattice in $[0, 1]^d$ (see the definition
\eqref{equation:lattice_design} for the precise formulation of this
assumption), the points $\zvec_1, \dots, \zvec_p$ can simply be taken
to be $\xvec_1, \dots, \xvec_n$  and, in this case, more explicit
expressions can be given for the estimators (see
\autoref{section:esld} for details). It should be noted that the
lattice design is quite commonly used for theoretical studies in
multidimensional nonparametric function estimation (see e.g.,
\cite{nemirovski2000}) especially in connection with image analysis
(see e.g., \cite{chambolle2010introduction, condat2017discrete}).

We also investigate the accuracy properties of $\EMfitfun$ and
$\HKfitfun$ via the study of their risk behavior under the standard
fixed design squared error loss
function. Specifically, we define the risk of an estimator $\fhat$ by
\begin{equation}\label{rislo}
\Risk(\fhat, \fstar) \defn \E \Loss(\fhat, \fstar) \qt{where $\quad \Loss(\fhat, \fstar)
\defn \frac{1}{\numobs} \sum_{i=1}^\numobs (\fhat(\xvec_i) - \fstar(\xvec_i))^2$}.
\end{equation}
We prove results on the risk of $\EMfitfun$ and $\HKfitfun$ in the case of the aforementioned lattice design. In this setting, our main results are described below.

We analyze the risk of $\EMfitfun$ under the (well-specified)
assumption that $f^* \in \EMClass$. We prove in \autoref{theorem:NNLS_worst_case} that, for $\numobs \ge 1$,
\begin{equation}
  \label{cmwcintro}
  \Risk(\EMfitfun, \fstar) \leq \frac{C(d, \sigma, V^*)}{\numobs^{2/3}} (\log
    (e \numobs))^{\frac{2 \usedim -1}{3}}
\end{equation}
where
\begin{equation}\label{vstaintro}
  V^* := \fstar(1, \dots, 1) - \fstar(0, \dots, 0)
\end{equation}
and $C(\usedim, \sigma, V^*)$ depends only on $d$, $\sigma$ and
$V^*$ (see statement of \autoref{theorem:NNLS_worst_case} for the
explicit form of $C(\usedim, \sigma, V^*)$). Note that the dimension
$d$ appears in \eqref{cmwcintro} only through the logarithmic term
which means that we obtain ``dimension independent rates'' ignoring
logarithmic factors. Some intuition for why the constraint of entire
monotononicity is able to mitigate the usual curse of dimensionality
is provided in \autoref{section:discussion}. Other nonparametric
estimators exhibiting such dimension independent rates can be found in
\cite{barron1993universal,   lin2000tensor, chkifa2018polynomial,
  niyogi1999generalization, sadhanala2017higher,
  wahba1995smoothing}. In \autoref{theorem:em_minimax}, we prove
a minimax lower bound which implies that the dependence on $d$ through
the logarithmic term in \eqref{cmwcintro} cannot be avoided for any
estimator.


We also prove in \autoref{theorem:NNLS_adapt} that $\Risk(\EMfitfun,
\fstar)$ is smaller than the bound given by \eqref{cmwcintro} when $f^* \in
\EMClass$ is \textit{rectangular piecewise
  constant}. Loosely speaking, we say that $f : [0, 1]^d \rightarrow \R$ is
rectangular piecewise constant if it is constant on each set in a
partition of $[0, 1]^d$ into axis-aligned rectangles and the smallest cardinality
of such a partition shall be denoted by $k(f)$ (see
\autoref{definition:peecee} for the precise definitions). In
\autoref{theorem:NNLS_adapt}, we prove that
whenever $f^* \in \EMClass$ is rectangular piecewise constant, we have
\begin{equation}\label{cmadintro}
  \Risk(\EMfitfun, \fstar) \leq C_\usedim \sigma^2 \frac{k(f^*)}{n}
  (\log(e \numobs))^{\frac{3 \usedim }{2}}
  (\log ( e \log ( e \numobs)))^{\frac{2 \usedim -1}{2}}
\end{equation}
for a positive constant $C_d$ which only depends on $d$. Note that
when $k(f^*)$  is not too large, the right hand side of
\eqref{cmadintro} converges to zero as $n \rightarrow \infty$ at a
faster rate compared to the right hand side of \eqref{cmwcintro}. Thus
rectangular piecewise constant functions which also satisfy the
constraint of entire monotonicity are estimated at nearly the
parametric rate (ignoring the logarithmic factor) by the LSE $\EMfitfun$.


Let us now describe our results for the other estimator $\HKfitfun$. In
\autoref{theorem:lasso_worst_case} we prove that when $\HKVar(\fstar)
\leq V$ (note that $V$ is the tuning parameter in the definition of
$\HKfitfun$), then
\begin{equation}\label{hkwcintro}
  \Risk(\HKfitfun, \fstar) \leq \frac{C(d, \sigma, V)}{n^{2/3}}
  (\log(en))^{\frac{2d-1}{3}}.
\end{equation}
Note that the right sides of the bounds \eqref{hkwcintro} and
\eqref{cmwcintro}  are the same and thus the estimator $\HKfitfun$
also achieves dimension independent rates (ignoring logarithmic
factors) (see \autoref{section:discussion} for an explanation of this
phenomenon). We also prove a minimax lower bound in
\autoref{THEOREM:MINIMAX} which implies that the dependence on $d$ in
the logarithmic term in \eqref{hkwcintro} cannot be completely
removed for any estimator.



In univariate total variation denoising, it is known that one
obtains faster rates than given by the bound~\eqref{hkwcintro} when $f^*: [0, 1]
\rightarrow \R$ is piecewise constant with not too many pieces. Indeed
if $\fstar$ is piecewise constant for $\usedim = 1$ with $k(\fstar)$ pieces, then
it has been proved that
 \begin{equation}\label{hkad1}
  \Risk(\HKfitfun, \fstar) \leq C(c) \sigma^2 \frac{k(\fstar)}{\numobs} \log(e \numobs)
\end{equation}
provided $V = \HKVar(f^*)$ and $f^*$ satisfies a minimum length
condition in that each constant piece has length at least $c/k(f^*)$
(the multiplicative term $C(c)$ in \eqref{hkad1} only depends on this
$c$ appearing in the minimum length condition). A proof of this result
can be found in \cite[Corollary 2.3]{guntuboyina2017spatial} and, for
other similar results, see \cite{lin2017sharp, dalalyan2017tvd,
  ortelli2018total, zhang2019element}. In light of this
univariate result, it is plausible to expect a bound similar to
\eqref{cmadintro} for $\HKfitfun$ when $f^*$ is an axis-aligned rectangular
piecewise constant function provided that the tuning parameter $V$ is
taken to be equal to $\HKVar(f^*)$ and provided that $f^*$ satisfies a
minimum length condition. We prove such a result for a class of simple
rectangular piecewise constant functions $f^* :
[0, 1]^\usedim \rightarrow \R$ of the form
\begin{equation}\label{spc}
  f^*(\cdot) = a_1 \Ind_{[\xvec^*, \onevec]}(\cdot) + a_0
\end{equation}
for some $a_1, a_0 \in \R$ and $\xvec^* \in
[0, 1]^d$ (here $\Ind$ stands for the indicator function). It is easy
to see that \eqref{spc} represents a rectangular
piecewise constant function with $k(\fstar) \leq 2^d$. In
\autoref{theorem:lasso_adaptive_d}, we prove that when $\fstar$ is of the above
form~\eqref{spc}, then
\begin{equation}\label{hkadintro}
  \Risk(\HKfitfun, \fstar) \leq C(c, \usedim) \frac{\sigma^2}{\numobs}
(\log(e \numobs))^{\frac{3\usedim}{2}} (\log (e \log (e \numobs)))^{\frac{2\usedim -
  1}{2}}
\end{equation}
provided the tuning parameter $V$ equals $\HKVar(f^*)$ and $\xvec^* \in [0, 1]^d$ satisfies a
\textit{minimum size condition}~\eqref{equation:min_length_strong}.
This latter condition, which is
analogous to the minimum length condition in the univariate case,
involves a positive constant $c$ and the constant $C(c, d)$ appearing in
\eqref{hkadintro}  only depends on $c$ and the dimension $d$. In the
specific case when $d = 2$, the minimum length condition
\eqref{equation:min_length_strong} can be weakened,
as discussed in \autoref{section:adaptation2}.

We are unable to prove versions of \eqref{hkadintro} for more
general rectangular piecewise constant functions. However, some
results in that direction have been proved in a very recent paper by
\citet{ortelli2018total}. Their results are of a different flavor as
they work with a similar but different estimator and a smaller loss
function. Their proof techniques are also completely different from
ours.


The rest of the paper is organized as follows. The notions of entire
monotonicity and Hardy-Krause variation are formally defined for arbitrary $d
\geq 1$ in \autoref{SECTION:CM_HK} where we also collect some of
their relevant properties. In \autoref{SECTION:COMPUTATION}, we
discuss the computational aspects for solving the optimization
problems in \eqref{equation:cmfitfun}
and~\eqref{equation:hkfitfun}. The risk results for $\EMfitfun$ are
given in \autoref{section:nnls}  while the risk bounds for
$\HKfitfun$ are in \autoref{section:lasso}. We discuss the connections
of our contributions with other related work in
\autoref{section:discussion}. The proofs for our risk results are
given in \autoref{SECTION:RISK_PROOFS} while the proofs of the results
in \autoref{SECTION:CM_HK} and \autoref{SECTION:COMPUTATION} are given
in \autoref{section:proofs23}. Additional technical results used in
the proofs of \autoref{SECTION:RISK_PROOFS} are proved in
\autoref{section:lemma_proofs}.
\autoref{section:adaptation2} contains another risk bound for $\HKfitfun$,
and \autoref{section:simulation} contains the results of some simulations
that includes depictions of the two estimators,
as well as an application to estimation in the bivariate current status model.

\section{Entire monotonicity and Hardy-Krause variation}
\label{SECTION:CM_HK}
The aim of this section is to provide formal definitions of entire
monotonicity and HK$\zerovec$ variation for the convenience of the
reader. We roughly follow the notation of
\citet{aistleitner2014functions} and \citet{owen2005multidimensional}.

Let us first introduce some basic notation that will be used
throughout the paper. We let $\zerovec = (0, \dots, 0)$ and $\onevec =
(1, \dots, 1)$. Given an integer $m$, we take $[m] \defn \{1,
\ldots, m\}$. For two points $\avec = (a_1, \dots, a_{\usedim})$ and
$\bvec = (b_1, \dots, b_{\usedim}) \in [0, 1]^\usedim$, we write
\begin{equation}
 \avec \prec \bvec ~~ \text{ if and only if } ~~ a_j < b_j \text{
   for every } j = 1, \dots, d
\end{equation}
and
\begin{equation}
 \avec \preceq \bvec ~~ \text{ if and only if } ~~ a_j \leq b_j \text{
   for every } j = 1, \dots, d.
\end{equation}
When $\avec \preceq \bvec$, we write
\begin{align}
  [\avec, \bvec] &\defn \left\{ \xvec : \avec \preceq \xvec \preceq
    \bvec \right\} \defn \prod_{j=1}^d [a_j, b_j], \label{clorec}
    \\
  [\avec, \bvec) &\defn \left\{\xvec : \avec \preceq \xvec \prec \bvec
  \right\} \defn \prod_{j=1}^d [a_j, b_j).
\end{align}
Note that $[\avec, \bvec]$ is a closed axis-aligned rectangle and it
has nonempty interior when $\avec \prec \bvec$.

Given a function $f: [0, 1]^d \rightarrow \R$ and two distinct points
$\avec = (a_1, \dots, a_d), \bvec = (b_1, \dots, b_d) \in [0, 1]^d$
with $\avec \preceq \bvec$, we define the \textit{quasi-volume}
$\QVolume(f; [\b{a}, \b{b}])$ by
\begin{equation}\label{equation:quasivolume}
\sum_{j_1 = 0}^{J_1} \cdots \sum_{j_\usedim = 0}^{J_\usedim}
(-1)^{j_1 + \dots + j_d} f \parens*{
    b_1 + j_1(a_1 - b_1), \ldots,
    b_d + j_d(a_d - b_d)
},
\end{equation}
where $J_i \defn \Ind\{a_i \ne b_i\}$ for each $i$. For example, when
$d = 2$, it is easy to see that $\QVolume(f; [\b{a}, \b{b}])$ equals
\begin{equation}\label{d2qv}
  \begin{aligned}
   f(b_1, b_2) - f(b_1, a_2) - f(a_1, b_2) + f(a_1, a_2) ~~ \text{if }
   \avec \prec \bvec \\
   f(b_1, b_2) - f(b_1, a_2)    ~~  \text{if } a_1 = b_1 \text{ and }
   a_2 < b_2 \\
  f(b_1, b_2) - f(a_1, b_2) ~~ \text{if } a_2 = b_2 \text{ and } a_1 <
  b_1.
  \end{aligned}
\end{equation}
We are now ready to define entire monotonicity.
\begin{definition}[Entire monotonicity]\label{definition:CM}
  We say that a function $f: [0, 1]^d \rightarrow \R$ is
  \emph{entirely monotone} if
  \begin{equation}\label{equation:cm_def}
    \QVolume(f; [\b{a}, \b{b}]) \geq 0 \qt{for every $\avec \neq \bvec
      \in [0, 1]^d$ with $\avec \preceq \bvec$}.
  \end{equation}
\end{definition}
In words, for a entirely monotone function $f$, every quasi-volume
$\QVolume(f; [\avec, \bvec])$ is nonnegative. The class of such
functions will be denoted by $\EMClass$. By \eqref{d2qv}, note that
entire monotonicity is equivalent to \eqref{sp2} for $d = 2$.



A more common generalization of monotonicity
to multiple dimensions is the class $\MClass$ consisting of all functions $f : [0, 1]^d \rightarrow \R$ satisfying
\begin{equation}\label{equation:m_def}
  f(a_1, \dots, a_d) \leq f(b_1, \dots, b_d), \qt{for $0 \leq a_i \leq b_i \leq 1, \quad i = 1, \dots, d$}.
\end{equation}
As the following result shows
(see \autoref{section:proof_cm_m} for a proof),
$\EMClass$ is a strict subset of $\MClass$  when $d \ge 2$ (e.g., when $d = 2$, functions in $\EMClass$ need to additionally satisfy the second constraint in \eqref{sp2}) and thus the estimator~\eqref{equation:cmfitfun} is distinct from the LSE over $\MClass$ for $d \geq 2$. This latter estimator is the classical multivariate isotonic regression estimator  \cite{RWD88}.

\begin{lemma}\label{lemma:cm_m}
When $\usedim = 1$, entire monotonicity coincides with monotonicity, i.e., $\EMClassPlain^1 = \MClassPlain^1$.
For $\usedim \ge 2$, we have $\EMClass \subsetneq \MClass$.
\end{lemma}

It is well-known that entirely monotone functions are closely related
to cumulative distribution functions of nonnegative measures. The
following result taken from \citet[Theorem
3]{aistleitner2014functions} makes this connection precise.

\begin{lemma}[{\cite[Theorem 3]{aistleitner2014functions}}]
    \label{lemma:em_right_continuous}
  \begin{enumerate}
  \item For every nonnegative Borel measure $\nu$ on $[0,
    1]^{\usedim}$, the function $f(\xvec) := \nu([\zerovec, \xvec])$
    belongs to $\EMClass$.
  \item If $f \in \EMClass$ is right-continuous, then there exists a
    unique nonnegative Borel measure $\nu$ on $[0, 1]^\usedim$ such
    that $f(\xvec) - f(\zerovec) = \nu([\zerovec, \xvec])$.
  \end{enumerate}
\end{lemma}

We shall now define the notion of HK$\zerovec$ variation. The
HK$\zerovec$ variation is defined through another variation called the
\textit{Vitali variation}. Let us first define the Vitali variation of
a function $f: [0, 1]^d \rightarrow \R$. To do so, we need some
notation. By a partition of the univariate interval $[0, 1]$, we mean
a set of points $0 = x_0 < x_1 < \dots < x_k = 1$ for some $k \geq
1$. Given $d$ such univariate partitions:
\begin{equation}\label{dpar}
  0 = x_0^{(s)} < x_1^{(s)} < \dots < x_{k_s}^{(s)} = 1, \qquad \quad \mbox{for $s = 1, \dots, d$,}
\end{equation}
 we can define a collection $\mathcal{P}$ of
subsets of $[0, 1]^d$
consisting of all sets of the form $A_1 \times \dots \times A_d$
where for each $1 \le s \leq d$, $A_s = [x_{l_s}^{(s)}, x_{l_s +
  1}^{(s)}]$ for some $0 \leq l_s \leq k_s - 1$. Note that each set in
$\mathcal{P}$ is an axis-aligned closed rectangle and the cardinality
of $\mathcal{P}$ equals $k_1 \dots k_d$. The rectangles in
$\mathcal{P}$ are not disjoint but they form a \textit{split} of $[0,
1]^d$ in the sense of \citet[Definition
3]{owen2005multidimensional} and we shall refer to $\mathcal{P}$ as the split
generated by the $d$ univariate partitions \eqref{dpar}.
\begin{definition}[Vitali variation]\label{definition:vitali}
  The Vitali variation of a function $f: [0, 1]^d \rightarrow \R$ is
  defined as
  \begin{equation}\label{equation:vitali}
    \VVar{\usedim}(f; [0, 1]^d) \defn \sup_{\mathcal{P}} \sum_{A \in
      \mathcal{P}} |\Delta(f; A)|
  \end{equation}
  where $\Delta(f; A)$ is the quasi-volume defined in
  \eqref{equation:quasivolume} and the supremum above is taken over
  all splits $\mathcal{P}$ that are generated by $d$ univariate
  partitions in the manner described above.
\end{definition}
The following observations about the Vitali variation will be useful
for us. Note first that when $d = 1$, Vitali variation is simply total
variation \eqref{equation:TV_def}  since the rectangles in this case
are intervals. The second fact is that when $f$ is smooth (in the
sense that the partial derivatives appearing below exist and are
continuous on $[0, 1]^d$), we have
\begin{equation}\label{vvs}
  \VVar{\usedim}(f; [0, 1]^d) = \int_0^1 \dots \int_0^1
  \left|\frac{\partial^d f}{\partial x_1 \dots \partial x_d} \right| dx_1
  \dots dx_d.
\end{equation}
The third observation is that $\VVar{\usedim}(f; [0, 1]^d)$ can be written
out explicitly when $f$ is a rectangular piecewise constant
function. In order to state this result, let us formally define the
notion of a rectangular piecewise constant function on $[0,
1]^d$. Given $d$ univariate partitions as in \eqref{dpar}, let
$\mathcal{P}^*$ denote the collection of all sets of the form $A_1
\times \dots \times A_d$ where for each $1 \leq s \leq d$, $A_s$ is
either equal to $[x_{l_s}^{(s)}, x_{l_s + 1}^{(s)})$ for some $0 \leq
l_s \leq k_s - 1$ or the singleton $\{1\}$. Note that, unlike
$\mathcal{P}$, the sets in $\mathcal{P}^*$ are disjoint and hence
$\mathcal{P}^*$ forms a partition of $[0, 1]^d$. We shall refer to
$\mathcal{P}^*$ as the partition generated by the $d$ univariate
partitions \eqref{dpar}.
\begin{definition}[Rectangular piecewise constant
  function]\label{definition:peecee}
  We say that $f: [0, 1]^d \rightarrow \R$ is rectangular piecewise
  constant if there exists a partition $\mathcal{P}^*$ generated by $d$
  univariate partitions as described above such that $f$ is constant
  on each set in $\mathcal{P}^*$.    We use $\rpc$ to denote the class
  of all rectangular piecewise constant functions on $[0, 1]^d$. For
  $f \in \rpc$, we define $k(f)$ as the smallest value of $k_1\dots
  k_d$ for which there exist $d$ univariate partitions of lengths
  $k_1, \dots, k_d$ such that $f$ is constant on each of the sets in
  $\mathcal{P}^*$ generated by these $d$ univariate partitions.
\end{definition}

The following lemma (proved in \autoref{section:proof_vvrpc}) provides
a formula for the Vitali variation of a rectangular piecewise constant
function $f$ on $[0, 1]^\usedim$. Note that this lemma implies, in
particular, that the Vitali variation of every rectangular piecewise
constant function is finite.

\begin{lemma}\label{vvrpc}
  Suppose $f$ is rectangular piecewise constant on $[0, 1]^d$ with
  respect to a partition $\mathcal{P}^*$ generated by $d$ univariate
  partitions and let $\mathcal{P}$ denote the split generated by these
  univariate partitions. Then
  \begin{equation}
    \VVar{\usedim}(f; [0, 1]^d) = \sum_{A \in \mathcal{P}} |\Delta(f; A)|.
  \end{equation}
\end{lemma}
Despite these interesting properties, the Vitali variation is not
directly suitable for our purposes because there exist many
non-constant functions $f$ on $[0, 1]^d$ (such as $f(x, y) := x$) whose Vitali variation is
zero. This weakness of the Vitali variation is well-known (see
e.g.,~\citet{owen2005multidimensional} or
\citet{aistleitner2014functions}) and motivates the
following definition of the HK$\zerovec$ variation.

Given a nonempty subset of indices $S \subseteq [\usedim] = \{1,
\dots, \usedim\}$, let
\begin{equation}\label{equation:face_adjacent}
U_S \defn \{(u_1, \ldots, u_\usedim) \in [0, 1]^\usedim : u_j = 0, j \notin S\}.
\end{equation}
Note that $U_S$ is a face of $[0, 1]^\usedim$ adjacent to
$\zerovec$. By ignoring the components not in $S$,
the restriction of the function $f$ on $[0, 1]^d$ to the set $U_S$ can
be viewed as a function $\tilde{f} : [0, 1]^{\abs{S}} \to \R$. The
Vitali variation of $\tilde{f}$ viewed as a function of $[0, 1]^{|S|}$
will be denoted by
\begin{equation}
\VVar{\abs{S}}(f; S ; [0, 1]^\usedim)
\defn \VVar{\abs{S}}(\tilde{f}; [0, 1]^{\abs{S}}).
\end{equation}
The \emph{Hardy-Krause variation (anchored at $\zerovec$)}
of $f : [0, 1]^\usedim \to \R$ is defined by
\begin{equation}\label{equation:hkvariation}
\HKVar(f; [0, 1]^\usedim) := \sum_{\varnothing \neq S \subseteq [\usedim]}
\VVar{{\abs{S}}}(f; S; [0, 1]^\usedim).
\end{equation}
That is, the HK$\zerovec$ variation is the sum of the Vitali
variations of $f$ restricted to each face of $[0, 1]^\usedim$ adjacent
to $\zerovec$.  Note the special role played by the point $\zerovec$
in this definition and this is the reason for the phrase ``anchored at
$\zerovec$''. It is also common to anchor the HK variation at
$\onevec$ (see e.g.,~\citet{aistleitner2014functions}) but we
focus only on $\zerovec$ as the anchor in this paper.  Because of the
addition of the lower-dimensional Vitali variations, it is clear that
the HK$\zerovec$ variation equals zero only for constant functions and
this property is the reason why the HK$\zerovec$ variation is usually
preferred to the Vitali variation.

Let us now remark that the HK$\zerovec$ variation is quite different
from the usual notion of multivariate total variation. Indeed, when
$f$ is smooth, the multivariate total variation of $f$ only involves
the first order partial derivatives of $f$. On
the other hand, as can be seen from \eqref{vvs}, the HK$\zerovec$
variation is defined in terms of higher order mixed partial
derivatives of $f$.

An important property of the HK$\zerovec$ variation
is that it is finite for rectangular piecewise constant
functions. This is basically a consequence of \autoref{vvrpc} and
the fact that the restriction of a rectangular piecewise constant
function to each set $U_S$ in \eqref{equation:face_adjacent} is also
rectangular piecewise constant.


The following lemma formally establishes the connection between
entire monotonicity and HK$\zerovec$ variation,
as mentioned earlier in the Introduction. 

\begin{lemma}\label{lemma:hkvar_properties}
The following properties hold:
\begin{enumerate}[(i)]
\item \label{lemma:hk_diff_of_cm}  If $f : [0, 1]^\usedim \to \R$ has finite HK$\zerovec$ variation, then
there exist unique $f_+, f_- \in \EMClass$ such that
$f_+(\zerovec) = f_-(\zerovec) = 0$ and
\begin{equation}
f(\xvec) - f(\zerovec) = f_+(\xvec) - f_-(\xvec),
\qquad \xvec \in [0, 1]^\usedim
\end{equation}
and
\begin{equation}
\HKVar(f; [0, 1]^\usedim)
= \HKVar(f_+; [0, 1]^\usedim)
+ \HKVar(f_-; [0, 1]^\usedim).
\end{equation}
\item \label{lemma:hkvar_of_cm} If $f \in \EMClass$, then
\begin{equation}
\HKVar(f; [0, 1]^\usedim) = f(\onevec) - f(\zerovec).
\end{equation}
\end{enumerate}
\end{lemma}
The first fact in the above lemma is quite standard (see e.g.,~\cite[Theorem 2]{aistleitner2014functions}). We could not find an exact
  reference for the second fact so we included a proof in
  \autoref{section:proof_hkvar_of_cm}.

Finally, let us mention that it is well-known that a result analogous
to \autoref{lemma:em_right_continuous} holds for the connection
between functions with finite HK$\zerovec$ variation and cumulative
distribution functions for signed measures. This result is stated
next.

\begin{lemma}[{\cite[Theorem
    3]{aistleitner2014functions}}]\label{lemma:hk_right_continuous}
  \begin{enumerate}
  \item For every signed Borel measure $\nu$ on $[0,
    1]^{\usedim}$, the function $f(\xvec) := \nu([\zerovec, \xvec])$
    has finite HK$\zerovec$ variation.
  \item If $f$ has finite HK$\zerovec$ variation and is
    right-continuous, then there exists a unique finite signed Borel
    measure $\nu$ on $[0, 1]^\usedim$ such
    that $f(\xvec) = \nu([\zerovec, \xvec])$.
  \end{enumerate}
\end{lemma}

\section{Computational feasibility}
\label{SECTION:COMPUTATION}
The goal of this section is to describe procedures for computing the
two estimators~\eqref{equation:cmfitfun}
and~\eqref{equation:hkfitfun}. We shall specifically show that
the estimators~\eqref{equation:cmfitfun}
and~\eqref{equation:hkfitfun} can be computed by solving a NNLS problem and a LASSO problem respectively, with
a suitable design matrix that is the same for both the problems and that
depends only on $\xvec_1, \dots, \xvec_n$. This design matrix will be
the matrix $\Altdesignmat$ whose columns are the distinct elements of
the finite set
\begin{equation}\label{equation:ureval_set}
\IndSet \equiv \IndSet_{\xvec_1, \ldots, \xvec_\numobs}
\defn \{\UReval{\zvec} : \zvec \in [0, 1]^\usedim\}
\subseteq \{0, 1\}^\numobs,
\end{equation}
where
\begin{equation}\label{vzdef}
\UReval{\zvec}
\equiv \URevalplain_{\xvec_1, \ldots, \xvec_\numobs}(\zvec)
\defn
(\Ind_{[\zvec, \onevec]}(\xvec_1),
\Ind_{[\zvec, \onevec]}(\xvec_2),
\ldots,
\Ind_{[\zvec, \onevec]}(\xvec_\numobs)).
\end{equation}
We assume without loss of generality that the first column of $\Altdesignmat$
is $\UReval{\zerovec} = \onevec = (1,\ldots, 1) \in \R^n$. Note that $\Altdesignmat$ has
dimensions $\numobs \times \altdim$ where $\altdim \equiv
\altdim(\xvec_1, \ldots, \xvec_\numobs) \defn \abs{\IndSet}$.  By
definition, there exist distinct points $\zvec_1, \ldots,
\zvec_\altdim \in [0, 1]^d$ with $\zvec_1 = \zerovec$ such that the
$j$th column of $\Altdesignmat$ is $\UReval{\zvec_j}$ for each
$j$.

Our first result below deals with problem
\eqref{equation:cmfitfun}. Given the design matrix $\Altdesignmat$, we
can define the following NNLS problem
\begin{equation}\label{equation:cm_nnls}
\coefhatcm \in \argmin_{\coef \in \R^\altdim : \coefplain_j \ge 0, \forall j \ge 2}
\norm{\yvec - \Altdesignmat \coef}^2
\end{equation}
where $\yvec$ is the $\numobs \times 1$ vector consisting of the
observations $y_1, \dots, y_n$ coming from model
\eqref{obmod}. \eqref{equation:cm_nnls} is clearly a finite
dimensional convex optimization problem (in fact, a quadratic
optimization problem with linear constraints). Its solution
$\coefhatcm$ is
not necessarily unique but the vector $\Altdesignmat
\coefhatcm$ is the projection of the observation vector $\yvec$ onto
the closed convex cone $\{\Altdesignmat \coef : \min_{j \geq 2}
\beta_j \geq 0\}$ and is thus unique. The next result (proved in
\autoref{section:proof_cm_nnls}) shows how to
obtain a solution to problem \eqref{equation:cmfitfun} using any solution
$\coefhatcm$ of \eqref{equation:cm_nnls}.

\begin{proposition}\label{proposition:cm_nnls}
One solution for the optimization problem~\eqref{equation:cmfitfun} is
\begin{equation}\label{equation:cmfitfun_nnls}
\EMfitfun \defn
\sum_{j = 1}^\altdim (\coefhatcmplain)_j \cdot
\Ind_{[\zvec_j, \onevec]},
\end{equation}
where $\coefhatcm = ((\coefhatcmplain)_1, \dots, (\coefhatcmplain)_p)$
is any solution to~\eqref{equation:cm_nnls}.
\end{proposition}
Thus, one way to compute the estimator~\eqref{equation:cmfitfun}
is to solve the NNLS problem~\eqref{equation:cm_nnls} and use the
resulting coefficients in the above manner \eqref{equation:cmfitfun_nnls}. It is
interesting to note that the solution~\eqref{equation:cmfitfun_nnls} is a
rectangular piecewise constant function and the quantity
$k(\EMfitfun)$ (see \autoref{definition:peecee}) will be controlled by the
sparsity of $\coefhatcm$. The key to proving
\autoref{proposition:cm_nnls}  is the following characterization of
$\EMClass$ (proved in \autoref{section:proof_cm_discrete}).

\begin{proposition}[Discretization of entirely monotone functions]
\label{proposition:cm_discrete}
For every set of design points $\xvec_1, \dots, \xvec_n \in [0, 1]^d$,
we have
\begin{equation}
\braces*{\Altdesignmat \coef : \coefplain_j \ge 0, \forall j \ge 2}
= \braces*{(f(\xvec_1), \ldots, f(\xvec_\numobs)) : f \in \EMClass}.
\end{equation}
\end{proposition}
Note that \autoref{proposition:cm_discrete} immediately implies that
for every minimizer $\EMfitfun$ of \eqref{equation:cmfitfun}, the
vector $(\EMfitfun(\xvec_1), \dots, \EMfitfun(\xvec_n))$ equals
$\Altdesignmat \coefhatcm$ and is thus unique.

We now turn to problem \eqref{equation:hkfitfun}. Given the matrix
$\Altdesignmat$ and a tuning parameter $\LASSOrad > 0$,
we can define the following LASSO problem:
\begin{equation}\label{equation:hk_lasso}
\coefhathk \in
\argmin_{\coef \in \R^{\altdim}: \sum_{j \ge 2} \abs{\coefplain_j} \le \LASSOrad}
\norm{\yvec - \Altdesignmat\coef}^2.
\end{equation}
Again $\coefhathk$ may not be unique but $A \coefhathk$ is unique as
it is the projection of $\yvec$ onto the closed convex set
\begin{equation}\label{equation:lasso_ball}\noeqref{equation:lasso_ball}
\LASSOBall(\LASSOrad)
\defn
\braces*{
   \Altdesignmat \coef :
    \sum_{j \ge 2} \abs{\coefplain_j} \le \LASSOrad
}.
\end{equation}
The next result (proved in \autoref{section:proof_hk_lasso}) shows how to
obtain a solution to
\eqref{equation:hkfitfun} using any solution $\coefhathk$ of
\eqref{equation:hk_lasso}.

\begin{proposition}\label{proposition:hk_lasso}
One solution for the optimization problem~\eqref{equation:hkfitfun} is
\begin{equation}\label{equation:hkfitfun_lasso}
\HKfitfun := \sum_{j=1}^\altdim (\coefhathkplain)_j \cdot
\Ind_{[\zvec_j, \onevec]},
\end{equation}
where $\coefhathk = ((\coefhathkplain)_1, \dots, (\coefhathkplain)_p)$
is the solution to the LASSO problem~\eqref{equation:hk_lasso}.
\end{proposition}
Thus, one way to compute the estimator~\eqref{equation:hkfitfun}
is to solve the LASSO problem~\eqref{equation:hk_lasso}
and use the resulting coefficients to construct the rectangular
piecewise constant function \eqref{equation:hkfitfun}. Note the strong similarity
between the two expressions \eqref{equation:cmfitfun_nnls} and
\eqref{equation:hkfitfun_lasso}. The following result (proved in
\autoref{section:proof_hk_discrete}) is the key ingredient in proving the above.

\begin{proposition}\label{proposition:hk_discrete}
For every set of design points $\xvec_1, \dots, \xvec_n \in [0, 1]^d$,
we have
\begin{equation}
\LASSOBall(\LASSOrad)
= \{(f(\xvec_1), \ldots, f(\xvec_\numobs)) : \HKVar(f; [0, 1]^\usedim) \le \LASSOrad\}.
\end{equation}
\end{proposition}
\autoref{proposition:hk_discrete} immediately implies
that for every minimizer $\HKfitfun$ of \eqref{equation:hkfitfun}, the vector
$(\HKfitfun(\xvec_1), \dots, \HKfitfun(\xvec_n))$ equals $\Altdesignmat
\coefhathk$ and is thus unique.

We have thus shown that the LSEs defined by
\eqref{equation:cmfitfun}  and \eqref{equation:hkfitfun}
can be computed via NNLS and LASSO estimators with respect
to the design matrix $\Altdesignmat$ whose columns are the elements of
the finite set $\IndSet$ defined in \eqref{equation:ureval_set}. Once
the design matrix $\Altdesignmat$ is formed, we can use existing
quadratic program solvers to solve the NNLS and LASSO problems. The
key to forming
$\Altdesignmat$ is to enumerate the elements of $\IndSet$ and we
address this issue now. We first state the following result which
provides a worst case upper bound on $\altdim \equiv \altdim(\xvec_1,
\ldots, \xvec_\numobs)$, the cardinality of $\IndSet$.

\begin{lemma}\label{lemma:VC_app}
The cardinality of $\IndSet$ satisfies
\begin{equation}\label{equation:VC_lemma}
\altdim(\xvec_1, \ldots, \xvec_\numobs)
\le \sum_{j=0}^\usedim \binom{\numobs}{j}
\end{equation}
for every $\xvec_1, \dots, \xvec_n \in \R^d$.
\end{lemma}
\autoref{lemma:VC_app} is a consequence of the Vapnik-Chervonenkis
lemma~\cite{VapnikCervonenkis71events} and is proved in
\autoref{section:proof_VC_app}. Note that the upper
bound~\eqref{equation:VC_lemma} can be further bounded by $(e \numobs
/ \usedim)^\usedim$.

We emphasize here that \autoref{lemma:VC_app} gives a worst case upper
bound for $p(\xvec_1, \dots, \xvec_n)$ (here worst case is in terms of
the design configurations $\xvec_1, \dots, \xvec_n$). For specific
choices of $\xvec_1, \dots, \xvec_n$, the quantity $\altdim(\xvec_1,
\ldots, \xvec_\numobs)$ can be much smaller than the right hand side
of \eqref{equation:VC_lemma}. For example, if  $\xvec_1, \ldots,
\xvec_\numobs$ are an enumeration of the grid points
$\braces*{(i_1 / \numobs^{1 / \usedim}, \ldots, i_\usedim / \numobs^{1
    / \usedim}) : i_1, \ldots, i_\usedim \in \{1, \ldots, \numobs^{1 /
    \usedim}\}}$
(or form any other full grid)
then $\altdim(\xvec_1, \ldots, \xvec_\numobs) = \numobs$
whereas the upper bound in \eqref{equation:VC_lemma} is of order
$\numobs^\usedim$. However, there exist design configurations
$\xvec_1, \ldots, \xvec_\numobs$ where the upper bound can be tight.
For instance, when $\usedim = 2$, if $\xvec_1, \ldots, \xvec_\numobs$
lie on the anti-diagonal (the line segment connecting $(0,1)$ and $(1,0)$),
then $\altdim(\xvec_1, \ldots, \xvec_\numobs) = \frac{\numobs (\numobs + 1)}{2}$,
so the upper bound $\frac{\numobs (\numobs + 1)}{2} + 1$ in
\eqref{equation:VC_lemma} is nearly tight for $\altdim(\xvec_1,
\ldots, \xvec_\numobs)$.

The task of enumerating $\IndSet$ in general can be simplified if we
show that we only need to check the value of $\IndUpperRight{\zvec}$
on the design points $\xvec_1, \ldots, \xvec_\numobs$
for all $\zvec$ in some finite set $S$, rather than all $\zvec \in (0,1]^\usedim$
as in definition~\eqref{equation:ureval_set}. Then we can list all
$|S|$ evaluation vectors (and remove duplicates if necessary) to form
$\Altdesignmat$. The following two strategies can be used to construct
the set $S$:

\begin{enumerate}
    \item \textbf{Na\"ive gridding.}
    The simplest idea is to let $S$ be the smallest grid that contains
    the design points
    $\xvec_1, \ldots, \xvec_\numobs$.
    That is, let $S = S_1 \times \cdots \times S_\usedim$ where
    $S_j \defn \{(\xvec_1)_j, \ldots, (\xvec_\numobs)_j\}$
    is the set of unique $j$th component values among the design points.
    It is simple to check that for any $\zvec \in (0,1]^\usedim$,
    the value of $\IndUpperRight{\zvec}$ on the design points
    is the same as $\IndUpperRight{\zvec'}$, where $\zvec'$ is the
    smallest element of $S$ such that $\zvec \preceq \zvec'$. In the
    worst case, $\abs{S_j} = n$ for each $j$, so we would need to
    check at most $\abs{S} = n^\usedim$ vectors.

    \item \textbf{Component-wise minimum.}
    A better approach is to let
    \begin{equation}
    S \defn \{\min\{\xvec_i : i \in I\} : I \subseteq [\numobs], \abs{I} \le \usedim\},
    \end{equation}
    where ``$\min$'' denotes component-wise minimum of vectors. That is,
    for each subset of the design points of size $\le \usedim$,
    we take the component-wise minimum and include that vector in
    $S$. To see why this definition of $S$ suffices, consider any
    $\zvec \in [0, 1]^\usedim$
    and note the $\IndUpperRight{\zvec}$ has the same values on the design points
    as $\IndUpperRight{\zvec'}$, where $\zvec' \defn \min\{\xvec_i : i \in J\}$
    and $J \defn \{i : \zvec \preceq \xvec_i\}$. Furthermore, by the
    same reasoning as in our VC dimension computation above,
    there must exist some subset $I \subseteq J$ of size $\le \usedim$
    such that $\min\{\xvec_i : i \in J\} = \min\{\xvec_i : i \in I\}$,
    which proves $\zvec' \in S$.
    In the worst case, we would need to check
    $\abs{S} = \sum_{j=0}^\usedim \binom{\numobs}{j}$ vectors,
    which is the VC upper bound \eqref{equation:VC_lemma}.
\end{enumerate}

\subsection{Special Case: the equally-spaced lattice design} \label{section:esld}
The results stated so far in the section hold for every configuration
of design points $\xvec_1, \dots, \xvec_n \in [0, 1]^d$. We now
specialize to the setting where $\xvec_1, \dots, \xvec_n$ form an
equally-spaced lattice (precisely defined below). Our theoretical
results described in the next section work under this
setting. Moreover, some of the estimators from the literature that are
related to $\EMfitfun$ and $\HKfitfun$ are defined only under the
lattice design so a discussion of the form of our estimators in this setting
will make it easier for us to compare and contrast them with existing
estimators (this comparison is the subject of
\autoref{section:discussion}).

Given positive integers $n_1, \dots, n_d$ with $n = n_1 \dots n_d$, by
a lattice design of dimensions $n_1 \times \dots \times n_d$, we
mean that $\xvec_1, \dots, \xvec_n$ form an enumeration of the points
in
\begin{equation}\label{equation:lattice_design}
\LatticeDesign \defn \braces*{
    (i_1 / \numobs_1, \ldots, i_\usedim / \numobs_\usedim) : 0 \le i_j \le \numobs_j - 1,
    j = 1, \ldots, \usedim
}
\end{equation}
Note that, in this setting, the set $\IndSet$ (defined in
\eqref{equation:ureval_set}) can be enumerated by
$\IndSet = \{\UReval{\xvec_1}, \ldots, \UReval{\xvec_\numobs}, \zerovec\}$.
Without loss of generality, we may ignore the $\zerovec$ element
and assume the columns of $\Altdesignmat$ are
$\UReval{\xvec_1}, \ldots, \UReval{\xvec_\numobs}$
so that the $i,j$ entry of $\Altdesignmat$ is given by
$\Altdesignmat(i, j) = \Ind_{[\xvec_j,   \onevec]}(\xvec_i) =
\Ind\{\xvec_j \preceq \xvec_i\}$. We also take
$\xvec_1 := \zerovec$ (corresponding to $i_1 = \dots = i_d = 0$) so
that the first column of $\Altdesignmat$ is the vector of
ones. Therefore in the lattice design setting, the optimization
problems \eqref{equation:cm_nnls} and \eqref{equation:hk_lasso} for
computing the two estimators $\EMfitfun$ and $\HKfitfun$ can be
rewritten as
\begin{equation}\label{bexp}\noeqref{bexp}
  \coefhatcm =  \argmin_{\coef \in \R^\altdim : \coefplain_j \ge 0, \forall j \ge 2}
\sum_{i=1}^{\numobs} \left(y_i - \sum_{j=1}^n \Ind\{\xvec_j \preceq
  \xvec_i\} \coef_j \right)^2
\end{equation}
and
\begin{equation}\label{bexpv}\noeqref{bexpv}
\coefhathk =
\argmin_{\coef \in \R^{\altdim}: \sum_{j \ge 2} \abs{\coefplain_j} \le
  \LASSOrad} \sum_{i=1}^{\numobs} \left(y_i - \sum_{j=1}^n \Ind\{\xvec_j \preceq
  \xvec_i\} \coef_j \right)^2
\end{equation}
respectively. It also turns out that, in the lattice design setting, the
 matrix $\Altdesignmat$ is square and invertible
(\autoref{lemma:span}). As a result, it is possible to write down the
vectors $(\EMfitfun(\xvec_1), \dots, \EMfitfun(\xvec_n))$ and
$(\HKfitfun(\xvec_1), \dots, \HKfitfun(\xvec_n))$ as solutions to
more explicit constrained quadratic optimization problems. This is the
content of the next result which is proved in
\autoref{section:proof_boann}. Here, it will be convenient to
represent
vectors in $\R^n$ as tensors indexed  by $\ivec := (i_1, \dots, i_d) \in
\IndexSet$ where
\begin{equation}\label{AllInd}
\IndexSet \defn \Big\{
   \ivec = (i_1, \ldots, i_\usedim) : i_j \in \{0, 1, \ldots, \numobs_j - 1\}
    \text{ for every } j = 1, \dots, d
\Big\}.
\end{equation}
In other words, we write the components of a vector $\param \in \R^n$
by $\paramplain_{\ivec}$ for $\ivec = (i_1, \dots, i_d) \in
\IndexSet$. We will also denote the observation corresponding to the
design point $(i_1/n_1, \dots, i_d/n_d)$ by $y_{\ivec} = y_{i_1,
  \dots, i_d}$.

\begin{lemma}\label{boann}
 Consider the setting of the lattice design of dimensions $n_1 \times
 \dots \times n_d$. For each $\param \in \R^n$, associate the
 ``differenced'' vector $D \param \in \R^n$ whose $\ivec^{th}$ entry
 is given by
\begin{equation}\label{equation:diff_def}
\sum_{j_1=0}^1 \dots \sum_{j_d = 0}^1
  I\{i_1 - j_1 \ge 0, \dots, i_d - j_d \ge 0 \}
  (-1)^{j_1 + \dots + j_d}
  \theta_{i_1 - j_1, \dots, i_d - j_d}
\end{equation}
for every $\ivec = (i_1, \dots, i_d) \in \IndexSet$. Then:
\begin{enumerate}
\item The vector $  \left(\EMfitfun \left(i_1/n_1, \dots, i_d/n_d  \right) : \ivec =
    (i_1, \dots, i_d) \in \IndexSet\right)$ is the solution to the
  optimization problem
\begin{equation}\label{fexp}
  \argmin \left\{\sum_{\ivec \in \IndexSet} \left(y_{\ivec} -
      \theta_{\ivec} \right)^2 : (D \param)_{\ivec} \geq 0 \text{ for
      all } \ivec \neq \zerovec \right\}.
\end{equation}
\item The vector $\left(\HKfitfun \left(i_1/n_1, \dots, i_d/n_d
    \right) : \ivec = (i_1, \dots, i_d) \in \IndexSet \right)$ is the
  solution to the optimization problem
\begin{equation}\label{fexpv}
  \argmin \left\{\sum_{\ivec} \left(y_{\ivec} -
      \theta_{\ivec} \right)^2 : \sum_{\ivec \neq \zerovec}
    \left|(D \param)_{\ivec} \right| \leq V \right\}.
\end{equation}
\end{enumerate}
\end{lemma}

\begin{remark}[The special case of $d = 2$]
  When $d = 2$, it is easy to see that the differenced vector
  $D \param$ is given by
\[
 (D \param)_{(i_1, i_2)} =
  \begin{cases}
   \theta_{i_1, i_2} - \theta_{i_1 - 1 ,i_2} - \theta_{i_1, i_2 -1} +
   \theta_{i_1-1,i_2-1} & \text{if } i_1 > 0, i_2 > 0 \\
   \theta_{i_1, 0} - \theta_{i_1 -1, 0}       & \text{if } i_1 > 0,
   i_2 = 0 \\
   \theta_{0, i_2} - \theta_{0, i_2 - 1} & \text{if } i_1 = 0, i_2 > 0
   \\
   \theta_{0, 0} & \text{if } i_1 = i_2 = 0.
  \end{cases}
\]
Using this, it is easy to see that \eqref{fexpv} can be rewritten for
$d = 2$ as
\begin{align}\label{latd2}
  \argmin &\left\{\sum_{i_1 = 0}^{n_1 - 1} \sum_{i_2=0}^{n_2-1} \left(y_{i_1, i_2} -
      \theta_{i_1, i_2} \right)^2 : \right.
      \\
      &\phantom{{}\Big\{{}} \qquad
      \sum_{i_1=1}^{n_1-1} \sum_{i_2 =1}^{n_2-1}
    \left|\theta_{i_1, i_2} - \theta_{i_1 - 1, i_2} - \theta_{i_1, i_2
      - 1} + \theta_{i_1-1, i_2 -1} \right|
      \\
      &\phantom{{}\Big\{{}} \qquad +  \left. \sum_{i_1=1}^{n_1-1}
  |\theta_{i_1, 0} - \theta_{i_1 -1, 0}| + \sum_{i_2 = 1}^{n_2-1} |\theta_{0,
    i_2} - \theta_{0, i_2 -1}| \leq V \right\}
\end{align}
and a similar formula can be written for \eqref{fexp} for $d = 2$.
\end{remark}

As mentioned in the Introduction, an estimator similar to $\HKfitfun$
has been described by
\citet{mammen1997locally} for $d = 2$ under the lattice design
setting. Specifically, the estimator of \cite{mammen1997locally} for
the vector $(f^*(i_1/n_1, i_2/n_2), 0 \leq i_1 \leq n_1-1, 0 \leq i_2
\leq n_2-1)$ is given by the solution to the optimization problem:
\begin{align}\label{latd2mv}
&  \argmin \left\{\sum_{i_1, i_2} \left(y_{i_1, i_2} -
      \theta_{i_1, i_2} \right)^2 \right.
      \\ &+ \lambda_1 \sum_{i_1, i_2 \geq 1}
    \left|\theta_{i_1, i_2} - \theta_{i_1 - 1, i_2} - \theta_{i_1, i_2
      - 1} + \theta_{i_1-1, i_2 -1} \right|
      \\
      & + \lambda_2 \left. \sum_{i_1 \geq 1}
  |\bar{\theta}^{(1)}_{i_1} - \bar{\theta}^{(1)}_{i_1 -1}| + \lambda_2
                                                  \sum_{i_2\geq 1}
                                                  |\bar{\theta}^{(2)}_{i_2} - \bar{\theta}^{(2)}_{i_2 -1}|  \right\}
\end{align}
where $\lambda_1$ and $\lambda_2$ are positive tuning parameters,
$\bar{\theta}_{i_1}^{(1)} := \frac{1}{n_2} \sum_{i_2=0}^{n_2-1}
\theta_{i_1, i_2}$ and $\bar{\theta}_{i_2}^{(2)} := \frac{1}{n_1} \sum_{i_1=0}^{n_1-1}
\theta_{i_1, i_2}$. This optimization problem is similar to
\eqref{latd2} in that the first term in the penalty is the same in both problems. However the
remaining terms in the penalty above  are different from the terms in
\eqref{latd2} although they are of the same spirit in that both are
penalizing lower dimensional variations. Moreover, our
estimator \eqref{latd2} has one tuning parameter (in the constrained
form) and \eqref{latd2mv}  has two tuning parameters in the penalized
form. It should also be noted that we defined our estimators for
arbitrary design points $\xvec_1, \dots, \xvec_n$ while
\citet{mammen1997locally} only considered the lattice design for $d =
2$.


\section{Risk results}
In this section, risk bounds for the estimators $\EMfitfun$ and $\HKfitfun$
are presented. We define risk under the standard fixed design squared
error loss function (see \eqref{rislo}). Throughout this section, we
assume that we are working with the lattice design of dimensions $n_1
\times \dots \times n_d$ with $n = n_1 \times \dots \times n_d$ and $n_j \ge 1$ for all $ j = 1, \ldots, d$.

\subsection{Risk results for \texorpdfstring{$\EMfitfun$}{entirely
    monotone regression}} \label{section:nnls}
In this subsection, we present bounds on the risk $\Risk(\EMfitfun, f^*)$ of $\EMfitfun$
under the well-specified assumption where we assume that $f^* \in \EMClass$.
The first result below (proved in
\autoref{section:nnls_worst_proof}) bounds the risk
in terms of the HK$\zerovec$ variation of $f^*$. Note that from part
\eqref{lemma:hkvar_of_cm} of \autoref{lemma:hkvar_properties},
$\HKVar(\fstar; [0, 1]^\usedim) = \fstar(\onevec) - \fstar(\zerovec)$
as $\fstar \in \EMClass$.

\begin{theorem}
\label{theorem:NNLS_worst_case}
Let $\fstar \in \EMClass$ and $V^* \defn \HKVar(\fstar; [0, 1]^\usedim)$.
For the lattice design~\eqref{equation:lattice_design},
the estimator $\EMfitfun$ satisfies
\begin{align}
\begin{split}
\Risk(\EMfitfun, \fstar)
&\le C_\usedim \parens*{\frac{\noisestd^2 V^*}{\numobs}}^{\frac{2}{3}}
\parens*{
    \log \left(2 + \frac{V^* \sqrt{\numobs}}{\noisestd} \right)
}^{\frac{2 \usedim - 1}{3}}
\\
&\qquad
+ C_\usedim \frac{\noisestd^2}{\numobs}
(\log (e \numobs))^{\frac{3\usedim}{2}}
(\log (e \log (e \numobs)))^{\frac{2 \usedim - 1}{2}}.
\end{split}
\label{equation:worst_case}
\end{align}
where $C_\usedim$ is a constant that depends only on the dimension $\usedim$.
\end{theorem}


Note that the bound~\eqref{cmwcintro} in the Introduction is the dominant first term of this bound~\eqref{equation:worst_case}.

\begin{remark}[Model misspecification]
\label{remark:cm_misspec}
\autoref{theorem:NNLS_worst_case} is stated under the well-specified assumption
$\fstar \in \EMClass$.
In the misspecified setting where $\fstar \notin \EMClass$,
our LSE $\EMfitfun$ will not be close to $\fstar$,
but rather to
$$\ftilde \in \argmin_{f \in \EMClass} \sum_{i=1}^\numobs (f(\xvec_i) - \fstar(\xvec_i))^2,$$
so it is reasonable to consider $\Risk(\EMfitfun, \ftilde)$
rather than $\Risk(\EMfitfun, \fstar)$.
By the argument outlined in \autoref{remark:gencha_misspec},
one can show that $\Risk(\EMfitfun, \ftilde)$
is upper bounded by the right hand side of~\eqref{equation:worst_case}
after re-defining $\LASSOrad^*$ as $\HKVar(\ftilde; [0, 1]^\usedim)$.
\end{remark}


As mentioned in the Introduction, when $\usedim = 1$, the estimator
$\EMfitfun$ is simply the isotonic LSE for which \citet{Zhang02}
proved that
\begin{equation}
\label{equation:iso_1d}
\Risk(\EMfitfun, \fstar)
\le C \parens*{
    \frac{
        \noisestd^2
        V^*
    }{\numobs}
}^{\frac{2}{3}}
+ C \frac{\noisestd^2}{\numobs} \log(e \numobs)
\end{equation}
for some constant $C > 0$. It is interesting to note that our risk
bound \eqref{equation:worst_case} for general $d \geq 2$ has the same
terms as the univariate bound \eqref{equation:iso_1d} with additional
logarithmic factors which depend on $d$. It is natural to ask
therefore if these additional logarithmic factors are indeed necessary or
merely artifacts of our analysis. The next result (a minimax lower
bound) shows that every estimator pays a logarithmic multiplicative
price of $\log n$ for $d = 2$ and $(\log n)^{2(d-2)/3}$ for $d \geq 3$
in the first $n^{-2/3}$ term. We do not, unfortunately, know if the
$(\log n)^{3d/2} (\log \log n)^{(2d-1)/2}$ factor in the second term in
\eqref{equation:worst_case}  is necessary or artifactual, although we can
prove that it can be removed by a
modification of the estimator $\EMfitfun$ (see
\autoref{theorem:constrained_em} below).



The next result (proved in \autoref{section:em_minimax_proof}) proves a lower bound for the minimax
risk:
\begin{equation}\label{minem}
\MinimaxRiskEM \defn \inf_{\fhat_\numobs}
\sup_{\fstar \in \EMClass : \HKVar(\fstar) \le \LASSOrad}
\E_{\fstar} \Loss(\fhat_\numobs, \fstar),
\end{equation}
where the expectation is with respect to model~\eqref{obmod}.

\begin{theorem}\label{theorem:em_minimax}
Let $d \geq 2$, $V > 0$, $\sigma > 0$
and let $\numobs_j \ge c_s \numobs^{1/\usedim}$ for all $j = 1, \ldots, \usedim$
for some $c_s \in (0, 1]$.
Then there exists a positive constant $C_\usedim$ depending only on $\usedim$ and $c_s$,
such that the minimax risk
on the lattice design~\eqref{equation:lattice_design} satisfies
\begin{equation}
\MinimaxRiskEM
\ge C_\usedim \parens*{\frac{\noisestd^2 \LASSOrad}{\numobs}}^{\frac{2}{3}}
\left(\log \left(\frac{\LASSOrad\sqrt{\numobs} }{\noisestd} \right)
\right)^{\frac{2(\usedim - 2)}{3}}
\end{equation}
provided $\numobs$ is larger than a positive constant $c_{\usedim,
  \noisestd^2 / \LASSOrad^2}$ depending only on $\usedim$,
$\noisestd^2 / \LASSOrad^2$, and $c_s$.
In the case $\usedim = 2$, this bound can be tightened to
\begin{equation}\label{min2}
\MinimaxRiskEM
\ge C \parens*{\frac{\noisestd^2 \LASSOrad}{\numobs}}^{\frac{2}{3}}
\log \left(\frac{\LASSOrad\sqrt{\numobs}}{\noisestd}
\right).
\end{equation}
\end{theorem}
Note that the assumption $\numobs_j \ge c_s \numobs^{1/\usedim}$ for
all $j$ is reasonable, since if,
for instance, $\numobs_{\usedim
'+1} = \numobs_{\usedim' + 2} \cdots = \numobs_\usedim = 1$
then we simply have a $\usedim'$-dimensional problem where $\usedim' < \usedim$,
which should have a smaller minimax risk.

As mentioned before, the above result shows that some dependence on
dimension $\usedim$ in the logarithmic term cannot be avoided for any
estimator. Note also, that for $\usedim = 2$, the minimax lower bound
\eqref{min2} matches our upper bound in
\autoref{theorem:NNLS_worst_case} implying minimaxity of $\EMfitfun$
for $d = 2$. For $\usedim > 2$, there remains a gap of $\log n$ between
our minimax lower bound and the upper bound in \autoref{theorem:NNLS_worst_case}.
This gap is due to a logarithmic gap between an upper bound and lower bound
given by \citet[Theorem 1.1]{blei2007metric}
for the metric entropy of cumulative distribution functions
of probability measures on $[0, 1]^\usedim$,
a gap that essentially reduces to improving estimates of a small ball probability
of Brownian sheets (see discussion in \cite{blei2007metric} for more detail and references).

As mentioned earlier, the logarithmic factor $(\log n)^{3d/2}
(\log \log n)^{(2d-1)/2}$ appearing in the second term of
\eqref{equation:worst_case} can be removed by a modification of the
estimator $\EMfitfun$. This is shown in the next result. For a tuning
parameter $\LASSOrad \ge 0$, let
\begin{equation}
\EMfitfuntwo \in \argmin_{f \in \EMClass : \HKVar(f) \le \LASSOrad}
\frac{1}{\numobs} \sum_{i=1}^\numobs (y_ i- f(\xvec_i))^2.
\end{equation}
Note that this differs from the original estimator~\eqref{equation:cmfitfun}
only by the introduction of the additional constraint $\HKVar(f) \le \LASSOrad$.

\begin{theorem}
\label{theorem:constrained_em}
Let $\fstar \in \EMClass$ and $V^* \defn \HKVar(\fstar; [0,1]^\usedim)$.
Assume the lattice design~\eqref{equation:lattice_design}.
If the tuning parameter $V$ is such that $V \geq V^*$,
then the estimator $\EMfitfuntwo$ satisfies
\begin{equation}
\label{equation:constrained_em}
\Risk(\EMfitfuntwo, \fstar)
\le
C_\usedim \parens*{\frac{\noisestd^2 V}{\numobs}}^{\frac{2}{3}}
\parens*{
    \log \left(2 + \frac{V \sqrt{\numobs}}{\noisestd} \right)
}^{\frac{2 \usedim - 1}{3}}
+
C_\usedim \frac{\noisestd^2}{\numobs}.
\end{equation}
\end{theorem}
Note that the second term in \eqref{equation:constrained_em} is just
$\sigma^2/n$ and smaller than the second term in
\eqref{equation:worst_case} but this comes at the cost of introducing
a tuning parameter $\LASSOrad$ that needs to be at least $V^*$.

We will now prove near-parametric rates for $\EMfitfun$ when $f^*$ is
rectangular piecewise constant. To motivate these results, note first
that when $f^*$ is constant on $[0, 1]^d$, we have $V^* = 0$ and thus
the bound given by \eqref{equation:worst_case} is $\sigma^2/n$ up to
logarithmic factors. In the next result (proved in
\autoref{section:nnls_adapt_proof}), we
generalize this fact and show that $\EMfitfun$ achieves
nearly the parametric rate for rectangular piecewise constant
functions $f^* \in \EMClass$. Recall the definition of the class $\rpc$
of all rectangular piecewise constant functions and the associated
mapping $k(f), f \in \rpc$, from \autoref{definition:peecee}.

\begin{theorem}
\label{theorem:NNLS_adapt}
For every
$\fstar : [0, 1]^\usedim \to \R$, the LSE $\EMfitfun$ satisfies
\begin{equation}
\Risk(\EMfitfun, \fstar)
\le
\inf_{f \in \rpc \cap \EMClass}
\left\{
\Loss(f, f^*)
    + C_\usedim \noisestd^2 \frac{
        \sparrect(f)
    }{\numobs}
    (\log (e \numobs))^{\frac{3\usedim}{2}}
    (\log (e \log (e \numobs)))^{\frac{2 \usedim - 1}{2}}
\right\}.
\end{equation}
\end{theorem}

\autoref{theorem:NNLS_adapt}  gives a sharp
oracle inequality in the sense of \cite{bellec2018sharp} as it applies
to every function $f^*$ (even in the misspecified case when $f^*
\notin \EMClass$) and the
constant in front of the first term inside the infimum equals 1. Even
though the inequality holds for every $f^*$, the right hand side will
be small only  when $f^*$ is close to some function $f$ in $\rpc \cap
\EMClass$. This implies that when $f^* \in \rpc
\cap \EMClass$, we can take $f = f^*$ in the right hand side to
obtain that the risk of $\EMfitfun$ decays as $\sigma^2k(f^*)/n$ up to
logarithmic factors. This rate will be faster than the rate given by
\autoref{theorem:NNLS_worst_case} provided $k(f^*)$ is not too
large. Note that one can combine the two bounds given by
\autoref{theorem:NNLS_worst_case} and \autoref{theorem:NNLS_adapt} by
taking their minimum.  In the case $\usedim = 1$, \autoref{theorem:NNLS_adapt}
reduces to the adaptive rates for
isotonic regression \cite{chatterjee2015risk,bellec2018sharp} but with
worse logarithmic factors.

We would also like to mention here that $\rpc \cap \EMClass$ is a
smaller class compared to $\rpc \cap \MClass$ (recall that $\MClass$
is defined via \eqref{equation:m_def}). Risk results over the
class $\rpc \cap \MClass$ for the LSE over
$\MClass$ and other related estimators have been proved in
\citet{han2017isotonic} and \citet{deng2018isotonic}.

Before closing this subsection, let us briefly describe the main ideas
underlying the proofs of Theorems~\ref{theorem:NNLS_worst_case},
\ref{theorem:em_minimax}, \ref{theorem:constrained_em} and
\ref{theorem:NNLS_adapt}. For Theorem  \ref{theorem:NNLS_worst_case},
we use standard results on the accuracy of LSEs on closed convex sets
which related the risk of $\EMfitfun$ to covering numbers of local balls of
the form $\left\{f \in \EMClass : \Loss(f, f^*) \leq t^2 \right\}$ for
$t > 0$ sufficiently small in the pseudometric given by the square-root of
the loss function $\Loss$. We calculated the covering numbers of these
local balls by relating the functions
in $\EMClass$ to distribution functions of signed measures on $[0,
1]^d$ and using existing covering number results for
distribution functions of signed measures from
\citet{blei2007metric} and \citet{gao2013bracketing}. The proof of
\autoref{theorem:em_minimax} is also based on covering number arguments
as we use general minimax lower bounds from
\citet{YangBarron}. Finding lower bounds for the covering numbers
under the pseudometric $\sqrt{\Loss}$ seems somewhat involved and we
used a multiscale construction from \citet[Section
4]{blei2007metric} for this purpose. The bound in \autoref{theorem:constrained_em} for $\EMfitfuntwo$
is a quick consequence of the proof of the
risk bound for $\HKfitfun$ (\autoref{theorem:lasso_worst_case})
which is stated in the next subsection. For
\autoref{theorem:NNLS_adapt}, we used standard results relating
$\Risk(\EMfitfun, \fstar)$ to a certain size-related measure
(statistical dimension) of the
tangent cone to $\EMfitfun$ at $\fstar$. When $f^* \in \rpc$ (or when
$f^*$ is approximable by a function in $\rpc$), this tangent cone is decomposable into tangent cones of certain lower-dimensional
tangent cones. The statistical dimension of these lower-dimensional
tangent cones is then bounded via an application of
\autoref{theorem:NNLS_worst_case} in the case when $V^* = 0$.

\subsection{Risk results for \texorpdfstring{$\HKfitfun$}{HK variation denoising}}
\label{section:lasso}
In this subsection, we present bounds on the risk $\Risk(\HKfitfun,
f^*)$ of the estimator $\HKfitfun$. Note that the estimator
$\HKfitfun$ involves a tuning parameter $V$ and therefore these
results will require some conditions on $V$. Our first result below
assumes that $V \geq V^* := \HKVar(\fstar; [0, 1]^\usedim)$  and gives
the $n^{-2/3}$ rate up to logarithmic factors. The proof of this
result is given in~\autoref{section:lasso_worst_proof}.

\begin{theorem}
\label{theorem:lasso_worst_case}
Assume the lattice design~\eqref{equation:lattice_design}. If the
tuning parameter $V$ is such that $V \geq V^* := \HKVar(\fstar; [0,
1]^\usedim)$, then the estimator $\HKfitfun$ satisfies
\begin{equation}
\label{equation:lasso_worst_case}
\Risk(\HKfitfun, \fstar)
\le
C_\usedim \parens*{\frac{\noisestd^2 V}{\numobs}}^{\frac{2}{3}}
\parens*{
    \log \left(2 + \frac{V \sqrt{\numobs}}{\noisestd} \right)
}^{\frac{2 \usedim - 1}{3}}
+
C_\usedim \frac{\noisestd^2}{\numobs}.
\end{equation}
\end{theorem}

\begin{remark}\label{remark:vdg}
As mentioned earlier, \citet{mammen1997locally} (see also the very recent paper
\citet{ortelli2019oracle}) proposed the estimator \eqref{latd2mv} that
is similar to $\HKfitfun$. \citet{mammen1997locally} also proved a
risk result for their estimator giving the rate $n^{-(1+d)/(1+2d)}$
which is strictly suboptimal compared to our rate in
\eqref{equation:lasso_worst_case} for $d \geq 2$. This suboptimality is
likely due to the use of suboptimal covering number bounds in
\cite{mammen1997locally}.
\end{remark}

\begin{remark}[Model misspecification]
\label{remark:hk_misspec}
\autoref{theorem:lasso_worst_case} is stated under the well-specified assumption
$\HKVar(\fstar; [0, 1]^\usedim) \le \LASSOrad$.
In the misspecified setting where $\HKVar(\fstar; [0, 1]^\usedim) > \LASSOrad$, our LSE $\HKfitfun$ will not be close to $\fstar$, but  to
$\ftilde \in \argmin_{f : \HKVar(f) \le \LASSOrad} \sum_{i=1}^\numobs (f(\xvec_i) - \fstar(\xvec_i))^2$,
so it is reasonable to consider $\Risk(\HKfitfun, \ftilde)$
rather than $\Risk(\HKfitfun, \fstar)$.
By the argument outlined in \autoref{remark:gencha_misspec},
$\Risk(\EMfitfun, \ftilde)$
is upper bounded by the right hand side of~\eqref{equation:lasso_worst_case}.
\end{remark}

In the next result, we prove a
complementary minimax lower bound to
\autoref{theorem:lasso_worst_case} which proves that, for $\usedim
\geq 2$, the risk of every
estimator over the class $\{\fstar : \HKVar(\fstar) \leq V\}$ is bounded from
below by $n^{-2/3} (\log n)^{2(d-1)/3}$ (ignoring terms depending on
$d$, $V$ and $\sigma$).  This implies that the logarithmic terms in
\eqref{equation:lasso_worst_case} can perhaps be reduced slightly but
cannot be removed altogether and must necessarily increase with the
dimension $d$. Let
\begin{equation}
\MinimaxRiskHK \defn \inf_{\fhat_\numobs}
\sup_{\fstar : \HKVar(\fstar) \le \LASSOrad}
\E_{\fstar} \Loss(\fhat_\numobs, \fstar),
\end{equation}
where the expectation is with respect to model~\eqref{obmod}. Note
that $\{\fstar \in \EMClass : \HKVar(\fstar) \leq V\} \subseteq
\{\fstar : \HKVar(\fstar) \leq V\}$ which implies that
\begin{equation*}
\MinimaxRiskHK \geq \MinimaxRiskEM
\end{equation*}
where $\MinimaxRiskEM$ is defined in \eqref{minem}. This implies, in
particular, that the lower bounds on $\MinimaxRiskEM$ from
\autoref{theorem:em_minimax} are also lower bounds on
$\MinimaxRiskHK$. However the next result (whose proof is in \autoref{section:minimax_proof}) gives a strictly larger lower bound for
$\MinimaxRiskHK$ for $d > 2$ than that given by
\autoref{theorem:em_minimax}.

\begin{theorem}\label{THEOREM:MINIMAX}
Let $d \geq 2$, $V > 0$, $\sigma > 0$ and let
$\numobs_j \ge c_s \numobs^{1/\usedim}$ for $j = 1, \ldots, \usedim$,
where $c_s \in (0, 1]$.
Then there exists a positive constant $C_\usedim$ depending only on $\usedim$ and $c_s$, such that
\begin{equation}\label{minlobo}
\MinimaxRiskHK
\ge
C_\usedim
\parens*{
    \frac{\noisestd^2 \LASSOrad}{\numobs}
}^{\frac{2}{3}}
\left(\log \left( \frac{\LASSOrad \sqrt{\numobs}}{\noisestd} \right) \right)^{\frac{2 (\usedim - 1)}{3}}
\end{equation}
provided $\numobs$ is larger than a positive constant $c_{\usedim,
  \noisestd^2 / \LASSOrad^2}$ depending only on $\usedim$,
$\noisestd^2 / \LASSOrad^2$,
and $c_s$.
In the case $\usedim = 2$, this bound can be tightened to
\begin{equation}
\MinimaxRiskHK
\ge C \parens*{\frac{\noisestd^2 \LASSOrad}{\numobs}}^{\frac{2}{3}}
\log \left( \frac{\LASSOrad \sqrt{\numobs}}{\noisestd} \right).
\end{equation}
\end{theorem}

Theorems \ref{theorem:lasso_worst_case} and \ref{THEOREM:MINIMAX}
together imply that $\HKfitfun$ is minimax optimal over $\{\fstar:
\HKVar(\fstar) \leq V\}$ for $d = 2$ and only possibly off by a factor
of $(\log n)^{1/3}$ for $d > 2$.

We next explore the possibility of near parametric rates for
$\HKfitfun$ for rectangular piecewise constant functions. In the
univariate case $\usedim = 1$, it is known
(see \cite[Theorem 2.2]{guntuboyina2017spatial}) that
$\HKfitfun$ satisfies the near-parametric risk bound
\eqref{hkad1}  provided (a) the tuning parameter $V$ is taken to be
close to $V^*$, (b) $f^*$ is piecewise constant, and (c) the length of
each constant piece of $f^*$ is bounded from below by $c/k(f^*)$ for
some $c > 0$. The next result (proved
in \autoref{section:proof_lasso_adaptive}) provides evidence that a
similar story holds true for estimating certain rectangular piecewise
constant functions.

For a given constant $0 < c \le 1/2$, let $\OneJumpClassStrong(c)$
denote the collection of functions $f: [0, 1]^\usedim \to \R$
of the form
\begin{equation}\label{equation:two_piece}
f = a_1 \Ind_{[\xvec^*, \onevec]} + a_0
\end{equation}
for some $a_1, a_0 \in \R$ and $\xvec^* \in [0, 1]^\usedim$
satisfying the minimum size condition
\begin{equation}\label{equation:min_length_strong}
\min\{
    \abs{\LatticeDesign \cap [\xvec^*, \onevec]},
    \abs{\LatticeDesign \cap [\zerovec, \xvec^*)}
\}
\ge c \numobs.
\end{equation}
To gain more intuition about the above condition, note first that we
are working with the lattice design so that $\LatticeDesign =
\{\xvec_1, \dots, \xvec_n\}$ is the set containing all design
points. Roughly speaking, \eqref{equation:min_length_strong} ensures
that $\xvec^*$ is not too
close to the boundary of $[0, 1]^\usedim$ so that each of the
rectangles $[\xvec^*, \onevec]$ and $[\zerovec, \xvec^*)$ contain at
least some constant fraction of the $\numobs$ design points.

It is clear that $\OneJumpClassStrong(c)$ is a subset of $\rpc$, i.e., every function
of the form \eqref{equation:two_piece} is rectangular piecewise
constant. Indeed, it is easy to see that $k(f) \leq 2^d$ for every $f
\in \OneJumpClassStrong(c)$. The following result (proved in
\autoref{section:proof_lasso_adaptive}) bounds the risk of
$\HKfitfun$ for $f^* \in \OneJumpClassStrong(c)$.

\begin{theorem}\label{theorem:lasso_adaptive_d}
Consider the lattice design \eqref{equation:lattice_design} with $\numobs > 1$. Fix
$\fstar : [0, 1]^d \to \R$ and consider the estimator $\HKfitfun$ with
a tuning parameter $\LASSOrad$. Then for every $0 < c \leq 1/2$, we
have
\begin{align}\label{equation:lasso_onejump}
\Risk(\HKfitfun, \fstar)
\le
\inf_{\substack{f \in \OneJumpClassStrong(c) : \\ \HKVar(f) = \LASSOrad}}
&\left\{
   \Loss(f, \fstar) + C(c, \usedim) \frac{\noisestd^2}{\numobs}
    (\log \numobs)^{\frac{3 \usedim}{2}} (\log \log  \numobs)^{\frac{2 \usedim - 1}{2}}
\right\}
\end{align}
for a constant $C(c, \usedim)$ that depends only on $c$ and $\usedim$.
\end{theorem}

\autoref{theorem:lasso_adaptive_d} applies to every function
$f^*$  but the infimum on the right hand side of
\eqref{equation:lasso_onejump} is over all functions $f$ in
$\OneJumpClassStrong(c)$ with $\HKVar(f) = V$. Therefore,
\autoref{theorem:lasso_adaptive_d} implies that the risk of
the estimator $\HKfitfun$ with tuning parameter $V$ at $\fstar$ is the
near-parametric rate $\frac{\sigma^2}{n} (\log en)^{3d/2} (\log \log
n)^{(2d-1)/2}$ provided $f^*$ is close to some function $f$ in
$\OneJumpClassStrong(c)$ with $V = \HKVar(f)$. As an immediate
consequence, we obtain that if  $\fstar \in \OneJumpClassStrong(c)$
and $\LASSOrad = \HKVar(\fstar)$, then
\begin{equation}
\Risk(\HKfitfun, \fstar)
\le C(c, \usedim)
\frac{\noisestd^2}{\numobs} (\log (e \numobs))^{\frac{3 \usedim}{2}}
(\log (e \log (e \numobs)))^{\frac{2d-1}{2}}.
\end{equation}
Functions in $\OneJumpClassStrong(c)$ are constrained to
satisfy the minimum size condition
\eqref{equation:min_length_strong}. A comparison of
\autoref{theorem:lasso_adaptive_d}  with the
corresponding univariate
results shows that the near-parametric rate cannot be achieved without
any minimum size condition (see e.g., \cite[Remark
2.5]{guntuboyina2017spatial} and \cite[Section
4]{fan2018approximate}). However, condition
\eqref{equation:min_length_strong} might sometimes be too
stringent for $d \geq 2$. For example, it rules out the case when $\xvec^*
:= (0.5, 0,
\dots, 0)$ which means that the function class $\OneJumpClassStrong(c)$
excludes simple functions such as $f(\xvec) \defn \Ind\{x_1 \ge
1/2\}$. In \autoref{theorem:lasso_adaptive}
(deferred to \autoref{section:adaptation2}),
we show that when $\usedim = 2$, it is
possible to obtain the same risk bound under a weaker minimum size
condition which does not rule out functions such as $f(\xvec) \defn \Ind\{x_1
\ge 1/2\}$.

The implication of Theorems~\ref{theorem:lasso_adaptive_d} and
\ref{theorem:lasso_adaptive} is that there exists a subclass of
$\rpc$ consisting of indicators of upper right rectangles in $[0,
1]^d$ over which the  estimator $\HKfitfun$, when ideally tuned,
achieves the near-parametric rate with some logarithmic
factors. Simulations (see \autoref{section:adaptation_simulation})
indicate that this
should also be true for a larger subclass of $\rpc$ consisting of all
functions in $\rpc$ satisfying some minimum size condition, but
our proof technique does not currently work in this generality.
\citet{ortelli2018total} recently proved, for $d = 2$, near-parametric rates for
the estimator \eqref{latd2mv} for a more general class of piecewise constant functions,
but for a smaller loss function. Their proof technique is completely
different from our approach.

Let us now briefly discuss the key ideas behind the proofs of
Theorems
\ref{theorem:lasso_worst_case}, \ref{THEOREM:MINIMAX} and
\ref{theorem:lasso_adaptive_d}. \autoref{theorem:lasso_worst_case}
is proved via covering number arguments which relate $\Risk(\HKfitfun,
\fstar)$ to covering numbers of $\{f : \HKVar(f) \leq V\}$ and these
covering numbers are controlled by invoking connections to
distribution functions of signed measures. \autoref{THEOREM:MINIMAX}
is proved by Assouad's lemma with a multiscale construction of
functions with bounded HK$\zerovec$ variation. This multiscale
construction is involved and taken from \citet[Section
4]{blei2007metric}.

The ideas for the proof of Theorem~\ref{theorem:lasso_adaptive_d} (and also
\autoref{theorem:lasso_adaptive}) is borrowed from the proofs for the univariate case in
\citet{guntuboyina2017spatial} although the situation for $d \geq 2$
is much more complicated. At a high level, we use tangent cone
connections where the goal
is to control an appropriate size measure (Gaussian width) of the
tangent cone of $\{f : \HKVar(f) \leq V^*\}$
at $f^*$. This tangent cone can be explicitly computed (see
\autoref{lemma:tangent_cone}). To bound its Gaussian width, our key
observation is that for functions $f^*$ in $\OneJumpClassStrong(c)$,
every element of the tangent cone can be broken
down into lower-dimensional elements each of which is either nearly
entirely monotone or has low HK$\zerovec$ variation. The Gaussian
width of the tangent cone can then be bounded by a combination of
(suitably strengthened) versions of \autoref{theorem:NNLS_adapt} and
\autoref{theorem:lasso_worst_case}. This method unfortunately does not
seem to work for arbitrary functions $f^* \in \rpc$ because of certain
technical issues which are mentioned in \autoref{remark:many_jumps}.

\section{On the ``dimension-independent'' rate \texorpdfstring{$n ^{-2/3}$}{n(-2/3)} in
  \autoref{theorem:NNLS_worst_case} and
  \autoref{theorem:lasso_worst_case}} \label{section:discussion}

As mentioned previously, the dimension $d$ appears in the bounds given
by \autoref{theorem:NNLS_worst_case} and
\autoref{theorem:lasso_worst_case} only through the logarithmic term
which means that $\EMfitfun$ and $\HKfitfun$ attain
``dimension-independent rates'' ignoring logarithmic factors. We shall
provide some insight and put these results in proper historical
context in this section. In nonparametric statistics, it is well-known
 that the rate of estimation of smooth functions based on $n$
 observations is $n^{-2m/(2m+d)}$ where $d$ is the dimension and $m$
 is the order of smoothness \cite{stone1982optimal}.
 The constraints of entire monotonicity
 and having finite HK$\zerovec$ variation can be loosely viewed as
 smoothness constraints of order $m = d$. This is because, for smooth functions
 $f$, entire monotonicity is equivalent to
 \begin{equation*}
   \frac{\partial^{|S|} f}{\prod_{j \in S} \partial x_{j}} \geq 0
   \qt{for every $\emptyset \neq S \subseteq \{1, \dots, d\}$}
 \end{equation*}
 and the constraint of finite HK$\zerovec$ variation is equivalent to
 \begin{equation}\label{ohk}
   \frac{\partial^{|S|} f}{\prod_{j \in S} \partial x_{j}} \in L^1
   \qt{for every $\emptyset \neq S \subseteq \{1, \dots, d\}$}.
 \end{equation}
Because derivatives of order $d$ appear in these expressions, these
constraints should be considered as smoothness constraints of order
$d$. Note that taking $m = d$ in $n^{-2m/(2m+d)}$ gives $n^{-2/3}$.

Some other papers which studied such higher order constraints to
obtain estimators having nearly dimension-free rates include \cite{barron1993universal,   lin2000tensor, chkifa2018polynomial,
  niyogi1999generalization, sadhanala2017higher,
  wahba1995smoothing}. In particular, \citet{lin2000tensor} studied
estimation under the constraint:
 \begin{equation}\label{linhk}
   \frac{\partial^{|S|} f}{\prod_{j \in S} \partial x_{j}} \in L^2
   \qt{for every $\emptyset \neq S \subseteq \{1, \dots, d\}$}.
 \end{equation}
The difference between \eqref{ohk} and \eqref{linhk} is that $L^1$ in \eqref{ohk} is
replaced by $L^2$ in \eqref{linhk}. \citet{lin2000tensor} proved that
the minimax rate of convergence under \eqref{linhk} is $n^{-2/3} (\log
n)^{2(d-1)/3}$ and constructed a linear estimator which is optimal
over the class \eqref{linhk}. Let us remark here that the $L^2$ constraint makes
the class smaller compared to \eqref{ohk} and also enables linear estimators to achieve the
optimal rate. However, linear estimators will not be optimal over $\{f
: \HKVar(f) \leq V \}$ as is well-known in $d = 1$ (see
\citet{donoho98minimaxwavelet}) and the estimator of
\citet{lin2000tensor} will also not adapt to rectangular piecewise
constant functions (note that it is not possible to extend
\eqref{linhk} to nonsmooth functions in such a way that the constraint
is satisfied by rectangular piecewise constant functions).

Let us also mention here that, in approximation theory, it is known
that classes of smooth functions $f$ on $[0, 1]^d$ satisfying mixed
partial derivative constraints such as \eqref{ohk} or \eqref{linhk}
allow one to overcome the curse of dimensionality to some extent from
the perspective of metric entropy, approximation and interpolation
(see e.g., \cite{donoho2000high, temlyakov2018multivariate,
  bungartz2004sparse}).

Another way to impose higher order smoothness is to impose the
constraint:
\begin{equation}\label{sadh}
  \frac{\partial^d f}{\partial x_j^d} \in L^1 \qt{for each $j = 1,
    \dots, d$}
\end{equation}
as in the Kronecker Trend filtering method of order $k + 1 = d$ of
\citet{sadhanala2017higher} who also proved that this leads to the
dimension-free rate $n^{-2/3}$ up to logarithmic factors. There are
some differences between the constraints \eqref{ohk} and
\eqref{sadh}. For example, product functions $f(x_1, \dots, x_d) :=
f_1(x_1) \dots f_d(x_d)$ satisfy \eqref{ohk} provided each $f_j$
satisfies $f_j' \in L_1$ while they will satisfy \eqref{sadh} provided
$f_j^{(d)}  \in L_1$.

Finally, let us mention that, in the usual multivariate extensions of
isotonic regression and total variation denoising, one uses
partial derivatives only of the first order  which
leads to rates of convergence that are exponential in the dimension
$d$. For example, the usual multivariate isotonic regression (see
e.g.,~\citet[Section 1.3]{RWD88}) considers the class $\MClass$ of
multivariate monotone functions which only imposes first order
constraints. The rate of convergence here is given by $n^{-1/d}$ as
recently shown in~\citet{han2017isotonic}. This rate is exponentially
slow in the dimension $d$. One sees the same rate behavior for the
multivariate total variation denoising estimator (which also imposes
only first order constraints) originally proposed
by \citet{rudin1992nonlinear} and whose theoretical behavior is
studied in \citet{hutter2016optimal, sadhanala2016total,
  chatterjee2019new, ortelli2019synthesis, ruiz2018frame}.


\section{Another adaptation result for the Hardy-Krause variation denoising estimator}
\label{section:adaptation2}
The goal of this section is to prove a result that is similar to but
stronger than \autoref{theorem:lasso_adaptive_d} for $d =
2$. Specifically, the minimum length condition appearing in
\eqref{equation:min_length_strong} is relaxed for the next result. We
take $\usedim = 2$ in this section. For a given
constant $0 < c \leq 1/2$, let $\OneJumpClass(c)$ denote the collection of functions
$f : [0, 1]^2 \to \R$ of the form~\eqref{equation:two_piece}
for some $a_1, a_0 \in \R$ and $\xvec^* = (x^*_1, x^*_2) \in [0,1]^2$
satisfying
\begin{equation}\label{equation:min_length}
\min\braces*{
\abs{\LatticeDesign \cap [\xvec^*, \onevec]},
\abs{\LatticeDesign \setminus [\xvec^*, \onevec]}
} \ge c \numobs.
\end{equation}
Note that the above condition is implied by the earlier
minimum size condition~\eqref{equation:min_length_strong} because
$[\zerovec, \xvec^*) \subseteq [\xvec^*, \onevec]^c$. Therefore we
have $\rpcplain^2_1(c) \subseteq \OneJumpClass(c)$. Note also that
$\xvec^* := (0.5, 0,
\dots, 0)$ satisfies \eqref{equation:min_length}. The next result
(proved in \autoref{section:proof_lasso_adaptive})
is the analogue of
\autoref{theorem:lasso_adaptive_d} for $d = 2$ which works under the
weaker minimum size condition \eqref{equation:min_length}.

\begin{theorem}\label{theorem:lasso_adaptive}
Consider the lattice design \eqref{equation:lattice_design}. Fix
$\fstar : [0, 1]^2 \to \R$ and consider the estimator $\HKfitfun$ with
a tuning parameter $\LASSOrad$. Then for every $0 < c \leq 1/2$, we
have
\begin{align}\label{equation:lasso_onejump_strong}
\Risk(\HKfitfun, \fstar)
\le
\inf_{\substack{f \in \OneJumpClass(c) : \\ \HKVar(f) = \LASSOrad}}
&\left\{ \Loss(f, f^*)
    + C(c) \frac{\noisestd^2}{\numobs}
    (\log (e \numobs))^{3} (\log (e \log (e n)))^{\frac{3}{2}}
\right\}
\end{align}
for a constant $C(c)$ that depends only on $c$.
\end{theorem}
When $\fstar \in \OneJumpClass(c)$ and $\LASSOrad = \HKVar(\fstar)$,
inequality
\eqref{equation:lasso_onejump_strong} readily implies
\begin{equation}
\Risk(\HKfitfun, \fstar)
\le C \frac{\noisestd^2}{\numobs} (\log (e \numobs))^3 (\log(e \log(e \numobs)))^{\frac{3}{2}}.
\end{equation}
Note that previously we were only able to claim this result for
functions $f^*$ in the smaller class $\rpcplain^2_1(c)$.


\section{Simulation studies}
\label{section:simulation}

Here we discuss some simulations we performed with the two
estimators $\EMfitfun$~\eqref{equation:cmfitfun} and $\HKfitfun$~\eqref{equation:hkfitfun} for $d = 2$.


\subsection{Examples of the estimators}
\label{section:simulation_examples}
We start by visual illustrations of our estimators for specific values
of $f^*$. In \autoref{figure:em1_grid} we depict an example of
$\EMfitfun$
when fit on a $10 \times 10$ grid of observations (i.e., $n_1 = n_2 =
10$ and $n = 100$)
from an EM function $\fstar$.
In \autoref{figure:em2_grid}, we consider a different example
where $f^*$ has $k(\fstar) = 4$
and depict the estimate $\EMfitfun$
computed on a $10 \times 10$ grid of observations.

\begin{figure}[!ht]
\includegraphics[width=\textwidth]{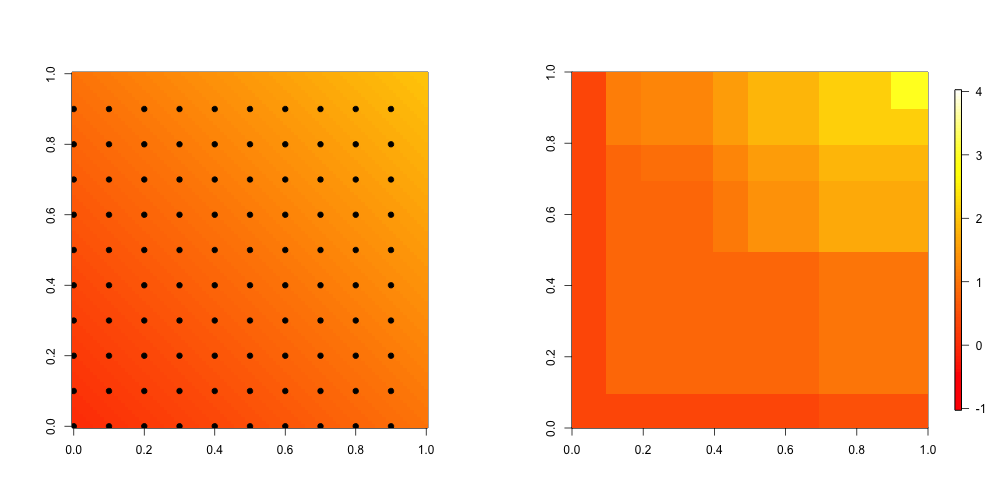}
\caption{
    The function $\fstar(x_1, x_2) = x_1 + x_2$ (left),
    and the estimate $\EMfitfun$ (right)
    performed on observations from $\fstar$ on the grid design ($\numobs_1 = \numobs_2 = 10$)
    with standard Gaussian noise ($\sigma^2 = 1$).
}
\label{figure:em1_grid}
\end{figure}

\begin{figure}[!ht]\centering
\includegraphics[width=\textwidth]{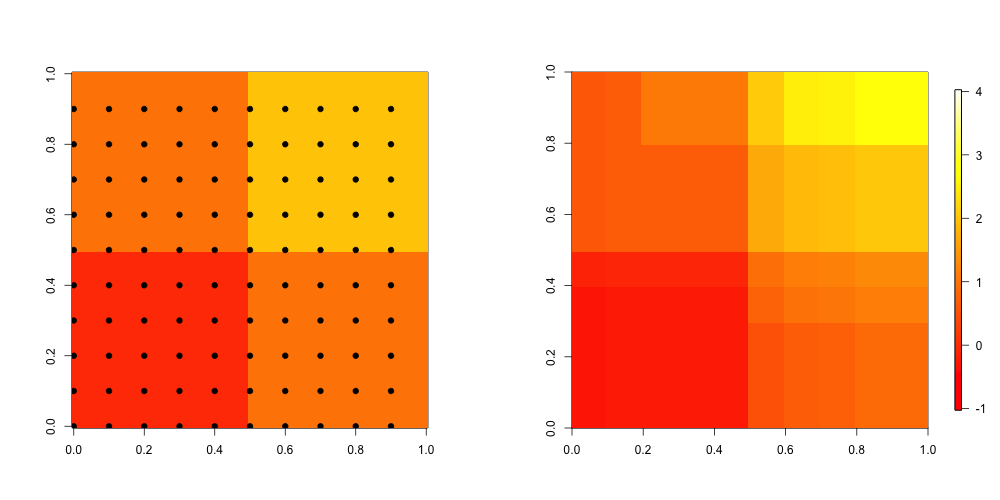}

\caption{
    The function $\fstar(x_1, x_2) = \Ind\{x_1 \ge 0.5\} + \Ind\{x_2 \ge 0.5\}$ (left),
    and the estimate $\EMfitfun$ (right)
    performed on observations from $\fstar$ with the grid design ($\numobs_1 = \numobs_2 = 10$)
    and standard Gaussian noise ($\sigma^2 = 1$).
}
\label{figure:em2_grid}
\end{figure}

In \autoref{figure:hk2_grid} we consider a function $\fstar \in \OneJumpClassStrong(1/4)$
(see equations~\eqref{equation:two_piece} and~\eqref{equation:min_length_strong})
and depict
our estimate $\HKfitfun$ computed from a $10 \times 10$ grid of observations for various values of the tuning parameter $\LASSOrad$.

\begin{figure}[!ht]\centering
\includegraphics[width=\textwidth]{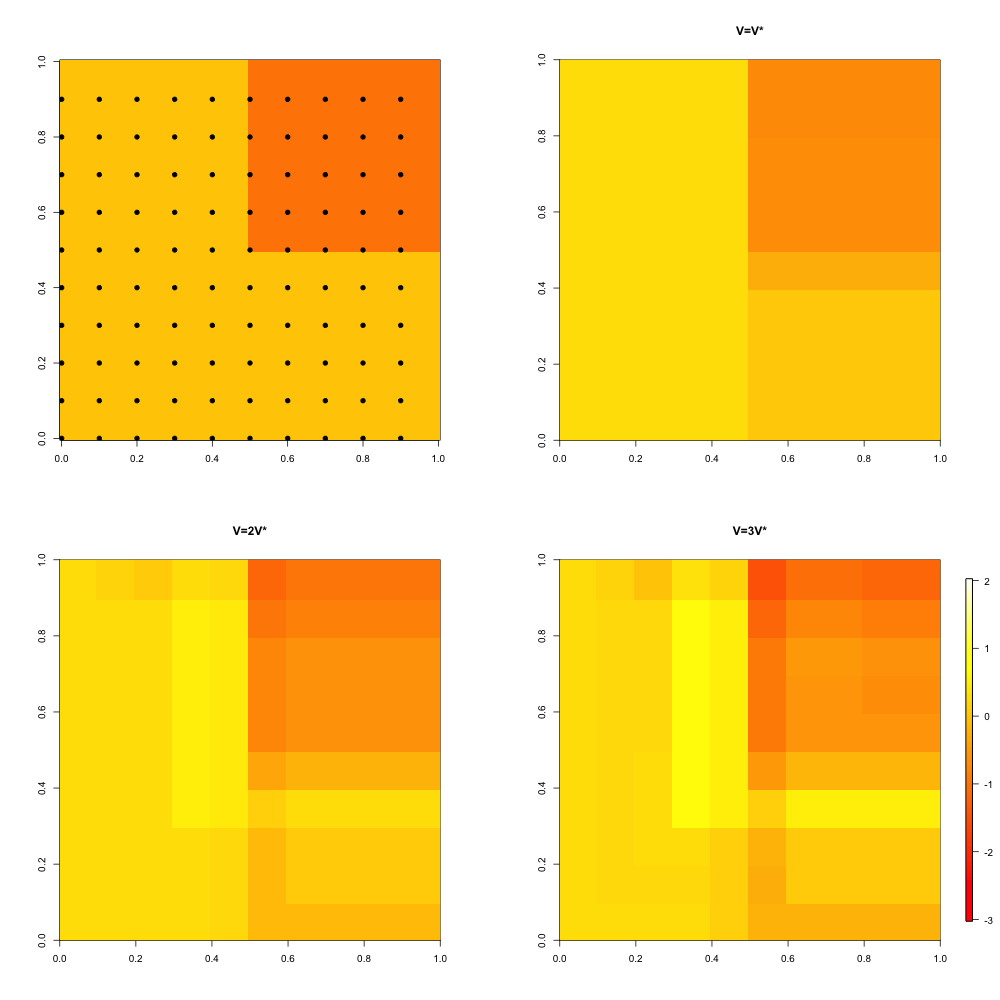}
\caption{
    The function $\fstar(x_1, x_2) = - \Ind\{x_1 \ge 0.5, x_2 \ge 0.5\}$ (upper left),
    and the estimate $\HKfitfun$ for $\LASSOrad = V^*, 2V^*, 3V^*$,
    performed on observations from $\fstar$ on the grid design ($\numobs_1 = \numobs_2 = 10$)
    with standard Gaussian noise ($\sigma^2 = 1$).
}
\label{figure:hk2_grid}
\end{figure}

\begin{figure}[!ht]\centering
\includegraphics[width=\textwidth]{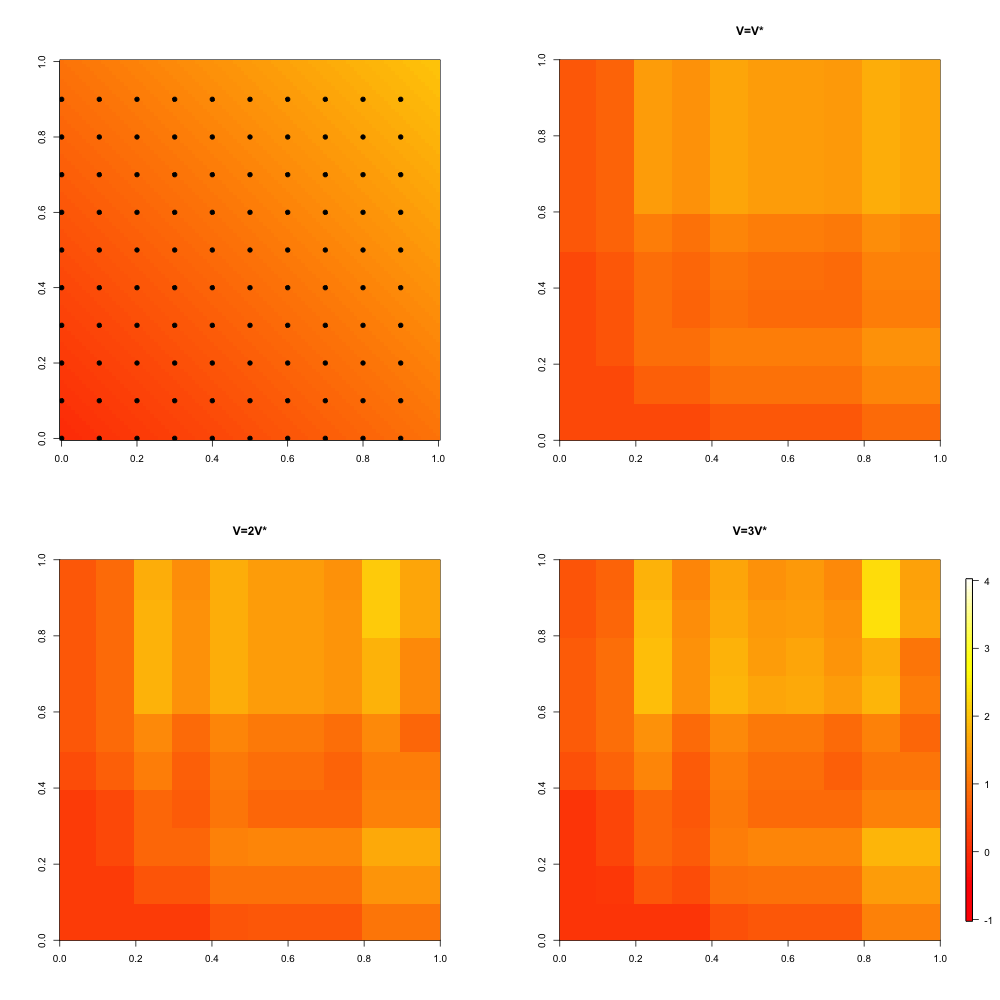}
\caption{
    The function $\fstar(x_1, x_2) = x_1 + x_2$ (upper left),
    and the estimate $\HKfitfun$ for $\LASSOrad = V^*, 2V^*, 3V^*$,
    performed on observations from $\fstar$ on the grid design ($\numobs_1 = \numobs_2 = 10$)
    with standard Gaussian noise ($\sigma^2 = 1$).
}
\label{figure:hk3_grid}
\end{figure}

We remark that in these examples,
we have chosen the estimator to be rectangular piecewise constant,
with values obtained by solving the finite-dimensional NNLS or LASSO problem
as discussed in \autoref{SECTION:COMPUTATION}.
Additionally, one can observe
in Figures~\ref{figure:hk3_grid} and \ref{figure:hk2_grid}
that the performance of $\HKfitfun$ improves as $\LASSOrad$ approaches the optimal $V^*$.
Note also that in Figures~\ref{figure:em2_grid} and \ref{figure:hk2_grid} (in the case $\LASSOrad = V^*$)
where $\fstar$ is rectangular piecewise constant, the estimate is also rectangular piecewise constant with relatively few ``jumps.''

\clearpage

Although our theorems in \autoref{section:nnls}
and \autoref{section:lasso} only apply in case of lattice design,
we can still compute the estimator for arbitrary design.
In \autoref{figure:random}, we used the ``na\"{i}ve gridding''
approach described in \autoref{SECTION:COMPUTATION} to
compute the design matrix for the NNLS optimization problem.
Note that the ``jumps'' in our estimate are located at design points.

\begin{figure}[!ht]\centering
\includegraphics[width=\textwidth]{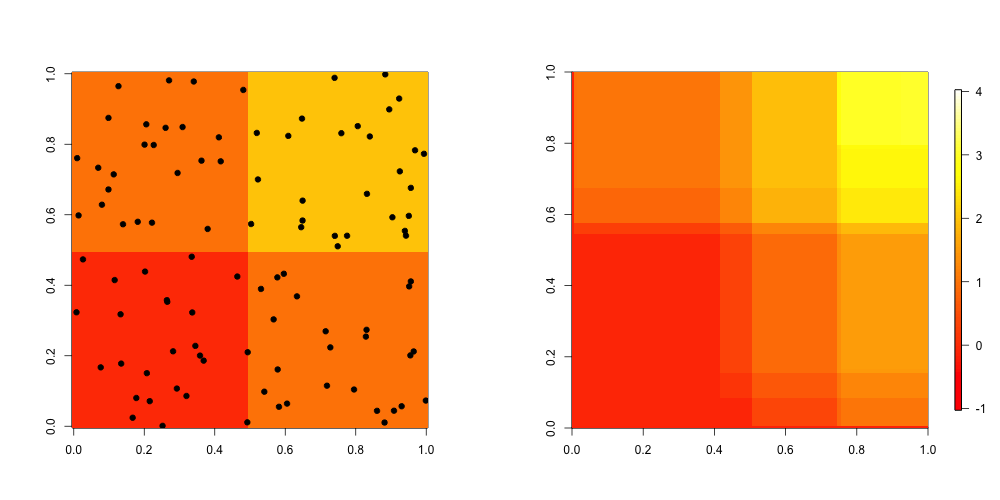}
\caption{
    The function $\fstar(x_1, x_2) = \Ind\{x_1 \ge 0.5\} + \Ind\{x_2 \ge 0.5\}$ (left),
    and the estimate $\EMfitfun$ (right)
    performed on observations from $\fstar$ on a uniformly drawn random design ($\numobs = 100$)
    and standard Gaussian noise ($\sigma^2 = 1$).
}
\label{figure:random}
\end{figure}

\subsection{Bivariate current status model}

One practical setting where our estimator may be useful is
in the bivariate current status model, which is a particular variant of the interval censoring problem
\cite{groeneboom2013bivariate,groeneboom2011estimators,maathuis2005reduction}.
In this setting we
observe $(\xvec_i, y_i)$
where the $y_i$ are independent Bernoulli random variables $y_i$
with success parameter $F_0(\xvec_i)$,
for some  bivariate CDF $F_0$.
Since $F_0$ is
an entirely monotone function of $\xvec$, it is plausible to use
our EM estimator~\eqref{equation:cmfitfun} on these observations to estimate $F_0$.
In \autoref{figure:bern_simple},
we simulated $n = 500$ observations
in the case where $F_0(\xvec) = \frac{1}{2}(x_1^2 x_2 + x_1 x_2^2)$ on $[0,1]^2$
(the CDF of the density $f_0(\xvec) = x_1 + x_2$),
and where $\xvec_i$ are drawn uniformly from $[0, 1]^2$,
and where $y_i \mid \xvec_i \sim \op{Bern}(F_0(\xvec_i))$.
We get a fairly reasonable estimate of the original CDF
on the interior of the square $[0,1]^2$.
The estimated function is not a proper CDF, as it can take values outside of $[0, 1]$,
which happens often along the boundaries of the square $[0, 1]^2$.
One could avoid this by modifying the estimator $\EMfitfun$
by restricting the least squares optimization
to functions in $\EMfitfun$ that take values in $[0, 1]$,
which would amount to adding two more linear constraints on the corresponding
NNLS problem~\eqref{equation:cm_nnls}.
This issue of obtaining an estimate that is not a proper CDF
also occurs with a plug-in estimator studied by \citet{groeneboom2013bivariate},
which they address by proposing a truncation procedure on the boundaries of the square.

\begin{figure}[!ht]\centering
\includegraphics[width=\textwidth]{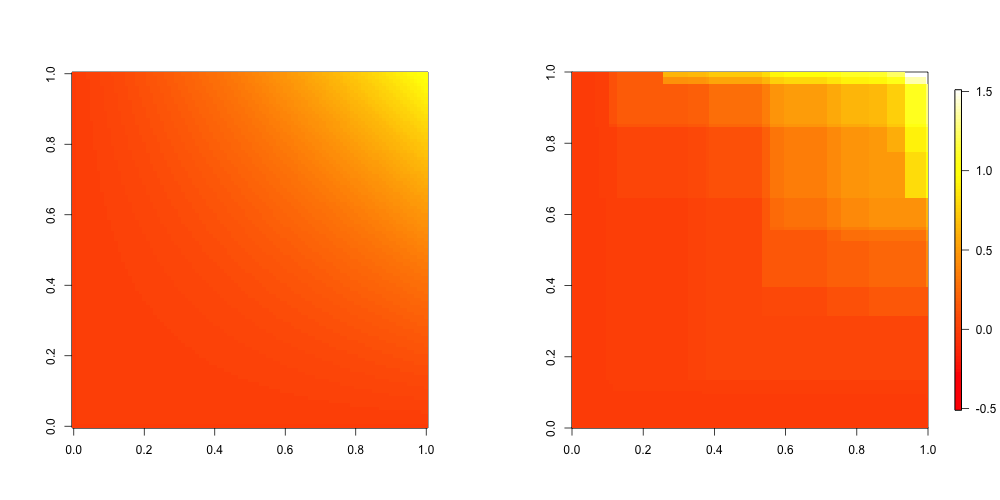}
\caption{The CDF $F_0(x_1, x_2) = \frac{1}{2}(x_1 x_2^2 + x_1^2 x_2)$ (left),
and the estimate $\EMfitfun$ (right) applied on $\numobs = 500$ observations
of the form
$(\xvec_i, y_i)$ where $y_i \sim \op{Bern}(F_0(\xvec_i))$.}
\label{figure:bern_simple}
\end{figure}

\subsection{Adaptation to more general rectangular piecewise constant functions}
\label{section:adaptation_simulation}

One severe limitation of Theorems~\ref{theorem:lasso_adaptive_d}
and \ref{theorem:lasso_adaptive} is that they only consider
functions of the form~\eqref{equation:two_piece}, which only has one ``jump''
and two continguous constant pieces.

The following simulation study suggests that the upper bound of $\numobs^{-1} (\log \numobs)^\gamma$
that we proved in Theorems~\ref{theorem:lasso_adaptive_d}
and \ref{theorem:lasso_adaptive}
may also hold for a larger subclass of rectangular piecewise constant functions $\rpc$.

The function $\fstar : [0,1]^2 \mapsto \R$
we consider is
\begin{equation}
\fstar(\xvec)
= \begin{cases}
1 & \xvec \in
    \parens*{
        [\frac{1}{3}, \frac{2}{3}) \times ([0, \frac{1}{3}) \cup [\frac{2}{3}, 1)
    }
    \cup \parens*{
        ([0, \frac{1}{3}) \cup [\frac{1}{3}, 1]) \times [\frac{1}{3}, \frac{2}{3})
    }
\\
0 & \text{otherwise}.
\end{cases}
\end{equation}
One can check that $V^* = 12$.
Visually, it has a checkered pattern (see \autoref{figure:checkered}).

\begin{figure}[!ht]
\includegraphics[width=\textwidth]{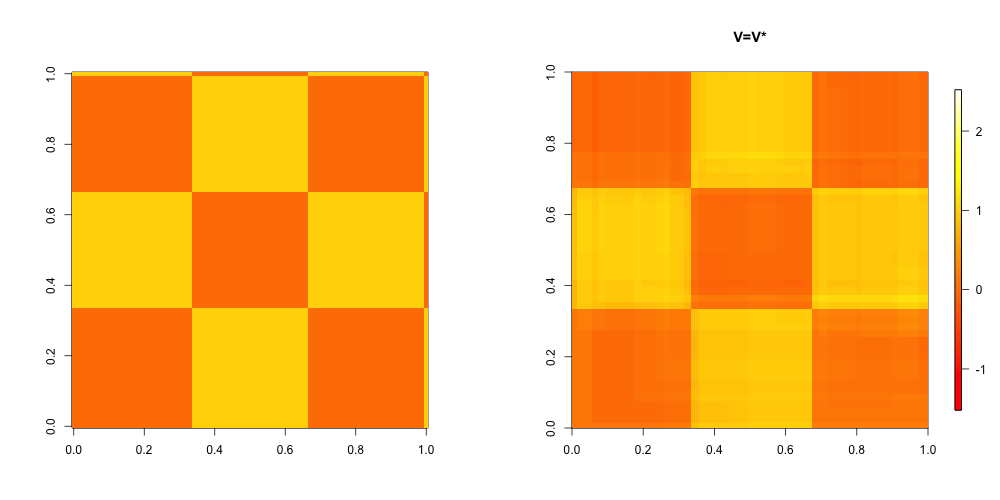}
\caption{Depiction of $\fstar$ (left), and an example of $\HKfitfun$ (right) when
given noisy measurements ($\sigma = 0.5$) from $\fstar$ on the grid design ($\numobs_1 = \numobs_2 = 50$).}
\label{figure:checkered}
\end{figure}

We considered the lattice design $\LatticeDesign$ with $\numobs_1 = \numobs_2 \in \{50, 60, 80, 95, 110\}$
(note that consequently $\numobs = \numobs_1 \numobs_2$ ranges between $2500$ and $12100$).
For each value of $\numobs$, we performed $20$ trials of generating observations $y_1, \ldots, y_\numobs$
with noise $\sigma = 0.5$, computed $\HKfitfun$ with $\LASSOrad = V^*$,
and computed the error $\frac{1}{\numobs}\sum_{i=1}^n (\HKfitfun(\xvec_i) - \fstar(\xvec_i))^2$.
Averaging over the $20$ trials gives us an estimate $r_\numobs$ of $\Risk(\HKfitfun, \fstar)$ for that value of $\numobs$.

As shown in \autoref{figure:loglog}
A linear regression of $\log r_\numobs$ over $\log \numobs$ yielded a slope of $-0.85$
which indicates that the estimator is performing better
than the worst-case rate of $\numobs^{-2/3}$ given in \autoref{theorem:lasso_worst_case}.
A linear regression of $\log r_\numobs$ over $\log \frac{\numobs}{\log \numobs}$
yielded a slope of $-0.96$,
while a regression of $\log r_\numobs$ over $\log \frac{\numobs}{(\log \numobs)^2}$
yielded a slope of $-1.11$.
Thus these simulations suggest that the estimator $\HKfitfun$ has risk on the order of
$\numobs^{-1} (\log \numobs)^\gamma$
(possibly for $\gamma \le 2$)
for rectangular piecewise constant functions
beyond the ones considered in
Theorems~\ref{theorem:lasso_adaptive_d}
and \ref{theorem:lasso_adaptive}.

\begin{figure}[!ht]
\includegraphics[width=0.45\textwidth]{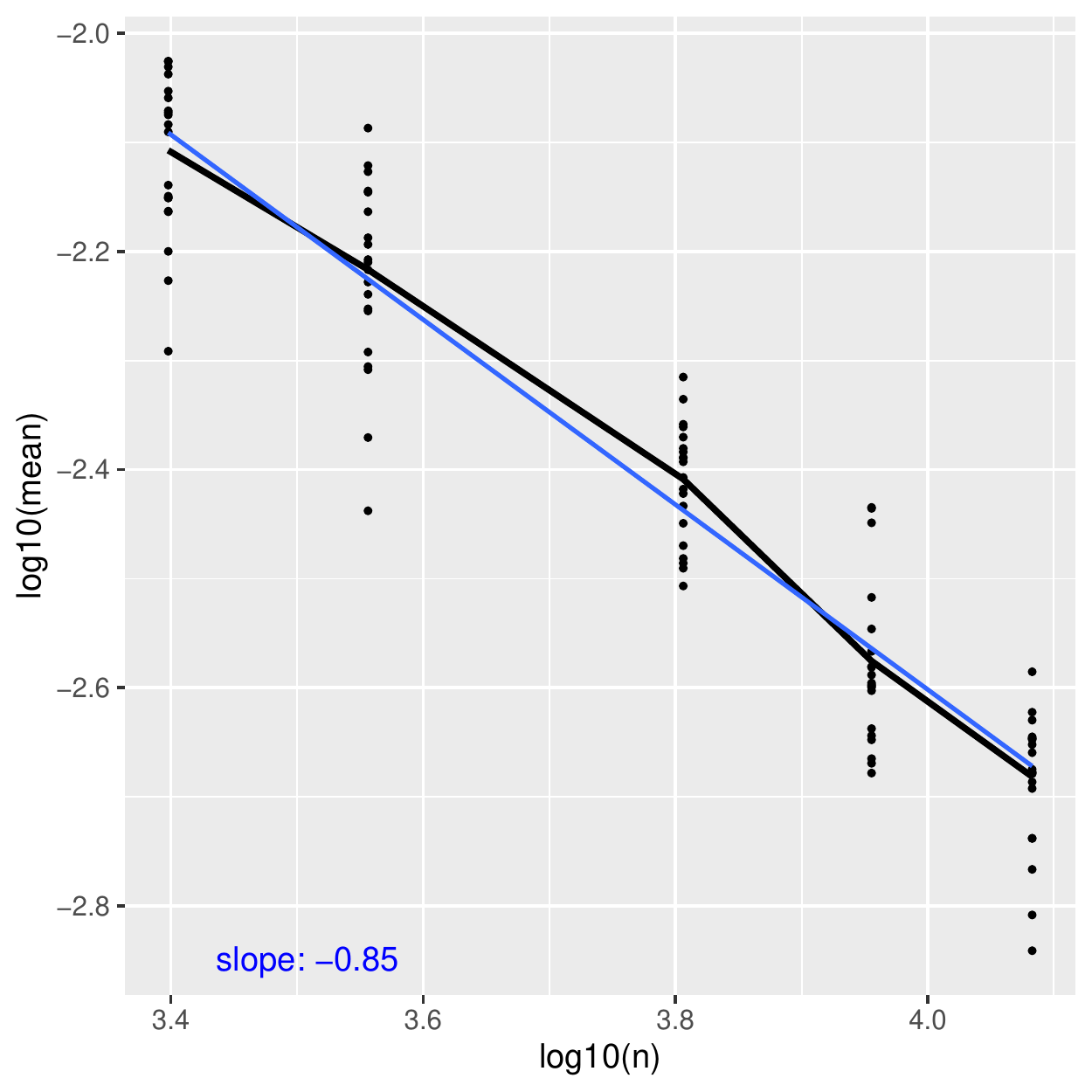}
\includegraphics[width=0.45\textwidth]{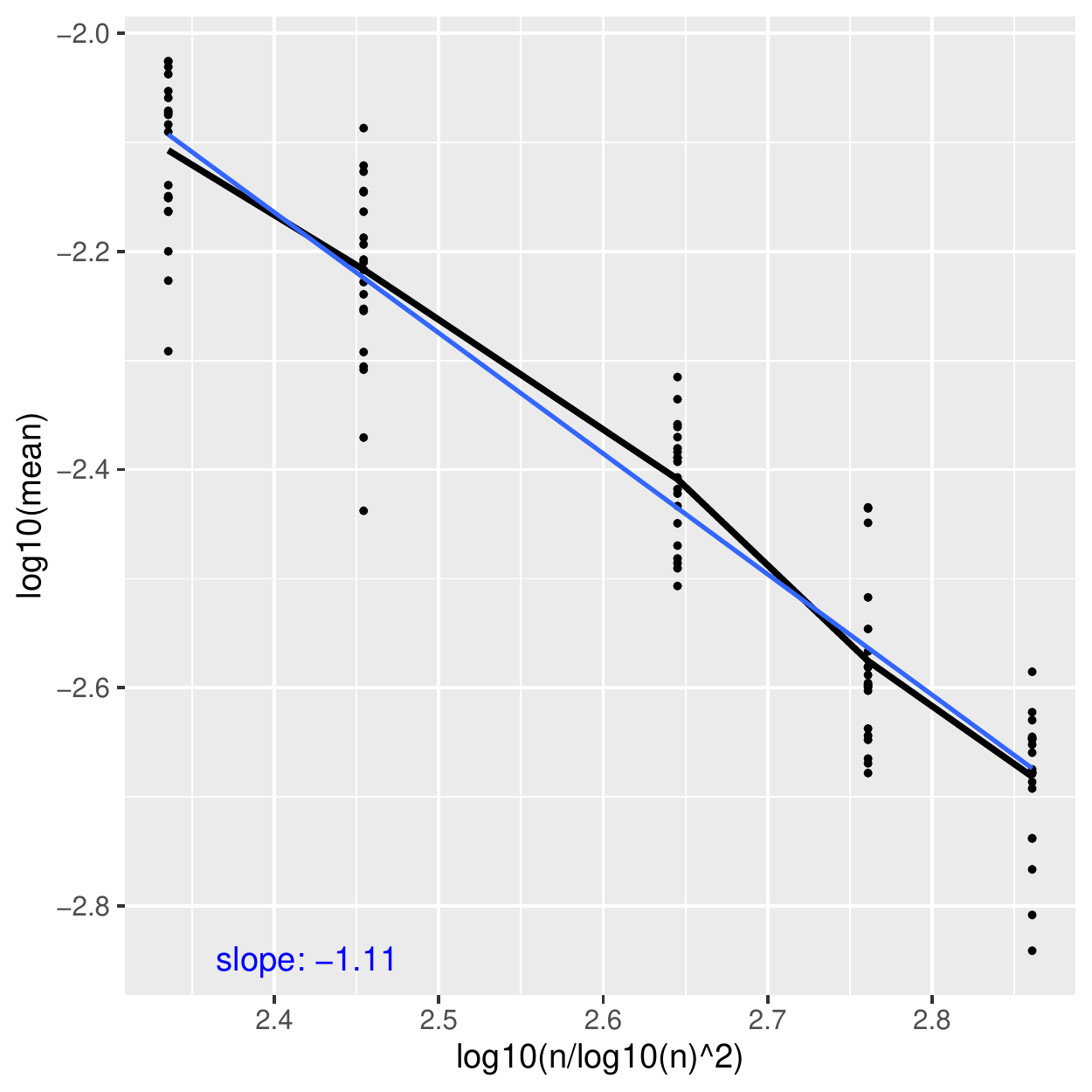}
\caption{Plot of estimate of $\log \Risk(\HKfitfun, \fstar)$ vs. $\log \numobs$ (left) and vs. $\log \frac{\numobs}{(\log \numobs)^2}$ (right).}
\label{figure:loglog}
\end{figure}


\section{Proofs of Risk Results}\label{SECTION:RISK_PROOFS}

\subsection{Preliminaries}
Note that the risks $\Risk(\EMfitfun, \fstar)$ and $\Risk(\HKfitfun,
\fstar)$ both only depend on the values of the estimators $\EMfitfun$
and $\HKfitfun$ at the design points $\xvec_1,\dots, \xvec_\numobs$. Also by
the results from \autoref{SECTION:COMPUTATION}, it is clear that
the vectors $(\EMfitfun(\xvec_1), \dots, \EMfitfun(\xvec_\numobs))$ and
$(\HKfitfun(\xvec_1), \dots, \HKfitfun(\xvec_\numobs))$ are Euclidean
projections of the data vector $\yvec = (y_1, \dots, y_\numobs)$ on the
closed convex sets
\begin{equation*}
  \left\{\Altdesignmat \coef : \min_{j \geq 2} \beta_j \geq 0
  \right\} ~~ \text{ and } ~~ \left\{\Altdesignmat \coef : \sum_{j \geq 2} |\beta_j| \leq V
  \right\}
\end{equation*}
respectively. Consequently, we can apply general results from the
theory of convex-constrained LSEs to prove the
risk results for $\EMfitfun$ and $\HKfitfun$. This theory is, by now,
well established (see e.g., \citet{VandegeerBook,
  vaartwellner96book, hjort2011asymptotics, chatterjee2014new}). The following
result from \citet{chatterjee2014new} provides upper bounds for the
risk of general convex-constrained LSEs. This
result will be used in the proofs of \autoref{theorem:NNLS_worst_case}
and \autoref{theorem:lasso_worst_case}.
\begin{theorem}[\citet{chatterjee2014new}]\label{theorem:gencha}
  Let $\ConvexSet$ be a closed convex set in $\R^\numobs$ and let
\begin{equation}\label{genk}
  \paramhat \defn \argmin_{\param \in \ConvexSet}
    \norm{\yvec - \param}^2,
\end{equation}
  where $\yvec \sim \Normal_\numobs(\paramstar, \Imat_\numobs)$ for some
  $\paramstar \in \R^\numobs$ (not necessarily in $\ConvexSet$).
  Then there exists a universal positive constant $C$ such that
  \begin{equation}\label{gencha.eq}
\E \norm{\paramhat - \paramstar}^2 \le C \max(\radstar^2, 1)
  \end{equation}
  for every $\radstar > 0$ which satisfies
  \begin{equation}\label{equation:crit_ineq}
    \E \brackets*{
        \sup_{\param \in \ConvexSet : \|\param - \paramstar\| \leq \radstar}
        \inner{\xivec, \param - \paramstar}
      }
      \leq \frac{\radstar^2}{2}
    \qt{where $\xivec \sim \Normal_\numobs(\zerovec, \Imat_\numobs)$}.
  \end{equation}
\end{theorem}

\autoref{theorem:gencha} is sufficient to prove \autoref{theorem:NNLS_worst_case}
and \autoref{theorem:lasso_worst_case}. However, in order to handle
the misspecified setting discussed in \autoref{remark:cm_misspec}
and \autoref{remark:hk_misspec},
one needs the following generalization of \autoref{theorem:gencha}.
Below,
\begin{equation}
\Proj_{\ConvexSet}(\vvec) \defn \argmin_{\param \in \ConvexSet}
\norm{\vvec - \param}^2
\end{equation}
denotes the projection of $\vvec$ onto the closed convex set
$\ConvexSet$. The following result generalizes
\autoref{theorem:gencha} to the case of model misspecification. It is
similar to related generalizations of \autoref{theorem:gencha} from
\citet{chen2017note} and \citet{bellec2017optimistic}. We omit the
proof of this result as it can be proved by a straightforward
generalization of the proof of the original result,
\autoref{theorem:gencha}, from \citet{chatterjee2014new}.

\begin{theorem}\label{THEOREM:GENCHA_MISSPEC}
Let $\ConvexSet$ be a closed convex set in $\R^\numobs$,
and let $\paramhat \defn \Proj_{\ConvexSet}(\yvec)$
be as defined above~\eqref{genk},
with $\yvec \sim \Normal_\numobs(\paramstar, \Imat_\numobs)$
and $\paramstar \in \R^\numobs$.
Then there exists a universal positive constant $C$ such that
\begin{equation}
\E \norm{\paramhat - \Proj_{\ConvexSet}(\paramstar)}^2 \le C \max(\radstar^2, 1),
\end{equation}
for every $\radstar > 0$ which satisfies
\begin{equation}\label{equation:crit_ineq_misspec}
\E \brackets*{
    \sup_{\param \in \ConvexSet : \norm{\param - \Proj_\ConvexSet(\paramstar)} \le \radstar}
    \inner{\xivec,
    \param - \Proj_{\ConvexSet}(\paramstar)}
} \le \frac{\radstar^2}{2} \qt{where $\xivec \sim \Normal_\numobs(0,
  \Imat_\numobs)$}.
\end{equation}
\end{theorem}
Note that in the well-specified setting $\paramstar \in \ConvexSet$,
we have $\Proj_\ConvexSet(\paramstar) = \paramstar$,
and thus \autoref{theorem:gencha} and \autoref{THEOREM:GENCHA_MISSPEC} are identical.
On the other hand, in the misspecified setting $\paramstar \notin \ConvexSet$,
the two results differ in the risk quantity they control:
$\E \norm{\paramhat - \paramstar}^2$
and
$\E \norm{\paramhat - \Proj_{\ConvexSet}(\paramstar)}^2$
respectively and the fact that $\paramstar$ appearing in
\eqref{equation:crit_ineq} is replaced by
$\Proj_{\ConvexSet}(\paramstar)$ in
\eqref{equation:crit_ineq_misspec}.

\begin{remark}[Risk bounds under misspecification] \label{remark:gencha_misspec}
\autoref{theorem:NNLS_worst_case} and
\autoref{theorem:lasso_worst_case} are proved via
\autoref{theorem:gencha} by establishing \eqref{equation:crit_ineq}
for an appropriate $\radstar$. If we replace $\paramstar$ in these
proofs by $\Proj_{\ConvexSet}(\paramstar)$ and replace the use of
\autoref{theorem:gencha} with that of
\autoref{THEOREM:GENCHA_MISSPEC}, we obtain the risk bounds under
misspecification described in \autoref{remark:cm_misspec} and
\autoref{remark:hk_misspec}.
\end{remark}

The risk of the estimator $\paramhat$ in \eqref{genk} can also be
related to the tangent cones of the closed convex set $K$ at
$\paramstar$. To describe these results, we need some notation and
terminology. The tangent cone of $\ConvexSet$ at $\param
\in \ConvexSet$ is defined as
\begin{equation}\label{tcod}
  \TCone_\ConvexSet(\param) \def \mathrm{closure}\{t(\etavec - \param) : t \geq 0, \etavec
  \in \ConvexSet\}.
\end{equation}
Informally, $\TCone_\ConvexSet(\param)$ represents all directions in which one can
move from $\param$ and still remain in $K$. Note that $\TCone_\ConvexSet(\param)$ is
a cone which means that $a \alphavec \in \TCone_\ConvexSet(\param)$ for every $\alphavec
\in \TCone_\ConvexSet(\param)$ and $a \ge 0$. It is also easy to see that
$\TCone_\ConvexSet(\param)$ closed and convex.

The statistical dimension of a closed convex cone $\TCone \subseteq \R^\numobs$
is defined as
\begin{equation}\label{equation:df_stat_dim}
  \delta(\TCone) \defn \E \|\Proj_\TCone(Z)\|^2, \qt{where $Z \sim \Normal_\numobs(0, \Imat_\numobs)$}
\end{equation}
and $\Proj_\TCone(Z) \defn \argmin_{\uvec \in \TCone} \|Z - \uvec\|^2$ is the projection of
$Z$ onto $T$. The terminology of statistical dimension is due to
\citet{amelunxen2014living} and we refer the reader to this paper for
many properties of the statistical dimension.

The relevance of these notions to the estimator $\paramhat$
(defined in \eqref{genk}) is that
the risk of $\paramhat$ can be related to the statistical dimension
of tangent cones of $K$. This is the content of the following result
due to \citet[Corollary 2.2]{bellec2018sharp}.
\begin{theorem}\label{belma}
Suppose $Y \sim \Normal_\numobs(\paramstar, \sigma^2 \Imat_\numobs)$ for some $\paramstar \in
\R^n$ and $\sigma^2 > 0$ and consider the estimator $\paramhat$
defined in \eqref{genk} for a closed convex set $K$. Then
\begin{equation}\label{belma.eq}
\E \|\paramhat - \paramstar\|^2\leq \inf_{\param \in \ConvexSet} \left[
  \|\param - \paramstar\|^2 + \noisestd^2 \delta(\TCone_\ConvexSet(\param))\right].
\end{equation}
\end{theorem}

The statistical dimension $\delta(\TCone)$ of a closed convex cone $\TCone$ is
closely related to the Gaussian width of $\TCone$ which is defined as
\begin{align}\label{GauWid}
  \GWidth(\TCone) \defn \E \left[\sup_{\param \in \TCone : \|\param\| \leq 1} \inner{Z, \param} \right] \qt{where $Z \sim \Normal_\numobs(0, \Imat_\numobs)$}.
\end{align}
Indeed, it has been shown in \citet[Proposition
10.2]{amelunxen2014living} that
\begin{equation}
\label{equation:statdim_gw_sandwich}
\GWidth^2(\TCone) \leq \delta(\TCone) \leq \GWidth^2(\TCone) + 1
\end{equation}
for every closed convex cone $T$. Using this relation in
conjunction with \eqref{belma.eq}, we obtain the following bound on
the risk of the estimator $\paramhat$ defined in \eqref{genk} when
$Y \sim \Normal_\numobs(\paramstar, \sigma^2 \Imat_\numobs)$:
\begin{equation}
  \label{wma}
  \E \|\paramhat - \paramstar\|^2\leq \inf_{\param \in K} \left[
  \|\param - \paramstar\|^2 + \sigma^2 + \sigma^2 \GWidth^2(T_{K}(\param))\right].
\end{equation}

\subsection{Proof of \autoref{theorem:NNLS_worst_case}}
\label{section:nnls_worst_proof}

Let
\begin{equation}\label{cmhtp}
\paramhat \defn (\EMfitfun(\xvec_1), \ldots, \EMfitfun(\xvec_\numobs))
~~ \text{ and } ~~ \paramstar \defn (\fstar(\xvec_1), \ldots,
\fstar(\xvec_\numobs))
\end{equation}
and note that
\begin{equation*}
  \Risk(\EMfitfun, \fstar) = \E
  \frac{1}{\numobs}  \norm{\paramhat - \paramstar}^2
\end{equation*}
where $\norm{\cdot}$ denotes the usual Euclidean norm in $\R^n$.

Observe that by
\autoref{proposition:cm_nnls}, it follows that $\paramhat =
\Altdesignmat \coefhatcm$ is the projection of the data vector $\yvec$
on the closed convex cone
\begin{equation}\label{equation:NNLSSet_def}
\NNLSSet \defn \{ \Altdesignmat \coef : \coefplain_j \ge 0, \forall j
\ge 2\} = \{(f(\xvec_1), \dots, f(\xvec_\numobs)) : f \in \EMClass\}.
\end{equation}
where $\Altdesignmat$ is the design matrix introduced in
\autoref{SECTION:COMPUTATION}. Note that, under the lattice design
\eqref{equation:lattice_design}, the set $\NNLSSet$ is completely
determined by the values of $\numobs_1, \ldots, \numobs_\usedim$. We
can therefore employ \autoref{theorem:gencha} to bound the risk $\E
\|\paramhat - \paramstar\|^2/n$.

First, we claim that it suffices to prove the theorem
under the assumption $\numobs_j \ge 2$ for all $j = 1, \ldots, \usedim$.
To see this, note first that when $\numobs =
\numobs_1 \cdots \numobs_\usedim = 1$, we have $\paramhat = \yvec$ so that
$\Risk(\paramhat, \paramstar) = \noisestd^2 / \numobs$ and the result
holds which means that we can assume that $\max_j n_j \geq 2$ for some
$j$. Now if $\numobs_j = 1$ for some values of $j$, we can simply
ignore these components
and focus on the equivalent problem with a lattice design~\eqref{equation:lattice_design}
in a lower-dimensional space that has
at least two grid points in each component.
We can apply the bound~\eqref{equation:worst_case} to this lower-dimensional problem
(for instance, the dimension would be $\usedim' = \#\{j : \numobs_j \ge 2\}$ instead of $\usedim$) and then remark that the bound~\eqref{equation:worst_case} for the original problem
is even larger.

Next, we claim that it suffices to prove the theorem under the assumption $\noisestd^2 = 1$.
Indeed in general we may consider the rescaled problem with
$\ftilde \defn \fstar / \noisestd$,
$\tilde{V}^* \defn V^* / \noisestd$,
and $\tilde{y}_i \sim \Normal(\ftilde(\xvec_i), 1)$,
apply the bound~\eqref{equation:worst_case},
and then multiply the risk bound by $\noisestd^2$
to account for rescaling the fitted function by $\noisestd$. This is possible because $\EMClass$ is a cone.

So, we assume $\numobs_j \ge 2$ for all $j = 1, \ldots, \usedim$ and $\noisestd^2 = 1$.
As mentioned above, we want to bound $\E \norm{\paramhat - \paramstar}^2 / \numobs$
using \autoref{theorem:gencha}.
For this, we need to obtain upper
bounds for
\begin{equation}\label{equation:nnls_gw}
G(t)
\defn \E \sup_{\param \in \NNLSSet \cap \Ball_2(\paramstar, \rad)}
\inner{\noisevec, \param - \paramstar}
\end{equation}
where $\noisevec \sim \Normal_\numobs(0, \Imat_\numobs)$ and $\Ball_2(\paramstar, \rad)
\defn \{\param : \norm{\param - \paramstar} < \rad\}$
denotes the ball of radius $\rad$ centered at $\paramstar$.

In what follows, we sometimes treat vectors in $\R^\numobs$
as arrays in $\R^{\numobs_1 \times \cdots \times \numobs_\usedim}$
indexed by $\ivec = (i_1, \ldots, i_\usedim)$ for $0 \le i_j \le \numobs_j - 1$
and $j = 1, \ldots, \usedim$.

For each $j \in 1, \ldots, \usedim$, let
\begin{equation}
S^{(0)}_j \defn \{i_j : 0 \le i_j \le \frac{\numobs_j}{2} - 1\},
\qquad
S^{(1)}_j \defn \{i_j : \frac{\numobs_j}{2} - 1 < i_j \le \numobs_j - 1\},
\end{equation}
so that
\begin{align}
\inner{\noisevec, \param - \paramstar}
&= \sum_{i_1 = 0}^{\numobs_1 - 1}
\cdots
\sum_{i_\usedim = 0}^{\numobs_\usedim - 1}
\noise_{\ivec} (\paramplain_{\ivec} - \paramplainstar_{\ivec})
= \sum_{\zvec \in \{0, 1\}^\usedim}
\sum_{\ivec \in S_1^{(z_1)} \times \cdots \times S_\usedim^{(z_\usedim)}}
\noise_{\ivec} (\paramplain_{\ivec} - \paramplainstar_{\ivec}).
\end{align}
We then obtain the bound
\begin{equation}\label{equation:binary_slices}
G(t) \le \sum_{\zvec \in \{0,1\}^\usedim}
\underbrace{
    \E \sup_{\param \in \NNLSSet \cap \Ball_2(\paramstar, \rad)}
    \sum_{\ivec \in S_1^{(z_1)} \times \cdots \times S_\usedim^{(z_\usedim)}}
    \noise_{\ivec} (\paramplain_{\ivec} - \paramplainstar_{\ivec})
}_{\eqqcolon H_{\zvec}(t)}.
\end{equation}

We now bound $H_{\zvec}(t)$ for fixed $\zvec \in \{0, 1\}^\usedim$.
For each $j = 1, \ldots, \usedim$
let $K_j$ denote the largest positive integer $k_j$ for which
\begin{equation}
\braces*{
    i_j
    \in S^{(z_j)}_j
    \numobs_j 2^{-(k_j + 1)} - 1 + z_j \numobs_j / 2
    < i_j
    \le \numobs_j 2^{-k_j} - 1 + z_j \numobs_j / 2
}
\end{equation}
is nonempty.
Let $\KIndexSet \defn \bigtimes_{j=1}^\usedim \{1, \ldots, K_j\}$
and note that $\abs{\KIndexSet} = K_1 K_2 \cdots K_\usedim$.
For $\kvec \defn (k_1, \ldots, k_\usedim) \in \KIndexSet$
and
$\param \in \R^{\numobs_1 \times \cdots \times \numobs_\usedim}$,
let
\begin{align}
\param^{(\kvec)}
&\defn
\Big\{
    \param_{\ivec} :
    \numobs_j 2^{-(k_j + 1)} -  1 + z_j \numobs_j / 2
    < i_j
    \le \numobs_j 2^{-k_j} - 1 + z_j \numobs_j / 2, \,
\\
&\qquad\qquad\qquad
    i_j = 0, \ldots, \numobs_j - 1, \,
    j = 1, \ldots, \usedim
\Big\}.
\end{align}
Let $\MIndexSet \defn \{(m_{\kvec})_{\kvec \in \KIndexSet} : 1 \le m_{\kvec} \le \abs{\KIndexSet}, \sum_{\kvec \in \KIndexSet} m_{\kvec} \le 2 \abs{\KIndexSet}\}$.
For $\mvec \in \MIndexSet$, we define
\begin{equation}
T_{\mvec}(\rad) \defn \braces*{
    \param \in \NNLSSet \cap \Ball_2(\paramstar, \rad) :
    \norm{\param - \paramstar} \le t,
    \norm{\param^{(\kvec)} - (\paramstar)^{(\kvec)}}^2
    \le \frac{m_{\kvec} t^2}{\abs{\KIndexSet}}, \forall \kvec \in \KIndexSet
}.
\end{equation}

We claim
\begin{equation}\label{equation:break_up}
\NNLSSet \cap \Ball_2(\paramstar, \rad)
\subseteq
\bigcup_{\mvec \in \MIndexSet} T_{\mvec}(\rad).
\end{equation}
Indeed suppose $\param \in \NNLSSet \cap \Ball_2(\paramstar, \rad)$;
then we have
$\rad^2 \ge \norm{\param - \paramstar}^2
\ge \norm{\param^{(\kvec)} - (\paramstar)^{(\kvec)}}^2$
for each $\kvec$,
and thus there exists $1 \le m_{\kvec} \le \abs{\KIndexSet}$ such that
\begin{equation}
m_{\kvec} - 1 \le
\abs{\KIndexSet}
\frac{\norm{\param^{(\kvec)} - (\paramstar)^{(\kvec)}}^2}{\rad^2}
\le m_{\kvec}.
\end{equation}
This implies
\begin{equation}
1
\ge \rad^{-2} \norm{\param - \paramstar}^2
\ge \rad^{-2} \sum_{\kvec \in \KIndexSet}
\norm{\param^{(\kvec)} - (\paramstar)^{(\kvec)}}^2
\ge \abs{\KIndexSet}^{-1} \sum_{\kvec \in \KIndexSet} (m_{\kvec} - 1)
\end{equation}
and thus $\sum_{\kvec \in \KIndexSet} m_{\kvec} \le 2 \abs{\KIndexSet}$,
so $\mvec \in \MIndexSet$ and $\param \in T_{\mvec}(\rad)$,
which verifies the claim~\eqref{equation:break_up}.

Using this claim~\eqref{equation:break_up} we obtain
\begin{equation}
H_{\zvec}(\rad) \le
\E \max_{\mvec \in \MIndexSet}
\sup_{\param \in T_{\mvec}}
\sum_{\ivec \in S_1^{(z_1)} \times \cdots \times S_\usedim^{(z_\usedim)}}
\noise_{\ivec} (\paramplain_{\ivec} - \paramplainstar_{\ivec}).
\end{equation}
Lemma D.1 from \cite{guntuboyina2017spatial} then implies
\begin{equation}
H_{\zvec}(\rad)
\le \max_{\mvec \in \MIndexSet}
\E \sup_{\param \in T_{\mvec}(\rad)}
\sum_{\ivec \in S_1^{(z_1)} \times \cdots \times S_\usedim^{(z_\usedim)}}
\noise_{\ivec} (\paramplain_{\ivec} - \paramplainstar_{\ivec})
+ \rad \sqrt{2 \log \abs{\MIndexSet}}
+ \rad \sqrt{\pi / 2}.
\end{equation}
Because the number of $\abs{\KIndexSet}$-tuples of positive integers summing to $p$ is $\binom{p-1}{\abs{\KIndexSet} - 1} = \binom{p - 1}{p - \abs{\KIndexSet}}$,
we can bound the cardinality of $\MIndexSet$ by
\begin{equation}
\abs{\MIndexSet}
\le \sum_{p = \abs{\KIndexSet}}^{2 \abs{\KIndexSet}} \binom{p - 1}{p - \abs{\KIndexSet}}
\le \sum_{p = \abs{\KIndexSet}}^{2 \abs{\KIndexSet}} \binom{2 \abs{\KIndexSet} - 1}{p - \abs{\KIndexSet}}
\le 2^{2 \abs{\KIndexSet} - 1}.
\end{equation}
Thus,
\begin{equation}\label{equation:pull_out_max}
H_{\zvec}(\rad)
\le \max_{\mvec \in \MIndexSet}
\underbrace{
    \E \sup_{\param \in T_{\mvec}(\rad)}
    \sum_{\ivec \in S_1^{(z_1)} \times \cdots \times S_\usedim^{(z_\usedim)}}
    \noise_{\ivec} (\paramplain_{\ivec} - \paramplainstar_{\ivec})
}_{\eqqcolon U_{\zvec, \mvec}(\rad)}
+ 2 \rad \sqrt{\abs{\KIndexSet}}
+ \rad \sqrt{\pi / 2}.
\end{equation}

Since
$\sum_{\ivec \in S_1^{(z_1)} \times \cdots \times S_\usedim^{(z_\usedim)}}
\noise_{\ivec} (\paramplain_{\ivec} - \paramplainstar_{\ivec})
= \sum_{\kvec \in \KIndexSet}
\inner*{\noisevec^{(\kvec)}, \param^{(\kvec)} - (\paramstar)^{(\kvec)}}$,
we have
\begin{equation}\label{equation:u_break_bound}
U_{\zvec, \mvec}(\rad)
\le \sum_{\kvec \in \KIndexSet}
\underbrace{
\E \sup_{
    \substack{
    \param \in \NNLSSet \cap \Ball_2(\paramstar, \rad) :
    \\
    \norm{\param^{(\kvec)} - (\paramstar)^{(\kvec)}}^2
    \le \rad^2 m_{\kvec} / \abs{\KIndexSet}
    }
}
\inner*{\noisevec^{(\kvec)}, \param^{(\kvec)} - (\paramstar)^{(\kvec)}}
}_{\eqqcolon U_{\zvec, \mvec, \kvec}(\rad)}.
\end{equation}
We claim that for any $\param \in \NNLSSet \cap \Ball_2(\paramstar, \rad)$
and any $\ivec \in \bigtimes_{j=1}^\usedim \{0, 1, \ldots, \numobs_j - 1\}$
and $\kvec \in \KIndexSet$
satisfying
\begin{equation}\label{equation:ij_bound}
\numobs_j 2^{-(k_j + 1)} - 1 + z_j \numobs_j / 2
< i_j
\le \numobs_j 2^{-k_j} - 1 + z_j \numobs_j / 2,
\end{equation}
then $\paramplain_{\ivec}$ can be bounded as
\begin{equation}\label{equation:entry_sandwich}
\paramplainstar_{\zerovec}
- \rad (2^{\usedim + k_+} / \numobs)^{1/2}
\le \paramplain_{\ivec}
\le \paramplainstar_{\nvec - \onevec}
+ \rad (2^{\usedim + k_+} / \numobs)^{1/2}.
\end{equation}
where $k_+ \defn k_1 + \cdots + k_\usedim$.
We prove each bound by contradiction.
If the upper bound of~\eqref{equation:entry_sandwich} does not hold,
then
\begin{equation}
\paramplain_{\lvec}
\ge \paramplain_{\ivec}
> \paramplainstar_{\nvec - \onevec}
+ \rad (2^{\usedim + k_+} / \numobs)^{1/2}
\ge \paramplainstar_{\lvec}
+ \rad (2^{\usedim + k_+} / \numobs)^{1/2}
\end{equation}
as long as $\lvec \succeq \ivec$,
which yields
\begin{equation}
\rad^2
\ge \norm{\param - \paramstar}^2
\ge \sum_{\lvec \succeq \ivec}
(\paramplain_{\lvec} - \paramplainstar_{\lvec})^2
> \rad^2 2^{\usedim + k_+} \numobs^{-1}
\cdot \prod_{j=1}^\usedim (\numobs_j - i_j).
\end{equation}
Noting that our condition on $i_j$~\eqref{equation:ij_bound} implies
$\numobs_j - i_j \ge \numobs_j (1 - z_j / 2 - 2^{-k_j}) \ge \numobs_j 2^{-(k_j + 1)}$,
we obtain
$\prod_{j=1}^\usedim (\numobs_j - i_j + 1) \ge \numobs 2^{-(\usedim + k_+)}$
which yields the contradiction $\rad^2 > \rad^2$.

Similarly if the lower bound of~\eqref{equation:entry_sandwich} does not hold,
then
\begin{equation}
\paramplain_{\lvec}
\le \paramplain_{\ivec}
< \paramplainstar_{\nvec - \onevec}
- \rad (2^{\usedim + k_+} / \numobs)^{1/2}
\le \paramplainstar_{\lvec}
- \rad (2^{\usedim + k_+} / \numobs)^{1/2}
\end{equation}
as long as $\lvec \preceq \ivec$,
which yields
\begin{equation}
\rad^2
\ge \norm{\param - \paramstar}^2
\ge \sum_{\lvec \preceq \ivec}
(\paramplain_{\lvec} - \paramplainstar_{\lvec})^2
> \rad^2 2^{\usedim + k_+} \numobs^{-1}
\cdot \prod_{j=1}^\usedim (i_j + 1).
\end{equation}
Noting that our condition on $i_j$~\eqref{equation:ij_bound} implies
$i_j + 1 > \numobs_j 2^{-(k_j+ 1)}$
we obtain
$\prod_{j=1}^\usedim (i_j + 1) \ge \numobs 2^{-(\usedim + k_+)}$
which yields the contradiction $\rad^2 > \rad^2$.

Thus, the bounds~\eqref{equation:entry_sandwich} hold.
So, for each $\param \in \NNLSSet \cap \Ball_2(\paramstar, \rad)$
and $\kvec \in \KIndexSet$,
the number of entries in $\param^{(\kvec)}$ is at most
\begin{equation}
\prod_{j=1}^\usedim (\numobs_j 2^{-k_j} - \numobs_j 2^{-(k_j + 1)})
\le \numobs 2^{-(\usedim + k_+)},
\end{equation}
and each entry lies in the interval
\begin{equation}
[a,b] \defn
\brackets*{
    \paramplainstar_{\zerovec}
- \rad (2^{\usedim + k_+} / \numobs)^{1/2}
, \;
\paramplainstar_{\nvec - \onevec}
+ \rad (2^{\usedim + k_+} / \numobs)^{1/2}
}.
\end{equation}
Moreover, $\param^{(\kvec)}$ lies in some
$\tilde{\NNLSSetPlain} \defn
\NNLSSetPlain_{\tilde{\numobs}_1, \ldots, \tilde{\numobs}_\usedim}$
where $\tilde{\numobs}_1, \ldots, \tilde{\numobs}_\usedim$
are the dimensions of $\param^{(\kvec)}$ as a sub-array.

We make use of the following metric entropy result,
proved in \autoref{section:proof_metric_entropy_nnls_rectangle}

\begin{lemma}\label{lemma:metric_entropy_nnls_rectangle}
For $a < b$, we have
\begin{equation}
\log \covernum_2(\radcover, \NNLSSet \cap [a,b]^\numobs)
\le C_\usedim \frac{(b-a) \sqrt{\numobs}}{\radcover}
\parens*{
    \log \frac{(b-a) \sqrt{\numobs}}{\radcover}
  }^{\usedim - \frac{1}{2}}
\Ind\{\radcover \le (b - a) \sqrt{\numobs}\}.
\end{equation}
\end{lemma}

Combining this metric entropy bound with
Dudley's entropy bound \cite{Dudley67}
(for instance see \cite[Thm. 3.2]{chatterjee2015matrix})
yields
\begin{equation}
U_{\zvec, \mvec, \kvec}(\rad)
\le c \int_0^{\rad \sqrt{m_{\kvec} / \abs{\KIndexSet}}}
\sqrt{\frac{B}{\radcover} \parens*{\log \frac{B}{\radcover}}^{\usedim - \frac{1}{2}}}
\, d\radcover,
\end{equation}
where
\begin{align}
B
&\defn
(\numobs 2^{-(\usedim + k_+)})^{1/2}
(V^* + 2 \rad (\numobs 2^{-(d + k_+)})^{-1/2})
\\
&= (\numobs 2^{-(\usedim + k_+)})^{1/2} V^* + 2 \rad
\label{equation:break_magnitude}
\end{align}
and $V^* = \fstar(\onevec) - \fstar(\zerovec) \ge \paramplainstar_{\nvec - \onevec} - \paramplainstar_{\zerovec}$.
Note that $\radcover \le \rad \sqrt{m_{\kvec}/\abs{\KIndexSet}} \le \rad < B$,
so $\log(B/\radcover) > 0$.

The following lemma (proved in \autoref{section:proof_bound_dudley_integral})
allows us to bound the above integral.

\begin{lemma}\label{lemma:bound_dudley_integral}
For every $\usedim \ge 1$ there exists a positive constant $C_\usedim$ such that
for every $s \in (0, B]$, the following inequality holds.
\begin{equation}
\int_0^s \sqrt{
    \frac{B}{\radcover}
    \parens*{\log \frac{B}{\radcover}}^{\usedim - \frac{1}{2}}
}
\, d\radcover
\le C_\usedim \sqrt{sB}
\parens*{\log \frac{B}{s}}^{\frac{2 \usedim - 1}{4}}
\end{equation}
\end{lemma}

Applying \autoref{lemma:bound_dudley_integral} with $s \defn \rad \sqrt{m_{\kvec} / \abs{\KIndexSet}} \ge \rad / \sqrt{\abs{\KIndexSet}}$ yields
\begin{align}
U_{\zvec, \mvec, \kvec}(\rad)
\le C_\usedim \sqrt{B \rad} (m_{\kvec} / \abs{\KIndexSet})^{1/4}
\parens*{\log \frac{e B \sqrt{\abs{\KIndexSet}}}{\rad}}^{\frac{2 \usedim - 1}{4}}.
\end{align}

We bound this with two terms depending on which of the two terms
in the definition~\eqref{equation:break_magnitude} of $B$ is larger.
In the case $V^* (\numobs 2^{-(d+k_+)})^{1/2} > 2 \rad$,
we have $B \le 2 V^* (\numobs 2^{-(d+k_+)})^{1/2} \le 2 V^* \sqrt{\numobs}$
and
\begin{equation}
U_{\zvec, \mvec, \kvec}(\rad)
\le C_\usedim \sqrt{t V^*} (\numobs 2^{-k_+})^{1/4}
\parens*{\log \frac{2 e V^* \sqrt{\numobs \abs{\KIndexSet}}}{\rad}}^{\frac{2 \usedim - 1}{4}}.
\end{equation}
In the other case where
$V^* (\numobs 2^{-(d+k_+)})^{1/2} \le 2 \rad$, we have $B \le 3 \rad$,
which yields
\begin{equation}
U_{\zvec, \mvec, \kvec}(\rad)
\le C_\usedim \rad (m_{\kvec} / \abs{\KIndexSet})^{1/4}
(\log(2 e \sqrt{\abs{\KIndexSet}}))^{\frac{2 \usedim - 1}{4}}.
\end{equation}
Combining the two cases and using the indicator bounds
$\Ind\{V^* (\numobs 2^{-(d+k_+)})^{1/2} > 2 \rad\} \le \Ind\{V^* \sqrt{\numobs} > \rad\}$
and $\Ind\{V^* (\numobs 2^{-(d+k_+)})^{1/2} \le 2 \rad\} \le 1$,
we obtain
\begin{align}
U_{\zvec, \mvec, \kvec}(\rad)
&\le
C_\usedim \sqrt{t V^*} (\numobs 2^{- k_+})^{1/4}
\parens*{\log \frac{2 e V^* \sqrt{\numobs \abs{\KIndexSet}}}{\rad}}^{\frac{2 \usedim - 1}{4}}
\Ind\{V^* \sqrt{\numobs} > \rad\}
\\
&\qquad
+
C_\usedim \rad (m_{\kvec} / \abs{\KIndexSet})^{1/4}
(\log(2 e \sqrt{\abs{\KIndexSet}}))^{\frac{2 \usedim - 1}{4}}.
\end{align}
Applying this observation to the earlier bound $U_{\zvec, \mvec}(\rad) \le \sum_{\kvec \in \KIndexSet} U_{\zvec, \mvec, \kvec}(\rad)$ from~\eqref{equation:u_break_bound} yields
\begin{align}
U_{\zvec, \mvec}(\rad)
&\le
C_\usedim \sqrt{t V^*} \numobs^{1/4}
\parens*{\log \frac{2 e V^* \sqrt{\numobs \abs{\KIndexSet}}}{\rad}}^{\frac{2 \usedim - 1}{4}}
\Ind\{V^* \sqrt{\numobs} > \rad\}
\sum_{\kvec \in \KIndexSet} {2^{- k_+ / 4}}
\\
&\qquad
+
C_\usedim \rad
(\log(2 e \sqrt{\abs{\KIndexSet}}))^{\frac{2 \usedim - 1}{4}}
\sum_{\kvec \in \KIndexSet} (m_{\kvec} / \abs{\KIndexSet})^{1/4}.
\end{align}
The first sum can be bounded as
\begin{equation}
\sum_{\kvec \in \KIndexSet} 2^{-k_+ / 4}
\le \prod_{j=1}^\usedim \sum_{k_j = 1}^\infty 2^{-k_j / 4}
\le C_\usedim.
\end{equation}
For the second sum, note that H\"{o}lder's inequality combined with the fact that $\sum_{\kvec \in \KIndexSet} m_{\kvec} \le 2 \abs{\KIndexSet}$ yields
\begin{equation}
\sum_{\kvec \in \KIndexSet} m_{\kvec}^{1/4}
\le \parens*{\sum_{\kvec \in \KIndexSet} m_{\kvec}}^{1/4} \abs{\KIndexSet}^{3/4}
\le 2^{1/4} \abs{\KIndexSet}.
\end{equation}

Additionally, note that $K_j \le C \log \numobs_j$ for each $j$, so
$\log \abs{\KIndexSet} \le \sum_{j=1}^\usedim \log(C \log \numobs_j) \le C_\usedim \log \numobs$,
which allows us to bound the logarithmic term as
\begin{equation}
\log \frac{2 e V^* \sqrt{\numobs \abs{\KIndexSet}}}{\rad}
\le \log \frac{2 e V^* \sqrt{\numobs}}{\rad}
+ \frac{1}{2} \log \abs{\KIndexSet}
\le C_\usedim \log \frac{e V^* \sqrt{\numobs}}{\rad}.
\end{equation}

Finally, note that
\begin{equation}
\parens*{\log \frac{e V^* \sqrt{\numobs}}{\rad}}^{\frac{2 \usedim - 1}{4}} \Ind\{V^* \sqrt{\numobs} > t\}
= \parens*{\logplus \frac{e V^* \sqrt{\numobs}}{\rad}}^{\frac{2 \usedim - 1}{4}},
\end{equation}
where $\logplus(x) \defn \max(0, \log x)$.

Combining these four observations yields
\begin{equation}
U_{\zvec, \mvec}(\rad)
\le
C_\usedim \sqrt{t V^*} \numobs^{1/4}
\parens*{\logplus \frac{e V^* \sqrt{\numobs}}{\rad}}^{\frac{2 \usedim - 1}{4}}
+
C_\usedim \rad
\abs{\KIndexSet}^{3/4}
(\log(2 e \sqrt{\abs{\KIndexSet}}))^{\frac{2 \usedim - 1}{4}}
.
\end{equation}
Combining this bound with the earlier
bound~\eqref{equation:pull_out_max} on $H_{\zvec}(\rad)$
yields
\begin{align}
H_{\zvec}(\rad)
&\le \max_{\mvec \in \MIndexSet} U_{\zvec, \mvec}(\rad)
+ 2 \rad \sqrt{\abs{\KIndexSet}} + \rad \sqrt{\pi / 2}
\\
&\le
\underbrace{
    C_\usedim \sqrt{t V^*} \numobs^{1/4}
    \parens*{\logplus \frac{e V^* \sqrt{\numobs}}{\rad}}^{\frac{2 \usedim - 1}{4}}
}_{\eqqcolon G_1(\rad)}
+
\underbrace{
    C_\usedim \rad
    \abs{\KIndexSet}^{3/4}
    (\log(2 e \sqrt{\abs{\KIndexSet}}))^{\frac{2 \usedim - 1}{4}}
}_{\eqqcolon G_2(\rad)}
.
\end{align}
By observing the earlier bound~\eqref{equation:binary_slices},
we see that the above upper bound for $H_{\zvec}(\rad)$ also
holds for $G(\rad)$ (after multiplying the constants by $2^\usedim$).
That is,
\begin{equation}\label{equation:gw_bound}
G(\rad) \le G_1(\rad) + G_2(\rad).
\end{equation}
where $G_1$ and $G_2$ are the two terms of the previous inequality.
Let
\begin{equation}
\rad_1 \defn \max\{1, (4 C_\usedim)^{2/3}\} (\sqrt{\numobs} V^*)^{1/3}
\brackets*{
    \max\{1, \logplus(e (\sqrt{\numobs} V^*)^{2/3})\}
}^{\frac{2 \usedim - 1}{6}}.
\end{equation}
Then $\rad_1 \ge (\sqrt{\numobs} V^*)^{1/3}$, so for $\rad \ge \rad_1$ we have
\begin{align}
\frac{G_1(\rad)}{\rad^2}
&= C_\usedim \frac{\sqrt{V^*} \numobs^{1/4}}{\rad^{3/2}}
\parens*{\logplus \frac{e V^* \sqrt{\numobs}}{\rad}}^{\frac{2 \usedim - 1}{4}}
\\
&\le C_\usedim \frac{\sqrt{V^*} \numobs^{1/4}}{\rad^{3/2}}
\parens*{\logplus (e (V^* \sqrt{\numobs})^{2/3})}^{\frac{2 \usedim - 1}{4}}
\le \frac{1}{4}.
\end{align}
Next, with the definition
\begin{equation}
\rad_2 \defn 4 C_\usedim
\abs{\KIndexSet}^{3/4}
(\log(2 e \sqrt{\abs{\KIndexSet}}))^{\frac{2 \usedim - 1}{4}},
\end{equation}
for $\rad \ge \rad_2$ we have
\begin{equation}
\frac{G_2(\rad)}{\rad^2}
= \frac{C_\usedim \abs{\KIndexSet}^{3/4}}{\rad}
(\log(2 e \sqrt{\abs{\KIndexSet}}))^{\frac{2 \usedim - 1}{4}}
\le \frac{1}{4}.
\end{equation}
Combining the two inequalities, we obtain $G(\rad) \le \rad^2 / 2$ for $\rad \ge \max\{\rad_1, \rad_2\}$.
By \autoref{theorem:gencha} and the bound $K_j \le c \log \numobs_j$, we obtain
\begin{align}
\Risk(\EMfitfun, \fstar)
&= \E \frac{1}{\numobs} \norm{\paramhat - \paramstar}^2
\le \frac{\rad_1^2 + \rad_2^2}{\numobs}
\\
&\le C_\usedim \parens*{\frac{V^*}{\numobs}}^{\frac{2}{3}}
\brackets*{
    \max\{1, \logplus(e (\sqrt{\numobs} V^*)^{2/3})\}
}^{\frac{2 \usedim - 1}{3}}
\\
&\qquad
+ \frac{C_\usedim}{\numobs}
\parens*{\prod_{j=1}^\usedim \log \numobs_j}^{\frac{3}{2}}
\parens*{
    \sum_{j=1}^\usedim \log (e \log \numobs_j)
}^{\frac{2 \usedim - 1}{2}}
\\
&\le C_\usedim \parens*{\frac{V^*}{\numobs}}^{\frac{2}{3}}
\brackets*{
    \log(2 + \sqrt{\numobs} V^*)
}^{\frac{2 \usedim - 1}{3}}
\\
&\qquad
+ \frac{C_\usedim}{\numobs}
(\log \numobs)^{\frac{3\usedim}{2}}
(\log (e \log \numobs))^{\frac{2 \usedim - 1}{2}}.
\end{align}

\subsection{Proof of \autoref{theorem:NNLS_adapt}}
\label{section:nnls_adapt_proof}
We use the earlier notation \eqref{cmhtp}. As observed in the proof of
\autoref{theorem:NNLS_worst_case}, it follows from
\autoref{proposition:cm_nnls} and \autoref{proposition:cm_discrete}
that $\paramhat$ is the projection of the data vector $\yvec$ onto the
closed convex cone \eqref{equation:NNLSSet_def}. We then apply
\autoref{belma} to obtain
\begin{equation*}
    \Risk(\EMfitfun, \fstar) = \E
  \frac{1}{\numobs}  \norm{\paramhat - \paramstar}^2 \leq
  \inf_{\param \in K} \left\{\frac{1}{\numobs}
    \|\param - \paramstar\|^2 + \frac{\sigma^2}{\numobs}
    \delta(T_{K}(\param)) \right\}
\end{equation*}
where $K = \NNLSSet$ is the set \eqref{equation:NNLSSet_def}. Using
the notation $\param_f := (f(\xvec_1), \dots, f(\xvec_\numobs))$ for $f \in
\EMClass$, we can rewrite the above inequality as
\begin{align*}
  \Risk(\EMfitfun, \fstar)  &\leq \inf_{f \in \EMClass}
  \left\{\frac{1}{n} \sum_{i=1}^n \left(f(\xvec_i) - f^*(\xvec_i)
    \right)^2 + \frac{\sigma^2}{\numobs} \delta(T_{K}(\param_f)) \right\}  \\
&\leq \inf_{f \in \rpc \cap \EMClass}
  \left\{\frac{1}{n} \sum_{i=1}^n \left(f(\xvec_i) - f^*(\xvec_i)
    \right)^2 + \frac{\sigma^2}{\numobs} \delta(T_{K}(\param_f)) \right\}.
\end{align*}
Therefore to complete the proof of \autoref{theorem:NNLS_adapt}, it is
enough to show that
\begin{equation}\label{rram}
  \delta(T_{K}(\param_f)) \leq C_d k(f)
  (\log (e \numobs))^{\frac{3\usedim}{2}} (\log \log \numobs)^{\frac{2 \usedim - 1}{2}}
  \qt{for every $f \in \rpc \cap \EMClass$}.
\end{equation}
Fix $f \in \rpc \cap \EMClass$ with $k(f) = k$. By the definition of
$\rpc$, there exist $d$ univariate partitions as in \eqref{dpar} such
that $f$ is constant on each of the $k$ rectangles
\begin{equation}\label{arbrecti}
R_{l_1, \dots, l_d} :=  \prod_{s=1}^d [x_{l_s}^{(s)}, x_{l_s+1}^{(s)})
\qt{$l_s = 0, 1, \dots, k_s-1$ and $s = 1, \dots, d$}.
\end{equation}
For every $s = 1, \dots, d$ and $l_s = 0, 1, \dots, k_s-1$, let
$n_s(l_s)$ be the number of indices $i_s = 0, 1, \dots, n_s - 1$ such
that $i_s/n_s \in [x_{l_s}^{(s)}, x_{l_s+1}^{(s)})$. It will be
convenient in the sequel to, as in \autoref{section:esld}, index
vectors in $\R^n$ by $(i_1, \dots, i_d) \in \IndexSet$ (recall that
$\IndexSet$ is defined as in \eqref{AllInd}). Specifically the
components of $\param \in
R^n$ will be denoted by $\theta_{i_1, \dots, i_d}, (i_1, \dots, i_d)
\in \IndexSet$. Also, for $\param \in \R^n$
and the rectangle \eqref{arbrecti}, let $\param(R_{l_1, \dots, l_d})$
denote the vector in $\R^{n_1(l_1)} \times \dots \times \R^{n_d(l_d)}$
with components given by $\theta_{i_1, \dots, i_d}$ as each $i_s$
varies over the
indices in $0, 1, \dots, n_s-1$ such that $i_s/n_s \in [x_{l_s}^{(s)},
x_{l_s+1}^{(s)})$. We now make the key observation that for every
$\param \in \NNLSSet$ and rectangle $R_{l_1, \dots, l_d}$ in
\eqref{arbrecti}, we have
\begin{equation}\label{fasub}
  \param(R_{l_1, \dots, l_d}) \in \mathcal{D}_{n_1(l_1), \dots,
    n_d(l_d)}.
\end{equation}
To see this, fix $\param \in \NNLSSet$ and let $f \in \EMClass$ be
such that $\theta_{i_1, \dots, i_d} = f(i_1/n_1, \dots, i_d/n_d)$ for
every $i_1, \dots, i_d$. Then
\begin{align*}
\param(R_{l_1, \dots, l_d}) &= \left\{\left(f(\frac{i_1}{n_1}), \dots,
  f(\frac{i_d}{n_d}) \right) : \frac{i_s}{n_s} \in [x_{l_s}^{(s)}, x_{l_s +
  1}^{(s)}), s = 1, \dots, d \right\} \\
&= \left\{\left(g(\frac{j_1}{n_1(l_1)}, \dots, \frac{j_d}{n_d(l_d)}) \right) : j_s = 0,
  1, \dots, n_s(l_s) - 1, s = 1, \dots, d\right\}
\end{align*}
where $g: [0, 1]^d \rightarrow \R$ is defined as
\begin{align*}
  g(x_1, \dots, x_d) := f \left((1 - x_1) x_{l_1}^{(1)} + x_1
  x_{l_1+1}^{(1)}, \dots, (1 - x_d) x_{l_d}^{(d)} + x_d
  x_{l_1+d}^{(d)} \right)
\end{align*}
It is easy to see that $g \in \EMClass$ which proves
\eqref{fasub}. The fact \eqref{fasub} will be used to prove \eqref{rram}
in the following way. We first observe that
\begin{align}\label{facon}
  T_K(\param_f) \subseteq \left\{v \in \R^n : v(R_{l_1, \dots, l_d})
  \in \mathcal{D}_{n_1(l_1), \dots, n_d(l_d)}, \forall l_s =
  0, 1, \dots, k_s-1, \forall s = 1, \dots, d  \right\}.
\end{align}
To prove \eqref{facon}, note first that, by the definition of the
tangent cone, we have
\begin{align*}
  T_K(\param_f) = \mathrm{Closure} \left\{\alpha(\param - \param_f) :
  \theta \in K, \alpha \geq 0 \right\}.
\end{align*}
Since the right hand side of \eqref{facon}  is a closed set, we only
need to show  that $v = \alpha(\param - \param_f)$ belongs to the
right hand side of \eqref{facon} for every $\param \in K$ and $\alpha
\geq 0$. Fix $l_1, \dots, l_d$. By \eqref{fasub}, we have that
$\param(R_{l_1, \dots, l_d}) \in \mathcal{D}_{n_1(l_1), \dots,
  n_d(l_d)}$. On the other hand, $\param_f(R_{l_1, \dots, l_d})$ is a
constant vector, because $f$ is constant on $R_{l_1, \dots, l_d}$. As
a result, with $R = R_{l_1, \dots, l_d}$, we obtain that $v(R) =
\alpha(\param(R) - \param_f(R)) \in \mathcal{D}_{n_1(l_1), \dots,
  n_d(l_d)}$  as $\mathcal{D}_{n_1(l_1), \dots,
  n_d(l_d)}$ is a cone that is invariant under translation by constant
vectors. This proves \eqref{facon}.

The observation \eqref{facon} implies (using the monotonicity of
statistical dimension; see \citet[Proposition
3.1]{amelunxen2014living}) that $\delta(T_K(\param_f)) \leq \delta(T)$
where $T$ denotes the right hand side of \eqref{facon}. It is now easy
to see that
\begin{equation}
\delta(T) = \sum_{l_1=0}^{k_1 - 1} \dots \sum_{l_d = 0}^{k_d-1} \E
\norm{\Proj_{\mathcal{D}_{n_1(l_1), \dots,
  n_d(l_d)}} (Z(\rect_{l_1, \dots, l_d}))}^2
\end{equation}
where $Z \sim \Normal_\numobs(0, \Imat_\numobs)$ and $\Proj_{\mathcal{D}_{n_1(l_1), \dots,
  n_d(l_d)}}$  is the projection operator on the closed convex set $\mathcal{D}_{n_1(l_1), \dots,
  n_d(l_d)}$ Each addend on the right-hand side is simply the
risk of the NNLS estimator $\EMfitfun$ when the design points are
$(j_1/n_1(l_1), \dots, j_d/n_d(l_d)), j_s = 0, 1, \dots, n_s(l_s) - 1,
s = 1, \dots, d$ and when the true function $f^*$ is constantly equal
to zero. Thus, by the second term in \eqref{equation:worst_case},
and noting that the number of design points here is
$\prod_{s=1}^\usedim \numobs_s(l_s) \le \numobs$,
we obtain
\begin{equation}
\delta(T_K(\param_f)) \leq \delta(T) \le C_\usedim k
(\log (e \numobs))^{\frac{3\usedim}{2}} (\log (e \log (e \numobs)))^{\frac{2 \usedim - 1}{2}},
\end{equation}
which proves \eqref{rram} and completes the proof of
\autoref{theorem:NNLS_adapt}.


\subsection{Proof of \autoref{theorem:lasso_worst_case}}
\label{section:lasso_worst_proof}
Let
\begin{equation}
  \label{hkvth}
\paramhat \defn (\HKfitfun(\xvec_1), \ldots, \HKfitfun(\xvec_\numobs))
= \Altdesignmat \coefhathk ~~ \text{ and } ~~ \paramstar \defn
(\fstar(\xvec_1), \ldots, \fstar(\xvec_\numobs))
\end{equation}
where $\coefhathk$ is defined by the LASSO
problem~\eqref{equation:hk_lasso}. Note that $\Risk(\HKfitfun, \fstar)
= \frac{1}{\numobs} \E \norm{\paramhat   - \paramstar}^2$.

Similar to the proof of \autoref{theorem:NNLS_worst_case}, we take
$\noisestd = 1$ without loss of generality. To see this, note that we
can consider the scaled problem
$y_i / \noisestd = \fstar(\xvec_i) / \noisestd + \noise_i / \noisestd$
so that noise is scaled to have variance $1$
and the variation is now
$\HKVar(\fstar / \noisestd, [0, 1]^\usedim) \le \LASSOrad / \noisestd$.
Note also that
the estimator for the scaled problem is $\HKfitfun / \noisestd$
where $\HKfitfun$ is the estimator in the original problem.
We may apply the bound~\eqref{equation:lasso_worst_case} to the scaled problem,
and convert this into a bound on the risk of the original problem
by multiplying the bound by $\noisestd^2$
and replacing the variation term $\LASSOrad / \noisestd$
with $\LASSOrad$.
Thus, for the rest of the proof we assume $\noisestd = 1$.

Observe first that $\paramhat$ is the projection of $\yvec$ on the
closed convex set $\LASSOBall(\LASSOrad)$ defined in
\eqref{equation:lasso_ball}. We use \autoref{theorem:gencha} to bound $\E \norm{\paramhat
  - \paramstar}^2$ and the key is to bound the quantity
\begin{equation}\label{equation:crit_inequality}
\GFun(t) \defn \E \sup_{\param \in \LASSOBall(\LASSOrad) : \norm{\param - \paramstar}_2 \le t}
\inner{\noisevec, \param - \paramstar} \qt{for $t > 0$}
\end{equation}
where $\noisevec \sim \Normal(\zerovec, \Imat_\numobs)$ in order to
find $\radstar > 0$ such that $G(\radstar) \leq \radstar^2/2$.

Throughout, $\Altdesignmat$ is the design matrix from
\autoref{SECTION:COMPUTATION}.
If $\param = \Altdesignmat \coef$ and $\paramstar = \Altdesignmat \coefstar$
both belong to $\LASSOBall(\LASSOrad)$
then $\sum_{j = 2}^\numobs \abs{\coefplain_j - \coefplainstar_j} \le \sum_{j = 2}^\numobs \abs{\coefplain_j} + \sum_{j = 2}^\numobs \abs{\coefplainstar_j} \le 2 \LASSOrad$, so we have
\begin{equation}
\GFun(\rad) \le H(t) := \E \sup_{\alphavec \in
  \LASSOBall(2\LASSOrad) : \norm{\alphavec}_2 \le \rad}
\inner{\noisevec, \alphavec}.
\end{equation}
Let $\LASSOBall(\LASSOrad, \rad) \defn \LASSOBall(\LASSOrad) \cap \Ball_2(\zerovec, \rad)$.
We now use Dudley's entropy bound
(see \citet[Thm. 3.2]{chatterjee2015matrix})
to control the right hand side
above:
\begin{equation}
H(\rad)
\le c \int_0^\rad
\sqrt{\log \covernum(\radcover, \LASSOBall(2 \LASSOrad, \rad))}
\, d\radcover.
\end{equation}
The covering numbers above are bounded in the following lemma whose
proof is deferred to \autoref{section:proof_lasso_metric_entropy}.

\begin{lemma}\label{lemma:lasso_metric_entropy}
For every $V > 0$ and $t > 0$, we have
\begin{equation}
\log \covernum(\radcover, \LASSOBall(\LASSOrad, \rad))
\le
C_\usedim
\parens*{\frac{\LASSOrad \sqrt{\numobs}}{\radcover} + 1}
\parens*{
    \log\parens*{\frac{2 \LASSOrad \sqrt{\numobs}}{\radcover} + 1}
}^{\usedim - \frac{1}{2}}
+ \log\parens*{
    2 + 2 \frac{t + \LASSOrad \sqrt{\numobs}}{\radcover}
}.
\end{equation}
\end{lemma}
\autoref{lemma:lasso_metric_entropy} and the inequality $\sqrt{a^2 +
  b^2} \le a+b$ for $a,b \ge 0$ together give
\begin{align}
\sqrt{\log \covernum(\radcover, \LASSOBall(2\LASSOrad, t))}
&\le
C_\usedim
\sqrt{
    \parens*{\frac{\LASSOrad\sqrt{\numobs}}{\radcover} + 1}
    \parens*{\log\parens*{\frac{2\LASSOrad\sqrt{\numobs}}{\radcover} + 1}}^{\usedim - \frac{1}{2}}
}
\\
&\qquad +
C_\usedim
\sqrt{\log\parens*{
    2 + 2 \frac{t + \LASSOrad \sqrt{\numobs}}{\radcover}
}}
\end{align}
and thus
\begin{align}
H(t) \le
    &C_\usedim \int_0^{\rad}
    \sqrt{
        \parens*{\frac{\LASSOrad\sqrt{\numobs}}{\radcover} + 1}
        \parens*{\log\parens*{\frac{2\LASSOrad\sqrt{\numobs}}{\radcover} + 1}}^{\usedim - \frac{1}{2}}
    }
    \, d\radcover
    \\ &
    +
    C_\usedim
    \int_0^{\rad}
    \sqrt{\log\parens*{
        2 + 2 \frac{t + \LASSOrad \sqrt{\numobs}}{\radcover}
    }}
    \, d\radcover
\end{align}
We can upper bound the second integral as follows.

Let $B \defn 4 \rad + 2 \LASSOrad \sqrt{\numobs}$.
Using the fact that $\radcover \le \rad$ in the integral, and peforming some substitutions and integration by parts, we obtain
\begin{align}
&\int_0^{\rad}
\sqrt{\log\parens*{
    2 + 2 \frac{t + \LASSOrad \sqrt{\numobs}}{\radcover}
}}
\, d\radcover
\\
&\le
\int_0^\rad
\sqrt{\log \frac{4t +2V \sqrt{\numobs}}{\radcover}}
\, d\radcover
\\
&=
\int_0^\rad
\sqrt{\log \frac{B}{\radcover}}
\, d\radcover
\\
&=B \int_\alpha^\infty u^{1/2} e^{-u} \, du
& u = \log \frac{B}{\radcover}, \alpha \defn \log \frac{B}{\rad}
\\
&= B \sqrt{\alpha} e^{-\alpha}
+ B \int_\alpha^\infty \frac{e^{-u}}{2 \sqrt{u}} \, du
\end{align}
where the last step is due to integration by parts.
The last integral can be bounded by
\begin{equation}
\int_\alpha^\infty \frac{e^{-u}}{2 \sqrt{u}} \, du
\le \frac{1}{2 \sqrt{\alpha}} \int_\alpha^\infty
e^{-u} \, du
\le \frac{1}{2 \sqrt{\alpha}} e^{-\alpha}.
\end{equation}
Noting that
$\alpha = \log(B/\rad) \ge \log(4)$
and $B e^{-\alpha} = \rad$,
we obtain
\begin{align}
\int_0^{\rad}
\sqrt{\log\parens*{
    2 + 2 \frac{t + \LASSOrad \sqrt{\numobs}}{\radcover}
}}
\, d\radcover
&\le B e^{-\alpha}
\parens*{\sqrt{\alpha} + \frac{1}{2\sqrt{\alpha}}}
\\
&\le C \rad \sqrt{1 + \log(B/\rad)}
\\
&\le C \rad
\sqrt{\log(4 + 2 V \sqrt{\numobs} / \rad)}.
\end{align}

We now return to the first integral.
\begin{align}
&C_\usedim \int_0^{\rad}
\sqrt{
    \parens*{\frac{\LASSOrad\sqrt{\numobs}}{\radcover} + 1}
    \parens*{\log\parens*{\frac{2\LASSOrad\sqrt{\numobs}}{\radcover} + 1}}^{\usedim - \frac{1}{2}}
}
\, d\radcover
\\
&\le C_\usedim \int_0^\rad
\sqrt{
    \frac{\LASSOrad\sqrt{\numobs} + \rad}{\radcover}
    \parens*{\log \frac{2\LASSOrad\sqrt{\numobs} + \rad}{\radcover}}^{\usedim - \frac{1}{2}}
}
\, d\radcover
\\
&\le C_\usedim \sqrt{t (2V \sqrt{\numobs} + \rad)}
\parens*{\log \frac{e (2 \LASSOrad \sqrt{\numobs} + \rad)}{\rad}}^{\frac{2 \usedim - 1}{4}}
\\
&\le C_\usedim
\parens*{
    \rad + \sqrt{2 \rad \LASSOrad \sqrt{\numobs}}
}
\parens*{
    \log (1 + 2 e \LASSOrad\sqrt{\numobs} / \rad)
}^{\frac{2 \usedim - 1}{4}},
\end{align}
where we have used \autoref{lemma:bound_dudley_integral}
to bound the integral.

Combining these two terms yields
\begin{align}
\begin{split}\label{equation:lasso_gw}
G(\rad)
&\le
C_\usedim
\parens*{
    \rad + \sqrt{2 \rad \LASSOrad \sqrt{\numobs}}
}
\parens*{
    \log (1 + 2 e \LASSOrad\sqrt{\numobs} / \rad)
}^{\frac{2 \usedim - 1}{4}}
\\
&\qquad +
C_\usedim
\rad
\sqrt{\log(4 + 2 V \sqrt{\numobs} / \rad)}.
\end{split}
\end{align}

As always, the constants $C_\usedim$ that appear below vary from line to line.
We have
\begin{equation}
C_\usedim \rad
\parens*{
    \log (1 + 2 e \LASSOrad\sqrt{\numobs} / \rad)
}^{\frac{2 \usedim - 1}{4}}
\le \frac{\rad^2}{6}
\end{equation}
whenever $\rad \ge C_\usedim \max\braces*{
    1,
    \parens*{
        \log\parens*{1 + 2 e \LASSOrad\sqrt{\numobs}}
    }^{\frac{2 \usedim - 1}{4}}
}$.
We have
\begin{equation}
C_\usedim  \sqrt{2 \rad \LASSOrad \sqrt{\numobs}}
\parens*{
    \log (1 + 2 e \LASSOrad\sqrt{\numobs} / \rad)
}^{\frac{2 \usedim - 1}{4}}
\le \frac{\rad^2}{6}
\end{equation}
whenever $\rad \ge c_ \usedim \max\braces*{
    1,
    (\LASSOrad \sqrt{\numobs})^{1/3}
    \parens*{
        \log\parens*{1 + 2 e \LASSOrad\sqrt{\numobs}}
    }^{\frac{2 \usedim - 1}{6}}
}$.
Finally, we have
\begin{equation}
2 C_\usedim \rad
\sqrt{\log(4 + 2 \LASSOrad \sqrt{\numobs} / \rad)}
\le \frac{\rad^2}{8}
\end{equation}
whenever $\rad \ge C_\usedim\max\braces*{
    1,
    \sqrt{\log(4 + 2 \LASSOrad \sqrt{\numobs})}
}$.
So, with
\begin{equation}
\rad = C_\usedim \max\braces*{
    (\LASSOrad \sqrt{\numobs})^{1/3}
    \parens*{
        \log\parens*{1 + 2 e \LASSOrad\sqrt{\numobs}}
    }^{\frac{2 \usedim - 1}{6}},
    \sqrt{\log(4 + 2 \LASSOrad \sqrt{\numobs})},
    \parens*{
        \log\parens*{1 + 2 e \LASSOrad\sqrt{\numobs}}
    }^{\frac{2 \usedim - 1}{4}},
    1
}
\end{equation}
the above three inequalities hold, and we obtain
$G(\rad) \le \rad^2 / 2$,
and we may then use \autoref{theorem:gencha}
to obtain
\begin{align}
\Risk(\LASSOfit, \fstar)
\le \frac{\rad^2}{\numobs}
\le C_\usedim \max \Bigg\{&
    \parens*{\frac{\LASSOrad}{\numobs}}^{\frac{2}{3}}
    \parens*{\log(1 + 2 e \LASSOrad \sqrt{\numobs})}^{\frac{2 \usedim - 1}{3}},
    \frac{1}{\numobs}
    \log(4 + 2 \LASSOrad \sqrt{\numobs}),
\\
&\qquad
    \frac{1}{\numobs}
    \parens*{\log(1 + 2 e \LASSOrad \sqrt{\numobs})}^{\frac{2 \usedim - 1}{2}},
    \frac{1}{\numobs}
\Bigg\}.
\end{align}
We claim we can remove the log terms in the second and third terms as well.
Note that $\log(4 + x) \le x^{2/3}$ for $x \ge 3$.
Thus, we may bound the second term by
\begin{equation}
\frac{1}{\numobs}
\log(4 + 2 \LASSOrad \sqrt{\numobs})
\le \parens*{\frac{2 \LASSOrad}{\numobs}}^{\frac{2}{3}}
\Ind\{2 \LASSOrad \sqrt{\numobs} \ge 3\}
+ \frac{\log(7)}{\numobs}
\Ind\{2 \LASSOrad \sqrt{\numobs} < 3\}
\end{equation}
Similarly, $\log(1+x)^{\frac{2 \usedim - 1}{2}} \le x^{2/3}$ for $x \ge C_\usedim$, so
we may bound the third term by
\begin{align}
&\frac{1}{\numobs}
\parens*{\log(1 + 2 e \LASSOrad \sqrt{\numobs})}^{\frac{2 \usedim - 1}{2}}
\\
&\le \parens*{\frac{2 e \LASSOrad}{\numobs}}^{\frac{2}{3}}
\Ind\{2 e \LASSOrad \sqrt{\numobs} \ge C_\usedim\}
+ \frac{(\log(1 + C_\usedim))^{\frac{2 \usedim - 1}{2}}}{\numobs}
\Ind\{2 e \LASSOrad \sqrt{\numobs} < C_\usedim\}.
\end{align}

This allows us to rewrite our risk bound as
\begin{equation}
\Risk(\LASSOfit, \fstar)
\le
C_\usedim
\parens*{\frac{\LASSOrad}{\numobs}}^{\frac{2}{3}}
\parens*{\log(1 + 2 e \LASSOrad \sqrt{\numobs})}^{\frac{2 \usedim - 1}{3}}
+ C_\usedim
\frac{1}{\numobs}
\end{equation}
which is the desired bound in the case $\noisestd^2 = 1$.
The general result can be obtained by rescaling as discussed earlier.

\subsection{Proof of \autoref{theorem:constrained_em}}
\label{section:proof_constrained_em}

Let
\begin{equation}
\paramtilde \defn (\EMfitfuntwo(\xvec_1), \ldots, \EMfitfuntwo(\xvec_\numobs))
\text{ and }
\paramstar \defn (\fstar(\xvec_1), \ldots, \fstar(\xvec_\numobs))
\end{equation}
As discussed in \autoref{SECTION:COMPUTATION},
$\NNLSSet \cap (\paramplain_\numobs - \paramplain_1)
= \NNLSSet \cap \LASSOBall(\LASSOrad)$
(since if $\param = \Altdesignmat \coef \in \NNLSSet$
then $\paramplain_\numobs - \paramplain_1 = \sum_{j\ge 2} \coefplain_j = \sum_{j \ge 2} \abs{\coefplain_j}$),
and we have
\begin{equation}
\paramtilde = \argmin_{\param \in \NNLSSet \cap \LASSOBall(\LASSOrad)} \norm{\yvec - \param}^2
\end{equation}
As in \autoref{section:lasso_worst_proof}, we may without loss of generality assume $\noisestd^2 = 1$,
and then rescale to handle the general case.

We again appeal to \autoref{theorem:gencha}.
We need to bound
\begin{equation}
\E \sup_{
    \param \in \NNLSSet \cap \LASSOBall(\LASSOrad) :
    \norm{\param - \paramstar} \le \rad
}
\inner{\noise, \param - \paramstar}
\end{equation}
for $\rad > 0$
where $\noise \sim \Normal_\numobs(\zerovec, \Imat_\numobs)$.
But by removing the $\NNLSSet$ constraint in the supremum,
we immediately see that this quantity is bounded from above
by $G(\rad)$ as defined above~\eqref{equation:crit_inequality}.
Thus we may exactly follow the argument that bounds $G(\rad)$
in \autoref{section:lasso_worst_proof}, and ultimately end up with the same
bound~\eqref{equation:lasso_worst_case} in \autoref{theorem:lasso_worst_case}.


\subsection{Proof of Theorem~\ref{THEOREM:MINIMAX}}
\label{section:minimax_proof}

See the end of \autoref{section:em_minimax_proof}
for the proof of the tighter bound in the case $\usedim = 2$.

We use Assouad's lemma \cite{Assouad} (see also \cite{Yu97lecam} for
more discussion) in the following form:
\begin{lemma}[Assouad's lemma {\cite[Lemma 2]{Yu97lecam}}]\label{lemma:assouad}
Let $q$ be a positive integer, and assume that for every $\etavec \in \{-1, 1\}^q$
there is an associated function $f_{\etavec}$ satisfying $\HKVar(f_{\etavec}) \le \LASSOrad$.
Then
\begin{equation}
\MinimaxRisk \ge \frac{q}{2} \min_{\etavec \ne \etavec'}
\frac{\Loss(f_{\etavec}, f_{\etavec'})}{\Hamming(\etavec, \etavec')}
\min_{\Hamming(\etavec, \etavec') = 1}
\parens*{
    1 - \tvnorm{\P_{f_{\etavec}} - \P_{f_{\etavec'}}}
},
\end{equation}
where $\Loss(f,g) \defn \frac{1}{\numobs} \sum_{i=1}^\numobs (f(\xvec_i) - g(\xvec_i))^2$,
where $\P_f$ denotes the probability measure of $y_1, \ldots, y_n$ drawn
from the model~\eqref{obmod} where $\fstar = f$,
and where $\Hamming(\etavec,
\etavec') \defn \sum_{j=1}^q \Ind\{\eta_j \neq \eta'_j\}$ denotes the
Hamming distance.
\end{lemma}

Below we construct a collection of functions $\{f_{\etavec}, \etavec \in \{-1, 1\}^q\}\}$
such that the right-hand side of Assouad's bound above is the resulting bound
$C_\usedim (\noisestd^2 \LASSOrad / \numobs)^{2/3} (\log (\numobs(\LASSOrad/\noisestd)^2))^{2(\usedim - 1)/3}$
of
\autoref{THEOREM:MINIMAX},
but under the assumption that $\numobs_1 = \cdots = \numobs_\usedim$
and that $\numobs_1$ is a power of $2$.

Our construction of the functions $\{f_{\etavec}, \etavec \in \{-1,
1\}^q\}$ closely roughly mirrors that of
\citet[Section 4]{blei2007metric}. First let
\begin{equation}\label{equation:ell_choice}
\ell \defn \ceil*{
\frac{1}{3 \log 2} \parens*{
    \log(C_\usedim \numobs \LASSOrad^2 / \noisestd^2)
    - (\usedim  - 1)\log \log (C_\usedim \numobs \LASSOrad^2 / \noisestd^2)
}
}.
\end{equation}
The particular choice of this integer $\ell$ will be relevant
later. We define the index set
\begin{equation}
\MIndexSet_\ell \defn \braces*{
    (m_1, \ldots, m_\usedim) \in \mathbb{N}^\usedim :
    \sum_{j=1}^\usedim m_j = \ell,\
    \max_{j \in [\usedim]} m_j \le 2 \ell / \usedim
},
\end{equation}
and for each $\mvec \in \MIndexSet_\ell$ we define
\begin{equation}
\IIndexSet_{\mvec} \defn \braces*{
    (i_1, \ldots, i_\usedim) \in \mathbb{N}^\usedim :
    i_j \in [2^{m_j}] \text{ for each $j \in [\usedim]$}
}.
\end{equation}
One can check that $\abs{\IIndexSet_{\mvec}} = \prod_{j=1}^\usedim 2^{m_j} = 2^\ell$
for each $\mvec \in \MIndexSet_\ell$.
We also have the following lower bound which is proved in
\autoref{section:proof_mindexset_lb}.
\begin{lemma}\label{lemma:mindexset_lb}
There exist positive constants $a_\usedim$ and $c'_{\usedim,
  \noisestd^2 / \LASSOrad^2}$ such that
\begin{equation}
\abs{\MIndexSet_\ell} \ge a_\usedim \ell^{\usedim - 1} \qt{for all
  $\numobs \ge c'_{\usedim, \noisestd^2 / \LASSOrad^2}$}.
\end{equation}
\end{lemma}
Finally, let
\begin{equation}\label{equation:cardinality_lb}
q \defn \abs{\MIndexSet_\ell} \cdot 2^\ell
\end{equation}
be the cardinality of the set $\{(\mvec, \ivec) : \mvec \in \MIndexSet_\ell, \ivec \in \IIndexSet_{\mvec}\}$.
We index the components of $\etavec \in \{-1, 1\}^q$
by $\eta_{\mvec, \ivec}$ for
$\mvec \in \MIndexSet_\ell, \ivec \in \IIndexSet_{\mvec}$.

We now define a function $f_{\etavec}$
for each $\etavec \in \{-1, 1\}^q$.
For natural numbers $m$ and natural number $i \in [2^{m_j}]$ we define
the function $\phi_{m,i} : [0, 1]\to \R$ by
\begin{equation}\label{equation:phi_def}
\phi_{m, i}(x) \defn \begin{cases}
0 & x \notin [(i-1) 2^{-m}, i 2^{-m}],
\\
2^{-m - 2} & x = (i - \frac{3}{4}) 2^{-m},
\\
-2^{-m - 2} & x = (i - \frac{1}{4}) 2^{-m},
\\
\text{linear} & \text{otherwise}.
\end{cases}
\end{equation}
Note that consequently
\begin{equation}
\phi'_{m,i}(x) = \begin{cases}
1 & x \in ((i-1) 2^{-m}, (i - \frac{3}{4}) 2^{-m}) \cup
((i-\frac{1}{4}) 2^{-m},  i 2^{-m}),
\\
-1 & x \in ((i - \frac{3}{4}) 2^{-m}, (i - \frac{1}{4}) 2^{-m}).
\end{cases}
\end{equation}
We define the function $f_{\etavec} : [0, 1]^\usedim \to \R$ as
\begin{equation}
f_{\etavec}
\defn \frac{\LASSOrad}{\sqrt{\abs{\MIndexSet_\ell}}}
\sum_{\mvec \in \MIndexSet_\ell} \sum_{\ivec \in \IIndexSet_{\mvec}}
\eta_{\mvec, \ivec} \bigotimes_{j=1}^\usedim \phi_{m_j, i_j},
\end{equation}
that is,
\begin{equation}
f_{\etavec}(\xvec)
\defn \frac{\LASSOrad}{\sqrt{\abs{\MIndexSet_\ell}}}
\sum_{\mvec \in \MIndexSet_\ell} \sum_{\ivec \in \IIndexSet_{\mvec}}
\eta_{\mvec, \ivec} \prod_{j=1}^\usedim \phi_{m_j, i_j}(x_j).
\end{equation}

The following lemma (proved in \autoref{section:proof_3steps})
contains the key ingredients for the application of
\autoref{lemma:assouad}.

\begin{lemma}\label{lemma:3steps}
  For the functions $f_{\etavec}$ defined above, the following three
  inequalities hold.
\begin{equation}
  \label{eq:step1}
  \HKVar(f_{\etavec}; [0, 1]^\usedim) \le \LASSOrad,
\end{equation}
\begin{equation}
  \label{eq:step2}
\max_{\Hamming(\etavec, \etavec')=1}
\tvnorm{\P_{f_{\etavec}} - \P_{f_{\etavec'}}}
\le
\sqrt{
\frac{\numobs}{\noisestd^2}
\frac{\LASSOrad^2}{\abs{\MIndexSet_\ell}}
2^{-3 \ell - 4 \usedim}
},
\end{equation}
and
\begin{equation}
  \label{eq:step3}
  \min_{\etavec \ne \etavec'}
\frac{\Loss(f_{\etavec}, f_{\etavec'})}{\Hamming(\etavec, \etavec')}
\ge
\frac{4 \LASSOrad^2}{\abs{\MIndexSet_\ell}}
2^{-3 \ell - 6 \usedim}.
\end{equation}
\end{lemma}
The three inequalities in the above lemma, together with
\autoref{lemma:assouad}, \autoref{lemma:mindexset_lb} and
equation~\eqref{equation:cardinality_lb} imply
\begin{align}
\MinimaxRisk
&\ge
\frac{q}{2}
\cdot
\frac{4 \LASSOrad^2}{\abs{\MIndexSet_\ell}}
2^{-3 \ell - 6 \usedim}
\brackets*{
    1 - \sqrt{
        \frac{\numobs}{\noisestd^2}
        \frac{\LASSOrad^2}{\abs{\MIndexSet_\ell}}
        2^{-3 \ell - 4 \usedim}
    }
}
\\
&\ge
2^{\ell + 1}
\LASSOrad^2
2^{-3 \ell - 6 \usedim}
\brackets*{
    1 - \sqrt{
        \frac{\numobs}{\noisestd^2}
        \frac{\LASSOrad^2}{a_\usedim \ell^{\usedim - 1}}
        2^{-3 \ell - 4 \usedim}
    }
}
\\
&\ge \LASSOrad^2 2^{-2 \ell - 6 \usedim + 1}
\brackets*{
    1 - \sqrt{
        C_\usedim
        \frac{\numobs}{\noisestd^2}
        \frac{\LASSOrad^2}{\ell^{\usedim - 1}}
        2^{-3 \ell}
    }
}
\label{equation:minimax_step}
\end{align}
where $C_\usedim \defn 2^{-4\usedim} / a_\usedim$.

Note that our choice~\eqref{equation:ell_choice} of $\ell$ implies
\begin{equation}
2^{-\ell} = \parens*{
    \frac{\noisestd^2}{C_\usedim \numobs \LASSOrad^2}
}^{\frac{1}{3}}
\parens*{\log (C_\usedim \numobs \LASSOrad^2 / \noisestd^2)}^{\frac{\usedim - 1}{3}}.
\end{equation}
Then
\begin{align}
C_\usedim
\frac{\numobs}{\noisestd^2}
\frac{\LASSOrad^2}{\ell^{\usedim - 1}}
2^{-3 \ell}
&=
\parens*{\frac{
    \log (C_\usedim \numobs \LASSOrad^2 / \noisestd^2) \cdot  \log 2
}{
    \frac{1}{3} \log(C_\usedim \numobs \LASSOrad^2 / \noisestd^2)
    - \frac{\usedim  - 1}{3}\log \log (C_\usedim \numobs \LASSOrad^2 / \noisestd^2)
}}^{\usedim - 1}
\\
&= \parens*{
    \frac{3}{\log 2}
    \parens*{
        1 - (\usedim - 1) \frac{
            \log \log (C_\usedim \numobs \LASSOrad^2 / \noisestd^2)
        }{
            \log (C_\usedim \numobs \LASSOrad^2 / \noisestd^2)
        }
    }
}^{-(\usedim - 1)}.
\label{equation:pinsker_term}
\end{align}
For all $x > 1$ we have $\frac{\log \log x}{\log x} \le (\log x)^{-1/2}$.
Thus if we have
\begin{equation}\label{equation:large_enough}
\numobs \LASSOrad^2 / \noisestd^2 \ge e^{\usedim^2 / 4} / C_\usedim
\end{equation}
then we obtain
\begin{equation}
\frac{
    \log \log (C_\usedim \numobs \LASSOrad^2 / \noisestd^2)
}{
    \log (C_\usedim \numobs \LASSOrad^2 / \noisestd^2)
}
\le (\log (C_\usedim \numobs \LASSOrad^2 / \noisestd^2))^{-1/2}
\le \frac{2}{\usedim}.
\end{equation}
Applying this bound to the earlier equality~\eqref{equation:pinsker_term}
yields
\begin{equation}
C_\usedim
\frac{\numobs}{\noisestd^2}
\frac{\LASSOrad^2}{\ell^{\usedim - 1}}
2^{-3 \ell}
\le \parens*{\frac{2 \log 2}{3}}^{\usedim - 1}
\le \frac{1}{2}.
\end{equation}
Thus continuing from the earlier lower bound~\eqref{equation:minimax_step},
we obtain
\begin{align}
\MinimaxRisk
&\ge \tilde{c}_\usedim \LASSOrad^2
\parens*{
    \frac{\noisestd^2}{C_\usedim \numobs \LASSOrad^2}
}^{\frac{2}{3}}
\parens*{\log (C_\usedim \numobs \LASSOrad^2 / \noisestd^2)}^{\frac{2(\usedim - 1)}{3}}
\\
&= c'_\usedim
\parens*{
    \frac{\noisestd^2 \LASSOrad}{\numobs}
}^{\frac{2}{3}}
\parens*{\log (C_\usedim \numobs \LASSOrad^2 / \noisestd^2)}^{\frac{2(\usedim - 1)}{3}},
\end{align}
where $\tilde{c}_\usedim \defn 2^{-6 \usedim + 1} (1 - 2^{-1/2})$ and $c'_\usedim \defn C_\usedim^{-2/3} \tilde{c}_\usedim$,
provided the sample size condition~\eqref{equation:large_enough}
holds.

We claim we may replace $\log (C_\usedim \numobs \LASSOrad^2 / \noisestd^2)$
with $\log (\numobs (\LASSOrad / \noisestd)^2)$
in the above lower bound for sufficiently large $\numobs$.
Indeed as long as
$\numobs (\LASSOrad / \noisestd)^2 \ge C_\usedim^{-2}$
we have
$\log (C_\usedim \numobs \LASSOrad^2 / \noisestd^2) \ge \frac{1}{2} \log (\numobs (\LASSOrad / \noisestd)^2)$,
so we obtain
\begin{equation}
\MinimaxRisk
\ge c''_\usedim \parens*{
    \frac{\noisestd^2 \LASSOrad}{\numobs}
}^{\frac{2}{3}}
(\log (\numobs (\LASSOrad / \noisestd)^2))^{\frac{2(\usedim - 1)}{3}}
\end{equation}
for all $\numobs$ larger than a constant depending only on $\usedim$ and $\noisestd^2 / \LASSOrad^2$.

\paragraph{Relaxing assumptions}
We have proved the theorem under the assumption $\numobs_1 = \cdots = \numobs_\usedim$
with $\numobs_1$ a power of $2$.
We now argue that this suffices to handle the general case.
First, suppose $\numobs_1 = \cdots = \numobs_\usedim$,
but $\numobs_1$ is not a power of $2$.
Let $\numobs'_1$ be the largest power of $2$ less than $\numobs_1$, and let $\numobs' = (\numobs'_1)^\usedim$.
Then we may apply the argument on the $\numobs'_1 \times \cdots \times \numobs'_1$
and obtain a collection $\{f_{\etavec}, \etavec \in \{-1, 1\}^q\}$ such that the right-hand side
of Assouad's bound is
$C_\usedim (\noisestd^2 \LASSOrad / \numobs')^{2/3} (\log (\numobs'(\LASSOrad/\noisestd)^2))^{2(\usedim - 1)/3}$.
We now adapt this collection for our original $\numobs_1 \times \cdots \times \numobs_\usedim$ grid.
Since $\Loss(f_{\etavec}, f_{\etavec'})$ and $\tvnorm{\P_{f_{\etavec}} - \P_{f_{\etavec'}}}$ depend only
the values of the functions at the design points $\xvec_i$, we may assume without loss of generality
that the functions are piecewise constant with respect to the $\numobs'_1 \times \cdots \times \numobs'_1$ grid,
since keeping the values of $f_{\etavec}(\xvec_i)$ intact for all $\etavec$ and $\xvec_i$ while making the function piecewise constant elsewhere can only decrease the HK-variation, and thus not violate the $\HKVar(f_{\etavec}) \le \LASSOrad$ condition.
Note that $\numobs_1 - \numobs'_1 < \numobs'_1$. To move from the
$\numobs'_1 \times \cdots \times \numobs'_1$ grid
$\bigtimes_{j=1}^\usedim \{0, \frac{1}{\numobs'_1}, \ldots, \frac{\numobs'_1 - 1}{\numobs'_1}\}$
to a $\numobs_1 \times \cdots \times \numobs_\usedim$ grid,
we simply include the $\numobs_1 - \numobs'_1$
extra points $\frac{1}{2\numobs'_1}, \frac{3}{2\numobs'_1}, \ldots, \frac{2(\numobs_1 - \numobs'_1) - 1}{2 \numobs'_1}$
to the set $\{0, \frac{1}{\numobs'_1}, \ldots, \frac{\numobs'_1 - 1}{\numobs'_1}\}$
before taking the Cartesian product $\usedim$ times.
This is not an evenly spaced grid, but we may consider an isotonic function
$g$ that maps these $\numobs_1$ points
\begin{equation}
0, \frac{1}{2 \numobs'_1}, \frac{2}{2 \numobs'_1}, \ldots, \frac{2(\numobs_1 - \numobs'_1) - 1}{2\numobs'_1}, \frac{\numobs_1 - \numobs'_1}{\numobs'_1}, \frac{\numobs_1 - \numobs'_1+ 1}{\numobs'_1}, \ldots \frac{\numobs'_1 - 1}{\numobs'_1}
\end{equation}
to the evenly spaced grid
$0, \frac{1}{\numobs_1}, \ldots, \frac{\numobs_1 - 1}{\numobs_1}$, and let
$\ftilde_{\etavec} = f_{\etavec} \circ G$
where $G = \bigotimes_{i=1}^\usedim g$.

We now account for how the right-hand side of Assouad's bound (\autoref{lemma:assouad})
changes when using $\{\ftilde_{\etavec}\}$
on the full $\numobs_1 \times \cdots \times \numobs_\usedim$ grid
instead of $\{f_{\etavec}\}$ on the smaller grid.
Since HK variation is invariant under ``stretching'' of the domain,
$\HKVar(\ftilde_{\etavec}) = \HKVar(f_{\etavec}) \le \LASSOrad$.
Furthermore, since the $f_{\etavec}$ are piecewise constant,
the addition of the extra points simply means that certain values of $f_{\etavec}$
on the smaller grid appear up to $2^\usedim$ times as values of $\ftilde_{\etavec}$
on the larger grid (since $\numobs_j < 2 \numobs'_1$ for each $j$, and $\numobs < 2^\usedim \numobs'$).
Thus, using the fact that $\numobs' < \numobs < 2^\usedim \numobs'$,
the loss $\tilde{\Loss}(\ftilde_{\etavec}, \ftilde_{\etavec'})$
with respect to the larger grid
satisfies
\begin{equation}
2^{-\usedim} \Loss(f_{\etavec}, f_{\etavec'})
\le
\tilde{\Loss}(\ftilde_{\etavec}, \ftilde_{\etavec'})
\le \Loss(f_{\etavec}, f_{\etavec'})
\end{equation}
where $\Loss(f_{\etavec}, f_{\etavec'})$ is with respect to the smaller grid.
In particular, we still have the bound in~\eqref{eq:step2}
for $\tvnorm{\P_{\ftilde_{\etavec}} - \P_{\ftilde_{\etavec'}}}$,
since in the proof of~\eqref{lemma:3steps} we show
$\tvnorm{\P_{\ftilde_{\etavec}} - \P_{\ftilde_{\etavec'}}} \le
\sqrt{\frac{\numobs}{4 \noisestd^2} \Loss(\ftilde_{\etavec}, \ftilde_{\etavec'})}$.
For~\eqref{eq:step3}, we need to multiply the right-hand side by a factor of $2^{-\usedim}$,
which amounts to changing a few constants that depend on $\usedim$.
Thus, up to this $\usedim$-dependent factor, the result of \autoref{lemma:3steps} hold,
and we can apply Assouad's bound as before, with the only changes being an adjustment
in the constants that depend on $\usedim$.
Thus, we obtain a final lower bound of the form
\begin{equation}
\MinimaxRisk
\ge C_\usedim \parens*{
    \frac{\noisestd^2 \LASSOrad}{\numobs'}
}^{\frac{2}{3}}
(\log (\numobs' (\LASSOrad / \noisestd)^2))^{\frac{2(\usedim - 1)}{3}}.
\end{equation}
To conclude, note that $2^{-\usedim} \numobs \le \numobs' \le \numobs$,
so we have
\begin{equation}
\MinimaxRisk
\ge C'_\usedim \parens*{
    \frac{\noisestd^2 \LASSOrad}{\numobs}
}^{\frac{2}{3}}
(\log (\numobs (\LASSOrad / \noisestd)^2))^{\frac{2(\usedim - 1)}{3}}.
\end{equation}
for $\numobs$ larger than a [now slightly larger] constant depending only on $(\noisestd / \LASSOrad)^2$
and $\usedim$.

We have now proven the theorem under the assumption $\numobs_1 = \cdots = \numobs_\usedim$
where $\numobs_1$ is any sufficiently large positive integer.
The argument for relaxing this assumption to $\numobs_j \ge c \numobs^{1/\usedim}$ is similar.
We can consider a smaller square grid $\numobs'_1 \times \cdots \times \numobs'_1$
where $\numobs'_1 = c_s \numobs^{1/\usedim}$, and use the above argument
to obtain a collection of $f_{\etavec}$ (which may be assumed to be rectangular piecewise constant on the small grid)
for which Assouad's bound yields
$C_\usedim (\noisestd^2 \LASSOrad / \numobs')^{2/3} (\log (\numobs'(\LASSOrad/\noisestd)^2))^{2(\usedim - 1)/3}$
where $\numobs' = c_s^\usedim \numobs$.
To move to the larger grid, we need to add $\numobs_j - c_s \numobs'_1$ points to each dimension of the grid in the same fashion as above, by distributing them evenly among the gaps between the points of the smaller grid.
We can again make this larger grid evenly spaced by stretching the domain as before to obtain a new collection of functions $\ftilde_{\etavec}$.
Since we have enlarged the grid by a factor of $c_s^{-\usedim}$, each value of $f_{\etavec}$ on the small grid appears at most $c_s^{-\usedim}$ times as values of $\ftilde_{\etavec}$ on the larger grid.
Thus,
\begin{equation}
c_s^\usedim \Loss(f_{\etavec}, f_{\etavec'})
\le
\tilde{\Loss}(\ftilde_{\etavec}, \ftilde_{\etavec'})
\le \Loss(f_{\etavec}, f_{\etavec'})
\end{equation}
We may then use the bounds in \autoref{lemma:3steps} (with the bound~\eqref{eq:step3}
having an extra factor of $c_s^\usedim$ that will later be absorbed into constants)
and apply Assouad's bound to obtain the same bound $C_\usedim (\noisestd^2 \LASSOrad / \numobs')^{2/3} (\log (\numobs'(\LASSOrad/\noisestd)^2))^{2(\usedim - 1)/3}$.
Substituting $\numobs' = c_s^\usedim \numobs$ and absorbing $c_s^\usedim$ into the constant
and taking $\numobs$ larger than a constant depending only on $c_s$, $(\noisestd/\LASSOrad)^2$, and $\usedim$ yields the desired bound.


\subsection{Proof of \autoref{theorem:em_minimax}}
\label{section:em_minimax_proof}

Let us first consider the case $\noisestd^2 = 1$.
Let $\EMClass(\LASSOrad) \defn \{f \in \EMClass : \HKVar(f) \le \LASSOrad\}$.

Let $\DFClass$ denote the class of cumulative distribution functions of probability distributions on $[0, 1]^\usedim$.
We immediately have $\LASSOrad \DFClass \subseteq \EMClass(\LASSOrad)$, which implies
\begin{equation}
\inf_{\fhat_\numobs} \sup_{\fstar \in \EMClass(\LASSOrad)} \E_{\fstar} \Loss(\fhat_\numobs, \fstar)
\ge \inf_{\fhat_\numobs} \sup_{\fstar \in \LASSOrad\DFClass} \E_{\fstar} \Loss(\fhat_\numobs, \fstar).
\end{equation}
Thus it suffices to prove a minimax lower bound for $\LASSOrad \DFClass$.
To do so, we employ the Yang and Barron bound \cite{YangBarron},
roughly in the form appearing in \cite[Thm. IV.1]{GuntuFdiv}
(after specializing the Kullback-Leibler divergence
to our Gaussian model):
\begin{equation}\label{equation:yang_barron}
\inf_{\fhat_\numobs} \sup_{\fstar \in \DFClass} \E_{\fstar} \Loss(\fhat_\numobs, \fstar)
\ge \frac{\eta^2}{4} \parens*{
    1 - \frac{\log 2 + \log \covernum(\epsilon / \LASSOrad; \DFClass)
    + \numobs \epsilon^2}{\log \packnum(\eta / \LASSOrad; \DFClass)}
}
\end{equation}
for any positive $\eta$ and  $\epsilon$.
Here, $\covernum(\epsilon; \DFClass)$ is the covering number of $\DFClass$
(cardinality $\covernum$ of the smallest set $g^1, \ldots, g^\covernum$ satisfying
$\min_j \Loss(f^j, g) \le \epsilon^2$
for any $g \in \DFClass$)
and $\packnum(\eta; \DFClass)$ is the packing number of $\DFClass$
(cardinality $\packnum$ of the largest set $g^1, \ldots, g^\packnum$ satisfying
$\Loss(f^j, f^k) > \eta^2$
for all $j \ne k$).

\textbf{The case $\usedim \ge 2$.} We first prove the general minimax bound for cases $\usedim \ge 2$,
before specializing to the case $\usedim = 2$.
We claim
\begin{subequations}
\begin{align}
\log \covernum(\epsilon'; \DFClass)
&\le
C_\usedim \frac{1}{\epsilon'} \parens*{
    \log \frac{1}{\epsilon'}
}^{\usedim - \frac{1}{2}},
& \epsilon' < e^{-1}
\label{equation:df_cover}
\\
\log \packnum(\eta'; \DFClass)
&\ge
C_\usedim \frac{1}{\eta'} \parens*{
    \log \frac{1}{\eta'}
}^{\usedim - 1}.
\label{equation:df_pack}
\end{align}
\end{subequations}
Assuming these two equations are true, then applying the Yang-Barron bound~\eqref{equation:yang_barron}
with $\epsilon = a_\usedim (\LASSOrad/\numobs)^{\frac{1}{3}} (\log (\numobs \LASSOrad^2))^{\frac{2 \usedim - 1}{6}}$
and $\eta = b_\usedim (\LASSOrad / \numobs)^{\frac{1}{3}} (\log (\numobs \LASSOrad^2))^{\frac{\usedim - 2}{3}}$,
for certain constants $a_\usedim$ and $b_\usedim$,
allows us to conclude the proof.
Specifically, we then have
$\numobs \epsilon^2
= a_\usedim^2 (\numobs \LASSOrad^2)^{\frac{1}{3}} (\log (\numobs \LASSOrad^2))^{\frac{2 \usedim - 1}{3}}$
as well as
\begin{align}
&\log \covernum(\epsilon / \LASSOrad; \DFClass)
\\
&= C_\usedim a_\usedim (\numobs \LASSOrad^2)^{\frac{1}{3}}
(\log (\numobs \LASSOrad^2))^{- \frac{2 \usedim - 1}{6}}
\brackets*{
    \frac{1}{3} \log (\numobs \LASSOrad^2 / a_\usedim^3) - \frac{2 \usedim - 1}{6} \log \log (\numobs \LASSOrad^2)
}^{\usedim - \frac{1}{2}}
\\
&\lesssim (\numobs \LASSOrad^2)^{\frac{1}{3}} (\log (\numobs \LASSOrad^2))^{\frac{2 \usedim - 1}{3}}
\end{align}
and
\begin{align}
&\log \packnum(\eta / \LASSOrad; \DFClass)
\\
&= C_\usedim b_\usedim (\numobs \LASSOrad^2)^{\frac{1}{3}}
(\log (\numobs \LASSOrad^2))^{- \frac{\usedim - 2}{3}}
\brackets*{
    \frac{1}{3} \log (\numobs \LASSOrad^2 / b_\usedim) -\frac{\usedim - 2}{3} \log \log (\numobs \LASSOrad^2)
}^{\usedim  -1}
\\
&\gtrsim (\numobs \LASSOrad^2)^{\frac{1}{3}} (\log (\numobs \LASSOrad^2))^{\frac{2 \usedim - 1}{3}}.
\end{align}
In particular, the quantities
$\numobs \epsilon^2$, $\log \covernum(\epsilon; \DFClass)$, and $\log \packnum(\eta; \DFClass)$
are of the same order,
so a judicious choice of constants $a_\usedim$ and $b_\usedim$
will make the Yang-Barron bound~\eqref{equation:yang_barron}
be on the order of
\begin{equation}\label{equation:em_minimax_1}
\eta^2 \asymp \parens*{\frac{\LASSOrad}{\numobs}}^{\frac{2}{3}} (\log (\numobs \LASSOrad^2))^{\frac{2(d-2)}{3}},
\end{equation}
which yields the desired minimax bound in the case $\noisestd^2 = 1$.
Note that $\numobs$ must be sufficiently large (larger than a constant depending on $\usedim$ and $\LASSOrad$)
in order for $\radcover / \LASSOrad < e^{-1}$ in order to use the covering number bound~\eqref{equation:df_cover}.
For general $\noisestd^2$ and $\LASSOrad$, we may rescale the problem
to have noise level $(\noisestd')^2 = 1$ and variation $\LASSOrad' = \LASSOrad / \noisestd$,
apply the above bound~\eqref{equation:em_minimax_1}, and multiply
by $\noisestd^2$ to obtain the final minimax bound that appears in \autoref{theorem:em_minimax}.

It now remains to verify the above two claims.
The first claim~\eqref{equation:df_cover} is due to \citet{blei2007metric}; see~\eqref{equation:blei} with $\TVrad = 1$ and note that our notion of distance in the present proof is normalized by $\numobs$.

We now turn to the other claim~\eqref{equation:df_pack}.
Let $\ell$, $\MIndexSet_\ell$, $q \defn \abs{\MIndexSet_\ell} 2^\ell$, and
$\{f_{\etavec} : \etavec \in \{-1, 1\}^\usedim\}$
be as defined in \autoref{section:minimax_proof}
(see \eqref{equation:ell_choice}, \eqref{equation:cardinality_lb}, etc.),
and let with $\LASSOrad = 1$.
Note that the $f_{\etavec}$ are continuous functions with $\HKVar(f_{\etavec}) \le 1$,
so they belong to $\DFClass - \DFClass$.

The Gilbert-Varshamov lemma (see \cite[Lemma 4.7]{Massart03Flour})
guarantees a subset $T \subseteq \{-1, 1\}^q$ satisfying $\log \abs{T} \gtrsim q$
and $\Hamming(\etavec, \etavec') \gtrsim q/2$ for all distinct $\etavec, \etavec' \in T$.
Recalling from \autoref{lemma:3steps} that 
\begin{equation}\label{equation:em_hamming_lb}
  \min_{\etavec \ne \etavec'}
\frac{\Loss(f_{\etavec}, f_{\etavec'})}{\Hamming(\etavec, \etavec')}
\ge
\frac{2^{-3 \ell - 6 \usedim + 2}}{\abs{\MIndexSet_\ell}},
\end{equation}
we obtain a packing set $\{f_{\etavec} : \etavec \in T\}$ of $\DFClass - \DFClass$
satisfying $\Loss(f_{\etavec}, f_{\etavec'}) \ge \frac{q \cdot 2^{-3 \ell - 6 \usedim + 1}}{\abs{\MIndexSet_\ell}}
= 2^{-2 \ell - 6 \usedim + 1} \eqqcolon (\eta')^2$.
Note that $\ell = c \log \frac{1}{\eta'}$.
Recalling $\abs{\MIndexSet_\ell} \gtrsim \ell^{\usedim - 1}$ from \autoref{lemma:mindexset_lb},
the log cardinality of this packing set $\{f_{\etavec} : \etavec \in T\}$ with radius $\eta'$ is
\begin{equation}
\log \packnum(\eta'; \DFClass - \DFClass)
= \abs{\MIndexSet_\ell} 2^\ell \gtrsim 2^\ell \ell^{\usedim - 1}
\asymp \frac{1}{\eta'} \parens*{\log \frac{1}{\eta'}}^{\usedim - 1}.
\end{equation}
Using basic relationships between covering numbers and packing numbers, we have
\begin{align}
\frac{1}{\eta'} \parens*{\log \frac{1}{\eta'}}^{\usedim - 1}
&\lesssim
\log \packnum(\eta'; \DFClass - \DFClass)
\\
&\le \log \covernum(\eta' / 2; \DFClass - \DFClass)
\\
&\overset{(*)}{\le} 2 \log \covernum(\eta' / 4; \DFClass)
\\
&\le 2 \log \packnum(\eta' / 4; \DFClass),
\end{align}
where the starred inequality is due to the fact that one can obtain a covering set for
$\DFClass - \DFClass$ by taking a covering set for $\DFClass$ with half the radius,
and taking the differences between all pairs drawn from the covering set.

The only place we used the assumption that
$\numobs_1 = \cdots = \numobs_\usedim$ with $\numobs_1$ a power of $2$
is in our appeal to the construction of $\{f_{\etavec}\}$ in proving
the packing bound~\eqref{equation:df_pack}. We may follow the same argument
as in the end of \autoref{section:minimax_proof} to relax these assumptions
to the setting of the theorem and obtain the same risk lower bound,
since the argument there only results in
changing the right-hand side of the lower bound~\eqref{equation:em_hamming_lb}
by a factor that depends on $\usedim$ and $c_s$.

\textbf{The case $\usedim = 2$.}
We now prove the tighter bound in the case $\usedim = 2$,
which will follow from tightening the packing number bound~\eqref{equation:df_pack}.

We again refer to notation in \autoref{section:minimax_proof}.
Let
\begin{equation}
\tilde{\MIndexSet}_\ell
\defn \{(m_1, m_2) \in \mathbb{N}^\usedim :
m_1 + m_2 = \ell, \text{ $m_1$ and $m_2$ both even}
\}
\end{equation}
and let $\tilde{q} \defn \abs{\tilde{\MIndexSet}_\ell} 2^\ell$.
Let $\phi_{m,i}$ be as before~\eqref{equation:phi_def}.
%
For $\etavec \in \{-1, 1\}^q$, we define
\begin{equation}
F_{\etavec, \mvec}(t_1, t_2)
\defn
\sum_{\ivec \in \IIndexSet_{\mvec}} \eta_{\mvec, \ivec}
\phi_{m_1, i_1}'(t_1) \phi_{m_2, i_2}'(t_2)
\end{equation}
and
\begin{equation}
\ftilde_{\etavec}(\xvec)
\defn \int_0^{x_1} \int_0^{x_2}
\prod_{\mvec \in \tilde{\MIndexSet}_\ell} \parens*{
    1 + F_{\etavec, \mvec}(t_1, t_2)
}
\, dt_1 \, dt_2
\end{equation}
Note that we can rewrite this function as
\begin{equation}
\ftilde_{\etavec}(\xvec)
= x_1 x_2
+ \sum_{\mvec \in \tilde{\MIndexSet}_\ell}
\sum_{\ivec \in \IIndexSet_{\mvec}} \eta_{\mvec, \ivec}
\phi_{m_1, i_1}(x_1) \phi_{m_2, i_2}(x_2)
+ Q_{\etavec}(\xvec),
\end{equation}
where
\begin{equation}
Q_{\etavec}(\xvec)
\defn
\sum_{P \ge 2}
\sum_{k_1, \ldots, k_P}
\int_0^{x_1} \int_0^{x_2}
\prod_{p=1}^P F_{\etavec, (k_p, \ell-k_p)}(t_1, t_2)
\, dt_1 \, dt_2,
\end{equation}
and the inner sum above is over even integers $0 \le k_1 < k_2 < \cdots < k_P \le \ell$.

These functions satisfying the following properties (proved in \autoref{section:proof_function_packing_two}).
\begin{lemma}\label{lemma:function_packing_two}
The functions $\ftilde_{\etavec}$ belong to $\DFClassTwo$
and satisfy
\begin{equation}
\label{equation:discrete_lb_two}
\min_{\etavec \ne \etavec'}
\frac{\Loss(\ftilde_{\etavec}, \ftilde_{\etavec'})}{\Hamming(\etavec, \etavec')}
\ge 2^{-3 \ell - 10}.
\end{equation}
\end{lemma}

From here, we apply the Gilbert-Varshamov lemma again
to obtain a subset $T \subseteq \{-1, 1\}^{\tilde{q}}$
satisfying $\log \abs{T} \gtrsim \tilde{q}$ and $\Hamming(\etavec, \etavec') \ge \tilde{q}/2$
for all distinct $\etavec, \etavec' \in T$.
From the above inequality, we can obtain
a packing set $\{\ftilde_{\etavec} : \etavec \in T\}$ of $\DFClassTwo$
satisfying $\Loss(\ftilde_{\etavec}, \ftilde_{\etavec'}) \ge \tilde{q} \cdot 2^{-3 \ell - 11} \gtrsim \ell \cdot 2^{-2\ell} \eqqcolon (\eta')^2$ where we have used $\tilde{q} = \abs{\tilde{\MIndexSet}_\ell} 2^\ell \gtrsim \ell \cdot 2^\ell$.
Note that then we have
\begin{equation}
\frac{1}{\eta'} \parens*{\log \frac{1}{\eta'}}^{3/2} = 2^\ell \ell^{-1/2} (c \ell - \frac{1}{2} \log \ell)^{3/2}
\lesssim \ell \cdot 2^{\ell}
\lesssim \tilde{q}
\le \log \packnum(\eta'; \DFClassTwo)
\end{equation}
since $\tilde{q} \lesssim \log \abs{T}$.
Note that this packing number bound
is of the same order as the earlier covering number bound~\eqref{equation:df_cover}.

We now return to the Yang-Barron bound~\eqref{equation:yang_barron} with
$\epsilon = a (\LASSOrad / \numobs)^{\frac{1}{3}} (\log(\numobs \LASSOrad^2))^{\frac{1}{2}}$
and $\eta = b (\LASSOrad / \numobs)^{\frac{1}{3}} (\log(\numobs \LASSOrad^2))^{\frac{1}{2}}$.
We have
$\numobs \epsilon^2 \asymp (\numobs \LASSOrad^2)^{\frac{1}{3}} \log(\numobs \LASSOrad^2)$
as well as
\begin{equation}
\log \covernum(\epsilon / \LASSOrad; \DFClassTwo)
\lesssim (\numobs \LASSOrad^2)^{\frac{1}{3}}  \log (\numobs \LASSOrad^2)
\lesssim
\log \packnum(\eta / \LASSOrad; \DFClassTwo).
\end{equation}
Thus with appropriate choices of constants, obtain a lower bound on the minimax risk on the order of
\begin{equation}
\eta^2 \asymp \parens*{\frac{\LASSOrad}{\numobs}}^{\frac{2}{3}} \log(\numobs \LASSOrad^2),
\end{equation}
in the case $\noisestd^2 = 1$. Repeating the rescaling argument produces the bound for general $\noisestd^2$.

We can relax the assumption that
$\numobs_1 = \numobs_2$ with $\numobs_1$ a power of $2$
in the same manner as before, and again, the result of applying
the same argument amounts to an additional factor depending only on $c_s$
for the lower bound in \autoref{equation:discrete_lb_two}.

Having proved the tighter minimax lower bound
of \autoref{theorem:em_minimax}
in the case $\usedim = 2$,
we note that the analogous bound of \autoref{THEOREM:MINIMAX} follows immediately, since
\begin{equation}
\{\fstar \in \EMClassPlain^2 : \HKVar(\fstar) \le \LASSOrad\}
\subseteq
\{\fstar : \HKVar(\fstar) \le \LASSOrad\}.
\end{equation}


\subsection{Proofs of \autoref{theorem:lasso_adaptive_d} and \autoref{theorem:lasso_adaptive}}
\label{section:proof_lasso_adaptive}
We shall first introduce some notation and state some auxiliary
results which will hold for every $\usedim \ge 1$ and which will used
in the proofs of both \autoref{theorem:lasso_adaptive_d} and
\autoref{theorem:lasso_adaptive}. After that we shall give the proofs
of \autoref{theorem:lasso_adaptive_d} and
\autoref{theorem:lasso_adaptive} separately in two subsections.

Throughout, $\Altdesignmat$ is the design matrix from
\autoref{SECTION:COMPUTATION}. As observed in
\autoref{section:esld}, $\Altdesignmat$ is
square and invertible (note that we are working under the assumption
that $\xvec_1, \dots, \xvec_\numobs$ come from the lattice design
\eqref{equation:lattice_design}). This means that every $\param \in
\R^n$ can be expressed as $\param = \Altdesignmat \coef$ for a unique
$\coef \in \R^n$. By an abuse of notation, we define
\begin{equation}
  \HKVar(\param)   := \sum_{j=2}^n |\beta_j|
\end{equation}
where $(\beta_1, \dots, \beta_\numobs)$ are the components of $\coef$. This
abuse of notation is justified by noting that if $\param =
\Altdesignmat \coef$, then $\param = (f(\xvec_1), \dots, f(\xvec_\numobs))$
for $f := \sum_{i=1}^m \beta_i \Ind_{[\xvec_i, 1]}$. For this function
$f$, it is easy to see that $\HKVar(f) = \sum_{j=2}^n |\beta_j|$. In
other words, we are defining $\HKVar(\param)$ to be equal to
$\HKVar(f)$ for a specific canonical function on $[0, 1]^d$ which
satisfies $f(\xvec_i) = \param_i$  for each $i = 1, \dots, n$.

We shall say that a vector $\param = \Altdesignmat \coef \in \R^n$ is
entirely monotone if $\min_{j \geq 2} \beta_{j} \geq 0$. This can be
justified by noting that the function $f := \sum_{i=1}^m \beta_i
\Ind_{[\xvec_i, 1]}$ belongs to $\EMClass$ if and only if $\min_{j
  \geq 2} \beta_{j} \geq 0$. We also say that $\param = \Altdesignmat
\coef$ is nearly entirely monotone if
\begin{equation}\label{ncm}
\sum_{j=2}^n \left( |\beta_j| - \beta_j \right)  \leq \delta
\end{equation}
for a small $\delta > 0$. Note that, by the definition of
$\HKVar(\param)$, this is equivalent to the inequality:
$\HKVar(\param) \leq \theta_\numobs - \theta_1 + \delta$. Note that if
$\param$ is entirely monotone, then \eqref{ncm} is true with $\delta
= 0$ and this justifies the terminology of nearly entirely
monotone.

We also use the notation in \eqref{hkvth}. Because $\HKfitfun$ is the LSE over the class $\LASSOBall(\LASSOrad)$,
inequality \eqref{wma} with $K = \LASSOBall(\LASSOrad)$ gives
\begin{equation}\label{equation:bellec_oracle2}
\Risk(\HKfitfun, \fstar) = \E \frac{1}{n} \norm{\paramhat - \paramstar}^2
\le \frac{1}{\numobs} \inf_{\paramtilde \in K}
\braces*{
     \norm{\paramtilde - \paramstar}^2
    + \noisestd^2 \GWidth^2(\TConeLASSO(\paramtilde))
    + \noisestd^2
}.
\end{equation}
To further bound the right hand side above, it is important to
understand the structure of the tangent cone
$\TConeLASSO(\paramtilde)$. The following result (proved in
\autoref{section:proof_tangent_cone}) provides an explicit
characterization of this tangent cone at $\paramtilde = A \coeftilde$.

\begin{lemma}
\label{lemma:tangent_cone}
Suppose $\coeftilde$ is such that $\Altdesignmat \coeftilde \in \LASSOBall(\LASSOrad)$.
Then the tangent cone of $\LASSOBall(\LASSOrad)$ at $\Altdesignmat \coeftilde$ is
\begin{equation}
\label{equation:tangent_cone}
\TConeLASSO (\Altdesignmat \coeftilde) = \braces*{
    \Altdesignmat \coef :
    \sum_{\substack{j \ge 2 : \coefplaintilde_j = 0}}
    \abs{\coefplain_j}
    \le
    - \sum_{\substack{j \ge 2 : \coefplaintilde_j \ne 0}}
    \coefplain_j \sign(\coefplaintilde_j)
},
\end{equation}
if $\sum_{j = 2}^{\numobs} \abs{\coefplaintilde_j} = \LASSOrad$; otherwise,
$\TConeLASSO (\Altdesignmat \coeftilde) = \R^\numobs$.
\end{lemma}
The structure of the tangent cone given above (in the case $\sum_{j = 2}^{\numobs}
\abs{\coefplaintilde_j} = \LASSOrad$) has the implication that, when
$\coeftilde$ corresponds to a function of the form
\eqref{equation:two_piece}, every vector in
$\TConeLASSO(\Altdesignmat \coeftilde)$ can be broken down
into lower-dimensional elements each of which is either nearly
entirely monotone or has low HK$\zerovec$ variation. This is the
content of the next result. For this result, it will be
necessary, as in \autoref{section:esld}, to view vectors in $\R^n$ as
arrays in $\R^{n_1} \times \dots \times \R^{n_d}$. Indeed, we shall denote the
elements $\param \in \R^n$ by $\param_{\ivec}, \ivec \in \IndexSet$
(where $\IndexSet$ is as defined in \eqref{AllInd}. Note that the
columns of the design matrix $\Altdesignmat$ can also be indexed
in this way so that the $\ivec^{th}$ column (where $\ivec = (i_1,
\dots, i_d)$) of $\Altdesignmat$ corresponds to the vector
\eqref{vzdef} with $\zvec = (i_1/n_1, \dots, i_d/n_d)$. Note that this
implies that the $\zerovec^{th}$ column is the column of ones.

\begin{lemma}\label{lemma:tcone_inequality}
Let $\coeftilde \in \R^{\numobs_1 \times \cdots \times \numobs_\usedim}$
satisfy $\coefplaintilde_{\ivec} = 0$ for all $\ivec \notin \{\zerovec, \ivec^*\}$
for some $\ivec^*$. Let $\ivec^u$ and $\ivec^\ell$ be two indices such
that $\ivec^* \preceq \ivec^u$ and $\ivec^\ell \nsucceq \ivec^*$,
and let $\LL_u \defn \{\ivec : \ivec \preceq \ivec^u\}$
and $\LL_\ell \defn \{\ivec : \ivec \preceq \ivec^\ell\}$.
Then for every $\alphavec = \Altdesignmat \coef \in \TConeLASSO(\Altdesignmat \coeftilde)$
where $\LASSOrad = \HKVar(\Altdesignmat \coeftilde) = \sum_{\ivec \ne
  \zerovec} \abs{\coefplaintilde_{\ivec}}$,
we have
\begin{equation}\label{equation:tcone_inequality}
\sum_{\ivec \notin \{\zerovec, \ivec^*\}}
\parens*{
    \abs{\coefplain_{\ivec}}
    - \mysign(\ivec) \coefplain_{\ivec}
}
\le
- \sign(\coefplaintilde_{\ivec^*}) (\alpha_{\ivec^u} - \alpha_{\ivec^\ell}),
\end{equation}
where
\begin{equation}\label{equation:sign_def}
\mysign(\ivec) \defn \begin{cases}
1 & \ivec \in \LL_u \cap \LL_\ell^c \setminus \{\ivec^*\}
\\
-1 & \ivec \in \LL_u^c \cap \LL_\ell
\\
0 & \ivec \in (\LL_u \cap \LL_\ell) \cup (\LL_u^c \cap \LL_\ell^c)
\setminus \{\zerovec\}
\end{cases}
\end{equation}
\end{lemma}

\autoref{lemma:tcone_inequality} will be used to bound the Gaussian
width $\GWidth(\TConeLASSO(\Altdesignmat \coeftilde))$  for every
$\coeftilde$ as in the statement of \autoref{lemma:tcone_inequality}
in the following way. Assume first that
$\ivec^u$ and $\ivec^\ell$ are chosen so that the right hand side of
\eqref{equation:tcone_inequality} is small. Specifically, for
$\coeftilde \in \R^{n_1 \times \dots \times n_d}$ and indices
$\ivec^*, \ivec^u, \ivec^{\ell}$ as in the statement of
\autoref{lemma:tcone_inequality} and a fixed $\delta \geq 0$, let
\begin{equation}\label{equation:tconesub_def}
\TConeSub(\ivec^u, \ivec^\ell, \delta)
\defn \braces*{
    \alphavec \in \TConeLASSO(\Altdesignmat \coeftilde) :
    \abs{\alpha_{\ivec^u} - \alpha_{\ivec^\ell}} \le \delta
}
\cap \Ball_2(\zerovec, 1),
\end{equation}
where $\Ball_2(\zerovec, 1) \defn \{\param : \norm{\param}
< 1\}$. The intersection with the unit ball here arises because of the
presence of the unit norm restriction in the definition of the
Gaussian width (see \eqref{GauWid}). For every $\alphavec =
\Altdesignmat \coef \in \TConeSub(\ivec^u, \ivec^\ell, \delta)$, it is
clear that:
\begin{equation}
  \label{equation:tcone_inequality_delta}
  \sum_{\ivec \notin \{\zerovec, \ivec^*\}}
\parens*{
    \abs{\coefplain_{\ivec}}
    - \mysign(\ivec) \coefplain_{\ivec}
}
\le |\alpha_{\ivec^u} - \alpha_{\ivec^\ell}| \leq \delta.
\end{equation}
Suppose now that $\delta$ is small. Then, if we
restrict the indices $\ivec$ to the set $L_u \cap L_\ell^c \setminus \{\ivec^*\}$,
we would have $\mysign(\ivec) = 1$ according to \eqref{equation:sign_def}
and, consequently,

The inequality \eqref{equation:tcone_inequality_delta} implies that
\begin{equation}\label{equation:global_ncm}
\sum_{\ivec \notin \LL_u \cap \LL_{\ell}^c \setminus \{\ivec^*\}}
(\abs{\coefplain_{\ivec}} - \coefplain_{\ivec})
\le \delta,
\end{equation}
which resembles the definition of nearly entire monotonicity~\eqref{ncm}.
This might suggest that the restriction of $\alphavec$
to its components indexed by $\LL_u \cap \LL_{\ell}^c \setminus \{\ivec^*\}$
is nearly entirely monotone, but there are a few issues,
one of which is that the definition of nearly entire monotonicity
for a sub-array $\alphavec_{\SubRect}$ of $\alphavec$ is not quite the same as
taking the condition~\eqref{equation:global_ncm} and taking the sum
only over indices $\ivec$ in the subset $\SubRect$
(specifically, the $\coefplain_{\ivec}$ terms should also be replaced
with the analogous quantities for $\alphavec_{\SubRect}$,
which are different than the original $\coefplain_{\ivec}$ terms derived
from the full array $\alphavec$).
Similarly we also have
$\sum_{\ivec \notin \LL_u^c \cap \LL_{\ell}}
(\abs{\coefplain_{\ivec}} + \coefplain_{\ivec})
\le \delta$
and
$\sum_{\ivec \notin (\LL_u \cap \LL_\ell) \cup (\LL_u^c \cap \LL_{\ell}^c) \setminus \{\zerovec\}}
\abs{\coefplain_{\ivec}}
\le \delta$,
which also might suggest nearly entire monotonicity of
$-\alphavec$ on $\LL_u^c \cap \LL_{\ell}$
and low HK$\zerovec$ variation on
$(\LL_u \cap \LL_\ell) \cup (\LL_u^c \cap \LL_{\ell}^c) \setminus \{\zerovec\}$
respectively,
but for similar reasons is not immediately true.



Another complication is
that the sets $L_u \cap L_\ell^c \setminus \{\ivec^*\}$ and $(\LL_u \cap \LL_\ell) \cup
(\LL_u^c \cap \LL_\ell^c) \setminus \{\zerovec\}$ are not necessarily
rectangular.
To deal with these above issues, we show that we can further
partition these sets into
rectangles such that $\alphavec$ restricted to each rectangle is indeed
either nearly entirely monotone or has small HK$\zerovec$
variation.
This observation
would allow us to bound $\GWidth(\TConeLASSO(\paramtilde))$ based on bounds
for the Gaussian width of nearly entirely monotone vectors and
vectors with small HK$\zerovec$ variation.

The following result gives conditions on a rectangle $Q$
such that the above holds.
To state this result, it will be
convenient to introduce the following notation.
For each $\param \in \R^n$, let $D \param$ denote the differenced
vector defined as in \eqref{equation:diff_def}. It is easy to check
that
\begin{align}
(\Diff \param)_{\zerovec} \defn \paramplain_{\zerovec} ~ \text{ and }
  ~ \paramplain_{\ivec} \defn \sum_{\ivec' : \ivec' \preceq \ivec} (\Diff \param)_{\ivec}
\text{ for } \ivec \ne \zerovec.
\label{equation:sum_lower}
\end{align}
As a result, it follows that $D \param = \Altdesignmat^{-1} \param$ or,
equivalently, $\param = \Altdesignmat(D \param)$.

Every two indices $\qvec^\ell$ and $\qvec^u$ in $\IndexSet$ with
$\qvec^\ell \preceq \qvec^u$ define the following rectangle in
$\IndexSet$:
\begin{equation}
  \label{Qrect}
  \SubRect \defn
[\qvec^\ell, \qvec^u] := \left\{\ivec \in \IndexSet : \qvec^\ell
  \preceq \ivec \preceq \qvec^u \right\}
\end{equation}
For this rectangle $Q$ and an arbitrary $\param \in \R^n$, we let $\param_Q$ be
the vector in $\R^{|Q|}$ given by the elements $\theta_{\ivec}, \ivec
\in Q$. For convenience, we shall index elements of $\param_Q$ by the
entries of $Q$ i.e., for every $q \in Q$, we have $(\param_Q)_q
:= \param_q$. We also define $D \param_Q := D(\param_Q)$ to be the
differencing operator applied to $\param_Q$ in a manner analogous to
\eqref{equation:diff_def}. Specifically, we take
\begin{equation}\label{equation:diff_def_Q}
  (D \param_{\SubRect})_{\ivec} = \sum_{\zvec \in \{0, 1\}^\usedim} \Ind\{\ivec - \zvec \succeq \qvec^\ell\}
(-1)^{z_1 + \cdots + z_\usedim} \paramplain_{\ivec - \zvec} \qt{for
  $\ivec \in Q$}
\end{equation}
Note that the elements of $D \param_{\SubRect}$ are also indexed by
the indices in $Q$. It is important to observe here that
$D \param_{\SubRect} = D(\param_{\SubRect})$ is different from
$(D \param)_{\SubRect}$. A formula for $D \param_{\SubRect}$   in
terms of $(D \param)_{\SubRect}$ is given in
\autoref{lemma:edgecoef}.

\noindent For the rectangle $Q$ in \eqref{Qrect} and every $\ivec = (i_1, \dots,
i_d) \in Q$, we let
\begin{equation}\label{equation:jivec}
  J(\ivec) := \left\{1 \leq j \leq d : i_j > q_j^{\ell} \right\}
  \qt{where $\qvec^{\ell} := (q_1^\ell, \dots, q_d^\ell)$}
\end{equation}
Also for $\ivec \in Q$ and $\ivec' \preceq \ivec$, let
\begin{equation}\label{equation:tii}
  t(\ivec', \ivec) := \Ind \left\{\ivec'_{J(\ivec)} = \ivec_{J(\ivec)}
  \right\}
\end{equation}
where we are using the notation $\kvec_J := (k_j : j \in J)$ for
$\kvec = (k_1, \dots, k_d) \in \IndexSet$ and $J \subseteq \{1,
\dots, d\}$.

\begin{lemma}\label{lemma:global_to_local}
Consider the same notation and setting as
\autoref{lemma:tcone_inequality} (in particular, the signs
$\mysign(\ivec)$ below come from \eqref{equation:sign_def}). Suppose
$\SubRect = [\qvec^\ell,
\qvec^u]$ is a rectangle satisfying the following.
\begin{enumerate}[(a)]
\item\label{enumerate:cond_no_istar}
If $\ivec \in \SubRect \setminus \{\qvec^\ell\}$ and $\ivec \succeq \ivec^*$, then $\edgecoef(\ivec^*, \ivec) = 0$ and $t(\zerovec, \ivec) = 0$.
\item\label{enumerate:cond_constant_sign}
Given $\ivec \in \SubRect \setminus \{\qvec^\ell\}$, the quantity $\mysign(\ivec')$ is constant
over all $\ivec'$ satisfying $\ivec' \preceq \ivec$, $\edgecoef(\ivec', \ivec) \ne 0$,
and $\mysign(\ivec') \ne 0$.
\item\label{enumerate:cond_contained}
$\SubRect$ is a subset of one of
$\LL_u \cap \LL_\ell$,
$\LL_u^c \cap \LL_\ell$,
$\LL_u \cap \LL_\ell^c$,
or $\LL_u^c \cap \LL_\ell^c$.
\end{enumerate}
Then for any $\alphavec \in \TConeSub(\ivec^u, \ivec^\ell, \delta)$,
\begin{equation}\label{equation:subrect_shape}
\sum_{\ivec \in \SubRect \setminus \{\qvec^\ell\}}
(\abs{(D \alphavec_{\SubRect})_{\ivec}} - \mysignalt(\ivec) (D \alphavec_{\SubRect})_{\ivec})
\le 2 \delta,
\end{equation}
where $\mysignalt(\ivec) \defn \mysign(\ivec)$ for $\ivec \succ \qvec^\ell$,
and otherwise for $\ivec \nsucc \qvec^\ell$ we have $\mysignalt(\ivec)
\defn \mysign(\ivec')$ for any $\ivec'$ satisfying $\ivec' \preceq
\ivec$, $\edgecoef(\ivec', \ivec) \ne 0$, and $\mysign(\ivec') \ne
0$.
\end{lemma}

As mentioned earlier, our idea will be to partition $\IndexSet$ into a
finite number of rectangles $Q$ each satisfying the conditions of
\autoref{lemma:global_to_local}. This will enable us to employ bounds
for the Gaussian width of nearly entirely monotone vectors and
vectors with small HK$\zerovec$ variation to bound
$\GWidth(\TConeLASSO(\Altdesignmat \coeftilde))$. The next result (proved in
\autoref{section:proof_nearly_cm}) bounds the  Gaussian width of
nearly entirely monotone vectors.
\begin{lemma}
\label{lemma:nearly_cm}
For every $\numobs \ge 1$, $\delta \ge 0$ and $t > 0$, we have
\begin{align}
&\E \sup_{\substack{
    \param : \norm{\param} \le t, \\
    \HKVar(\param) \le \paramplain_n - \paramplain_1 + \delta
}}
\inner{Z, \param}
\le C_\usedim  (t + \delta \sqrt{\numobs})
(\log (e \numobs))^{\frac{3 \usedim}{4}} (\log (e \log (e n)))^{\frac{2 \usedim - 1}{4}}
\end{align}
where $Z \sim \Normal(\zerovec, \Imat_\numobs)$.
\end{lemma}

For bounding the Gaussian width of vectors with small HK$\zerovec$
variation, we use the bound derived in \eqref{equation:lasso_gw} in
the proof of \autoref{theorem:lasso_worst_case}. This bound gives (here
$Z \sim \Normal(\zerovec, \Imat_{\numobs})$)
\begin{align}\label{smavar}
\E \sup_{\substack{
    \param : \norm{\param} \le 1, \\
    \HKVar(\param) \le 2V}} \inner{Z, \param}
&\le C_\usedim(1+\sqrt{2 \LASSOrad \sqrt{\numobs}})
\parens*{\log(1 + 2 e \LASSOrad \sqrt{\numobs})}^{\frac{2 \usedim - 1}{4}}
\\
&\qquad
+ C_\usedim \sqrt{\log(4 + 2 \LASSOrad \sqrt{\numobs})}
\end{align}
for every $V \geq 0$.

In addition to the above two Gaussian width bounds, we also need the
following result (proved in \autoref{section:proof_gw_mix}) for the
proof of \autoref{theorem:lasso_adaptive}.  This result is stated for
$d = 2$ as \autoref{theorem:lasso_adaptive} only applies to $d = 2$.
\begin{lemma}\label{lemma:gw_mix}
Let $d = 2$ and $Z \sim \Normal(\zerovec, \Imat_\numobs)$. For every
$\delta \geq 0$ and  $s_1, s_2 \in \{-1, 0, 1\}$, we have
\begin{align}
&\E \sup_{\substack{
    \param = \Altdesignmat \coef : \norm{\param} \le 1 \\
    \HKVar(\param) \le
    s_1(\paramplain_{\numobs_1, 1} - \paramplain_{1, 1})
    + s_2(\paramplain_{1, \numobs_2} - \paramplain_{1, 1})
    + \delta
    \\
    \coefplain_{\ivec} = 0, \, \forall \ivec \succ \zerovec
}}
\inner{Z, \param}
\\
&\qquad\qquad\qquad\le C \left\{
    (1 + \delta \sqrt{\numobs}) \sqrt{\log(e \numobs)}
    \Ind_{\{s_1 \ne 0\} \cup \{s_2 \ne 0\}}
    \right.
    \\
    &\qquad\qquad\qquad\qquad \left.
    + \brackets*{
        (\delta \sqrt{\numobs})^{\frac{1}{2}}
        + \sqrt{\log (e \numobs)}
    }
    \Ind_{\{s_1 = 0\} \cup \{s_2 = 0\}}
\right\}
+ \sqrt{2 / \pi}.
\end{align}
\end{lemma}

Before proceeding to the proofs of \autoref{theorem:lasso_adaptive_d}
and \autoref{theorem:lasso_adaptive}, let us add a brief remark below
on why our proof technique does not seem to work for
more general functions $f^*$  in $\rpc$.

\begin{remark}\label{remark:many_jumps}
The main technical reason why our adaptive results \autoref{theorem:lasso_adaptive_d}
and \autoref{theorem:lasso_adaptive} deal only with functions of the
form \eqref{equation:two_piece} and not more general functions in
$\rpc$ is that our proof technique seems to break down for these
general functions. In particular, for more complicated functions $f^*
\in \rpc$, it seems that it may not be possible
to obtain a partition of $\IndexSet$ into a constant (depending only
on $d$) number of rectangles $Q$ satisfying the conditions in
\autoref{lemma:global_to_local}.
\end{remark}

\subsubsection{Proof of \autoref{theorem:lasso_adaptive_d}}
We shall use \eqref{equation:bellec_oracle2}. Note that the right hand
side of \eqref{equation:bellec_oracle2} consists of infimum over all
$$\paramtilde \in K = \LASSOBall(\LASSOrad) = \left\{(f(\xvec_1),
  \dots, f(\xvec_n)) : \HKVar(f) \leq V \right\}.$$ It is clear then
that \eqref{equation:bellec_oracle2} will still be true if we restrict
the infimum to $\paramtilde$ belonging to any subset of $K$. We shall
consider the subset
\begin{equation*}
  \left\{(f(\xvec_1), \dots, f(\xvec_n)) : f \in
    \OneJumpClassStrong(c) \text{ and } \HKVar(f) = \LASSOrad
  \right\}.
\end{equation*}
We shall therefore fix a function $f \in \OneJumpClassStrong(c)$ with
$\HKVar(f) = \LASSOrad$ and bound the Gaussian width
\begin{equation}
\E \sup_{\alphavec \in \TConeLASSO(\paramtilde) \cap \Ball_2(\zerovec, 1)} \inner{Z, \alphavec}
\end{equation}
where $\paramtilde = \Altdesignmat \coeftilde = (f(\xvec_1), \dots,
f(\xvec_n))$. Due to the structure of $f$, there exists $\ivec^*$
such that $\coefplaintilde_{\ivec} = 0$ for all
$\ivec \notin \{\zerovec, \ivec^*\}$.
Explicitly, if $f = \Ind_{[\xvec^*, \onevec]}$,
then $\ivec^*$ is the index corresponding to the smallest design point
$\xvec = (i_1/\numobs_1, \ldots, i_\usedim / \numobs_\usedim)$
satisfying $\xvec \succeq \xvec^*$.

The minimum length assumption~\eqref{equation:min_length_strong} implies that
the sets $\{\ivec : \ivec \succeq \ivec^*\}$ and $\{\ivec : \ivec
\prec \ivec^*\}$ each have $\ge c \numobs$ elements.
Therefore if $\alphavec \in \TConeLASSO(\paramtilde) \cap \Ball_2(\zerovec, 1)$,
the pigeonhole principle and fact that $\norm{\alphavec} \le 1$
together imply that there exist $\ivec^u \succeq \ivec^*$ and $\ivec^\ell \prec \ivec^*$
such that $\abs{\alpha_{\ivec^u}} \le (cn)^{-1/2}$ and
$\abs{\alpha_{\ivec^\ell}} \le (cn)^{-1/2}$. This implies that
\begin{equation}
\TConeLASSO(\paramtilde)
\subseteq \bigcup_{\substack{
    \ivec^u, \ivec^\ell : \ivec^\ell \prec \ivec^* \preceq \ivec^u
}}
\TConeSub(\ivec^u, \ivec^\ell, 2 (c \numobs)^{-\frac{1}{2}}).
\end{equation}
where $\TConeSub(\ivec^u, \ivec^\ell, \delta)$ is defined in
\eqref{equation:tconesub_def}. By Lemma D.1 of
\cite{guntuboyina2017spatial} and noting that the above union is over
$\le \numobs^2$ indices, we obtain
\begin{align}
&\E \sup_{\alphavec \in \TConeLASSO(\paramtilde) \cap \Ball_2(\zerovec, 1)}
\inner{Z, \alphavec}
\\
&\qquad \le \max_{\substack{
    \ivec^u, \ivec^\ell : \ivec^\ell \prec \ivec^* \preceq \ivec^u
}}
\E \sup_{\alphavec \in \TConeSub(\ivec^u, \ivec^\ell, 2 (c \numobs)^{-\frac{1}{2}})}
\inner{Z, \alphavec}
+ \sqrt{4 \log \numobs}
+ \sqrt{\pi / 2}.
\end{align}
The following lemma bounds the expectations appearing on the
right-hand side above and is proved below.
\begin{lemma}\label{lemma:gw_tconesub_strong}
Let $\ivec^\ell$ and $\ivec^u$ satisfy $\ivec^\ell \prec \ivec^* \preceq \ivec^u$.
For $\delta \ge 0$,
\begin{equation}
\E \sup_{\alphavec \in \TConeSub(\ivec^u, \ivec^\ell, \delta)}
\inner{Z, \alphavec}
\le C_\usedim (1 + 2 \delta \sqrt{\numobs})
(\log (e \numobs))^{\frac{3 \usedim}{4}} (\log (e \log(e n)))^{\frac{2 \usedim - 1}{4}}.
\end{equation}
\end{lemma}
Plugging $\delta = 2 (c \numobs)^{-1/2}$
into \autoref{lemma:gw_tconesub_strong} yields
\begin{align}
\E \sup_{\alphavec \in \TConeLASSO(\paramtilde) \cap \Ball_2(\zerovec, 1)}
\inner{Z, \alphavec}
&\le C_\usedim
(\log (e \numobs))^{\frac{3 \usedim}{4}} (\log (e \log (e n)))^{\frac{2 \usedim - 1}{4}}
+ \sqrt{4 \log \numobs} + \sqrt{\pi / 2}
\\
&\le C_\usedim
(\log (e \numobs))^{\frac{3 \usedim}{4}} (\log (e \log (e n)))^{\frac{2 \usedim - 1}{4}}.
\end{align}
Plugging this bound into the oracle
inequality~\eqref{equation:bellec_oracle2} concludes the proof of
\autoref{theorem:lasso_adaptive_d}.

It therefore suffices to prove \autoref{lemma:gw_tconesub_strong}. For
every partition $\SubRect_1, \ldots, \SubRect_\numrect$ of $\IndexSet$
into rectangles of the form \eqref{Qrect}, we have
\begin{equation}\label{equation:sum_of_rectangles}
\E \sup_{\alphavec \in \TConeSub(\ivec^u, \ivec^\ell, \delta)}
\inner{Z, \alphavec}
\le \sum_{r = 1}^\numrect
\E \sup_{\alphavec \in \TConeSub(\ivec^u, \ivec^\ell, \delta)}
\inner{Z_{\SubRect_r}, \alphavec_{\SubRect_r}}.
\end{equation}
Our idea is to choose the partition such that each $\SubRect_r$
satisfies the conditions of \autoref{lemma:global_to_local} so that
then each $\alphavec_{\SubRect_r}$ for $\alphavec \in
\TConeSub(\ivec^u, \ivec^\ell, \delta)$
satisfies \eqref{equation:subrect_shape} which would allow us to bound
each expectation appearing in the right hand side above.

Here is how we construct the partition. For each $j \in \{1, \dots, d\}$ we
partition the interval $\{0, \ldots, \numobs_j - 1\}$ into at most $4$
intervals by splitting at $i^u_j + 0.5$, $i^\ell_j + 0.5$, and $i^*_j
- 0.5$. We then take the Cartesian product of these partitions over $j
= 1, \dots, d$ to obtain a partition $\SubRect_1, \ldots,
\SubRect_\numrect$ of $\IndexSet$ into at most $\numrect \le
4^\usedim$ rectangles.

We now check that the rectangles each satisfy the three conditions of
\autoref{lemma:global_to_local}. Let $\SubRect = [\qvec^\ell,
\qvec^u]$ be one of the rectangles of the above partition. Here the
auxiliary technical \autoref{lemma:edgecoef}  will be used. By the
second part of \autoref{lemma:edgecoef}, the quantity $t(\zerovec,
\ivec)$
is zero for all $\ivec \in \SubRect$ except
when $\ivec = \qvec^\ell$.
Now suppose $\ivec^* \preceq \qvec^u$.
Due to the splits at $i^*_j - 0.5$ for all $j$,
we have $\max\{q^\ell_j, i^*_j\} = q^\ell_j$ for all $j$,
so the second part of \autoref{lemma:edgecoef} implies $t(\ivec^*, \ivec) = 0$
for all $\ivec \in \SubRect$ except $\ivec = \qvec^\ell$.
Thus condition~\eqref{enumerate:cond_no_istar} is satisfied.

Recall that by assumption $\ivec^\ell \prec \ivec^* \preceq \ivec^u$,
so $\LL_u^c \cap \LL_\ell$ is empty.
Thus by definition~\eqref{equation:sign_def}, $\mysign(\ivec) \in \{0, 1\}$
for all $\ivec$. Thus condition~\eqref{enumerate:cond_constant_sign}
holds automatically. Finally, note that $\SubRect$ is contained in
either $\LL_u$ or $\LL_u^c$ due to the splits at $i^u_j + 0.5$ for all
$j \in [\usedim]$. Similarly $\SubRect$ is contained in either
$\LL_\ell$ or $\LL_u^c$. Thus
condition~\eqref{enumerate:cond_contained} holds.

We have thus proved that for each rectangle $\SubRect_r, r = 1, \dots,
R$, the inequality \eqref{equation:subrect_shape} holds. We now fix
such a rectangle $Q \in \{Q_1, \dots, Q_R\}$ and bound the expected
supremum term appearing on the right hand side of
\eqref{equation:sum_of_rectangles}. By
condition~\eqref{enumerate:cond_contained} of
\autoref{lemma:global_to_local}, there exists $s \in \{-1, 0, 1\}$
such that $\mysign(\ivec) = s$ for all $\ivec \in \SubRect$. We
separate the two cases where $s = 0$ and $s \neq 0$.

\paragraph{Case 1: $s = 0$.} Because $\LL_u^c \cap
\LL_\ell$ is empty, we must have $\mysignalt(\ivec) \in \{0, 1\}$ for
all $\ivec \in \SubRect \setminus \{\qvec^\ell\}$. For $\ivec \in
\SubRect$ such that $\ivec \succ \qvec^\ell$, we further have
$\mysignalt(\ivec) = s = 0$. Thus~\eqref{equation:subrect_shape} can
be rewritten as
\begin{equation}
\sum_{\ivec \in \SubRect \setminus \{\qvec^\ell\} : \ivec \succ \qvec^\ell}
\abs{(D \alphavec_{\SubRect})_{\ivec}}
+
\sum_{\ivec \in \SubRect \setminus \{\qvec^\ell\} : \ivec \nsucc \qvec^\ell}
(\abs{(D \alphavec_{\SubRect})_{\ivec}} - \mysignalt(\ivec) (D \alphavec_{\SubRect})_{\ivec})
\le 2 \delta.
\end{equation}
Using the fact that $- (D \alphavec_{\SubRect})_{\ivec} \le \abs{(D \alphavec_{\SubRect})_{\ivec}}$
and
\begin{equation}
(\abs{(D \alphavec_{\SubRect})_{\ivec}} - (D \alphavec_{\SubRect})_{\ivec})
\le 2 (\abs{(D \alphavec_{\SubRect})_{\ivec}} - \mysignalt(\ivec) (D \alphavec_{\SubRect})_{\ivec})
\end{equation}
for every $\ivec \in \SubRect \setminus \{\qvec^\ell\}$, we deduce
\begin{equation}
\sum_{\ivec \in \SubRect \setminus \{\qvec^\ell\}}
(\abs{(D \alphavec_{\SubRect})_{\ivec}} - (D \alphavec_{\SubRect})_{\ivec})
\le 4 \delta.
\end{equation}
Thus, \autoref{lemma:nearly_cm}
(with $4\delta$ in place of $\delta$, as well as $\rad=1$ and $\noisestd = 1$)
implies
\begin{align}
&\E \sup_{\alphavec \in \TConeSub(\ivec^u, \ivec^\ell, \delta)}
\inner{Z_{\SubRect}, \alphavec_{\SubRect}}
\\
&\le C_\usedim (1 + 4 \delta \sqrt{\numobs})
(\log (e \numobs))^{\frac{3 \usedim}{4}} (\log (e \log (e n)))^{\frac{2 \usedim - 1}{4}}
\label{equation:apply_nearly_cm}
\end{align}

\paragraph{Case 2. $s \neq 0$.} Then the fact that $\LL_u^c \cap
\LL_\ell$ is empty implies $s = 1$.
Thus $\mysignalt(\ivec) = 1$ for all $\ivec \in \SubRect \setminus \{\qvec^\ell\}$.
Therefore, the shape constraint~\eqref{equation:subrect_shape} can be rewritten as
\begin{equation}
\sum_{\ivec \in \SubRect \setminus \{\qvec^\ell\}}
(\abs{(D \alphavec_{\SubRect})_{\ivec}} - (D \alphavec_{\SubRect})_{\ivec})
\le 2 \delta.
\end{equation}
Thus the above bound~\eqref{equation:apply_nearly_cm} holds as well.

Returning to the earlier inequality~\eqref{equation:sum_of_rectangles}
and recalling the sum is over $\numrect \le 4^\usedim$ rectangles,
we obtain
\begin{equation}
\E \sup_{\alphavec \in \TConeSub(\ivec^u, \ivec^\ell, \delta)}
\inner{Z, \alphavec}
\le C_\usedim (1 + 2 \delta \sqrt{\numobs})
(\log (e \numobs))^{\frac{3 \usedim}{4}} (\log (e \log (e n)))^{\frac{2 \usedim - 1}{4}}.
\end{equation}
We have thus proved \autoref{lemma:gw_tconesub_strong} which completes
the proof of \autoref{theorem:lasso_adaptive_d}.

\subsubsection{Proof of \autoref{theorem:lasso_adaptive}}
In this proof we take $d = 2$. This proof is similar to but longer
than the proof of \autoref{theorem:lasso_adaptive_d}. We upper bound
the oracle inequality~\eqref{equation:bellec_oracle2} by taking the
infimum only over $\param$ of the form $\param = (f(\xvec_1), \ldots,
f(\xvec_\numobs))$ where $f \in \OneJumpClass(c)$ and $\HKVar(f) =
\LASSOrad$. It then suffices to control the Gaussian width
$\E \sup_{\alphavec \in \TConeLASSO(\paramtilde): \norm{\alphavec} \le
  1} \inner{Z, \alphavec}$ for such $\param = \Altdesignmat \coef$.

Let $\ivec^* \defn (i^*_1, i^*_2) \ne (0, 0)$ be the unique index
such that $\coefplaintilde_{\ivec^*} \ne 0$,
which is guaranteed by the form~\eqref{equation:two_piece}
of functions in $\OneJumpClass$.
Specifically, if $f \in \OneJumpClass$, it is of the form $a_1 \Ind_{[\xvec^*, \onevec]} + a_0$,
and $\ivec^*$ is the index corresponding to the smallest design point
$\xvec$ satisfying $\xvec \succeq \xvec^*$.

The minimum size assumption~\eqref{equation:min_length}
implies that the set $\{\ivec : \ivec \succeq \ivec^*\}$
and its complement have cardinality $\ge c \numobs$. By the pigeonhole
principle, for any $\alphavec$ satisfying $\norm{\alphavec} \le 1$,
there exists some $\ivec^u \succeq \ivec^*$
such that $\abs{\alpha_{\ivec^u}} \le (c \numobs)^{-1/2}$
and some $\ivec^\ell \nsucceq \ivec^*$
such that $\abs{\alpha_{\ivec^\ell}} \le (c \numobs)^{-1/2}$.
Then we have
\begin{equation}
\abs{\alpha_{\ivec^u} - \alpha_{\ivec^\ell}}
\le 2 (c \numobs)^{-\frac{1}{2}}.
\end{equation}
Thus,
\begin{equation}
\TConeLASSO(\paramtilde) \cap \Ball_2(\zerovec, 1)
\subseteq \bigcup_{\substack{\ivec^u ,\ivec^\ell : \ivec^u \succeq \ivec^*, \ivec^\ell \nsucceq \ivec^*}}
\TConeSub(\ivec^u, \ivec^\ell, 2 (c \numobs)^{-\frac{1}{2}}),
\end{equation}
where $\TConeSub(\ivec^u, \ivec^\ell, 2 (c \numobs)^{-\frac{1}{2}})$
is defined in~\eqref{equation:tconesub_def}.

Using Lemma D.1 of \cite{guntuboyina2017spatial} and noting the above union is over
$\le \numobs^2$ sets, we then have
\begin{align}
&\E \sup_{\alphavec \in \TConeLASSO : \norm{\alphavec} \le 1}
\inner{Z, \alphavec}
\\
&\qquad\le \max_{\substack{\ivec^u ,\ivec^\ell : \ivec^u \succeq \ivec^*,
    \ivec^\ell \nsucceq \ivec^*}}
\E \sup_{\alphavec \in \TConeSub(\ivec^u, \ivec^\ell, 2 (c \numobs)^{-\frac{1}{2}})}
\inner{Z, \alphavec}
+ \sqrt{4 \log \numobs} + \sqrt{\pi / 2}.
\end{align}
Therefore it remains to bound the expectation on the right-hand side
for each set $\TConeSub(\ivec^u, \ivec^\ell, 2 (c \numobs)^{-\frac{1}{2}})$.
This is the content of the following lemma.

\begin{lemma}\label{lemma:gw_tconesub}
For $\usedim = 2$, $\delta \ge 0$ and every $\ivec^u \succeq \ivec^*$
and $\ivec^\ell \nsucceq \ivec^*$,
\begin{align}
&\E \sup_{\alphavec \in \TConeSub(\ivec^u, \ivec^\ell, \delta)}
\inner{Z, \alphavec}
\\
&\le c
    (1 + (\delta \sqrt{\numobs})^{\frac{1}{2}}) (\log(\delta \sqrt{\numobs} + 1))^{\frac{3}{4}}
    \\
    &\qquad
    + (1 + \delta \sqrt{\numobs})
    \brackets*{
    (\log (e \numobs))^{\frac{3}{2}} (\log (e \log (e n)))^{\frac{3}{4}}
    + \sqrt{\log(4 + 2 \delta \sqrt{\numobs})}}.
\end{align}
\end{lemma}
The proof of this result is quite involved and given below. Note that
\autoref{lemma:gw_tconesub} only deals with $d = 2$ while
\autoref{lemma:gw_tconesub_strong} is true for arbitrary $d$. On the
other hand, for $d = 2$, \autoref{lemma:gw_tconesub} is stronger
than  \autoref{lemma:gw_tconesub_strong} because it applies to a more
general set of indices $\ivec^\ell$ (the condition $\ivec^\ell
\nsucceq \ivec^*$ is weaker than $\ivec^\ell \prec \ivec^*$).

Before proving \autoref{lemma:gw_tconesub}, let us quickly note that
plugging  in $\delta = 2(c \numobs)^{-\frac{1}{2}}$ in
\autoref{lemma:gw_tconesub} yields
\begin{equation}
\E \sup_{\alphavec \in \TConeSub(\ivec^u, \ivec^\ell, 2 (c \numobs)^{-\frac{1}{2}})}
\inner{Z, \alphavec}
\le C
(\log (e \numobs))^{\frac{3}{2}} (\log (e \log (e n)))^{\frac{3}{4}},
\end{equation}
which concludes the proof of \autoref{theorem:lasso_adaptive}.

Let $Q_1, \dots, Q_R$ be the partition constructed in the proof of
\autoref{theorem:lasso_adaptive_d}. We shall first prove that each
rectangle $Q = [\qvec^\ell, \qvec^u]$ in $\{Q_1, \dots, Q_R\}$
satisfies the three conditions of
\autoref{lemma:global_to_local}. Note that this was proved in the
proof of  \autoref{theorem:lasso_adaptive_d} under the stronger
condition $\ivec^\ell \prec \ivec^*$ but now we are working under the
weaker condition $\ivec^\ell \nsucceq
\ivec^*$. Conditions~\eqref{enumerate:cond_no_istar}
and~\eqref{enumerate:cond_contained} hold by exactly the same argument
as in proof of \autoref{theorem:lasso_adaptive_d}. To show
condition~\eqref{enumerate:cond_constant_sign}, we need to crucially
use $\usedim = 2$. If $\ivec \in \SubRect$ satisfies $\ivec \succ
\qvec^\ell$,
then $\edgecoef(\ivec', \ivec) = 0$ for all $\ivec' \preceq \ivec$ except $\ivec' = \ivec$,
so condition~\eqref{enumerate:cond_constant_sign} holds automatically.
We now consider $\ivec \in \SubRect \setminus \{\qvec^\ell\}$
such that $\ivec \nsucc \qvec^\ell$.
Suppose without loss of generality that
$\ivec = (q^\ell_1, i_2)$ for $i_2 > q^\ell_2$;
the other case $\ivec = (i_1, q^\ell_2)$ for $i_1 > q^\ell_1$ can be handled similarly.
Then $\edgecoef(\ivec', \ivec) = 1$ only when $\ivec'$ satisfies $i'_2 = i_2$ and $i'_1 \le q^\ell_1$.
Therefore, to verify condition~\eqref{enumerate:cond_constant_sign} for such $\ivec$,
it suffices to show the stronger claim that
$\mysign(\ivec')$ is constant over all $\ivec'$ in the set
\begin{equation}\label{equation:drop_set}
\{\ivec' : i'_1 \le q^\ell_1, \, i'_2\in [q^\ell_2+1, q^u_2]\}
\end{equation}
satisfying $\mysign(\ivec') \ne 0$. Suppose for sake of contradiction
that $\ivec'$ and $\ivec''$
belong to this set and satisfy $\mysign(\ivec') = 1$ and $\mysign(\ivec') = -1$.
Then $\ivec' \in \LL_u \cap \LL_\ell^c$
and $\ivec'' \in \LL_u^c \cap \LL_\ell$.
We then must have $i'_j \le i^u_j < i''_j$
for some $j \in \{1, 2\}$,
and $i''_j \le i^\ell_j < i'_j$
for some $j \in \{1, 2\}$.
From here, we deduce that either $i^u_2$ or $i^\ell_2$
lies in $[\min\{i'_2, i''_2\}, \max\{i'_2, i''_2\}) \subseteq [q^\ell_2, q^u_2)$. But due to the splits at $i^u_2 + 0.5$ and $i^\ell_2 + 0.5$ in the construction of the partition,
this is a contradiction.

A similar argument shows that
$\mysign(\ivec')$ is constant over all $\ivec'$ in the set
\begin{equation}\label{equation:drop_set2}
\{\ivec' : i'_2 \le q^\ell_2, \, i'_1\in [q^\ell_1+1, q^u_1]\}
\end{equation}
satisfying $\mysign(\ivec') \ne 0$.
Let this constant value be denoted by $s_1$, and let the constant
value for the earlier set~\eqref{equation:drop_set} be denoted by
$s_2$.  Thus condition~\eqref{enumerate:cond_constant_sign} holds as well, and
we have the inequality~\eqref{equation:subrect_shape} by
\autoref{lemma:global_to_local}.

We shall now bound the Gaussian width
\begin{equation*}
  \E \sup_{\alphavec \in \TConeSub(\ivec^u, \ivec^\ell, \delta)}
\inner{Z_{\SubRect}, \alphavec_{\SubRect}}
\end{equation*}
by splitting into the two cases $s \neq 0$ and $s = 0$ where $s$ is
the common value of $\mysign(\ivec)$ for $\ivec \in \SubRect$ (the
fact that $\mysign(\ivec)$ is the same for every $\ivec \in \SubRect$
is guaranteed by condition~\eqref{enumerate:cond_contained} of
\autoref{lemma:global_to_local}).

\paragraph{Case 1: $s \ne 0$} By definition $\mysign(\ivec) = s$ for
all $\ivec \in \SubRect \setminus \{\qvec^\ell\}$,
so~\eqref{equation:subrect_shape} can be written as
\begin{equation}
\sum_{\ivec \in \SubRect \setminus \{\qvec^\ell\}}
(\abs{(D (s \alphavec_{\SubRect}))_{\ivec}} - (D (s \alphavec_{\SubRect}))_{\ivec})
=
\sum_{\ivec \in \SubRect \setminus \{\qvec^\ell\}}
(\abs{(D \alphavec_{\SubRect})_{\ivec}} - s (D \alphavec_{\SubRect})_{\ivec})
\le 2 \delta.
\end{equation}
Since the sets $\TConeSub(\ivec^u, \ivec^\ell, \delta)$
and $-\TConeSub(\ivec^u, \ivec^\ell, \delta)$
have the same Gaussian width, we may apply \autoref{lemma:nearly_cm}
to obtain
\begin{align}
&\E \sup_{\alphavec \in \TConeSub(\ivec^u, \ivec^\ell, \delta)}
\inner{Z_{\SubRect}, \alphavec_{\SubRect}}
\\
&\le c (1 + 2 \delta \sqrt{\numobs})
(\log (e \numobs))^{\frac{3}{2}} (\log (e \log (e n)))^{\frac{3}{4}}
\label{equation:apply_nearly_cm2}
\end{align}

\paragraph{Case 2: $s = 0$} In this case $\mysignalt(\ivec) = 0$ for
all $\ivec \succ \qvec^\ell$,
and is otherwise equal to $s_1$ (if $i_2 = q^\ell_2$)
or $s_2$ (if $i_1 = q^\ell_1$) because we showed that
$\mysign(\ivec')$ is constant over the sets~\eqref{equation:drop_set}
and~\eqref{equation:drop_set2}. So,
inequality~\eqref{equation:subrect_shape} can be rewritten as
\begin{equation}\label{equation:key_nearly3}
\begin{split}
\sum_{\ivec \in \SubRect : \ivec \succ \qvec^\ell}
\abs{(\Diff (\alphavec_{\SubRect}))_{\ivec}}
+ \sum_{\substack{\ivec = (i_1, q^\ell_2) :\\ i_1 \in [q^\ell_1 + 1, q^u_1]}}
(\abs{(\Diff (\alphavec_{\SubRect}))_{\ivec}} - s_1 (\Diff (\alphavec_{\SubRect}))_{\ivec})
\\ + \sum_{\substack{\ivec = (q^\ell_1, i_2) : \\ i_2 \in [q^\ell_2 + 1, q^u_2]}}
(\abs{(\Diff (\alphavec_{\SubRect}))_{\ivec}} - s_2 (\Diff (\alphavec_{\SubRect}))_{\ivec})
\le 2 \delta.
\end{split}
\end{equation}
Let us define
$\alphavec_{\SubRect}^{(0)} \defn
\sum_{\ivec \in \SubRect : \ivec \succ \qvec^\ell}
(\Diff (\alphavec_{\SubRect}))_{\ivec}$
and
$\alphavec_{\SubRect}^{(1)} \defn
\sum_{\ivec \in \SubRect : \ivec \nsucc \qvec^\ell}
(\Diff (\alphavec_{\SubRect}))_{\ivec}$.
Since $\alphavec_{\SubRect} = \alphavec_{\SubRect}^{(0)} + \alphavec_{\SubRect}^{(1)}$,
we obtain
\begin{equation}\label{dectwo}
\E \sup_{\alphavec \in \TConeSub(\ivec^u, \ivec^\ell, \delta)}
\inner{Z_{\SubRect}, \alphavec_{\SubRect}}
\le
\E \sup_{\alphavec \in \TConeSub(\ivec^u, \ivec^\ell, \delta)}
\inner{Z_{\SubRect}, \alphavec_{\SubRect}^{(0)}}
+
\E \sup_{\alphavec \in \TConeSub(\ivec^u, \ivec^\ell, \delta)}
\inner{Z_{\SubRect}, \alphavec_{\SubRect}^{(1)}}.
\end{equation}
We now bound the first term in the right hand side above. Because
$\mysignalt(\ivec) = 0$ for $\ivec \succ \qvec^\ell$,
inequality~\eqref{equation:key_nearly3} implies
\begin{equation}
\HKVar(\alphavec_{\SubRect}^{(0)})
= \sum_{\ivec \in \SubRect : \ivec \succ \qvec^\ell}
\abs{(\Diff (\alphavec_{\SubRect}))_{\ivec}}
\le 2 \delta,
\end{equation}
so applying~\eqref{smavar} yields
\begin{align}
\E \sup_{\alphavec \in \TConeSub(\ivec^u, \ivec^\ell, \delta)}
\inner{Z_{\SubRect}, \alphavec_{\SubRect}^{(0)}}
&\le
\E \sup_{\param \in \R^{|Q|}: \|\param\| \leq 1, \HKVar(\param) \leq 2 \delta} \inner{Z_{\SubRect}, \param}
\\
&\le
C_\usedim (1 + \sqrt{2 \delta \sqrt{\numobs}})
\parens*{\log(1 + 2 e \delta \sqrt{\numobs})}^{\frac{3}{4}}
\\
&\qquad
+ C_\usedim \sqrt{\log(4 + 2 \delta \sqrt{\numobs})}.
\label{equation:alpha0}
\end{align}
We turn to the second term in
\eqref{dectwo}. Inequality~\eqref{equation:key_nearly3} implies
\begin{align}
\HKVar(\alphavec_{\SubRect}^{(1)})
&= \sum_{\ivec \in \SubRect \setminus \{\qvec^\ell\} : \ivec \nsucc \qvec^\ell}
\abs{(\Diff\alphavec_{\SubRect})_{\ivec}}
\\
&\le s_1
\sum_{\substack{\ivec = (i_1, q^\ell_2) :\\ i_1 \in [q^\ell_1 + 1, q^u_1]}}
(\Diff\alphavec_{\SubRect})_{\ivec}
+ s_2
\sum_{\substack{\ivec = (q^\ell_1, i_2) :\\ i_2 \in [q^\ell_2 + 1, q^u_2]}}
(\Diff\alphavec_{\SubRect})_{\ivec}
+ 2 \delta
\\
&= s_1 [
    (\alphavec_{\SubRect}^{(1)})_{q^u_1, q^\ell_2}
    - (\alphavec_{\SubRect}^{(1)})_{\qvec^\ell}
]
+ s_2 [
    (\alphavec_{\SubRect}^{(1)})_{q^\ell_1, q^u_2}
    - (\alphavec_{\SubRect}^{(1)})_{\qvec^\ell}
]
+ 2 \delta.
\end{align}
\autoref{lemma:gw_mix} then implies
\begin{align}
&\E \sup_{\alphavec \in \TConeSub(\ivec^u, \ivec^\ell, \delta)}
\inner{Z_{\LL}, \alphavec_{\LL}^{(1)}}
\\
&\le c \left\{
    (1 + \delta \sqrt{\numobs}) \sqrt{\log(e \numobs)}
    \Ind_{\{s_1 \ne 0\} \cup \{s_2 \ne 0\}}
    \right.
    \\
    &\qquad \left. + \brackets*{
        (\delta \sqrt{\numobs})^{\frac{1}{2}}
        + \sqrt{\log (e \numobs)}
    }
    \Ind_{\{s_1 = 0\} \cup \{s_2 = 0\}}
\right\}
+ \sqrt{2 / \pi}.
\label{equation:alpha1}
\end{align}
Summing the bounds~\eqref{equation:alpha0} and~\eqref{equation:alpha1}
yields
\begin{align}\label{equation:gw_case2}
&\E \sup_{\alphavec \in \TConeSub(\ivec^u, \ivec^\ell, \delta)}
\inner{Z_{\LL}, \alphavec_{\LL}}
\\
&\le c
    (1 + (\delta \sqrt{\numobs})^{\frac{1}{2}}) (\log(\delta \sqrt{\numobs} + 1))^{\frac{3}{4}}
    + c (1 + \delta \sqrt{\numobs})
    \brackets*{\sqrt{\log(e \numobs)} + \sqrt{\log(4 + 2 \delta \sqrt{\numobs})}}.
\end{align}
Having handled the two cases $s = 0$ and $s \ne 0$,
we take the maximum of~\eqref{equation:apply_nearly_cm2} and~\eqref{equation:gw_case2}
to obtain
\begin{align}
&\E \sup_{\alphavec \in \TConeSub(\ivec^u, \ivec^\ell, \delta)}
\inner{Z_{\LL}, \alphavec_{\LL}}
\\
&\le c
    (1 + (\delta \sqrt{\numobs})^{\frac{1}{2}}) (\log(\delta \sqrt{\numobs} + 1))^{\frac{3}{4}}
    \\
    &\qquad
    + (1 + \delta \sqrt{\numobs})
    \brackets*{
    (\log (e \numobs))^{\frac{3}{2}} (\log (e \log (e n)))^{\frac{3}{4}}
    + \sqrt{\log(4 + 2 \delta \sqrt{\numobs})}}.
\end{align}
Finally, in view of the inequality~\eqref{equation:sum_of_rectangles},
multiplying this bound by $4^2 = 16$
(the maximum number of rectangles in the partition constructed at the beginning of this proof)
produces the final bound given by \autoref{lemma:gw_tconesub} thereby
completing the proof of \autoref{theorem:lasso_adaptive}.

\section{Proofs of results from Section~\ref{SECTION:CM_HK} and Section~\ref{SECTION:COMPUTATION}}\label{section:proofs23}
This section contains the proofs of all the results from \autoref{SECTION:CM_HK} and
  \autoref{SECTION:COMPUTATION}. Specifically, we prove
  \autoref{lemma:cm_m}, \autoref{vvrpc}, part
  (\ref{lemma:hkvar_of_cm}) of \autoref{lemma:hkvar_properties},
  \autoref{proposition:cm_discrete}, \autoref{proposition:cm_nnls},
  \autoref{proposition:hk_discrete}, \autoref{proposition:hk_lasso}
  and \autoref{lemma:VC_app}. In addition, we also state and prove a
  result in \autoref{section:proof_spanlemma} which asserts that the
  columns of the design matrix $\Altdesignmat$ span $\R^n$ provided
  the design points $\xvec_1, \dots, \xvec_n$ are distinct.

\subsection{Proof of \autoref{lemma:cm_m}}
\label{section:proof_cm_m}

When $\usedim = 1$, the only rectangles are intervals $[a, b]$, so
the definition of entire monotonicity~\eqref{equation:cm_def}
reduces to $0 \le \QVolume(f, [a, b]) = f(b) - f(a)$ for all $0 \le a \le b \le 1$,
which is precisely the definition of $\MClass$~\eqref{equation:m_def}.

More generally for $\usedim \ge 1$, suppose $\avec, \bvec \in [0, 1]^\usedim$
agree in all but one component, that is, $\abs{\{i : a_i \ne b_i \}} = 1$.
Then entire monotonicity implies
$0 \le \QVolume(f, [\avec, \bvec]) = f(\bvec) - f(\avec)$.
To see how this inequality implies monotonicity~\eqref{equation:m_def}
note that for $\avec \preceq \bvec$ we can apply the above inequality repeatedly
to obtain
\begin{equation}
f(\avec)
\le f(b_1, a_2, \ldots, a_\usedim)
\le f(b_1, b_2, a_3, \ldots, a_\usedim)
\le \cdots \le f(\bvec).
\end{equation}
Thus $\EMClass \subseteq \MClass$ for $\usedim \ge 1$.

Finally, for $\usedim \ge 2$ consider the function $f:[0, 1]^\usedim \to \R$
defined by
\begin{equation}
f(\uvec)
\defn \begin{cases}
0 & \max\{u_1, u_2\} < 1/2
\\
3 & \min\{u_1, u_2\} \ge 1/2
\\
2 & \text{otherwise}
\end{cases}
\end{equation}
Note that $f$ is constant in all components except the first two.
One can directly check that $f \in \MClass$. However,
for $\avec = (\frac{1}{4}, \frac{1}{4}, 0, \ldots, 0)$
and $\bvec = (\frac{3}{4}, \frac{3}{4}, 0, \ldots, 0)$,
we have
\begin{equation}
\QVolume(f; [\avec, \bvec])
= 3 - 2 - 2 + 0 = -1 < 0,
\end{equation}
so $f \notin \EMClass$.




\subsection{Proof of \autoref{vvrpc}}\label{section:proof_vvrpc}
Let $\mathcal{P}^*$ be given by the $d$ univariate partitions
\eqref{dpar} and let $\mathcal{P}$ be the split of $[0, 1]^d$ formed
from these univariate partitions (as described after
\eqref{dpar}). Because $\mathcal{P}$ forms a split of $[0, 1]^d$, it
follows from \citet[Lemma 1]{owen2005multidimensional} that
\begin{equation}
  \VVar{\usedim}(f; [0, 1]^d) = \sum_{A \in \mathcal{P}} \VVar{\usedim}(f; A)
\end{equation}
where $\VVar{\usedim}(f; A)$ is the Vitali variation of $f$ on the rectangle
$A$ (which is defined analogously to $\VVar{\usedim}(f; [0, 1]^d)$). Let us
now fix a rectangle $A = [\avec, \bvec] \in \mathcal{P}$ where $\avec
= (a_1, \dots, a_d)$ and $\bvec = (b_1, \dots, b_d)$. Because $f$ is
rectangular piecewise constant with respect to $\mathcal{P}^*$, it
follows that $f$ is constant on each of the sets $B_1 \times \dots
\times B_d$ where each $B_i$ is either $\{b_i\}$ or $[a_i,
b_i)$. Using this, it is easy to observe that
\begin{equation}
  \VVar{\usedim}(f; A) = \abs{\Delta(f; A)}
\end{equation}
which completes the proof of \autoref{vvrpc}.

\subsection{Proof of part (\ref{lemma:hkvar_of_cm}) of \autoref{lemma:hkvar_properties}}
\label{section:proof_hkvar_of_cm}

If $f \in \EMClass$ is entirely monotone, then one can check that
for each $S$,
\begin{equation}
\VVar{{\abs{S}}}(f; S; [0, 1]^\usedim) = \QVolume(f; U_S),
\end{equation}
where $U_S$ is the face adjacent to $\zerovec$
defined earlier~\eqref{equation:face_adjacent}.
Thus the HK variation of $f$ is the sum of quasi-volumes of all faces
adjacent to $\zerovec$.
From the definition of quasi-volume~\eqref{equation:quasivolume},
this sum involves only the value of $f$ at vertices of $[0, 1]^\usedim$
(possibly multiplied by $-1$),
and one can check that all terms cancel except for $f(\onevec) - f(\zerovec)$.

\subsection{Statement and proof of a fact about the design matrix
  \texorpdfstring{$\Altdesignmat$}{A}} \label{section:proof_spanlemma}
Recall the definition of $\Altdesignmat$
as the matrix whose columns are the elements of the finite set
$\IndSet \defn \{\UReval{\zvec} : \zvec \in [0, 1]^\usedim\}$.

\begin{lemma}\label{lemma:span}
Suppose $\xvec_1, \ldots, \xvec_\numobs$ are unique.
Then the columns of $\Altdesignmat$ span $\R^\numobs$.
\end{lemma}

\begin{proof}
It suffices to show the standard basis vector $\evec_i$
lies in the column space of $\Altdesignmat$, for each $i = 1, \ldots, \numobs$.

Fix $i$. If $\xvec_i = \onevec$, then $\evec_i = \UReval{\onevec} \in \IndSet$,
which concludes the proof.

Otherwise we assume $\xvec_i \ne \onevec$.
Let $\uvec^\delta$ be defined by $u^\delta_j \defn \min\{1, (\xvec_i)_j + \delta\}$
for $j = 1, \ldots, \usedim$.
There exists $\delta > 0$ such that the hyperrectangle $[\xvec_i, \uvec^\delta]$
contains no design point except $\xvec_i$.
Let $S \defn \{j : u^\delta_j \ne (\xvec_i)_j\}$,
and note that the rectangle $[\xvec_i, \uvec^\delta]$
is $\abs{S}$-dimensional.

For a subset $S' \subseteq [\usedim]$ let $\evec_{S'}$ denote the
indicator vector of $S'$; that is, $(\evec_{S'})_j$ is $1$ if $j \in S'$
and is zero otherwise.
We claim
\begin{equation}
\evec_i = \sum_{S' \subseteq S} (-1)^{\abs{S'}} \UReval{\xvec_i + \delta \evec_{S'}}.
\end{equation}
To verify this, note that an inclusion-exclusion argument shows that the right-hand side is
$(\Ind_{[\xvec_i, \uvec^\delta)}(\xvec_1), \ldots, \Ind_{[\xvec_i, \uvec^\delta)}(\xvec_\numobs))$,
and this is $\evec_i$
due to the fact that $[\xvec_i, \uvec^\delta)$ contains no design point except $\xvec_i$.

\end{proof}

\subsection{Proof of
  \autoref{proposition:cm_discrete}} \label{section:proof_cm_discrete}
If \autoref{proposition:cm_discrete} holds for a given design
$\xvec_1, \ldots, \xvec_\numobs$, then adding an additional
design point $\xvec_{\numobs+1} \defn \xvec_i$ that is a copy of one of the original
design points simply gives $\Altdesignmat$ a new row that is a copy of its $i$th row,
and one can observe that the equality in the proposition still holds even after adding this extra design point.
Thus without loss of generality we may assume the design points are distinct.

Suppose we replace the original design $\{\xvec_1, \ldots, \xvec_\numobs\}$
with $\mathcal{U} \defn \prod_{j=1}^\usedim \mathcal{U}_j$
where $\mathcal{U}_j = \{0, (\xvec_1)_j, \ldots, (\xvec_\numobs)_j\}$
for each $j = 1, \ldots, \usedim$.
This is a lattice that contains the original design.
Using this new design, we define a square matrix $\Altdesignmat'$
whose $k$th column is $(\Ind_{[\uvec_k, \onevec]}(\uvec_1), \ldots, \Ind_{[\uvec_k, \onevec]}(\uvec_m))$.
Let $\uvec_1 = \zerovec$ so that the first column of $\Altdesignmat'$ is $\onevec$.
If we let $K \defn (k_1, \ldots, k_\numobs)$ be such that $\uvec_{k_i} = \xvec_i$
so that it indexes the elements of the new design that are also in the old design,
then we claim
$\{(\Altdesignmat' \coef')_K : \coefplain'_k \ge 0, \forall k \ge 2\}
= \{\Altdesignmat \coef : \coefplain_j \ge 0, \forall j \ge 2\}$.
Indeed, this holds simply because each column of $(\Altdesignmat')_K$
is also a column in $\Altdesignmat$,
so both sets are linear combinations of the same columns with the same nonnegativity constraints.

Thus it remains to show
\begin{equation}
\{(\Altdesignmat' \coef')_K : \coefplain'_k \ge 0, \forall k \ge 2\}
= \{(f(\xvec_1), \ldots, f(\xvec_\numobs) : f \in \EMClass\}.
\end{equation}


We first show the forward inclusion $\subseteq$.
Suppose $\coef'$ satisfies $\coefplain'_k \ge 0$ for all $k \ge 2$.
If $f \defn \sum_{k=1}^m \coefplain'_k \cdot \Ind_{[\uvec_k, \onevec]}$,
then $(f(\xvec_1), \ldots, f(\xvec_\numobs)) = (\Altdesignmat' \coef')_K$.
We now show $f \in \EMClass$.
For each pair of distinct points
$\avec \preceq \bvec$ in $[0, 1]^\usedim$,
we want to show $\QVolume(f; [\avec, \bvec]) \ge 0$.
Then there exist a pair $\uvec_k \preceq \uvec_{k'}$ in $\mathcal{U}$
such that $f(\avec) = f(\uvec_k)$, $f(\bvec) = f(\uvec_{k'})$, and
$\{j : \avec_j \ne \bvec_j\} = \{j : (\uvec_k)_j \ne (\uvec_{k'})_j\}$,
so that $\QVolume(f; [\avec, \bvec]) = \QVolume(f; [\uvec_k, \uvec_{k'}])$.

Recall that $\QVolume(f; [\uvec_k, \uvec_{k'}])$ by definition
is the sum of terms of the form $f(\uvec_\ell)$ for some $\uvec_\ell \in \mathcal{U}$
(possibly with sign changes),
since $\mathcal{U}$ is a lattice.
Note that $f(\uvec_\ell) = \sum_{i : \uvec_i \preceq \uvec_\ell} \coefplain_i$
for each $\ell$.
Putting the pieces together with an inclusion-exclusion argument yields
\begin{equation}
\QVolume(f; [\uvec_k, \uvec_{k'}])
= \sum_{i : \uvec_k \prec \uvec_i \preceq \uvec_{k'}}
\coefplain'_i
\ge 0.
\end{equation}

We now show the reverse inclusion $\supseteq$.
The matrix $\Altdesignmat'$ is square and has spanning columns (\autoref{lemma:span}),
so it is invertible. Thus there exists $\coef'$ such that $\Altdesignmat' \coef'
= (f(\uvec_1),\ldots, f(\uvec_m))$. Sub-indexing by $K$ yields
$(\Altdesignmat' \coef')_K = (f(\xvec_1),\ldots, f(\xvec_\numobs))$.

\subsection{Proof of \autoref{proposition:cm_nnls}}
\label{section:proof_cm_nnls}

The optimization problem~\eqref{equation:cmfitfun}
only involves the values of the function at $\xvec_1, \ldots, \xvec_\numobs$.
Thus by \autoref{proposition:cm_discrete},
the solution $\EMfitfun$ to the optimization problem~\eqref{equation:cmfitfun} must satisfy
$(\EMfitfun(\xvec_1), \ldots, \EMfitfun(\xvec_\numobs)) = \Altdesignmat \coefhatcm$.
It remains to show that the function $\EMfitfun$
defined in the result~\eqref{equation:cmfitfun_nnls}
satisfies this equality
and also lies in $\EMClass$.

The equality holds by definition, since $\EMfitfun$ satisfies
\begin{equation}
\EMfitfun(\xvec_i)
= \sum_{j=1}^\altdim (\coefhatcmplain)_j \cdot
\Ind_{[\zvec_j, \onevec]}(\xvec_i)
= (\Altdesignmat \coefhatcm)_i,
\qquad i = 1, \ldots, \numobs.
\end{equation}

To check $\EMfitfun$ as defined in the result~\eqref{equation:cmfitfun_nnls} lies
in $\EMClass$, we need to show
$\QVolume(\EMfitfun; [\avec, \bvec]) \ge 0$
for any rectangle $[\avec, \bvec] \subseteq [0, 1]^\usedim$,
$\avec \ne \bvec$.
Similar to the proof in \autoref{section:proof_cm_discrete},
we consider the augmented design
$\mathcal{U} \defn \prod_{j=1}^\usedim \mathcal{U}_j$
where $\mathcal{U}_j = \{0, (\xvec_1)_j, \ldots, (\xvec_\numobs)_j\}$
for each $j = 1, \ldots, \usedim$.
This is a lattice that contains the original design.
Moreover, for each $\zvec_j$ there exists some $\uvec \in \mathcal{U}$
such that $\Ind_{[\zvec_j, \onevec]}(\xvec_i) = \Ind_{[\uvec, \onevec]}(\xvec_i)$
holds for all $\xvec_i$.
Thus the function defined in the result~\eqref{equation:cmfitfun_nnls}
can be written as
$\EMfitfun = \sum_{\uvec \in \mathcal{U}}
\coefplaintilde_{\uvec} \Ind_{[\uvec, \onevec]}$
for some coefficients $\{\coefplaintilde_{\uvec} : \uvec \in \mathcal{U}\}$
that are either zero or equal to $(\coefhatcmplain)_j$ for some $j$.
Then, as in \autoref{section:proof_cm_discrete},
there exist a pair $\uvec_k \preceq \uvec_{k'}$ in $\mathcal{U}$
such that $f(\avec) = f(\uvec_k)$, $f(\bvec) = f(\uvec_{k'})$, and
$\{j : \avec_j \ne \bvec_j\} = \{j : (\uvec_k)_j \ne (\uvec_{k'})_j\}$,
so that $\QVolume(\EMfitfun; [\avec, \bvec]) = \QVolume(\EMfitfun; [\uvec_k, \uvec_{k'}])$,
and by the same reasoning as in the earlier section,
$\QVolume(\EMfitfun; [\uvec_k, \uvec_{k'}]) = \sum_{\uvec \in \mathcal{U}: \uvec_k \prec \uvec \preceq \uvec_{k'}} \coefplaintilde_{\uvec} \ge 0$.

\subsection{Proof of \autoref{proposition:hk_discrete}}
\label{section:proof_hk_discrete}

If \autoref{proposition:hk_discrete} holds for a given design
$\xvec_1, \ldots, \xvec_\numobs$, then adding an additional
design point $\xvec_{\numobs+1} \defn \xvec_i$ that is a copy of one of the original
design points simply gives $\Altdesignmat$ a new row that is a copy of its $i$th row,
and one can observe that the equality in the proposition still holds even after adding this extra design point.
Thus without loss of generality we may assume the design points are distinct.

We claim that the feasible set $\LASSOBall(\LASSOrad)$~\eqref{equation:lasso_ball}
does not change if we append additional columns to $\Altdesignmat$
(and append corresponding components to $\coef$)
that are copies of columns already in $\Altdesignmat$.
Concretely, if $\Altdesignmat'$ is the augmented matrix
(without loss of generality assume the new columns are appended on the right)
and $\LASSOBall'(\LASSOrad) \defn \{\Altdesignmat' \coef' : \sum_{j \ge 2} \abs{\coefplain'_j} \le \LASSOrad\}$
is the analogue of $\LASSOBall(\LASSOrad)$,
then the inclusion $\LASSOBall(\LASSOrad) \subseteq \LASSOBall'(\LASSOrad)$
holds immediately by noting $\Altdesignmat \coef = \Altdesignmat' \coef'$
and $\sum_{j \ge 2} \abs{\coefplain_j} = \sum_{j \ge 2} \abs{\coefplain'_j}$
where $\coef'$ is the result of taking $\coef$ and having coefficients $0$
for the added components.
For the reverse inclusion, suppose we are given $\Altdesignmat' \coef'$
such that $\sum_{j \ge 2} \abs{\coefplain'_j} \le \LASSOrad$.
Then $\Altdesignmat' \coef' = \Altdesignmat \coef$
where $\coefplain_j \defn \sum_{k : \Altdesignmat'_{\cdot, k} = \Altdesignmat_{\cdot, j}} \coefplain'_j$
so the triangle inequality implies
\begin{equation}
\sum_{j \ge 2} \abs{\coefplain_j} \le \sum_{j \ge 2}
\abs*{\sum_{k : \Altdesignmat'_{\cdot, k} = \Altdesignmat_{\cdot, j}} \coefplain'_j}
\le \sum_{j \ge 2} \abs{\coefplain'_j} \le \LASSOrad.
\end{equation}
Above, $\Altdesignmat_{\cdot, j}$ denotes the $j$th column of $\Altdesignmat$,
and $\Altdesignmat'_{\cdot, k}$ denotes the $k$th column of $\Altdesignmat'$.

Thus, similar to \autoref{section:proof_cm_discrete},
we may assume without loss of generality that the columns of $\Altdesignmat$
are $\UReval{\uvec_1}, \ldots, \UReval{\uvec_m}$
where $\uvec_1, \ldots, \uvec_m$
are the elements of the lattice
$\prod_{j=1}^\usedim \mathcal{U}_j$
and $\mathcal{U}_j \defn \{0, (\xvec_1)_j, \ldots, (\xvec_\numobs)_j, 1\}$
for $j = 1, \ldots, \usedim$.
Note the inclusion of $0$ and $1$ in each $\mathcal{U}_j$,
so that the lattice spans the entire
hypercube $[0, 1]^\usedim$.
Without loss of generality we assume the $\uvec_j$ are ordered
such that $\uvec_{j'} \preceq \uvec_j$ implies $j' \le j$.
Note that as a result, $\uvec_1 = \zerovec$.

Fix $\coef$ and let $\Altdesignmat \coef$.
Let $f \defn \sum_{j=1}^m \coefplain_j \Ind_{[\uvec_j, \onevec]}$.
By construction we have $f(\xvec_i) = (\Altdesignmat \coef)_i$
for all $i = 1, \ldots, \numobs$.
It remains to compute the HK$\zerovec$ variation of $f$.
One can check that a maximizing partition in the definition
of the Vitali variation~\eqref{equation:vitali}
is the partition induced by the lattice $\prod_{j=1}^\usedim \mathcal{U}_j$
(that is, the unique partition $\Partition^*$
whose rectangles each intersect the lattice
only at its vertices).
That is,
\begin{equation}\label{equation:maximal_partition}
\VVar{\usedim}(f; [\zerovec, \xvec_\numobs])
= \sum_{\rect \in \Partition^*} \abs*{\QVolume(f; \rect)}.
\end{equation}
Similarly, the maximizing partitions for the
Vitali variations over each face $U_S$ adjacent to $\zerovec$
\eqref{equation:face_adjacent}
can also be shown to be induced by the corresponding face of the lattice.
By construction, the quasi-volume for the rectangle whose largest vertex is
$\uvec_j$ will turn out to be $\coefplain_j$, so by the definition
of HK$\zerovec$ variation~\eqref{equation:hkvariation},
$\HKVar(f; [0,1]^\usedim) = \sum_{j \ge 2} \abs{\coefplain_j} \le \LASSOrad$.

Conversely, suppose we are given $f : [0, 1]^\usedim \to \R$
with $\HKVar(f; [0,1]^\usedim) \le \LASSOrad$.
Suppose first that the original design $\xvec_1, \ldots, \xvec_\numobs$
is already a lattice spanning $[0, 1]^\usedim$,
i.e. $\{\xvec_1, \ldots, \xvec_\numobs\} = \prod_{j=1}^\usedim \mathcal{U}_j$
and $\numobs = m$.
We remove this assumption at the end of the proof.

Because $\Altdesignmat$
has full column rank (\autoref{lemma:span}),
there exists some $\coef$ such that
$(f(\xvec_1), \ldots, f(\xvec_\numobs)) = \Altdesignmat \coef$.
By the above argument, the function
$\tilde{f} \defn \sum_{j=1}^\numobs \coefplaintilde_j \Ind_{[\uvec_j, 1]}$
agrees with $f$ at all the $\xvec_i$ (i.e. all the lattice points $\uvec_j$)
and satisfies
$\HKVar(\tilde{f}; [0,1]^\usedim) = \sum_{j \ge 2} \abs{\coefplain_j}$.
It then suffices to show
$\HKVar(f; [0,1]^\usedim) \le \HKVar(\tilde{f}; [0,1]^\usedim)$.

Let $\Partition^*$ be the partition of $[0, 1]^\usedim$
induced by the lattice $\prod_{j=1}^\usedim \mathcal{U}_j$.
As noted already~\eqref{equation:maximal_partition}, this partition is maximal for the definition
of the Vitali variation of $\tilde{f}$ on $[0, 1]^\usedim$.
Therefore, since $f$ and $\tilde{f}$ agree on all the lattice points $\uvec_j$,
their quasi-volumes on all the rectangles of $\Partition^*$ are the same,
so we have
\begin{equation}
\VVar{\usedim}(\tilde{f}; [0, 1]^\usedim)
= \sum_{\rect \in \Partition^*} \abs*{\QVolume(\tilde{f}; \rect)}
= \sum_{\rect \in \Partition^*} \abs*{\QVolume(f; \rect)}
\le \VVar{\usedim}(f; [0, 1]^\usedim)
\le \LASSOrad.
\end{equation}
A similar argument on the lower-dimensional faces adjacent to $\zerovec$
shows that
$\VVar{\abs{S}}(\tilde{f}; S; [0, 1]^\usedim) \le \VVar{\abs{S}}(f; S; [0, 1]^\usedim)$
for all $S \subseteq [\usedim]$.
Summing these inequalities over all
Vitali variations in the definition of
HK$\zerovec$ variation~\eqref{equation:hkvariation} leads to
\begin{equation}
\sum_{j \ge 2} \abs{\coefplain_j}
= \HKVar(\tilde{f}; [0, 1]^\usedim)
\le \HKVar(f; [0, 1]^\usedim) \le \LASSOrad
\end{equation}
as desired.

We now consider the case when the design $\xvec_1, \ldots, \xvec_\numobs$
is not a lattice.
Recall we have assumed the columns of $\Altdesignmat$
are $\UReval{\uvec_1}, \ldots, \UReval{\uvec_m}$.
We can augment $\Altdesignmat$ further by redefining
$\UReval{\zvec}$ as $(\Ind_{[\zvec, \onevec]}(\uvec_1), \ldots, \Ind_{[\zvec, \onevec]}(\uvec_m))$ which amounts to adding new rows to $\Altdesignmat$.
This new matrix, call it $\Altdesignmat''$, is precisely the matrix that would have resulted
if our original design $\xvec_1, \ldots, \xvec_\numobs$
were the full lattice $\uvec_1, \ldots, \uvec_m$.
By the above argument, there exists $\coef$ such that
$\Altdesignmat'' \coef = (f(\uvec_1), \ldots, f(\uvec_m))$
and $\sum_{j \ge 2} \abs{\coefplain_j} \le \LASSOrad$.
Discarding the rows of $\Altdesignmat''$ that correspond to lattice points
$\uvec_j$ that do not belong to the original design $\{\xvec_1, \ldots, \xvec_\numobs\}$,
we obtain $\Altdesignmat \coef = (f(\xvec_1), \ldots, f(\xvec_\numobs))$.

\subsection{Proof of \autoref{proposition:hk_lasso}}
\label{section:proof_hk_lasso}

The optimization problem~\eqref{equation:hkfitfun}
only involves the values of the function at $\xvec_1, \ldots, \xvec_\numobs$.
Thus by \autoref{proposition:hk_discrete},
$\HKfitfun$ must satisfy
$(\HKfitfun(\xvec_1), \ldots, \HKfitfun(\xvec_\numobs))
= \Altdesignmat \coefhathk$.
Furthermore, in \autoref{section:proof_hk_discrete}
we construct precisely the function
in the result~\eqref{equation:hkfitfun_lasso}
and shows that it has HK$\zerovec$ variation
equal to $\sum_{j=2}^\altdim \abs{(\coefhathkplain)_j}$.

\subsection{Proof of \autoref{lemma:VC_app}}
\label{section:proof_VC_app}

We will argue that the Vapnik-Chervonenkis (VC) dimension of
``upper-right rectangles'': $\{(\zvec, \onevec] : \zvec \in [0,
1]^\usedim\}$ is $\usedim$. A direct application of the Vapnik-Chervonenkis
lemma~\cite{VapnikCervonenkis71events} would then yield
\autoref{lemma:VC_app}.

To show that the VC dimension is $d$, one can first check that the set
$\{\onevec - \frac{1}{2} \evec_1, \ldots, \onevec - \frac{1}{2} \evec_\usedim\}$
can be shattered by these rectangles, so the VC dimension is $\ge \usedim$.
To show that no set $\{\avec_1, \ldots, \avec_{\usedim + 1}\}$ of size $\usedim + 1$ can be shattered (so that the VC dimension is $\le \usedim$),
note that there must exist some point $\avec_i$ such that the component-wise
minimum of the $\usedim + 1$ points does not change after removing $\avec_i$;
thus the rectangles cannot select the other $\usedim$ points without
also selecting $\avec_i$.

\subsection{Proof of \autoref{boann}} \label{section:proof_boann}
    Let us first start by describing some basic notation. Since we are
    working in the lattice design setting, we shall write the
    components of a vector $\param \in \R^n$ by $\param_{\ivec}, \ivec
    \in \IndexSet$ (note that $\IndexSet$ is defined in
    \eqref{AllInd}). We shall also write the design points as
    $\xvec_{\ivec}, \ivec \in \IndexSet$ where
    \begin{equation*}
      \xvec_{\ivec} = \left(\frac{i_1}{n_1}, \dots, \frac{i_d}{n_d}
      \right) \qt{for $\ivec = (i_1, \dots, i_d)$}.
    \end{equation*}
    The design matrix $\Altdesignmat$ is $n \times n$. We shall index
    the rows and columns of $\Altdesignmat$ by $\IndexSet$ so that
    \begin{equation*}
      \Altdesignmat(\ivec, \jvec) = \Ind\{\xvec_{\jvec} \preceq
      \xvec_{\ivec}\} = \Ind\{\jvec \preceq \ivec\}
    \end{equation*}
    where $\jvec \preceq \ivec$ simply refers to $j_1 \leq i_1, \dots,
    j_d \leq i_d$. The key to proving \autoref{boann} is the
    observation that for every $\param \in \R^n$, we have
    \begin{equation}\label{invo}
      \Altdesignmat (D \param) = \param.
    \end{equation}
    In other words, the differencing operator $D$ is simply equal to
    the inverse of  $\Altdesignmat$. From \eqref{invo}, it should be
    clear that \eqref{fexp} and \eqref{fexpv} follow directly from
    immediately  \eqref{bexp} and \eqref{bexpv} respectively. To prove
    \eqref{invo}, we need to show that the $\ivec^{th}$ component of
    $\Altdesignmat (D \param)$ equals the $\ivec^{th}$ component of
    $\param$ for every $\ivec \in \IndexSet$. For this, we write
    \begin{align*}
      \left(\Altdesignmat(D \param) \right)_{\ivec} &= \sum_{\jvec \in \IndexSet}
                                       \Altdesignmat(\ivec, \jvec)
                                       (D \param)_{\jvec} \\
&= \sum_{\jvec \in \IndexSet} \Ind\{\jvec \preceq \ivec\}
  (D \param)_{\jvec} \\
&= \sum_{\jvec} \Ind\{\jvec \preceq \ivec\} \sum_{\lvec \in \{0,
  1\}^d} \Ind \left\{\lvec \preceq \jvec \right\} (-1)^{l_1 + \dots +
  l_d} \theta_{\jvec - \lvec} \\
&= \sum_{\kvec \in \IndexSet} \theta_{\kvec} \left(\sum_{\lvec \in
  \{0, 1\}^d} \Ind\{\zerovec \preceq \kvec \preceq \ivec - \lvec\}
  (-1)^{l_1 + \dots + l_d} \right) \\
&= \sum_{\kvec \in \IndexSet} \theta_{\kvec} \left(\prod_{u=1}^d
  \sum_{l_u = 0}^1 \Ind\{0 \leq k_u \leq i_u - l_u\} (-1)^{l_u}
  \right) \\
&= \sum_{\kvec \in \IndexSet} \theta_{\kvec} \prod_{u=1}^d \Ind \{k_u
  = i_u\} = \theta_{\ivec}.
    \end{align*}
This proves \eqref{invo} and completes the proof of \autoref{boann}.

\section{Proofs of technical lemmas from section~\ref{SECTION:RISK_PROOFS}}
\label{section:lemma_proofs}
In this section, we prove the all the lemmas  stated in
\autoref{SECTION:RISK_PROOFS}. Specifically, we provide proofs of
\autoref{lemma:metric_entropy_nnls_rectangle},
\autoref{lemma:bound_dudley_integral},
\autoref{lemma:lasso_metric_entropy},
\autoref{lemma:mindexset_lb}, \autoref{lemma:3steps},
\autoref{lemma:tangent_cone}, \autoref{lemma:tcone_inequality},
\autoref{lemma:global_to_local}, \autoref{lemma:nearly_cm} and
\autoref{lemma:gw_mix}. In addition, we also state and prove
\autoref{lemma:edgecoef} which was used in the proof of
\autoref{theorem:lasso_adaptive_d}  and which is also needed for the
proof of  \autoref{lemma:global_to_local}.

\subsection{Proof of \autoref{lemma:metric_entropy_nnls_rectangle}}
\label{section:proof_metric_entropy_nnls_rectangle}
Let $\Probspace \defn [0, 1]^\usedim$,
and let $\Probspacesm \defn \Probspace \setminus \{\xvec_1\}$
be the result
of removing the first design point $\xvec_1 \defn \zerovec$.
Recall that by definition~\eqref{equation:NNLSSet_def},
the elements of $\NNLSSet$ are of the form $\Altdesignmat \coef$
where $\coefplain_j \ge 0$ for $j \ge 2$.
Recall also that the $j$th column of $\Altdesignmat$ is
$\UReval{\xvec_j}$ due to the lattice design~\eqref{equation:lattice_design}
so $(\Altdesignmat \coef)_i = \sum_{i' : \xvec_{i'} \preceq \xvec_i} \coefplain_{i'}$
for $i=1, \ldots, \numobs$. This suggests we can express $\NNLSSet$ in terms of distribution functions.

Given such a $\coef$, we define a measure $\mu$ supported on $\Probspacesm$
by $\mu\{\xvec_j\} = \coefplain_j$ for $j \ge 2$.
We also let $b \defn \coefplain_1$.
If we consider the distribution function
$F_{\mu + b \dirac_{\xvec_1}}(\xvec) \defn (\mu + b \dirac_{\xvec_1})([\zerovec, \xvec])$
of the signed measure $\mu + b \dirac_{\xvec_1}$,
then $F_{\mu + b \dirac_{\xvec_1}}(\xvec_i) = \sum_{i' : \xvec'_i \preceq \xvec_i} \coefplain_{i'} = (\Altdesignmat \coef)_i$ for all $i = 1, \ldots, \numobs$

Conversely, given any measure $\mu$ supported on $\Probspacesm$ and real number $b$,
we may define $\coefplain_j \defn (\mu + b\dirac_{\xvec_1})\{\xvec_j\}$ for all $j = 1, \ldots, \numobs$ and note that it satisfies $\coefplain_j \ge 0$ for $j \ge 2$
and $F_{\mu + b \dirac_{\xvec_1}}(\xvec_i) = \sum_{i' : \xvec'_i \preceq \xvec_i} \coefplain_{i'} = (\Altdesignmat \coef)_i$ for all $i = 1, \ldots, \numobs$.

Therefore,
\begin{equation}
\NNLSSet =
\braces*{
    (F_{\mu + b \dirac_{\xvec_1}}(\xvec_1), \ldots, F_{\mu + b \dirac_{\xvec_1}}(\xvec_\numobs) ) : b \in \R,
    \text{ finite measure $\mu$ on $\Probspacesm$}
},
\end{equation}
Recall that the total variation of a signed measure $\nu$ on $\Probspace$ is defined by
$\tvnorm{\nu} \defn \nu_+(\Probspace) + \nu_-(\Probspace)$
where $\nu = \nu_+ - \nu_-$ is the Jordan decomposition of the signed measure.
We define the more restricted set
\begin{align}
\label{equation:NNLSSet_TV}
\NNLSSet(\TVrad)
&\defn
\big\{
    (F_{\mu + b \dirac_{\xvec_1}}(\xvec_1), \ldots, F_{\mu + b \dirac_{\xvec_1}}(\xvec_\numobs) ):
    b \in \R,
    \\&\qquad
    \text{ finite measure $\mu$ on $\Probspacesm$}, \tvnorm{\mu + b \dirac_{\xvec_1}} \le \TVrad
\big\},
\end{align}
which will be useful in our goal of bounding the metric entropy of
$\NNLSSet \cap \Ball_2(\zerovec, \smallrad)$.
Note that the total variation term can be written as
\begin{equation}\label{equation:tv_express}
\tvnorm{\mu + b \dirac_{\xvec_1}} = \mu(\Probspacesm) + \abs{b}.
\end{equation}




Let $\param \defn (F_{\mu + b \dirac_{\xvec_1}}(\xvec_1), \ldots, F_{\mu + b \dirac_{\xvec_1}}(\xvec_\numobs) )$
and
$\param' \defn (F_{\mu' + b' \dirac_{\xvec_1}}(\xvec_1), \ldots, F_{\mu' + b' \dirac_{\xvec_1}}(\xvec_\numobs) )$;
recall these distribution functions belong to the function class $\EMClass$
(\autoref{proposition:cm_discrete}).
The Euclidean distance on $\NNLSSet$ is related
to the $L^2$ distance on $\EMClass$, as
\begin{align}
&\numobs \int_{[0,1]^\usedim} (F_{\mu + b \dirac_{\xvec_1}} - F_{\mu' + b' \dirac_{\xvec_1}})^2 \, d\lambda
\\
&\qquad = \numobs \sum_{i=1}^\numobs (\paramplain_i - \paramplain'_i)^2
\lambda([\xvec_i, \xvec_i + (\numobs_1^{-1}, \ldots, \numobs_\usedim^{-1})])
= \norm{\param - \param'}^2,
\end{align}
where the integral is respect to the Lebesgue measure $\lambda$.
Note that this equality holds even when $\numobs_j = 1$ for some of the $j$.

Thus, the $\radcover$-metric entropy of $\NNLSSet(\TVrad)$ (in the Euclidean norm)
is bounded by the $\radcover/\sqrt{\numobs}$-metric entropy
of distribution functions of signed measures with total variation norm $\le \TVrad$
(in the $L^2$ norm).
As explained in \citet[Sec. 3]{blei2007metric} (see also
  \citet{gao2013bracketing}), we have:
\begin{equation}\label{equation:blei}
\log \covernum_2(\radcover, \NNLSSet(\TVrad))
\le C_\usedim \frac{\TVrad \sqrt{\numobs}}{\radcover}
\parens*{
    \log \frac{\TVrad \sqrt{\numobs}}{\radcover}
}^{\usedim - \frac{1}{2}}
\qquad \text{whenever } \frac{\radcover}{\TVrad \sqrt{\numobs}} < e^{-1}
\end{equation}
for $\usedim > 1$.
We remark again that this inequality holds even when $\numobs_j = 1$ for some of the $j$.

The following inclusions show that $\NNLSSet(\TVrad)$ is essentially
the same as $\NNLSSet \cap [-R, R]^\numobs$ up to a constant scaling factor.
\begin{equation}
\label{equation:tv_sandwich}
\NNLSSet(\TVrad)
\subseteq \NNLSSet \cap [-\TVrad, \TVrad]^\numobs
\subseteq \NNLSSet(3 \TVrad).
\end{equation}
To verify these inclusions, it is useful to recall that
for $\param \defn (F_{\mu + b \delta_{\xvec_1}}(\xvec_1), \ldots, F_{\mu + b \delta_{\xvec_1}}(\xvec_\numobs))$ we have
$\max_i \paramplain_i = \paramplain_\numobs$ and $\min_i \paramplain_i
= \paramplain_1$,
as well as the fact
that if $\param \in \NNLSSet$ is associated with the pair $(\mu, b)$,
then $\tvnorm{\mu + b \delta_{\xvec_1}} = \mu(\Probspacesm) + \abs{b}
= (\paramplain_\numobs - \paramplain_1) + \abs{\paramplain_1}$.
The first inclusion follows from the fact that $(\paramplain_\numobs - \paramplain_1) + \abs{\paramplain_1} \le \TVrad$ implies $\paramplain_\numobs \le R$ and $\paramplain_1 \ge -R$.
For the second inclusion,
note that $-\TVrad \le \paramplain_1 \le \paramplain_\numobs \le \TVrad$ implies
$(\paramplain_\numobs - \paramplain_1) + \abs{\paramplain_1} \le 3 \TVrad$.

The second inclusion~\eqref{equation:tv_sandwich} immediately yields
\begin{equation}
\log \covernum_2(\radcover, \NNLSSet \cap [-\TVrad, \TVrad]^\numobs)
\le C_\usedim \frac{3 \TVrad \sqrt{\numobs}}{\radcover}
\parens*{
    \log \frac{3 \TVrad \sqrt{\numobs}}{\radcover}
  }^{\usedim - \frac{1}{2}},
  \qquad \forall \epsilon < 3 \TVrad \sqrt{\numobs} / e.
\end{equation}
Because $\NNLSSet$ is translation invariant, we may translate a hyperrectangle
of the form $[a, b]^\numobs$ to $[-\TVrad, \TVrad]^\usedim$ for $\TVrad \defn \frac{b-a}{2}$,
and obtain
\begin{equation}
\log \covernum_2(\radcover, \NNLSSet \cap [a,b]^\numobs)
\le C_\usedim \frac{(b-a) \sqrt{\numobs}}{\radcover}
\parens*{
    \log \frac{(b-a) \sqrt{\numobs}}{\radcover}
  }^{\usedim - \frac{1}{2}}
\end{equation}
for $\epsilon < \frac{3}{2e} (b-a) \sqrt{\numobs}$,
where we have absorbed some constants into $C_\usedim$.

To show that this bound holds under the more general condition
$\epsilon \le \sqrt{n}(b-a)$,
simply observe that if
$\epsilon \ge \sqrt{n}(b-a) / 2$, then a single point whose entries are each $(a+b)/2$ covers $\NNLSSet \cap [a,b]^\numobs$, and so the log covering number is $0$,
which is bounded by the right-hand side as long as $\epsilon \le (b-a) \sqrt{n}$.

\subsection{Proof of \autoref{lemma:bound_dudley_integral}}
\label{section:proof_bound_dudley_integral}

The substitution $u = \frac{1}{2} \log \frac{B}{\radcover}$
and $du = -\frac{1}{2 \radcover} \, d\radcover$
allows us to rewrite the integral as
\begin{equation}
2^{\frac{2 \usedim + 3}{4}} B
\int_a^\infty
e^{-u} u^{\frac{2 \usedim - 1}{4}}
\, du,
\qquad \text{with } a \defn \frac{1}{2} \log \frac{B}{s}.
\end{equation}
It thus suffices to show
\begin{equation}\label{equation:upper_gamma}
I(a) \defn \int_a^\infty
e^{-u} u^{\frac{2 \usedim - 1}{4}}
\, du
\le C_\usedim e^{-a} (a + 1/2)^{\frac{2 \usedim - 1}{4}},
\qquad \forall a \ge 0.
\end{equation}

If $a \le 1$, then
\begin{equation}
I(a) \le \int_0^\infty
e^{-u} u^{\frac{2 \usedim - 1}{4}}
\, du
\le C_\usedim e^{-a} 2^{-\frac{2\usedim - 1}{4}}
\end{equation}
for $C_\usedim \ge e 2^{\frac{2\usedim - 1}{4}}
\int_0^\infty
e^{-u} u^{\frac{2 \usedim - 1}{4}}
\, du$, proving the claim~\eqref{equation:upper_gamma}.

Now suppose $a > 1$.
Let $v$ be the smallest
positive integer strictly larger than $\frac{2 \usedim - 1}{4}$.
Performing integration by parts $v$ times yields
\begin{align}
I(a)
&\le C_\usedim e^{-a}
\sum_{r = 1}^v a^{\frac{2 \usedim  -1}{4} - r + 1}
+ C_\usedim \int_a^\infty e^{-u} u^{\frac{2 \usedim - 1}{4} - v} \, du
\\
&\le C_\usedim e^{-a} a^{\frac{2 \usedim - 1}{4}} + C_\usedim e^{-a}
\\
&\le (C_\usedim + C_\usedim 2^{\frac{2 \usedim - 1}{4}}) e^{-a}
(a + 1/2)^{\frac{2 \usedim - 1}{4}},
\end{align}
which proves the claim~\eqref{equation:upper_gamma}.

\subsection{Proof of \autoref{lemma:lasso_metric_entropy}}
\label{section:proof_lasso_metric_entropy}

Suppose $\param = \Altdesignmat \coef \in \LASSOBall(\LASSOrad, \rad)$,
where $\Altdesignmat$ is the usual design matrix defined in~\autoref{SECTION:COMPUTATION}.
Note that for any $i$ we have
\begin{equation}
\abs{\paramplain_i - \paramplain_1}
= \abs*{
    \sum_{j : \xvec_j \preceq \xvec_i} \coefplain_j
    - \coefplain_1
}
\le \sum_{j = 2}^\numobs \abs{\coefplain_j}
\le \LASSOrad.
\end{equation}
Thus, using the simple inequality $(a+b)^2 \ge \frac{1}{2} a^2 - b^2$
along with the fact that $\norm{\param}^2 \le \rad$
we obtain
\begin{equation}
\paramplain_i^2
= (\paramplain_1 + \paramplain_i - \paramplain_1)^2
\ge \frac{1}{2} \paramplain_1^2
- (\paramplain_i - \paramplain_1)^2
\ge \frac{1}{2} \paramplain_1^2 - \LASSOrad^2
\end{equation}
for each $i$, and thus
\begin{equation}
\rad^2 \ge \sum_{i=1}^\numobs \paramplain_i^2
\ge \paramplain_1^2 + (n-1)
\parens*{\frac{1}{2} \paramplain_1^2 - \LASSOrad^2}.
\end{equation}
Rearranging this and applying the inequality $\sqrt{a+b} \le \sqrt{a} + \sqrt{b}$ for nonnegative $a,b$ yields
\begin{equation}
\abs{\paramplain_1}
\le \sqrt{\frac{2}{\numobs+1} \parens*{\rad^2 + (\numobs - 1) \LASSOrad^2}}
\le \rad\sqrt{\frac{2}{\numobs}}
+ \LASSOrad \sqrt{2}
\eqqcolon \radtilde.
\end{equation}

We fix $\delta > 0$, whose value will be chosen later.
If for an integer $k$ we define
\begin{equation}
\tilde{\LASSOBall}_k(\LASSOrad, \rad)
\defn \{\param \in \LASSOBall(\LASSOrad, \rad) :
k \delta \le \paramplain_1 \le (k + 1) \delta\},
\end{equation}
then we have
\begin{equation}
\LASSOBall(\LASSOrad, \rad)
\subseteq \bigcup_{- K - 1 \le k \le K}
\tilde{\LASSOBall}_k(\LASSOrad, \rad)
\end{equation}
where $K = \floor{\radtilde/\delta}$.
Then,
\begin{equation}\label{equation:metric_entropy_step_improved}
\log \covernum(\radcover, \LASSOBall(\LASSOrad, \rad))
\le \log\parens*{2 + \frac{\radtilde}{\delta}}
+ \max_{-K-1 \le k \le K}
\log \covernum(\radcover, \tilde{\LASSOBall}_k(\LASSOrad, \rad)).
\end{equation}
Since $\tilde{\LASSOBall}_{-k - 1}(\LASSOrad, \rad) = - \tilde{\LASSOBall}_k(\LASSOrad, \rad)$ for $k \ge 0$, we may restrict the maximum on the right-hand side to $0 \le k \le K$.

Fix $k \ge 0$.
If $\param = \Altdesignmat \coef \in \tilde{\LASSOBall}_k(\LASSOrad, t)$,
we let $\pospart(\param) \defn \Altdesignmat \coef^+$
and $\negpart(\param) \defn \Altdesignmat \coef^-$,
where $\coefplain^+_j \defn \max\{\coefplain_j, 0\}$
and $\coefplain^-_j \defn \max\{-\coefplain_j, 0\}$
so that $\param = \pospart(\param) - \negpart(\param)$.
Defining
\begin{align}
\LASSOBall_{\pospart}(\LASSOrad, t)
&\defn \{\pospart(\param) : \param \in \tilde{\LASSOBall}_k(\LASSOrad, t)\},
\\
\LASSOBall_{\negpart}(\LASSOrad, t)
&\defn \{\negpart(\param) : \param \in \tilde{\LASSOBall}_k(\LASSOrad, t)\},
\end{align}
we therefore have
\begin{equation}\label{equation:lasso_metric_entropy_step}
\log \covernum(\radcover, \tilde{\LASSOBall}_k(\LASSOrad, t))
\le \log \covernum(\radcover/2, \LASSOBall_{\pospart}(\LASSOrad, t))
+ \log \covernum(\radcover/2, \LASSOBall_{\negpart}(\LASSOrad, t)).
\end{equation}

We bound the second term first.
Because $\coefplain_1 = \paramplain_1 \ge k \delta \ge 0$
(recall the first column of $\Altdesignmat$ is the all-ones vector)
for $\param \in \tilde{\LASSOBall}_k(\LASSOrad, t)$,
we have $(\pospart(\param))_1 = \coefplain_1$ and $(\negpart(\param))_1 = 0$.
Also,
\begin{equation}
(\pospart(\param))_\numobs - (\pospart(\param))_1
+ (\negpart(\param))_\numobs - (\negpart(\param))_1
= \sum_{j = 2}^\numobs \coefplain_j \Ind_{\coefplain_j \ge 0}
- \sum_{j = 2}^\numobs \coefplain_j \Ind_{\coefplain_j \le 0}
= \sum_{j =2}^\numobs \abs{\coefplain_j} \le \LASSOrad,
\end{equation}
which implies
$\negpart(\param)_\numobs \le \LASSOrad$
for $\param \in \LASSOBall(\LASSOrad, t)$.
Since elements of $\LASSOBall_{\negpart}(\LASSOrad, t)$
are of the form $\Altdesignmat \coef$ with $\coefplain_j \ge 0$ for $j \ge 2$,
we use \autoref{proposition:cm_discrete}
and recall a definition~\eqref{equation:NNLSSet_TV} to obtain
the inclusion
$\LASSOBall_{\negpart}(\LASSOrad, t) \subseteq \NNLSSet \cap [0, \LASSOrad]^\numobs$.
Thus, using the bound~\eqref{equation:tv_sandwich}
along with \autoref{lemma:metric_entropy_nnls_rectangle}
we obtain
\begin{equation}
\log \covernum(\radcover/2, \LASSOBall_{\negpart}(\LASSOrad, t))
\le C_\usedim \frac{\LASSOrad \sqrt{\numobs}}{\radcover}
\parens*{\log \frac{2 \LASSOrad \sqrt{\numobs}}{\radcover}}^{\usedim - \frac{1}{2}}
\Ind\{\radcover \le 2 \LASSOrad \sqrt{\numobs}\}
\end{equation}
where we have absorbed constants into $C_\usedim$.

We now bound the first term from earlier~\eqref{equation:lasso_metric_entropy_step}.
We claim
\begin{equation}
\LASSOBall_{\pospart}(\LASSOrad, \rad)
\subseteq \NNLSSet \cap [0, \LASSOrad + \delta]^\numobs + \{k \delta\}.
\end{equation}
To see the last inclusion, note that if
$\etavec \in \LASSOBall_{\pospart}(\LASSOrad, t)$
satisfies
$k\delta \le \eta_1 \le (k+1)\delta$
then $\etavec - k \delta \onevec$ lies in $\NNLSSet$
and has all entries
lying in the interval $[0, \LASSOrad + \delta]$
(since $k \delta \le \eta_1 \le \eta_i$
and $\eta_i - k \delta \le \eta_i - (\eta_1 - \delta) \le \LASSOrad + \delta$).
Noting that $\NNLSSet$ is invariant under translation, we need only compute the metric entropy of $\NNLSSet \cap [0, \LASSOrad + \delta]^\numobs$.
Applying \autoref{lemma:metric_entropy_nnls_rectangle} again yields
\begin{equation}
\log \covernum(\radcover/2, \LASSOBall_{\pospart}(\LASSOrad, t))
\le C_\usedim \frac{(\LASSOrad + \delta) \sqrt{\numobs}}{\radcover}
\parens*{\log \frac{2 (\LASSOrad + \delta) \sqrt{\numobs}}{\radcover}}^{\usedim - \frac{1}{2}}
\Ind\{\radcover \le 2 (\LASSOrad + \delta) \sqrt{\numobs}\}.
\end{equation}
Choosing $\delta = \radcover / \sqrt{\numobs}$ yields
\begin{equation}
\log \covernum(\radcover/2, \LASSOBall_{\pospart}(\LASSOrad, t))
\le
C_\usedim
\parens*{\frac{\LASSOrad \sqrt{\numobs}}{\radcover} + 1}
\parens*{
    \log\parens*{\frac{2 \LASSOrad \sqrt{\numobs}}{\radcover} + 1}
}^{\usedim - \frac{1}{2}}.
\end{equation}
Returning to~\eqref{equation:lasso_metric_entropy_step}
we obtain
\begin{equation}
\log \covernum(\radcover, \tilde{\LASSOBall}_k(\LASSOrad, \rad))
\le C_\usedim
\parens*{\frac{\LASSOrad \sqrt{\numobs}}{\radcover} + 1}
\parens*{
    \log\parens*{\frac{2 \LASSOrad \sqrt{\numobs}}{\radcover} + 1}
}^{\usedim - \frac{1}{2}}.
\end{equation}
Going further back to~\eqref{equation:metric_entropy_step_improved}
and plugging our definitions of $\delta$ and $\radtilde$ yields
\begin{equation}
\log \covernum(\radcover, \LASSOBall(\LASSOrad, \rad))
\le
C_\usedim
\parens*{\frac{\LASSOrad \sqrt{\numobs}}{\radcover} + 1}
\parens*{
    \log\parens*{\frac{2 \LASSOrad \sqrt{\numobs}}{\radcover} + 1}
}^{\usedim - \frac{1}{2}}
+ \log\parens*{
    2 + 2 \frac{t + \LASSOrad \sqrt{\numobs}}{\radcover}
}.
\end{equation}

\subsection{Proof of \autoref{lemma:mindexset_lb}}
\label{section:proof_mindexset_lb}
Let $r \defn \floor{2 \ell / \usedim}$.
By an inclusion-exclusion argument, we have the following exact formula
for the cardinality.
\begin{equation}
\abs{\MIndexSet_\ell} = \sum_{k=0}^\usedim (-1)^k \binom{\usedim}{k} \binom{\ell - k r - 1}{\usedim - 1},
\end{equation}
with the convention that $\binom{a}{b} = 0$ if $a < b$.

If $k \ge \usedim / 2$ then we have
$\ell - kr \le \ell - \frac{\usedim}{2} \parens*{\frac{2 \ell}{\usedim} - 1} = \frac{\usedim}{2} < \usedim$
which implies $\binom{\ell - k r - 1}{\usedim - 1} = 0$.

Otherwise, for $k < \usedim / 2$ we have
\begin{equation}
\ell^{-(\usedim - 1)} \binom{\ell - k r - 1}{\usedim - 1}
= \frac{1}{(\usedim - 1)!} \prod_{i=1}^{\usedim - 1} \frac{\ell - k r - i}{\ell}.
\end{equation}
Noting that $\lim_{\ell \to \infty}\frac{\ell - k r - i}{\ell} = 1 - k \lim_{\ell \to \infty} \frac{r}{\ell} = 1 - \frac{2 k }{\usedim}$, we obtain
\begin{equation}
\lim_{\ell \to \infty}
\ell^{-(\usedim - 1)} \binom{\ell - k r - 1}{\usedim - 1}
= \frac{\parens*{1 - \frac{2 k }{\usedim}}^{\usedim - 1}}{(\usedim - 1)!}.
\end{equation}
for $k < \usedim / 2$.
Combining these observations for all $k$ yields
\begin{equation}
\lim_{\ell \to \infty} \frac{\abs{\MIndexSet_\ell}}{\ell^{\usedim - 1}}
= \sum_{k=0}^{\usedim} (-1)^k \binom{\usedim}{k}
\frac{\parens*{1 - \frac{2 k }{\usedim}}_+^{\usedim - 1}}{(\usedim - 1)!}
= \frac{\usedim^{\usedim - 1}}{(\usedim - 1)!}
\sum_{k=0}^\usedim
(-1)^k \binom{\usedim}{k} (d - 2 k)_+^{\usedim - 1},
\end{equation}
where $(x)_+ \defn \max\{x, 0\}$.
It then suffices to check
\begin{equation}
b_\usedim \defn \sum_{k=0}^\usedim
(-1)^k \binom{\usedim}{k} (d - 2 k)_+^{\usedim - 1}
> 0
\end{equation}
for each fixed $\usedim \ge 2$. Indeed, \citet{Goddard45} showed
\begin{equation}
\frac{b_\usedim}{2^\usedim (\usedim - 1)!} = \frac{1}{\pi}
\int_0^\infty \parens*{\frac{\sin x}{x}}^\usedim \, dx.
\end{equation}
When $\usedim$ is even, this clearly positive. When $\usedim$ is odd, we have
\begin{equation}
\int_0^\infty \parens*{\frac{\sin x}{x}}^\usedim \, dx
= \sum_{k = 0}^\infty \int_{k \pi}^{(k+1) \pi}
\parens*{\frac{\sin x}{x}}^\usedim \, dx
= \sum_{k = 0}^\infty (-1)^k  \int_0^\pi \parens*{\frac{ \sin x }{x+k\pi}}^\usedim \, dx,
\end{equation}
which is positive
because the last expression is
an alternating sum whose addends' magnitudes
$\int_0^\pi \parens*{\frac{ \sin x }{x+k\pi}}^\usedim \, dx$
form a positive decreasing sequence in $k$.

\subsection{Proof of \autoref{lemma:3steps}}
\label{section:proof_3steps}
We prove the three inequalities \eqref{eq:step1}, \eqref{eq:step2} and
\eqref{eq:step3} separately.
\begin{proof}[Proof of \eqref{eq:step1}]
For functions $f,g:[0,1]^\usedim \to \R$
we let $\norm{f}_2 \defn \parens*{\int_{[0,1]^\usedim} \abs{f(x)}^2 \, dx}^{1/2}$
and $\norm{f}_1 \defn \int_{[0,1]^\usedim} \abs{f(x)} \, dx$
denote the $L^2$ and $L^1$ norms on $[0, 1]^\usedim$ with respect to the Lebesgue measure,
and $\inner{f,g} \defn \int_{[0,1]^\usedim} f(x) g(x) \, dx$ denote the $L^2$ inner product.

Recall the definition of HK$\zerovec$ variation~\eqref{equation:hkvariation}
as the sum of Vitali variations over faces adjacent to $\zerovec$.
Because $f_{\etavec}(\xvec)$ is zero whenever $x_j = 0$ for some $j$,
all these Vitali variations are zero except for the Vitali variation over
the entire space $[0, 1]^\usedim$.
Thus, recalling that the Vitali variation can be written as the integral
of the magnitude of a mixed partial derivative~\eqref{vvs}, we have
\begin{align}
\HKVar(f_{\etavec}; [0, 1]^\usedim)
&= \VVar{\usedim}(f; [0, 1]^\usedim)
=
\norm*{\frac{\partial^\usedim f_{\etavec}}{\partial x_1 \cdots \partial x_\usedim}}_1
\\
&\le \norm*{\frac{\partial^\usedim f_{\etavec}}{\partial x_1 \cdots \partial x_\usedim}}_2
= \frac{\LASSOrad}{\sqrt{\abs{\MIndexSet_\ell}}}
\norm*{
    \sum_{\mvec \in \MIndexSet_\ell}
    g_{\etavec, \mvec}
}_2,
\end{align}
where
\begin{equation}
g_{\etavec, \mvec} \defn
\sum_{\ivec \in \IIndexSet_{\mvec}}
\eta_{\mvec, \ivec}
\bigotimes_{j=1}^\usedim \phi'_{m_j, i_j}.
\end{equation}
For natural numbers $m < m'$ and natural numbers $i \le 2^m$ and $i' \le 2^{m'}$,
the functions $\phi'_{m, i}$ and $\phi'_{m', i'}$ are orthogonal.
Thus for distinct $\mvec, \mvec' \in \MIndexSet_\ell$,
the functions $g_{\etavec, \mvec}$ and $g_{\etavec, \mvec'}$ are orthogonal as well.
Thus from above we have
\begin{align}
\HKVar(f_{\etavec}; [0, 1]^\usedim)
&\le \frac{\LASSOrad}{\sqrt{\abs{\MIndexSet_\ell}}}
\sqrt{
\sum_{\mvec \in \MIndexSet_\ell}
\norm*{
    g_{\etavec, \mvec}
}_2^2
}.
\end{align}
For a fixed natural number $m$ and distinct natural numbers $i,i' \le 2^m$,
the functions $\phi'_{m, i}$ and $\phi'_{m, i'}$ are also orthogonal
because they have different supports.
Thus for fixed $\mvec \in \MIndexSet$
and distinct $\ivec, \ivec' \in \IIndexSet_{\mvec}$,
the functions
$\bigotimes_{j=1}^\usedim \phi'_{m_j, i_j}$
and $\bigotimes_{j=1}^\usedim \phi'_{m_j, i'_j}$
are orthogonal.
Continuing from above, we obtain
\begin{align}
\HKVar(f_{\etavec}; [0, 1]^\usedim)
&\le \frac{\LASSOrad}{\sqrt{\abs{\MIndexSet_\ell}}}
\sqrt{
\sum_{\mvec \in \MIndexSet_\ell}
\sum_{\ivec \in \IIndexSet_{\mvec}}
\norm*{
    \bigotimes_{j=1}^\usedim \phi'_{m_j, i_j}
}_2^2
}
\\
&= \frac{\LASSOrad}{\sqrt{\abs{\MIndexSet_\ell}}}
\sqrt{
\sum_{\mvec \in \MIndexSet_\ell}
\sum_{\ivec \in \IIndexSet_{\mvec}}
2^{-\ell}
}
= V,
\end{align}
where we used the fact that $\abs{\IIndexSet_{\mvec}} = 2^\ell$
and
\begin{equation}
\norm*{
    \bigotimes_{j=1}^\usedim \phi'_{m_j, i_j}
}_2^2
= \prod_{j=1}^\usedim
\norm*{
    \phi'_{m_j, i_j}
}_2^2
= \prod_{j=1}^\usedim 2^{-m_j}
= 2^{-\ell}.
\end{equation}
\end{proof}

\begin{proof}[Proof of \eqref{eq:step2}]

By Pinsker's inequality, we can bound the total variation distance
between $\P_{f_{\etavec}}$ and $\P_{f_{\etavec'}}$
by their Kullback-Leibler divergence.
\begin{equation}
\tvnorm{\P_{f_{\etavec}} - \P_{f_{\etavec'}}}
\le \sqrt{\frac{1}{2} \KL(\P_{f_{\etavec}} \| \P_{f_{\etavec'}})}.
\end{equation}
The KL divergence can be computed as
\begin{equation}
\KL(\P_{f_{\etavec}} \| \P_{f_{\etavec'}})
= \frac{1}{2 \noisestd^2} \sum_{i = 1}^\numobs (f_{\etavec}(\xvec_i) - f_{\etavec'}(\xvec_i))^2
= \frac{\numobs}{2 \noisestd^2} \Loss(f_{\etavec}, f_{\etavec'}),
\end{equation}
where $\Loss$ denotes the discrete loss as defined earlier~\eqref{rislo}.

Note that
\begin{equation}
f_{\etavec} - f_{\etavec'}
= \frac{\LASSOrad}{\sqrt{\abs{\MIndexSet_\ell}}}
\sum_{\mvec \in \MIndexSet_\ell}
\sum_{\ivec \in \IIndexSet_{\mvec}}
(\eta_{\mvec, \ivec} - \eta'_{\mvec, \ivec})
\bigotimes_{j=1}^\usedim \phi_{m_j, i_j}.
\end{equation}

If $\Hamming(\etavec, \etavec') = 1$, then there exists a unique pair $\mvec \in \MIndexSet_\ell$
and $\ivec \in \IIndexSet_{\mvec}$ such that $\eta_{\mvec, \ivec} \ne \eta'_{\mvec, \ivec}$.
Then
\begin{equation}\label{equation:eta_diff}
f_{\etavec} - f_{\etavec'}
= \frac{\LASSOrad}{\sqrt{\abs{\MIndexSet_\ell}}}
(\eta_{\mvec, \ivec} - \eta'_{\mvec, \ivec})
\bigotimes_{j=1}^\usedim \phi_{m_j, i_j}.
\end{equation}
Thus, recalling that the design points $\xvec_1, \ldots, \xvec_\numobs$
come from the lattice $\LatticeDesign$ (see~\eqref{equation:lattice_design})
\begin{align}
\Loss(f_{\etavec}, f_{\etavec'})
&=
\frac{4\LASSOrad^2}{\numobs \abs{\MIndexSet_\ell}}
\sum_{k_1=0}^{\numobs_1 - 1}
\cdots
\sum_{k_\usedim = 0}^{\numobs_\usedim - 1}
\prod_{j=1}^\usedim
(\phi_{m_j, i_j}(k_j / \numobs_j))^2
\\
&=
\frac{4\LASSOrad^2}{\abs{\MIndexSet_\ell}}
\prod_{j=1}^\usedim
\parens*{
\frac{1}{\numobs_j}
\sum_{k_j=0}^{\numobs_j - 1}
(\phi_{m_j, i_j}(k_j / \numobs_j))^2
}.
\end{align}

Note that for each $j$, the number of nonzero
addends $(\phi_{m_j, i_j}(k_j / \numobs_j))^2$ (of the above inner sum)
is bounded by $\numobs_j 2^{-m_j}$, so we obtain
\begin{align}
\frac{1}{\numobs_j} \sum_{k_j = 0}^{\numobs_j - 1}
(\phi_{m_j, i_j}(k_j / \numobs_j))^2
\le \frac{1}{\numobs_j} \cdot \numobs_j 2^{-m_j} \cdot 2^{-2 m_j - 4}
= 2^{-3 m_j - 4}.
\end{align}
Multiplying over all $j$ yields
\begin{equation}
\prod_{j=1}^\usedim
\parens*{
\frac{1}{\numobs_j}
\sum_{k_j=0}^{\numobs_j - 1}
(\phi_{m_j, i_j}(k_j / \numobs_j))^2
}
= \prod_{j=1}^\usedim 2^{-3 m_j - 4}
= 2^{-3 \ell - 4 \usedim}.
\end{equation}
By combining our work above, we obtain
\begin{equation}
\max_{\Hamming(\etavec, \etavec') = 1}
\tvnorm{\P_{f_{\etavec}} - \P_{f_{\etavec'}}}
\le \sqrt{\frac{\numobs}{4 \noisestd^2} \Loss(f_{\etavec}, f_{\etavec'})}
\le \sqrt{
\frac{\numobs}{\noisestd^2}
\frac{\LASSOrad^2}{\abs{\MIndexSet_\ell}}
2^{-3 \ell - 4 \usedim}
}.
\end{equation}
\end{proof}

 \begin{proof}[Proof of \eqref{eq:step3}]
To compute the loss $\Loss(f_{\etavec}, f_{\etavec'})$
for some $\etavec, \etavec' \in \{-1, 1\}^q$,
only the values of $\phi_{m_j, i_j}$ at points $\{k / \numobs_j : k \in \{0, \ldots, \numobs_j-1\}\}$ matter.
In particular, for each fixed $j \in [\usedim]$ and $m_j \in \mathbb{N}$ and $i_j \in [2^{m_j}]$ we define the step function
$\phitilde: [0, 1] \to \R$ by
\begin{equation}\label{equation:phitilde_def}
\phitilde_{j, m_j, i_j}(x) \defn \phi_{m_j, i_j}(\floor{x \numobs_j} / \numobs_j).
\end{equation}
Recall our assumption that $\numobs_j$ is a power of $2$.
Thus function $\phitilde_{j,m_j,i_j}$ is a step function supported on
$[(i_j - 1) 2^{-m_j}, i_j 2^{-m_j}]$
that is constant on intervals $[k/\numobs_j, (k+1)/\numobs_j)$
for $k=0, \ldots, \numobs_j - 1$,
and agrees with the value of $\phi_{m_j, i_j}$ at points $k / \numobs_j$
for $k = 0, \ldots, \numobs_j - 1$.

If for $\etavec, \etavec' \in \{-1 ,1\}^q$ we define
\begin{equation}
g_{\etavec, \etavec'}
\defn \frac{\LASSOrad}{\sqrt{\abs{\MIndexSet_\ell}}}
\sum_{\mvec \in \MIndexSet_\ell} \sum_{\ivec \in \IIndexSet_{\mvec}}
(\eta_{\mvec, \ivec} - \eta'_{\mvec, \ivec})
\bigotimes_{j=1}^\usedim \phitilde_{m_j, i_j},
\end{equation}
then $g_{\etavec, \etavec'}$ agrees with $f_{\etavec} - f_{\etavec'}$
on points of the form $(k_1/\numobs_1, \ldots, k_\usedim / \numobs_\usedim)$
for $k_j = 0, \ldots, \numobs_j$ and all $j \in [\usedim]$,
and is constant on rectangles of the form
$\bigtimes_{j=1}^\usedim [k_j/\numobs_j, (k_j + 1)/\numobs_j)$.
Therefore,
\begin{equation}
\Loss(f_{\etavec}, f_{\etavec'})
\defn \frac{1}{\numobs}
\sum_{i=1}^\numobs (f_{\etavec}(\xvec_i) - f_{\etavec'}(\xvec_i))^2
= \int_{[0,1]^\usedim} (g_{\etavec, \etavec'}(\xvec))^2 \, d\xvec
= \norm{g_{\etavec, \etavec'}}_{L^2}^2.
\end{equation}

For natural number $m$ and $i \in [2^m]$, we define the function $h_{m,i} : [0, 1] \to \R$ by
\begin{equation}\label{equation:haar}
h_{m,i}(x) = \begin{cases}
2^{m/2} & x \in [(i-1)2^{-m}, (i - \frac{1}{2}) 2^{-m}],
\\
-2^{m/2} & x \in [(i-\frac{1}{2})2^{-m}, i 2^{-m}],
\\
0 & \text{otherwise}.
\end{cases}
\end{equation}
One can check $\{h_{m,i} : m \in \mathbb{N}, i \in [2^m]\}$
is an orthonormal set.
If we define $H_{\mvec, \ivec} \defn \bigotimes_{j=1}^\usedim h_{m_j, i_j}$,
then
$\braces*{
    H_{\mvec, \ivec} : \mvec \in \MIndexSet_\ell, \ivec \in \IIndexSet_{\mvec}
}$
is an orthonormal set of functions on $[0, 1]^\usedim$.


Thus by Bessel's inequality,
\begin{equation}
\Loss(f_{\etavec}, f_{\etavec'})
= \norm{g_{\etavec, \etavec'}}_2^2
\ge \sum_{\mvec' \in \MIndexSet_\ell}
\sum_{\ivec' \in \IIndexSet_{\mvec}}
\inner*{g_{\etavec, \etavec'}, H_{\mvec', \ivec'}}^2.
\end{equation}

Fix $\mvec' \in \MIndexSet_\ell$ and $\ivec' \in \IIndexSet_{\mvec}$.
We have
\begin{equation}
\inner*{g_{\etavec, \etavec'}, H_{\mvec', \ivec'}}
= \frac{\LASSOrad}{\sqrt{\abs{\MIndexSet_\ell}}}
\sum_{\mvec \in \MIndexSet_\ell}
\sum_{\ivec \in \IIndexSet_{\mvec}}
(\eta_{\mvec, \ivec} - \eta'_{\mvec, \ivec})
\prod_{j=1}^\usedim \inner{\phitilde_{j, m_j, i_j}, h_{m'_j, i'_j}}.
\end{equation}
We claim the inner products satisfy
\begin{equation}
\inner{\phitilde_{j, m_j, i_j}, h_{m'_j, i'_j}}
= \begin{cases}
0 & (m_j, i_j) \ne (m'_j, i'_j), m_j \le m'_j
\\
0 & \log_2(\numobs_j) \le m_j + 1
\\
2^{- 3 m_j / 2 - 3} & (m_j, i_j) = (m'_j, i'_j), \log_2(\numobs_j) \ge m_j + 2
\end{cases}
\end{equation}
For the first case, if $m_j = m'_j$ and $i_j \ne i'_j$, then the supports of $\phitilde_{j, m_j, i_j}$ and $h_{m'_j, i'_j}$ are disjoint so their inner product is zero.
If instead $m_j < m'_j$, then recall $\int_0^1 \phitilde_{j, m_j, i_j}(x) \, dx = 0$
and note that $h_{m'_j, i'_j}$ is constant on the support
$[(i_j - 1) 2^{-m_j}, i_j 2^{-m_j}]$ of $\phitilde_{j, m_j, i_j}$.
The second case is due to the fact that $\numobs_j$ is a power of $2$
and consequently $\phitilde_{j, m_j, i_j} \equiv 0$ when $\log_2 \numobs_j \le m_j + 1$.
For the third case where $(m_j, i_j) = (m'_j, i'_j)$ and $\log_2(\numobs_j) \ge m_j + 2$,
we have
\begin{align}
\inner{\phitilde_{j, m_j, i_j}, h_{m_j, i_j}}
&= 2 \cdot 2^{m_j/2} \int_{(i_j - 1) 2^{-m_j}}^{(i_j - \frac{1}{2}) 2^{-m_j}}
\phitilde_{j, m_j, i_j}(x) \, dx
\\
&= 2 \cdot 2^{m_j/2} \int_{(i_j - 1) 2^{-m_j}}^{(i_j - \frac{1}{2}) 2^{-m_j}}
\phi_{m_j, i_j}(x) \, dx
\\
&= 2^{m_j / 2} \cdot 2^{-m_j - 1} 2^{-m_j - 2} = 2^{-3 m_j / 2 - 3},
\end{align}
where the equality of integrals is a consequence of $\log_2(\numobs_j) \ge m_j + 2$ and
the fact that $\numobs_j$ is a power of $2$.

If $\mvec$ and $\mvec'$ both belong to $\MIndexSet$,
they satisfy $\sum_{i=1}^\usedim m_i = \sum_{i=1}^{\usedim} m'_i = \ell$.
Thus if $\mvec \ne \mvec'$, then because $\usedim \ge 2$
there is some $j$ for which $m_j < m'_j$,
and we obtain
\begin{equation}\label{equation:inner_haar}
\prod_{j=1}^\usedim \inner{\phitilde_{m_j, i_j}, h_{m'_j, i'_j}}
= \begin{cases}
0 & (\mvec, \ivec) \ne (\mvec', \ivec')
\\
2^{-3 \ell / 2 - 3 \usedim} & (\mvec, \ivec) = (\mvec', \ivec'), \
m_j + 2 \le \log_2 \numobs_j \forall j
\end{cases}
\end{equation}

Thus,
\begin{equation}
\inner*{g_{\etavec,\etavec'}, H_{\mvec', \ivec'}}
= \frac{\LASSOrad}{\sqrt{\abs{\MIndexSet_\ell}}}
(\eta_{\mvec', \ivec'} - \eta'_{\mvec', \ivec'})
2^{-3 \ell / 2 - 3 \usedim}
\prod_{j=1}^\usedim \Ind\{m'_j + 2 \le \log_2 \numobs_j\}.
\end{equation}
If we show that the above product of indicators is always equal to $1$,
then
plugging this into Bessel's inequality above yields
\begin{equation}
\Loss(f_{\etavec}, f_{\etavec'})
\ge \frac{4 \LASSOrad^2}{\abs{\MIndexSet_\ell}}
2^{-3 \ell - 6 \usedim} \Hamming(\etavec, \etavec')
\end{equation}
which would complete the proof of the desired claim~\eqref{eq:step3}.

It remains to show this last unverified claim about the product of indicators.
Equivalently, if $\mvec \in \MIndexSet_\ell$,
then we want to show
$m_j + 2 \le \log_2 n_j$ for all $j \in [\usedim]$, provided $\numobs$ is large enough.
Since $\numobs_j = \numobs^{1/\usedim}$ and since $\max_{j \in
  [\usedim]} m_j \le 2 \ell / \usedim$, it suffices to show
\begin{equation}\label{equation:indicator_bound}
\frac{2 \ell}{\usedim} + 2 \le \frac{1}{\usedim} \log \numobs
\end{equation}
Plugging in the definition~\eqref{equation:ell_choice} of $\ell$ yields
\begin{equation}
\frac{2}{3 \usedim \log 2}
\log(C_\usedim \numobs \LASSOrad^2 / \noisestd^2)
- \frac{2(\usedim  - 1)}{3 \usedim \log 2}
\log \log (C_\usedim \numobs \LASSOrad^2 / \noisestd^2)
+ 2
\le \frac{1}{\usedim} \log \numobs.
\end{equation}
For fixed $\usedim$ and $\noisestd^2 / \LASSOrad^2$, we have
\begin{align}
&\lim_{\numobs \to \infty}
\frac{\usedim}{\log \numobs}
\brackets*{
    \frac{2}{3 \usedim \log 2}
    \log(C_\usedim \numobs \LASSOrad^2 / \noisestd^2)
    - \frac{2(\usedim  - 1)}{3 \usedim \log 2}
    \log \log (C_\usedim \numobs \LASSOrad^2 / \noisestd^2)
    + 2
}
\\
&= \frac{2}{3 \log 2} < 1,
\end{align}
so there exists a constant $c_{\usedim, \noisestd^2/\LASSOrad^2}$ such that
the bound~\eqref{equation:indicator_bound} holds if
$\numobs \ge c_{\usedim, \noisestd^2/\LASSOrad^2}$.
 \end{proof}

 \subsection{Proof of \autoref{lemma:function_packing_two}}
\label{section:proof_function_packing_two}
We claim the functions $F_{\etavec, \mvec}$ and $F_{\etavec, \mvec'}$ are orthogonal
for distinct $\mvec, \mvec' \in \tilde{\MIndexSet}_\ell$.
We have
\begin{align}
&\int_0^1 \int_0^1
F_{\etavec, \mvec}(t_1, t_2)
F_{\etavec, \mvec'}(t_1, t_2)
\, dt_1 \, dt_2
\\
&= \sum_{\ivec \in \IIndexSet_{\mvec}}
\sum_{\ivec' \in \IIndexSet_{\mvec'}}
\etavec_{\mvec, \ivec}
\etavec_{\mvec', \ivec'}
\int_0^1 \phi'_{m_1, i_1}(t_1) \phi'_{m'_1, i'_1}(t_1) \, dt_1
\int_0^1 \phi'_{m_2, i_2}(t_2) \phi'_{m'_2, i'_2}(t_2) \, dt_2
\end{align}
Fix $(\mvec, \ivec)$ and $(\mvec', \ivec')$.
Since $\mvec$ and $\mvec'$ are distinct,
we must have $m_1 \ne m'_1$ (since $m_1 + m_2 = m'_1 + m'_2$).
Without loss of generality suppose $m_1 < m'_1$.
Then $\phi'_{m_1, i_1}$ is constant on the support of $\phi'_{m'_1, i'_1}$
for any $i_1 \in [2^{m_1}]$ and $i'_1 \in [2^{m'_1}]$, and thus
$\int_0^1 \phi'_{m_1, i_1}(t_1) \phi'_{m'_1, i'_1}(t_1) \, dt_1 = 0$.
The other case $m_1 > m'_1$ can be handled similarly.
In the end all terms in the above double sum are zero.

A similar argument shows that the integral
of the product of $F_{\etavec, \mvec^{(1)}}, \ldots, F_{\etavec, \mvec^{(k)}}$
for distinct $\mvec^{(1)}, \ldots, \mvec^{(k)}$ is zero, since
$m^{(1)}_1, \ldots, m^{(k)}_1$ are distinct in this case where $\usedim = 2$.

Note also that $1+F_{\etavec, \mvec} \ge 0$ for all $\mvec \in \tilde{\MIndexSet}_\ell$,
and thus $\partial^2 \ftilde_{\etavec} / (\partial x_1 \partial x_2) \ge 1$.
Consequently,
\begin{align}
\HKVar(\ftilde_{\etavec})
&= \norm*{\frac{\partial^2 \ftilde_{\etavec}}{\partial x_1 \partial x_2}}_1
=\norm*{
    \prod_{\mvec \in \tilde{\MIndexSet}_\ell} \parens*{
        1 + F_{\etavec, \mvec}(x_1, x_2)
    }
}_1
\\
&= \int_0^1 \int_0^1
\prod_{\mvec \in \tilde{\MIndexSet}_\ell} \parens*{
    1 + F_{\etavec, \mvec}(t_1, t_2)
}
\,dt_1 \, dt_2
\\
&= 1 + \sum_{\mvec \in \tilde{\MIndexSet}_\ell}
\int_0^1 \int_0^1 F_{\etavec, \mvec}(t_1, t_2) \, dt_1 \, dt_2
+ 0
\\
&= 1.
\end{align}
Combined with the fact that $\ftilde_{\etavec}$ is continuous, we have $\ftilde_{\etavec} \in \DFClassTwo$.

We now define $\gtilde_{\etavec}(x_1, x_2) \defn
\ftilde_{\etavec}(\floor{x_1 \numobs_1} / \numobs_1, \floor{x_2 \numobs_2} / \numobs_2)$.
This function agrees with $\ftilde_{\etavec}$
at the design points $(i/\numobs_1, j/\numobs_2)$,
and is piecewise constant on rectangles of the grid.
Thus for $\etavec \ne \etavec'$ we have
\begin{equation}
\Loss(\ftilde_{\etavec}, \ftilde_{\etavec'})
= \norm{\gtilde_{\etavec} - \gtilde_{\etavec'}}_{L^2}^2.
\end{equation}

Let us similarly define $\tilde{Q}_{\etavec}(x_1, x_2) \defn
\tilde{Q}_{\etavec}(\floor{x_1 \numobs_1} / \numobs_1, \floor{x_2 \numobs_2} / \numobs_2)$.
Let $h_{m,r}$ be as defined above~\eqref{equation:haar}.
We now show $\inner{\tilde{Q}_{\etavec}, h_{m_1, r_1} \otimes h_{m_2, r_2}}$
for all $\mvec \in \tilde{\MIndexSet}_\ell$ and $\vec{r} \in \IIndexSet_{\mvec}$.
Note
\begin{equation}
\tilde{Q}_{\etavec}(\xvec)
= \sum_{P \ge 2}
\sum
\int_0^{\floor{x_1 \numobs_1} / \numobs_1} \prod_{p=1}^P \phi'_{k_p, i_p}(t_1) \, dt_1
\int_0^{\floor{x_2 \numobs_2} / \numobs_2} \prod_{p=1}^P \phi'_{\ell - k_p, j_p}(t_2) \, dt_2
\end{equation}
where the inner sum above is over even integers $0 \le k_1 < k_2 < \cdots < k_P \le \ell$,
and all $1 \le i_p \le 2^{k_p}$, $1 \le j_p \le 2^{\ell - k_p}$,
$1 \le p \le P$.

Because $\phi'_{k_1, i_1}, \ldots, \phi'_{k_{P-1}, i_{P-1}}$
are constant on the support of $\phi'_{k_P, i_P}$ we have
for some constant $c_1$
\begin{equation}
\int_0^{\floor{x_1 \numobs_1} / \numobs_1}
\prod_{p=1}^P \phi'_{k_p, i_p}(t_1) \, dt_1
= c_1 \phi_{k_P, i_P}(\floor{x_1 \numobs_1} / \numobs_1)
\eqqcolon c_1 \phitilde_{1, k_P, i_P}(x_1),
\end{equation}
where the last equality is due to the earlier definition~\eqref{equation:phitilde_def}.
Similarly,
\begin{equation}
\int_0^{\floor{x_2 \numobs_2} / \numobs_2}
\prod_{p=1}^P \phi'_{\ell - k_p, j_p}(t_2) \, dt_2
= c_2 \phi_{\ell - k_1, j_1}(\floor{x_2 \numobs_2} / \numobs_2)
\eqqcolon c_2 \phitilde_{2, k_1, j_1}(x_2),
\end{equation}
Because $k_P + (\ell - k_1) > \ell = m_1 + m_2$,
we must have either $k_P > m_1$ or $\ell - k_1 > m_2$.
If $k_P > m_1$, then for any $1 \le r_1 \le 2^{m_1}$,
$h_{m_1, r_1}$ is constant on the support of $\phitilde_{k_P, i_P}$,
and thus
\begin{equation}
\int_0^1 \phitilde_{k_P, i_P}(x_1) h_{m_1, r_1}(x_1) \, dx_1
= c' \int_0^1 \phitilde_{k_P, i_P}(x_1) \, dx_1 = 0.
\end{equation}
Otherwise, if $\ell - k_1 > m_2$, then $\int_0^1 \phitilde_{\ell - k_1, j_1}(x_2) h_{m_2, r_2}(x_2) \, dx_2 = 0$
for all $1 \le r_2 \le 2^{m_2}$.
In either case we have
$\inner{\phitilde_{k_P, i_P} \otimes \phitilde_{\ell - k_1, j_1}, h_{m_1, r_1} \otimes h_{m_2, r_2}} = 0$,
and thus $\inner{\tilde{Q}_{\etavec}, h_{m_1, r_1} \otimes h_{m_2, r_2}} = 0$
for all $\mvec \in \tilde{\MIndexSet}_\ell$ and $\vec{r} \in \IIndexSet_{\mvec}$.
Therefore, using the earlier observation~\eqref{equation:inner_haar}
concerning inner products between $\phitilde_{m,i}$ and $h_{m',i'}$,
we obtain
\begin{align}
&\inner{\gtilde_{\etavec} - \gtilde_{\etavec'}, h_{m'_1, i'_1} \otimes h_{m'_2, i'_2}}
\\
&= \sum_{\mvec \in \tilde{\MIndexSet}_\ell}
\sum_{\ivec \in \IIndexSet_{\mvec}}
(\eta_{\mvec, \ivec} - \eta'_{\mvec, \ivec})
\inner{\phitilde_{m_1, i_1} \otimes \phitilde_{m_2, i_2},
h_{m'_1, i'_1} \otimes h_{m'_2, i'_2}}
\\
&= (\eta_{\mvec', \ivec'} - \eta'_{\mvec', \ivec'}) 2^{-3\ell/2 - 6}
\Ind\{m'_1 + 2 \le \log_2 \numobs_1,
m'_2 + 2 \le \log_2 \numobs_2\}.
\end{align}
As argued before~\eqref{equation:indicator_bound},
the event in the indicator function holds for sufficiently large $\numobs$,
so we may ignore it.
Applying Bessel's inequality yields
\begin{equation}
\Loss(\ftilde_{\etavec}, \ftilde_{\etavec'})
=
\norm{\gtilde_{\etavec} - \gtilde_{\etavec'}}_{L^2}^2
\ge \sum_{\mvec \in \tilde{\MIndexSet}_\ell} \sum_{\ivec \in \IIndexSet_{\mvec}}
(\eta_{\mvec, \ivec} - \eta'_{\mvec, \ivec})^2 2^{-3\ell - 12}
= \Hamming(\etavec, \etavec') 2^{-3\ell - 10}.
\end{equation}


\subsection{Proof of \autoref{lemma:tangent_cone}}
\label{section:proof_tangent_cone}
If $\sum_{j = 2}^{\numobs} \abs{\coefplaintilde_j} < R$, then
$\Altdesignmat \coef$ lies in the interior of $\LASSOBall(\LASSOrad)$,
so the tangent cone there is $\R^\numobs$.
Thus it remains to consider the case
$\sum_{j = 2}^{\numobs} \abs{\coefplaintilde_j} = R$.

Let $\TCone$ denote the right-hand side of the equality~\eqref{equation:tangent_cone}.
We first show $\TConeLASSO (\Altdesignmat \coeftilde) \subseteq \TCone$.
Since $\TCone$ is a closed convex cone, it suffices to show that
$\Altdesignmat \coef \defn \Altdesignmat (\coef' - \coeftilde)$
lies in $\TCone$ for any $\Altdesignmat \coef' \in \LASSOBall(\LASSOrad)$.
Indeed, using the fact that $\coefplain_j = \coefplain'_j$
whenever $\coefplaintilde_j = 0$, we have
\begin{align}
\sum_{\substack{j \ge 2 :\\ \coefplaintilde_j = 0}}
\abs{\coefplain_j}
+ \sum_{\substack{j \ge 2 :\\ \coefplaintilde_j \ne 0}}
\coefplain'_j \sign(\coefplaintilde_j)
&= \sum_{\substack{j \ge 2 :\\ \coefplain_j = 0}}
\abs{\coefplain'_j}
+ \sum_{\substack{j \ge 2 :\\ \coefplain_j \ne 0}}
\coefplain'_j \sign(\coefplain_j)
\\
&\le \sum_{j=2}^{\numobs} \abs{\coefplain'_j}
\le \LASSOrad
= \sum_{j=2}^{\numobs} \coefplaintilde_j \sign(\coefplaintilde_j).
\end{align}
Some rearrangement leads to $\TConeLASSO (\Altdesignmat \coeftilde) \subseteq \TCone$.

For the reverse inclusion, suppose $\Altdesignmat \coef \in \TCone$.
We claim that there exists some $c > 0$ such that
$\Altdesignmat \coeftilde + c \Altdesignmat \coef \in \LASSOBall(\LASSOrad)$.
Indeed, there exists a sufficiently small $c > 0$ such that
$\sign(\coefplaintilde_j + c \coefplain_j) = \sign(\coefplaintilde_j)$
for all $j$ satisfying $\coefplaintilde_j \ne 0$,
for which we have
\begin{align}
\sum_{j=2}^{\numobs} \abs{\coefplaintilde_j + c \coefplain_j}
&= c \sum_{\substack{j \ge 2 :\\ \coefplaintilde_j = 0}}
\abs{\coefplain_j}
+ \sum_{\substack{j \ge 2 :\\ \coefplaintilde_j = 0}}
(\coefplaintilde_j + c \coefplain_j) \sign(\coefplaintilde_j)
\\
&= \sum_{j=2}^{\numobs} \abs{\coefplaintilde_j}
+ c \underbrace{\parens*{
    \sum_{\substack{j \ge 2 :\\ \coefplaintilde_j = 0}}
    \abs{\coefplain_j}
    + \sum_{\substack{j \ge 2 :\\ \coefplaintilde_j \ne 0}}
    \coefplain_j \sign(\coefplaintilde_j)
}}_{\le 0}
\\
&\le \LASSOrad,
\end{align}
where the quantity in parentheses is nonpositive due to the definition of
$\Altdesignmat \coef \in \TCone$.
The above implies
$\Altdesignmat \coeftilde + c \Altdesignmat \coef \in \LASSOBall(\LASSOrad)$,
concluding the proof.

\subsection{Proof of \autoref{lemma:tcone_inequality}}
Using the fact that $\sum_{\ivec' : \ivec' \preceq \ivec} \coefplain_{\ivec'} = \alpha_{\ivec}$ we have
\begin{align}
\sign(\coefplaintilde_{\ivec^*}) (\alpha_{\ivec^u} - \alpha_{\ivec^\ell})
&= \sign(\coefplaintilde_{\ivec^*})
\parens*{
    \sum_{\ivec \in \LLup} \coefplain_{\ivec}
    - \sum_{\ivec \in \LLlw} \coefplain_{\ivec}
}
\\
&= \sign(\coefplaintilde_{\ivec^*})
\parens*{
    \sum_{\ivec \in \LLup \cap \LLlw^c} \coefplain_{\ivec}
    - \sum_{\ivec \in \LLup^c \cap \LLlw} \coefplain_{\ivec}
}
\\
&= \sign(\coefplaintilde_{\ivec^*})
\parens*{
    \coefplain_{\ivec^*}
    + \sum_{\ivec \in (\LLup \cap \LLlw^c) \setminus \{\ivec^*\}} \coefplain_{\ivec}
    - \sum_{\ivec \in \LLup^c \cap \LLlw} \coefplain_{\ivec}
}
\label{equation:subtract1}
\\
&\le \coefplain_{\ivec^*} \sign(\coefplaintilde_{\ivec^*})
+ \sum_{\ivec \notin \{\zerovec, \ivec^*\}} \abs{\coefplain_{\ivec}}
\label{equation:subtract2}
\\
&\le 0,
\end{align}
where the last inequality is due to the characterization~\eqref{equation:tangent_cone}
of $\TConeLASSO(\Altdesignmat \coeftilde)$.
The above chain of inequalities implies that
the difference between the expressions~\eqref{equation:subtract2}
and~\eqref{equation:subtract1} is bounded by
$- \sign(\coefplaintilde_{\ivec^*}) (\alpha_{\ivec^u} - \alpha_{\ivec^\ell})$.
This is precisely the desired inequality~\eqref{equation:tcone_inequality}.

\subsection{Statement and proof of a result connecting \texorpdfstring{$D(\param_Q)$}{D theta Q}
  and \texorpdfstring{$D \param$}{D theta}}
\begin{lemma}\label{lemma:edgecoef}
Consider $\SubRect$  as in \eqref{Qrect} for two indices $\qvec^\ell$
and $\qvec^u$ in $\IndexSet$ with $\qvec^\ell \preceq \qvec^u$. Recall
the notation \eqref{equation:jivec} and \eqref{equation:tii}. For
every $\param \in \R^n$, we have
\begin{equation}\label{equation:edgecoef_relation}
(D \param_{\SubRect})_{\ivec}
= \sum_{\ivec' \preceq \ivec}  t(\ivec', \ivec) (D \param)_{\ivec},
\qt{for every $\ivec \in \SubRect$}
\end{equation}
Furthermore, for
every $\ivec' \preceq \qvec^u$, there is a unique $\ivec \in
\SubRect$ such that $\ivec \succeq \ivec'$ and $\ivec'_{J(\ivec)} =
\ivec_{J(\ivec)}$; this $\ivec$ is given by $i_j \defn \max\{q^\ell_j,
i'_j\}, j = 1, \dots, d$.
\end{lemma}

\begin{proof}
For $\ivec \in \SubRect$, the identities~\eqref{equation:diff_def_Q}
and~\eqref{equation:sum_lower} together yield
\begin{align}
(D \param_{\SubRect})_{\ivec}
&= \sum_{\zvec \in \{0, 1\}^\usedim} \Ind\{\ivec - \zvec \succeq \qvec^\ell\}
(-1)^{z_1 + \cdots + z_\usedim}
\paramplain_{\ivec - \zvec}
\\
&= \sum_{\zvec \in \{0, 1\}^\usedim} \Ind\{\ivec - \zvec \succeq \qvec^\ell\}
(-1)^{z_1 + \cdots + z_\usedim}
\sum_{\ivec' \preceq \ivec - \zvec} (D \param)_{\ivec'}
\\
&= \sum_{\ivec' \preceq \ivec} (D \param)_{\ivec'}
  \sum_{\zvec \in \{0, 1\}^\usedim} \Ind\{\ivec - \zvec \succeq \qvec^\ell\}
(-1)^{z_1 + \cdots + z_\usedim}
\Ind\{\ivec' \preceq \ivec - \zvec\}.
\label{equation:edgecoef}
\end{align}
It then remains to show that the last inner sum equals $\edgecoef(\ivec', \ivec)$.
We have
\begin{align}
&\sum_{\zvec \in \{0, 1\}^\usedim} \Ind\{\ivec - \zvec \succeq \qvec^\ell\}
(-1)^{z_1 + \cdots + z_\usedim}
\Ind\{\ivec' \preceq \ivec - \zvec\}
\\
&= \prod_{j=1}^\usedim \sum_{z_j = 0}^1 (-1)^{z_j} \Ind\{i_j - z_j \ge q^\ell_j\}
\Ind\{i'_j \le i_j - z_j\}
\\
&= \prod_{j=1}^\usedim \parens*{
    \Ind\{q_j^\ell \le i_j; i'_j \le i_j\}
    - \Ind\{q_j^\ell \le i_j - 1; i'_j \le i_j - 1\}
}.
\end{align}
For $j \in J(\ivec)$ we have $i_j > q^\ell_j$, so the quantity in parentheses is
$\Ind\{i'_j \le i_j\} - \Ind\{i'_j \le i_j - 1\} = \Ind\{i'_j = i_j\}$.
For $j \notin J(\ivec)$ we have $i_j = q^\ell_j$, so the quantity in parentheses is $1$.
Thus the above product is $\Ind\{\ivec'_{J(\ivec)}  = \ivec_{J(\ivec)}\}$,
and we obtain \eqref{equation:edgecoef_relation}.

We now prove the second claim of the lemma.
Fix $\ivec' \preceq \qvec^u$.
We would like to produce $\ivec \in \SubRect$ such that
$\ivec \succeq \ivec'$ and
$i'_j = i_j$ for $j \in J(\ivec) = \{j : i_j > q^\ell_j\}$.
If $i'_j > q^\ell_j$, we have no choice but to let $i_j = i'_j$.
If $i'_j \le q^\ell_j$, we must let $i_j = q^\ell_j$ in order to have $\ivec \in \SubRect$.
This defines the unique $\ivec$ satisfying the conditions.
\end{proof}

\subsection{Proof of \autoref{lemma:global_to_local}}
Fix $\ivec' \preceq \ivec$ such that $\edgecoef(\ivec', \ivec) \ne 0$.
If $s(\ivec') \ne 0$, then $\mysignalt(\ivec) \defn \mysign(\ivec')$ so we have
\begin{equation}
\abs{\coefplain_{\ivec'}} - \mysignalt(\ivec) \edgecoef(\ivec', \ivec) \coefplain_{\ivec'}
= \abs{\coefplain_{\ivec'}} - \mysign(\ivec') \coefplain_{\ivec'}.
\end{equation}
Otherwise if $s(\ivec') = 0$, we have
\begin{equation}
\abs{\coefplain_{\ivec'}} - \mysignalt(\ivec) \edgecoef(\ivec', \ivec) \coefplain_{\ivec'}
\le 2 \abs{\coefplain_{\ivec'}}
= 2 (\abs{\coefplain_{\ivec'}} - \mysign(\ivec') \coefplain_{\ivec'}).
\end{equation}
Using these two observations along with the relation~\eqref{equation:edgecoef_relation},
we have
\begin{align}
&\sum_{\ivec \in \SubRect  \setminus \{\qvec^\ell\} : \ivec \nsucc \qvec^\ell}
(\abs{(D \alphavec_{\SubRect})_{\ivec}} - \mysignalt(\ivec) (D \alphavec_{\SubRect})_{\ivec})
\\
&= \sum_{\ivec \in \SubRect : \ivec \succ \qvec^\ell}
\braces*{
    \abs*{\sum_{\ivec' \preceq \ivec} \edgecoef(\ivec', \ivec) \coefplain_{\ivec'}}
    - \mysignalt(\ivec) \sum_{\ivec' \preceq \ivec} \edgecoef(\ivec', \ivec) \coefplain_{\ivec'}
}
\\
&= \sum_{\ivec \in \SubRect  \setminus \{\qvec^\ell\} : \ivec \nsucc \qvec^\ell}
\sum_{\ivec' \preceq \ivec} \edgecoef(\ivec', \ivec)
\braces*{
    \abs{\coefplain_{\ivec'}}
    - \mysignalt(\ivec) \edgecoef(\ivec', \ivec) \coefplain_{\ivec'}
}
\\
&\le 2 \sum_{\ivec \in \SubRect  \setminus \{\qvec^\ell\} : \ivec \nsucc \qvec^\ell}
\sum_{\ivec' \preceq \ivec}
\edgecoef(\ivec', \ivec)
\braces*{
    \abs{\coefplain_{\ivec'}}
    - \mysign(\ivec') \coefplain_{\ivec'}
}
\\
&= 2 \sum_{\ivec' \preceq \qvec^u}
\braces*{
    \abs{\coefplain_{\ivec'}}
    - \mysign(\ivec') \coefplain_{\ivec'}
}
\sum_{\ivec \in \SubRect}
\edgecoef(\ivec', \ivec)
\Ind\{\ivec' \preceq \ivec, \ivec \nsucc \qvec^\ell, \ivec \ne \qvec^\ell\}
\\
&\le 2 \sum_{\ivec' \preceq \qvec^u : \ivec' \nsucc \qvec^\ell, \ivec' \notin \{\zerovec, \ivec^*\}}
\braces*{
    \abs{\coefplain_{\ivec'}}
    - \mysign(\ivec') \coefplain_{\ivec'}
}.
\end{align}
In the last step we noted that for a given $\ivec' \preceq \qvec^u$,
there is a unique $\ivec$ such that $t(\ivec', \ivec)$ is nonzero
(second part of \autoref{lemma:edgecoef}),
so the inner sum only has one nonzero addend.
Then we used assumption~\eqref{enumerate:cond_no_istar} to note that
$\Ind\{\ivec' \preceq \ivec, \ivec \nsucc \qvec^\ell, \ivec \ne \qvec^\ell\}
\le \Ind\{\ivec' \nsucc \qvec^\ell, \ivec' \notin \{\zerovec, \ivec^*\}\}$
for any $\ivec \in \SubRect$ such that $\edgecoef(\ivec', \ivec) = 1$.

Finally, for $\ivec \in \SubRect$ such that $\ivec \succ \qvec^\ell$, let
$\mysignalt(\ivec) \defn \mysign(\ivec)$. Combining the above work
with the fact that \eqref{equation:edgecoef_relation} implies that $(D
\alphavec_Q)_{\ivec} = (D \alphavec)_{\ivec}$ for $\ivec \succ
\qvec^{\ell}$, we obtain
\begin{align}
&\sum_{\ivec \in \SubRect \setminus \{\qvec^\ell\}}
(\abs{(D \alphavec_{\SubRect})_{\ivec}} - \mysignalt(\ivec) (D \alphavec_{\SubRect})_{\ivec})
\\
&\le
\sum_{\ivec \in \SubRect : \ivec \succ \qvec^\ell}
\braces*{
    \abs{\coefplain_{\ivec}}
    - \mysign(\ivec) \coefplain_{\ivec}
}
+
2 \sum_{\ivec \preceq \qvec^u : \ivec \nsucc \qvec^\ell, \ivec \notin \{\zerovec, \ivec^*\}}
\braces*{
    \abs{\coefplain_{\ivec}}
    - \mysign(\ivec) \coefplain_{\ivec}
}
\\
&\le 2 \sum_{\ivec \notin \{\zerovec, \ivec^*\}}
\braces*{
    \abs{\coefplain_{\ivec'}}
    - \mysign(\ivec) \coefplain_{\ivec}
}.
\end{align}
The last inequality follows by noting that the two sums indexed by $\ivec$
are over disjoint sets.
Finally, the right-hand side can be bounded by $2 \delta$
due to~\eqref{equation:tcone_inequality_delta}.

\subsection{Proof of \autoref{lemma:nearly_cm}}
\label{section:proof_nearly_cm}
Without loss of generality we assume $t = 1$
(the general result can then be obtained by scaling and replacing $\delta$ by $\delta / t$).

Let $\coef$ be such that $\param = \Altdesignmat \coef$.
Let $\pospart(\param) \defn \Altdesignmat \coef^+$
and $\negpart(\param) \defn \Altdesignmat \coef^-$,
where $\coefplain_1^+ = \coefplain_1$ and $\coefplain_i^+ \defn \max\{\coefplain_i, 0\}$
for $i \ge 2$,
and where $\coef^- \defn \coef^+ - \coef$.
Then $\param = \pospart(\param) - \negpart(\param)$,
and both $\pospart(\param)$ and $\negpart(\param)$
are entirely monotone.

We have the following two equalities.
\begin{align}
\paramplain_\numobs - \paramplain_1
&= [(\pospart(\param))_n - (\pospart(\param))_1]
- (\negpart(\param))_n, 
\\
\HKVar(\param)
= \sum_{i = 2}^\numobs \abs{\coefplain_i}
&= [(\pospart(\param))_n - (\pospart(\param))_1]
+ (\negpart(\param))_n. 
\end{align}
Combining these two equalities shows that the constraint
$\HKVar(\param) \le \paramplain_n - \paramplain_1 + \delta$
is equivalent to
\begin{equation}
\sum_{i \ge 2}^n \coefplain^-_i
= (\negpart(\param))_n 
\le \frac{\delta}{2}.
\end{equation}
Then
\begin{equation}
\norm{\negpart(\param)}^2 \le \numobs (\negpart(\param))_\numobs^2
\le \frac{\delta^2}{4} \numobs.
\end{equation}
By the triangle inequality,
\begin{equation}
\norm{\pospart(\param)}
\le \norm{\param} + \norm{\negpart(\param)}
\le 1 + \frac{\delta}{2} \sqrt{\numobs}.
\end{equation}
Thus,
\begin{align}
&\E \sup_{\substack{
    \param : \norm{\param} \le 1, \\
    \HKVar(\param) \le \paramplain_n - \paramplain_1 + \delta
}}
\inner{\param, \noise}
&\le
\E \sup_{
    \substack{\pospart \in \NNLSSet : \\ \norm{\param} \le 1 + \delta \sqrt{\numobs} / 2
}}
\inner{\pospart, \noise}
+
\E \sup_{
    \substack{\negpart \in \NNLSSet : \\\norm{\negpart} \le \delta \sqrt{\numobs} / 2
}}
\inner{- \negpart, \noise}
\end{align}
Since $\NNLSSetPlain$ is a cone and since $\noise \overset{d}{=} - \noise$,
the right-hand side can be written as
\begin{equation}
\noisestd (1 + \delta \sqrt{\numobs}) \
\E \sup_{
    \param \in \NNLSSet : \norm{\param} \le 1
}
\inner{\param, \zvec},
\end{equation}
where $\zvec \sim \Normal(\zerovec, \Imat_\numobs)$.
From the earlier Gaussian width bound~\eqref{equation:gw_bound}
(with $\paramstar = 0$, $V^* = 0$, $\rad=1$,
and using the bounds $\abs{\KIndexSet} \le C (\log n)^\usedim$
and $\log(2 e \sqrt{\abs{\KIndexSet}}) \le C_\usedim \log (e \log n))$)
we have
\begin{equation}
\E \sup_{
    \param \in \NNLSSet : \norm{\param} \le 1
}
\inner{\param, \zvec}
\le C_\usedim (\log (e \numobs))^{\frac{3 \usedim}{4}} (\log (e \log (e n)))^{\frac{2 \usedim - 1}{4}}.
\end{equation}

\subsection{Proof of \autoref{lemma:gw_mix}}
\label{section:proof_gw_mix}
Because $\coefplain_{\ivec} = 0$ for all $\ivec \succ \zerovec$,
we have the following equality for all
$\ivec \in \{0,\ldots, \numobs_1-1\} \times \{0, \ldots, \numobs_2-1\}$.
\begin{equation}\label{equation:theta_to_theta}
\paramplain_{\ivec} = \sum_{\ivec' \preceq \ivec} \coefplain_{\ivec'}
= \sum_{i'_1 = 0}^{i_1} \coefplain_{i'_1, 0}
+ \sum_{i'_2 = 0}^{i_2} \coefplain_{0, i'_2}
- \coefplain_{0, 0}
= \paramplain_{i_1, 0} + \paramplain_{0, i_2} - \paramplain_{0, 0},
\qquad \forall \ivec.
\end{equation}
Let
$\paramplainbar_1 \defn \frac{1}{\numobs_1} \sum_{i_1 = 1}^{\numobs_1} \paramplain_{i_1, 1}$
and
$\paramplainbar_2 \defn \frac{1}{\numobs_2} \sum_{i_2 = 1}^{\numobs_2} \paramplain_{1, i_2}$

Note that the identity \eqref{equation:theta_to_theta} implies
\begin{align}
1 & \ge \norm{\param}^2
\\
&= \sum_{i_1=0}^{\numobs_1-1}
\sum_{i_2 = 0}^{\numobs_2-1}
[(\paramplain_{i_1, 0} - \paramplainbar_1)
+ (\paramplain_{0, i_2} - \paramplainbar_2)
- (\paramplain_{0, 0}- \paramplainbar_1 - \paramplainbar_2)]^2
\\
&= \numobs_2 \sum_{i_1 = 0}^{\numobs_1 - 1} (\paramplain_{i_1, 0} - \paramplainbar_1)^2
+ \numobs_1 \sum_{i_2 = 0}^{\numobs_2 0 1} (\paramplain_{0, i_2} - \paramplainbar_2)^2
+ \numobs_1 \numobs_2 (\paramplain_{0, 0}- \paramplainbar_1 - \paramplainbar_2)^2,
\end{align}
where the cross terms vanish in the last step
due to $\sum_{i_1 = 0}^{\numobs_1 - 1} (\paramplain_{i_1, 0} - \paramplainbar_1) = 0$
and $\sum_{i_2 = 0}^{\numobs_2 - 1} (\paramplain_{0, i_2} - \paramplainbar_2) = 0$.
Thus the vectors
$\sqrt{\numobs_2} (\paramplain_{i_1, 0} - \paramplainbar_1)_{i_1 = 0}^{\numobs_1 - 1}$,
$\sqrt{\numobs_2} (\paramplain_{0, i_2} - \paramplainbar_1)_{i_2 = 0}^{\numobs_2 - 1}$,
and $\sqrt{\numobs_1 \numobs_2} (\paramplain_{0, 0}- \paramplainbar_1 - \paramplainbar_2)$
each have norm bounded by $1$.

Let us view $Z$ as a $\numobs_1 \times \numobs_2$ matrix,
and define $Z_{\cdot, i_2} \defn \sum_{i_1 = 0}^{\numobs_1 - 1} Z_{i_1, i_2}$,
$Z_{i_1, \cdot} \defn \sum_{i_2 = 0}^{\numobs_2 - 1} Z_{i_1, i_2}$,
and $Z_{\cdot, \cdot} \defn \sum_{i_1 = 0}^{\numobs_1 - 1}
\sum_{i_2 = 0}^{\numobs_2 - 1} Z_{i_1, i_2}$.
Then, using the identity~\eqref{equation:theta_to_theta}
we can decompose the inner product as
\begin{align}
\inner{Z, \param}
&= \sum_{\ivec} Z_{\ivec} \paramplain_{\ivec}
= \sum_{i_1=0}^{\numobs_1-1}
\sum_{i_2=0}^{\numobs_2-1}
Z_{i_1, i_2}
(\paramplain_{i_1, 0} + \paramplain_{0, i_2} - \paramplain_{0, 0})
\\
&= \sum_{i_1=0}^{\numobs_1-1} Z_{i_1, \cdot} \paramplain_{i_1, 0}
+ \sum_{i_2=0}^{\numobs_2-1} Z_{\cdot, i_2} \paramplain_{0, i_2}
- Z_{\cdot, \cdot} \paramplain_{1, 1}
\\
&= \sum_{i_1=0}^{\numobs_1-1} Z_{i_1, \cdot} (\paramplain_{i_1, 0} - \paramplainbar_1)
+ \sum_{i_2=0}^{\numobs_2-1} Z_{\cdot, i_2} (\paramplain_{0, i_2} - \paramplainbar_2)
- Z_{\cdot, \cdot} (\paramplain_{0, 0} - \paramplainbar_1 - \paramplainbar_2)
\\
&= \sum_{i_1=0}^{\numobs_1-1}
\frac{Z_{i_1, \cdot}}{\sqrt{\numobs_2}} \sqrt{\numobs_2} (\paramplain_{i_1, 0} - \paramplainbar_1)
+ \sum_{i_2=0}^{\numobs_2-1}
\frac{Z_{\cdot, i_2}}{\sqrt{\numobs_1}} \sqrt{\numobs_1} (\paramplain_{0, i_2} - \paramplainbar_2)
\\
&\qquad - \frac{Z_{\cdot, \cdot}}{\sqrt{\numobs_1 \numobs_2}} \sqrt{\numobs_1 \numobs_2} (\paramplain_{0, 0} - \paramplainbar_1 - \paramplainbar_2).
\end{align}
Note that $(Z_{\cdot, i_2} / \sqrt{\numobs_1})_{i_1=1}^{\numobs_1}$, $(Z_{i_1, \cdot} / \sqrt{\numobs_2})_{i_2=1}^{\numobs_2}$,
and $Z_{\cdot, \cdot}/\sqrt{\numobs_1, \numobs_2}$ are each standard Gaussian vectors.

Finally, note that because $\coefplain_{\ivec} = 0$ for $\ivec \succ \onevec$,
the HK variation condition on $\param$ can be written as
\begin{equation}
\sum_{i_1 = 1}^{\numobs_1-1} (\abs{\coefplain_{i_1, 0}} - s_1 \coefplain_{i_1, 0})
+ \sum_{i_2 = 1}^{\numobs_2-1} (\abs{\coefplain_{0, i_2}} - s_2 \coefplain_{0, i_2})
\le \delta,
\end{equation}
and thus each of these two sums is bounded by $\delta$

Thus, we can bound the expectation in the lemma by
\begin{equation}
\E \sup_{\substack{
    \paramtilde \in \R^{\numobs_1} : \norm{\paramtilde} \le 1
    \\
    \HKVar(\paramtilde) \le s_1 (\paramplaintilde_{\numobs_1} - \paramplaintilde_1) + \delta
}}
\inner{Z_{\numobs_1}, \paramtilde}
+ \E \sup_{\substack{
    \paramtilde \in \R^{\numobs_2} : \norm{\paramtilde} \le 1
    \\
    \HKVar(\paramtilde) \le s_2 (\paramplaintilde_{\numobs_2} - \paramplaintilde_1) + \delta
}}
\inner{Z_{\numobs_2}, \paramtilde}
+ \E \sup_{\paramplaintilde \in \R : \abs{\paramplaintilde} \le 1}
Z_1 \paramplaintilde,
\end{equation}
where $Z_{\numobs_1}$, $Z_{\numobs_2}$, and $Z_1$ are standard Gaussian vectors
of the appropriate dimension.
The third term is readily computed to be $\E \abs{Z_1} = \sqrt{2 / \pi}$.

We now focus on the first term; the second term can be bounded analogously.
If $s_1 \in \{-1, 1\}$, then Lemma C.8 of \citet{guntuboyina2017spatial}
implies a bound of
\begin{equation}
c (1 + \delta \sqrt{\numobs_1}) \sqrt{\log(e \numobs_1)}
\end{equation}
Otherwise if $s_1 = 0$, then Lemma B.1 of the same paper \cite{guntuboyina2017spatial}
yields a bound of
\begin{equation}
c (\delta \sqrt{\numobs_1})^{\frac{1}{2}} + c \sqrt{\log(e \numobs_1)}.
\end{equation}

Handling the second term in the same fashion concludes the proof.






\section*{Acknowledgements}
We are extremely thankful to the Associate Editor and the two anonymous referees for several insightful comments that led to many improvements in the paper. We are also thankful to Frank Fuchang Gao for clarifying some
technical arguments in the paper~\cite{blei2007metric},  to Ming
Yuan for informing us of the paper~\cite{lin2000tensor} and to Jake
Soloff for helpful comments.

\bibliographystyle{chicago}
\bibliography{AG}

\def\noopsort#1{}
\begin{thebibliography}{}

\bibitem[\protect\citeauthoryear{Aistleitner and Dick}{Aistleitner and
  Dick}{2015}]{aistleitner2014functions}
Aistleitner, C. and J.~Dick (2015).
\newblock Functions of bounded variation, signed measures, and a general
  {K}oksma-{H}lawka inequality.
\newblock {\em Acta Arith.\/}~{\em 167\/}(2), 143--171.

\bibitem[\protect\citeauthoryear{Amelunxen, Lotz, McCoy, and Tropp}{Amelunxen
  et~al.}{2014}]{amelunxen2014living}
Amelunxen, D., M.~Lotz, M.~B. McCoy, and J.~A. Tropp (2014).
\newblock Living on the edge: phase transitions in convex programs with random
  data.
\newblock {\em Inf. Inference\/}~{\em 3\/}(3), 224--294.

\bibitem[\protect\citeauthoryear{Assouad}{Assouad}{1983}]{Assouad}
Assouad, P. (1983).
\newblock Deux remarques sur l'estimation.
\newblock {\em Comptes rendus des s{\'e}ances de l'Acad{\'e}mie des sciences.
  S{\'e}rie 1, Math{\'e}matique\/}~{\em 296}, 1021--1024.

\bibitem[\protect\citeauthoryear{Ayer, Brunk, Ewing, Reid, and Silverman}{Ayer
  et~al.}{1955}]{AyerEtAl55}
Ayer, M., H.~D. Brunk, G.~M. Ewing, W.~T. Reid, and E.~Silverman (1955).
\newblock An empirical distribution function for sampling with incomplete
  information.
\newblock {\em Ann. Math. Statist.\/}~{\em 26}, 641--647.

\bibitem[\protect\citeauthoryear{Barlow, Bartholomew, Bremner, and
  Brunk}{Barlow et~al.}{1972}]{BBBB72}
Barlow, R.~E., D.~J. Bartholomew, J.~M. Bremner, and H.~D. Brunk (1972).
\newblock {\em Statistical inference under order restrictions. {T}he theory and
  application of isotonic regression}.
\newblock John Wiley \& Sons, London-New York-Sydney.
\newblock Wiley Series in Probability and Mathematical Statistics.

\bibitem[\protect\citeauthoryear{Barron}{Barron}{1993}]{barron1993universal}
Barron, A.~R. (1993).
\newblock Universal approximation bounds for superpositions of a sigmoidal
  function.
\newblock {\em IEEE Transactions on Information theory\/}~{\em 39\/}(3),
  930--945.

\bibitem[\protect\citeauthoryear{Bellec}{Bellec}{2017}]{bellec2017optimistic}
Bellec, P.~C. (2017).
\newblock Optimistic lower bounds for convex regularized least-squares.
\newblock {\em arXiv preprint arXiv:1703.01332\/}.

\bibitem[\protect\citeauthoryear{Bellec}{Bellec}{2018}]{bellec2018sharp}
Bellec, P.~C. (2018).
\newblock Sharp oracle inequalities for least squares estimators in shape
  restricted regression.
\newblock {\em Ann. Statist.\/}~{\em 46\/}(2), 745--780.

\bibitem[\protect\citeauthoryear{Benkeser and Van Der~Laan}{Benkeser and Van
  Der~Laan}{2016}]{benkeser2016highly}
Benkeser, D. and M.~Van Der~Laan (2016).
\newblock The highly adaptive lasso estimator.
\newblock In {\em 2016 IEEE international conference on data science and
  advanced analytics (DSAA)}, pp.\  689--696. IEEE.

\bibitem[\protect\citeauthoryear{Blei, Gao, and Li}{Blei
  et~al.}{2007}]{blei2007metric}
Blei, R., F.~Gao, and W.~V. Li (2007).
\newblock Metric entropy of high dimensional distributions.
\newblock {\em Proc. Amer. Math. Soc.\/}~{\em 135\/}(12), 4009--4018.

\bibitem[\protect\citeauthoryear{Brunk}{Brunk}{1955}]{Brunk55}
Brunk, H.~D. (1955).
\newblock Maximum likelihood estimates of monotone parameters.
\newblock {\em Ann. Math. Statist.\/}~{\em 26}, 607--616.

\bibitem[\protect\citeauthoryear{Bungartz and Griebel}{Bungartz and
  Griebel}{2004}]{bungartz2004sparse}
Bungartz, H.-J. and M.~Griebel (2004).
\newblock Sparse grids.
\newblock {\em Acta numerica\/}~{\em 13}, 147--269.

\bibitem[\protect\citeauthoryear{Chambolle, Caselles, Cremers, Novaga, and
  Pock}{Chambolle et~al.}{2010}]{chambolle2010introduction}
Chambolle, A., V.~Caselles, D.~Cremers, M.~Novaga, and T.~Pock (2010).
\newblock An introduction to total variation for image analysis.
\newblock {\em Theoretical foundations and numerical methods for sparse
  recovery\/}~{\em 9\/}(263-340), 227.

\bibitem[\protect\citeauthoryear{Chatterjee}{Chatterjee}{2014}]{chatterjee2014new}
Chatterjee, S. (2014).
\newblock A new perspective on least squares under convex constraint.
\newblock {\em Ann. Statist.\/}~{\em 42\/}(6), 2340--2381.

\bibitem[\protect\citeauthoryear{Chatterjee and Goswami}{Chatterjee and
  Goswami}{2019}]{chatterjee2019new}
Chatterjee, S. and S.~Goswami (2019).
\newblock New risk bounds for 2d total variation denoising.
\newblock {\em arXiv preprint arXiv:1902.01215\/}.

\bibitem[\protect\citeauthoryear{Chatterjee, Guntuboyina, and Sen}{Chatterjee
  et~al.}{2015}]{chatterjee2015risk}
Chatterjee, S., A.~Guntuboyina, and B.~Sen (2015).
\newblock On risk bounds in isotonic and other shape restricted regression
  problems.
\newblock {\em Ann. Statist.\/}~{\em 43\/}(4), 1774--1800.

\bibitem[\protect\citeauthoryear{Chatterjee, Guntuboyina, and Sen}{Chatterjee
  et~al.}{2018}]{chatterjee2015matrix}
Chatterjee, S., A.~Guntuboyina, and B.~Sen (2018).
\newblock On matrix estimation under monotonicity constraints.
\newblock {\em Bernoulli\/}~{\em 24\/}(2), 1072--1100.

\bibitem[\protect\citeauthoryear{Chen, Guntuboyina, and Zhang}{Chen
  et~al.}{2017}]{chen2017note}
Chen, X., A.~Guntuboyina, and Y.~Zhang (2017).
\newblock A note on the approximate admissibility of regularized estimators in
  the gaussian sequence model.
\newblock {\em arXiv preprint arXiv:1703.00542\/}.

\bibitem[\protect\citeauthoryear{Chkifa, Dexter, Tran, and Webster}{Chkifa
  et~al.}{2018}]{chkifa2018polynomial}
Chkifa, A., N.~Dexter, H.~Tran, and C.~Webster (2018).
\newblock Polynomial approximation via compressed sensing of high-dimensional
  functions on lower sets.
\newblock {\em Mathematics of Computation\/}~{\em 87\/}(311), 1415--1450.

\bibitem[\protect\citeauthoryear{Condat}{Condat}{2013}]{condat2013direct}
Condat, L. (2013).
\newblock A direct algorithm for 1-d total variation denoising.
\newblock {\em IEEE Signal Process. Lett.\/}~{\em 20\/}(11), 1054--1057.

\bibitem[\protect\citeauthoryear{Condat}{Condat}{2017}]{condat2017discrete}
Condat, L. (2017).
\newblock Discrete total variation: New definition and minimization.
\newblock {\em SIAM Journal on Imaging Sciences\/}~{\em 10\/}(3), 1258--1290.

\bibitem[\protect\citeauthoryear{Dalalyan, Hebiri, and Lederer}{Dalalyan
  et~al.}{2017}]{dalalyan2017tvd}
Dalalyan, A., M.~Hebiri, and J.~Lederer (2017).
\newblock On the prediction performance of the lasso.
\newblock {\em Bernoulli\/}~{\em 23\/}(1), 552--581.

\bibitem[\protect\citeauthoryear{Deng and Zhang}{Deng and
  Zhang}{2018}]{deng2018isotonic}
Deng, H. and C.-H. Zhang (2018).
\newblock Isotonic regression in multi-dimensional spaces and graphs.
\newblock {\em arXiv preprint arXiv:1812.08944\/}.

\bibitem[\protect\citeauthoryear{Donoho}{Donoho}{2000}]{donoho2000high}
Donoho, D.~L. (2000).
\newblock High-dimensional data analysis: The curses and blessings of
  dimensionality.
\newblock {\em AMS math challenges lecture\/}~{\em 1\/}(32), 375.

\bibitem[\protect\citeauthoryear{Donoho and Johnstone}{Donoho and
  Johnstone}{1998}]{donoho98minimaxwavelet}
Donoho, D.~L. and I.~M. Johnstone (1998).
\newblock Minimax estimation via wavelet shrinkage.
\newblock {\em Ann. Statist.\/}~{\em 26\/}(3), 879--921.

\bibitem[\protect\citeauthoryear{Dudley}{Dudley}{1967}]{Dudley67}
Dudley, R.~M. (1967).
\newblock The sizes of compact subsets of {H}ilbert space and continuity of
  {G}aussian processes.
\newblock {\em J. Functional Analysis\/}~{\em 1}, 290--330.

\bibitem[\protect\citeauthoryear{Fan and Guan}{Fan and
  Guan}{2018}]{fan2018approximate}
Fan, Z. and L.~Guan (2018).
\newblock Approximate {$\ell_0$}-penalized estimation of piecewise-constant
  signals on graphs.
\newblock {\em Ann. Statist.\/}~{\em 46\/}(6B), 3217--3245.

\bibitem[\protect\citeauthoryear{Feller}{Feller}{2015}]{feller2015completely}
Feller, W. (2015).
\newblock Completely monotone functions and sequences.
\newblock In {\em Selected Papers I}, pp.\  497--510. Springer.

\bibitem[\protect\citeauthoryear{Gao}{Gao}{2013}]{gao2013bracketing}
Gao, F. (2013).
\newblock Bracketing entropy of high dimensional distributions.
\newblock In {\em High dimensional probability {VI}}, Volume~66 of {\em Progr.
  Probab.}, pp.\  3--17. Birkh\"{a}user/Springer, Basel.

\bibitem[\protect\citeauthoryear{Gao, Li, and Wellner}{Gao
  et~al.}{2010}]{gao2010many}
Gao, F., W.~V. Li, and J.~A. Wellner (2010).
\newblock How many {L}aplace transforms of probability measures are there?
\newblock {\em Proc. Amer. Math. Soc.\/}~{\em 138\/}(12), 4331--4344.

\bibitem[\protect\citeauthoryear{Gill, Laan, and Wellner}{Gill
  et~al.}{1995}]{gill1995inefficient}
Gill, R.~D., M.~J. Laan, and J.~A. Wellner (1995).
\newblock Inefficient estimators of the bivariate survival function for three
  models.
\newblock In {\em Annales de l'IHP Probabilit{\'e}s et statistiques},
  Volume~31, pp.\  545--597.

\bibitem[\protect\citeauthoryear{Goddard}{Goddard}{1945}]{Goddard45}
Goddard, L.~S. (1945).
\newblock The accumulation of chance effects and the {G}aussian frequency
  distribution.
\newblock {\em Philos. Mag. (7)\/}~{\em 36}, 428--433.

\bibitem[\protect\citeauthoryear{Groeneboom}{Groeneboom}{2013}]{groeneboom2013bivariate}
Groeneboom, P. (2013).
\newblock The bivariate current status model.
\newblock {\em Electronic Journal of Statistics\/}~{\em 7}, 1783--1805.

\bibitem[\protect\citeauthoryear{Groeneboom and Jongbloed}{Groeneboom and
  Jongbloed}{2014}]{groeneboom2014nonparametric}
Groeneboom, P. and G.~Jongbloed (2014).
\newblock {\em Nonparametric estimation under shape constraints}, Volume~38 of
  {\em Cambridge Series in Statistical and Probabilistic Mathematics}.
\newblock Cambridge University Press, New York.
\newblock Estimators, algorithms and asymptotics.

\bibitem[\protect\citeauthoryear{Groeneboom, Ketelaars, et~al.}{Groeneboom
  et~al.}{2011}]{groeneboom2011estimators}
Groeneboom, P., T.~Ketelaars, et~al. (2011).
\newblock Estimators for the interval censoring problem.
\newblock {\em Electronic Journal of Statistics\/}~{\em 5}, 1797--1845.

\bibitem[\protect\citeauthoryear{Guntuboyina}{Guntuboyina}{2011}]{GuntuFdiv}
Guntuboyina, A. (2011).
\newblock Lower bounds for the minimax risk using $f$ divergences, and
  applications.
\newblock {\em IEEE Transactions on Information Theory\/}~{\em 57}, 2386--2399.

\bibitem[\protect\citeauthoryear{Guntuboyina, Lieu, Chatterjee, and
  Sen}{Guntuboyina et~al.}{2017}]{guntuboyina2017spatial}
Guntuboyina, A., D.~Lieu, S.~Chatterjee, and B.~Sen (2017).
\newblock Adaptive risk bounds in univariate total variation denoising and
  trend filtering.
\newblock {\em Ann. Statist\/}.
\newblock \textit{(to appear)}; available at
  \url{https://arxiv.org/abs/1702.05113}.

\bibitem[\protect\citeauthoryear{Guntuboyina and Sen}{Guntuboyina and
  Sen}{2018}]{guntuboyina2017nonparametric}
Guntuboyina, A. and B.~Sen (2018).
\newblock Nonparametric {S}hape-{R}estricted {R}egression.
\newblock {\em Statist. Sci.\/}~{\em 33\/}(4), 568--594.

\bibitem[\protect\citeauthoryear{Guo and Wang}{Guo and
  Wang}{2006}]{guo2006quasi}
Guo, D. and X.~Wang (2006).
\newblock Quasi-monte carlo filtering in nonlinear dynamic systems.
\newblock {\em IEEE transactions on signal processing\/}~{\em 54\/}(6),
  2087--2098.

\bibitem[\protect\citeauthoryear{Han, Wang, Chatterjee, and Samworth}{Han
  et~al.}{2019}]{han2017isotonic}
Han, Q., T.~Wang, S.~Chatterjee, and R.~J. Samworth (2019).
\newblock Isotonic regression in general dimensions.
\newblock {\em Ann. Statist.\/}~{\em 47\/}(5), 2440--2471.

\bibitem[\protect\citeauthoryear{Hjort and Pollard}{Hjort and
  Pollard}{1993}]{hjort2011asymptotics}
Hjort, N.~L. and D.~Pollard (1993).
\newblock Asymptotics for minimisers of convex processes.
\newblock Technical report.
\newblock available at arXiv preprint arXiv:1107.3806.

\bibitem[\protect\citeauthoryear{Hobson}{Hobson}{1950}]{hobson1950theory}
Hobson, E.~W. (1950).
\newblock {\em The theory of functions of a real variable and the theory of
  Fourier's series}, Volume~1.
\newblock CUP Archive.

\bibitem[\protect\citeauthoryear{H{\"u}tter and Rigollet}{H{\"u}tter and
  Rigollet}{2016}]{hutter2016optimal}
H{\"u}tter, J.-C. and P.~Rigollet (2016).
\newblock Optimal rates for total variation denoising.
\newblock In {\em Conference on Learning Theory}, pp.\  1115--1146.

\bibitem[\protect\citeauthoryear{Kim, Koh, Boyd, and Gorinevsky}{Kim
  et~al.}{2009}]{kim2009ell_1}
Kim, S.-J., K.~Koh, S.~Boyd, and D.~Gorinevsky (2009).
\newblock $\ell_1$ trend filtering.
\newblock {\em SIAM review\/}~{\em 51\/}(2), 339--360.

\bibitem[\protect\citeauthoryear{Leonov}{Leonov}{1996}]{leonov1996total}
Leonov, A.~S. (1996).
\newblock On the total variation for functions of several variables and a
  multidimensional analog of {H}elly's selection principle.
\newblock {\em Mathematical Notes\/}~{\em 63\/}(1), 61--71.

\bibitem[\protect\citeauthoryear{Lin, Sharpnack, Rinaldo, and Tibshirani}{Lin
  et~al.}{2017}]{lin2017sharp}
Lin, K., J.~L. Sharpnack, A.~Rinaldo, and R.~J. Tibshirani (2017).
\newblock A sharp error analysis for the fused lasso, with application to
  approximate changepoint screening.
\newblock In {\em Advances in Neural Information Processing Systems}, pp.\
  6884--6893.

\bibitem[\protect\citeauthoryear{Lin}{Lin}{2000}]{lin2000tensor}
Lin, Y. (2000).
\newblock Tensor product space {ANOVA} models.
\newblock {\em Ann. Statist.\/}~{\em 28\/}(3), 734--755.

\bibitem[\protect\citeauthoryear{Maathuis}{Maathuis}{2005}]{maathuis2005reduction}
Maathuis, M.~H. (2005).
\newblock Reduction algorithm for the npmle for the distribution function of
  bivariate interval-censored data.
\newblock {\em Journal of Computational and Graphical Statistics\/}~{\em
  14\/}(2), 352--362.

\bibitem[\protect\citeauthoryear{Mammen and van~de Geer}{Mammen and van~de
  Geer}{1997}]{mammen1997locally}
Mammen, E. and S.~van~de Geer (1997).
\newblock Locally adaptive regression splines.
\newblock {\em Ann. Statist.\/}~{\em 25\/}(1), 387--413.

\bibitem[\protect\citeauthoryear{Massart}{Massart}{2007}]{Massart03Flour}
Massart, P. (2007).
\newblock {\em Concentration inequalities and model selection. Lecture notes in
  Mathematics}, Volume 1896.
\newblock Berlin: Springer.

\bibitem[\protect\citeauthoryear{Meyer and Woodroofe}{Meyer and
  Woodroofe}{2000}]{MW00}
Meyer, M. and M.~Woodroofe (2000).
\newblock On the degrees of freedom in shape-restricted regression.
\newblock {\em Ann. Statist.\/}~{\em 28\/}(4), 1083--1104.

\bibitem[\protect\citeauthoryear{Nemirovski}{Nemirovski}{2000}]{nemirovski2000}
Nemirovski, A. (2000).
\newblock Topics in non-parametric statistics.
\newblock In {\em Lectures on probability theory and statistics
  ({S}aint-{F}lour, 1998)}, Volume 1738 of {\em Lecture Notes in Math.}, pp.\
  85--277. Springer, Berlin.

\bibitem[\protect\citeauthoryear{Niyogi and Girosi}{Niyogi and
  Girosi}{1999}]{niyogi1999generalization}
Niyogi, P. and F.~Girosi (1999).
\newblock Generalization bounds for function approximation from scattered noisy
  data.
\newblock {\em Advances in Computational Mathematics\/}~{\em 10\/}(1), 51--80.

\bibitem[\protect\citeauthoryear{Ortelli and van~de Geer}{Ortelli and van~de
  Geer}{2018}]{ortelli2018total}
Ortelli, F. and S.~van~de Geer (2018).
\newblock On the total variation regularized estimator over the branched path
  graph.
\newblock {\em arXiv preprint arXiv:1806.01009\/}.

\bibitem[\protect\citeauthoryear{Ortelli and van~de Geer}{Ortelli and van~de
  Geer}{2019a}]{ortelli2019oracle}
Ortelli, F. and S.~van~de Geer (2019a).
\newblock Oracle inequalities for image denoising with total variation
  regularization.
\newblock {\em arXiv preprint arXiv:1911.07231\/}.

\bibitem[\protect\citeauthoryear{Ortelli and van~de Geer}{Ortelli and van~de
  Geer}{2019b}]{ortelli2019synthesis}
Ortelli, F. and S.~van~de Geer (2019b).
\newblock Synthesis and analysis in total variation regularization.
\newblock {\em arXiv preprint arXiv:1901.06418\/}.

\bibitem[\protect\citeauthoryear{Owen}{Owen}{2005}]{owen2005multidimensional}
Owen, A.~B. (2005).
\newblock Multidimensional variation for quasi-{M}onte {C}arlo.
\newblock In {\em Contemporary multivariate analysis and design of
  experiments}, Volume~2 of {\em Ser. Biostat.}, pp.\  49--74. World Sci.
  Publ., Hackensack, NJ.

\bibitem[\protect\citeauthoryear{Prause and Steland}{Prause and
  Steland}{2017}]{prause2017sequential}
Prause, A. and A.~Steland (2017).
\newblock Sequential detection of three-dimensional signals under dependent
  noise.
\newblock {\em Sequential Analysis\/}~{\em 36\/}(2), 151--178.

\bibitem[\protect\citeauthoryear{Robertson, Wright, and Dykstra}{Robertson
  et~al.}{1988}]{RWD88}
Robertson, T., F.~T. Wright, and R.~L. Dykstra (1988).
\newblock {\em Order restricted statistical inference}.
\newblock Wiley Series in Probability and Mathematical Statistics: Probability
  and Mathematical Statistics. John Wiley \& Sons, Ltd., Chichester.

\bibitem[\protect\citeauthoryear{Rudin, Osher, and Fatemi}{Rudin
  et~al.}{1992}]{rudin1992nonlinear}
Rudin, L.~I., S.~Osher, and E.~Fatemi (1992).
\newblock Nonlinear total variation based noise removal algorithms.
\newblock {\em Phys. D\/}~{\em 60\/}(1-4), 259--268.
\newblock Experimental mathematics: computational issues in nonlinear science
  (Los Alamos, NM, 1991).

\bibitem[\protect\citeauthoryear{Ruiz, Li, and Munk}{Ruiz
  et~al.}{2018}]{ruiz2018frame}
Ruiz, M. d.~{\'A}., H.~Li, and A.~Munk (2018).
\newblock Frame-constrained total variation regularization for white noise
  regression.
\newblock {\em arXiv preprint arXiv:1807.02038\/}.

\bibitem[\protect\citeauthoryear{Sadhanala, Wang, Sharpnack, and
  Tibshirani}{Sadhanala et~al.}{2017}]{sadhanala2017higher}
Sadhanala, V., Y.-X. Wang, J.~L. Sharpnack, and R.~J. Tibshirani (2017).
\newblock Higher-order total variation classes on grids: Minimax theory and
  trend filtering methods.
\newblock In {\em Advances in Neural Information Processing Systems}, pp.\
  5800--5810.

\bibitem[\protect\citeauthoryear{Sadhanala, Wang, and Tibshirani}{Sadhanala
  et~al.}{2016}]{sadhanala2016total}
Sadhanala, V., Y.-X. Wang, and R.~J. Tibshirani (2016).
\newblock Total variation classes beyond 1d: Minimax rates, and the limitations
  of linear smoothers.
\newblock In {\em Advances in Neural Information Processing Systems}, pp.\
  3513--3521.

\bibitem[\protect\citeauthoryear{Stone}{Stone}{1982}]{stone1982optimal}
Stone, C.~J. (1982).
\newblock Optimal global rates of convergence for nonparametric regression.
\newblock {\em Ann. Statist.\/}~{\em 10\/}(4), 1040--1053.

\bibitem[\protect\citeauthoryear{Temlyakov}{Temlyakov}{2018}]{temlyakov2018multivariate}
Temlyakov, V. (2018).
\newblock {\em Multivariate approximation}, Volume~32.
\newblock Cambridge University Press.

\bibitem[\protect\citeauthoryear{Tibshirani}{Tibshirani}{2014}]{tibshirani2014adaptive}
Tibshirani, R.~J. (2014).
\newblock Adaptive piecewise polynomial estimation via trend filtering.
\newblock {\em Ann. Statist.\/}~{\em 42\/}(1), 285--323.

\bibitem[\protect\citeauthoryear{van~de Geer}{van~de
  Geer}{2000}]{VandegeerBook}
van~de Geer, S.~A. (2000).
\newblock {\em Applications of empirical process theory}, Volume~6 of {\em
  Cambridge Series in Statistical and Probabilistic Mathematics}.
\newblock Cambridge University Press, Cambridge.

\bibitem[\protect\citeauthoryear{van~der Laan}{van~der
  Laan}{2017a}]{van2017finite}
van~der Laan, M. (2017a).
\newblock Finite sample inference for targeted learning.
\newblock {\em arXiv preprint arXiv:1708.09502\/}.

\bibitem[\protect\citeauthoryear{van~der Laan}{van~der Laan}{2017b}]{MR3724476}
van~der Laan, M. (2017b).
\newblock A generally efficient targeted minimum loss based estimator based on
  the highly adaptive {L}asso.
\newblock {\em Int. J. Biostat.\/}~{\em 13\/}(2), 20150097, 35.

\bibitem[\protect\citeauthoryear{van~der Laan, Benkeser, and Cai}{van~der Laan
  et~al.}{2019}]{van2019efficient}
van~der Laan, M.~J., D.~Benkeser, and W.~Cai (2019).
\newblock Efficient estimation of pathwise differentiable target parameters
  with the undersmoothed highly adaptive lasso.
\newblock {\em arXiv preprint arXiv:1908.05607\/}.

\bibitem[\protect\citeauthoryear{van~der Laan and Bibaut}{van~der Laan and
  Bibaut}{2017}]{van2017uniform}
van~der Laan, M.~J. and A.~F. Bibaut (2017).
\newblock Uniform consistency of the highly adaptive lasso estimator of
  infinite dimensional parameters.
\newblock {\em arXiv preprint arXiv:1709.06256\/}.

\bibitem[\protect\citeauthoryear{van~der Vaart and Wellner}{van~der Vaart and
  Wellner}{1996}]{vaartwellner96book}
van~der Vaart, A.~W. and J.~A. Wellner (1996).
\newblock {\em Weak convergence and empirical processes}.
\newblock Springer Series in Statistics. Springer-Verlag, New York.
\newblock With applications to statistics.

\bibitem[\protect\citeauthoryear{Vapnik and Chervonenkis}{Vapnik and
  Chervonenkis}{2015}]{VapnikCervonenkis71events}
Vapnik, V.~N. and A.~Y. Chervonenkis (2015).
\newblock On the uniform convergence of relative frequencies of events to their
  probabilities.
\newblock pp.\  11--30.
\newblock Reprint of Theor. Probability Appl. {{\bf{1}}6} (1971), 264--280.

\bibitem[\protect\citeauthoryear{Wahba, Wang, Gu, Klein, and Klein}{Wahba
  et~al.}{1995}]{wahba1995smoothing}
Wahba, G., Y.~Wang, C.~Gu, R.~Klein, and B.~Klein (1995).
\newblock Smoothing spline {ANOVA} for exponential families, with application
  to the {W}isconsin {E}pidemiological {S}tudy of {D}iabetic {R}etinopathy.
\newblock {\em Ann. Statist.\/}~{\em 23\/}(6), 1865--1895.

\bibitem[\protect\citeauthoryear{Widder}{Widder}{1941}]{widder1941princeton}
Widder, D.~V. (1941).
\newblock {\em The {L}aplace {T}ransform}.
\newblock Princeton Mathematical Series, v. 6. Princeton University Press,
  Princeton, N. J.

\bibitem[\protect\citeauthoryear{Yang and Barron}{Yang and
  Barron}{1999}]{YangBarron}
Yang, Y. and A.~Barron (1999).
\newblock Information-theoretic determination of minimax rates of convergence.
\newblock {\em Ann. Statist.\/}~{\em 27\/}(5), 1564--1599.

\bibitem[\protect\citeauthoryear{Young and Young}{Young and
  Young}{1924}]{young1924discontinuities}
Young, W. and G.~C. Young (1924).
\newblock On the discontinuities of monotone functions of several variables.
\newblock {\em Proceedings of the London Mathematical Society\/}~{\em 2\/}(1),
  124--142.

\bibitem[\protect\citeauthoryear{Yu}{Yu}{1997}]{Yu97lecam}
Yu, B. (1997).
\newblock Assouad, {F}ano, and {Le Cam}.
\newblock In D.~Pollard, E.~Torgersen, and G.~L. Yang (Eds.), {\em Festschrift
  for Lucien Le~Cam: Research Papers in Probability and Statistics}, pp.\
  423--435. New York: Springer-Verlag.

\bibitem[\protect\citeauthoryear{Zhang}{Zhang}{2002}]{Zhang02}
Zhang, C.-H. (2002).
\newblock Risk bounds in isotonic regression.
\newblock {\em Ann. Statist.\/}~{\em 30\/}(2), 528--555.

\bibitem[\protect\citeauthoryear{Zhang}{Zhang}{2019}]{zhang2019element}
Zhang, T. (2019).
\newblock Element-wise estimation error of a total variation regularized
  estimator for change point detection.
\newblock {\em arXiv preprint arXiv:1901.00914\/}.

\bibitem[\protect\citeauthoryear{Ziemer}{Ziemer}{2012}]{ziemer2012weakly}
Ziemer, W.~P. (2012).
\newblock {\em Weakly differentiable functions: Sobolev spaces and functions of
  bounded variation}, Volume 120.
\newblock Springer Science \& Business Media.

\end{thebibliography}

\end{document}